\documentclass[10pt]{amsart}

\usepackage{subfiles}
\usepackage{xr}
\usepackage{cite}

\usepackage{rotating}

\usepackage{graphicx}
\usepackage{amsmath,amsfonts,amssymb}
\usepackage{color}
\usepackage{textcomp}
\usepackage{mathtools}
\usepackage{ytableau}
\usepackage{stmaryrd}
\usepackage{soul}
\usepackage{tikz-cd}
\usepackage[utf8]{inputenc}
\usepackage{cmap}
\usepackage[T1]{fontenc}
\usepackage{tikz,environ}
\usepackage{float}
\usetikzlibrary{patterns,snakes}
\usetikzlibrary{cd}
\usepackage{ dsfont }
\usetikzlibrary{arrows}
\usepackage{multicol}

\usepackage{etoolbox}

\usepackage{url}

\usepackage[bookmarks=true,
    colorlinks=true,
    linkcolor=blue,
    citecolor=blue,
    filecolor=blue,
    menucolor=blue,
    urlcolor=blue,
    breaklinks=true]{hyperref}

\usepackage{fullpage}

\input xy
\usepackage[all]{xy}
\xyoption{line}
\xyoption{arrow}
\xyoption{color}
\SelectTips{cm}{}

\usepackage{tikz}
\usetikzlibrary{decorations.markings}
\usetikzlibrary{decorations.pathreplacing}
\newcommand{\hackcenter}[1]{
 \xy (0,0)*{#1}; \endxy}

\tikzstyle directed=[postaction={decorate,decoration={markings,
    mark=at position #1 with {\arrow{>}}}}]
\tikzstyle rdirected=[postaction={decorate,decoration={markings,
    mark=at position #1 with {\arrow{<}}}}]

\usetikzlibrary{calc}
\usepackage{relsize}

\tikzset{fontscale/.style = {font=\relsize{#1}}
    }

\usetikzlibrary{decorations.pathmorphing}
\usetikzlibrary{decorations.text}

\tikzset{snake it/.style={decorate, decoration=snake}}

\usepackage{graphicx}
\usepackage{color}

\def\P{\mathsf{P}}
\def\Q{\mathsf{Q}}
\def\C{\mathsf{C}}
\def\B{\mathsf{B}}
\def\K{\mathcal{K}}
\def\d{\mathsf{d}}

\theoremstyle{plain}
\newtheorem{theorem}{Theorem}[section]
\newtheorem{corollary}[theorem]{Corollary}
\newtheorem{proposition}[theorem]{Proposition}
\newtheorem{lemma}[theorem]{Lemma}

\newtheorem{example}[theorem]{Example}

\theoremstyle{definition}
\newtheorem{definition}[theorem]{Definition}

\theoremstyle{remark}
\newtheorem{remark}[theorem]{Remark}

\numberwithin{equation}{section}

\renewcommand{\to}{\rightarrow}

\def\Ind{{\mathrm{Ind}}}

\newcommand\nc{\newcommand}
\nc\rnc{\renewcommand}
\nc\Kar{\operatorname{Kar}}
\nc\End{\operatorname{End}}

\newcommand{\SYT}{\ensuremath\mathrm{SYT}}

\newlength\cellsize 

\savebox2{%
\begin{picture}(10,10)
\put(0,0){\line(1,0){10}}
\put(0,0){\line(0,1){10}}
\put(10,0){\line(0,1){10}}
\put(0,10){\line(1,0){10}}
\end{picture}}

\newcommand\cellify[1]{\def\thearg{#1}\def\nothing{}%
\ifx\thearg\nothing\vrule width0pt height\cellsize depth0pt%
  \else\hbox to 0pt{\usebox2\hss}\fi%
  \vbox to 10\unitlength{\vss\hbox to 10\unitlength{\hss$#1$\hss}\vss}}

\newcommand\tableau[1]{\vtop{\let\\=\cr
\setlength\baselineskip{-10000pt}
\setlength\lineskiplimit{10000pt}
\setlength\lineskip{0pt}
\halign{&\cellify{##}\cr#1\crcr}}}

\let\tilde=\widetilde

\let\theta=\vartheta
\let\epsilon=\varepsilon

\usepackage{bbm}
\def\N{{\mathds N}}

\def\Z{{\mathds Z}}
\def\H{{\mathcal{H}}}
\def\k{{\mathds{k}}}

\def\1{\mathsf{1}}%
\allowdisplaybreaks

\usepackage{verbatim}

\newtoggle{details}

\toggletrue{details}   

\iftoggle{details}{%
  \newcommand{\details}[1]{
      \ \\
      {
        \textbf{Details:} #1
      }
      \\
  }
}{%
  \newcommand{\details}[1]{}
}

\begin{document}


\title[Categorical Bernstein Operators and the Boson-Fermion Correspondence]{Categorical Bernstein Operators and the Boson-Fermion Correspondence}  

\author[Gonz\'{a}lez]{Nicolle E. S. Gonz\'{a}lez}
\address{Department of Mathematics, University of Southern California, 3620 S. Vermont Ave., Los Angeles, CA 90089-2532, U.S.A.}
\email{nesandov@usc.edu}


\date{\today}

\begin{abstract}
We prove a conjecture of Cautis and Sussan providing a categorification of the Boson-Fermion correspondence as formulated by Frenkel and Kac. We lift the Bernstein operators to infinite chain complexes in Khovanov's Heisenberg category $\H$ and from them construct categorical analogues of the Kac-Frenkel fermionic vertex operators. These fermionic functors are then shown to satisfy categorical Clifford algebra relations, solving a conjecture of Cautis and Sussan. We also prove another conjecture of Cautis and Sussan demonstrating that the categorical Fock space representation of $\H$ is a direct summand of the regular representation by showing that certain infinite chain complexes are categorical Fock space idempotents. In the process, we enhance the graphical calculus of $\H$ by lifting various Littlewood-Richardson branching isomorphisms to the Karoubian envelope of $\H$.
\end{abstract}

\dedicatory{ A mi pap\'{a} - qui\'{e}n siempre me ha ense\~{n}ado lo m\'{a}s importante... \\ “Hay que inyectarse cada día de fantasía para no morir de realidad”}


\maketitle

\setcounter{tocdepth}{1} 
\tableofcontents 

\section{Introduction}\label{sec:Introduction}

Building upon his ideas to categorify quantum groups and associated TQFTs, Igor Frenkel conjectured that conformal field theory could be categorified as well. As a starting point, he suggested that the Boson-Fermion correspondence should admit a categorification. One of the first breakthroughs came when Khovanov categorified the simplest Heisenberg algebra (see Section \ref{sec:ExtGraphicalCalc},\cite{Kh-H}). Shortly after, Cautis and Licata \cite{CL-Vertex} defined Heisenberg categories associated to every affine Dynkin diagram. Using this framework, they defined categorical vertex operators lifting the homogeneous realization of the basic representation of quantum affine algebras within the Frenkel-Kac-Segal construction \cite{Frenkel-Kac, Segal}. Specifically, they lifted vertex operators to infinite chain complexes in a Heisenberg category, inducing 2-representations of quantum affine algebras. The Boson-Fermion correspondence as formulated by Frenkel \cite{Frenkel-BF} and Kac \cite{Kac1} is analogous to the Frenkel-Kac-Segal construction in that it also defines vertex operators from a Heisenberg algebra. However, in the Boson-Fermion correspondence these vertex operators induce an action of the Clifford algebra instead of the quantum affine algebra. Motivated by this, Cautis and Sussan \cite{CS-BosonFermion} conjectured a categorical formulation of the Boson-Fermion correspondence within the framework of Khovanov's Heisenberg category $\H$ \cite{Kh-H}. One of our main results is the proof of this conjecture (see Theorem \ref{thm:catBF}).

Our proof relies on a family of chain complexes in the homotopy category of $\H$ which lift the \emph{Bernstein operators} introduced by Zelevinsky \cite{Zelevinsky}. These operators create and annihilate Schur functions when acting on the ring of symmetric functions \cite{MacDonald}. We show that these \emph{categorical Bernstein operators} create and annihilate Specht modules when acting on categorical Fock space.  

We also prove a series of Cautis-Sussan conjectures  \cite{CS-BosonFermion} regarding certain unbounded chain complexes analogous to those defined in \cite{CLS-BraidAction}. In that paper, Cautis, Licata, and Sussan prove that these functors satisfy braid group relations, thus showing that integrable 2-representations of the Heisenberg algebra always induce braid group actions. In this article we prove that the analogous complexes defined by Cautis and Sussan are Fock space idempotents (see Section \ref{sec:SigmaComplexes}). This is in line with the decategorified picture where any $\mathfrak{h}$-module generated by highest weight vectors decomposes into a direct sum of infinitely many copies of the trivial module. These functors are categorical analogues of orthogonal projectors onto a summand of Fock space. 

\subsection{The Boson-Fermion correspondence and its related algebras}
\emph{Dirac's sea of electrons} is a theoretical model for describing a system of fermionic quantum particles with infinitely many energy states. In this model energy states are indexed by half integers and are assumed to be entirely occupied below and unoccupied above certain values. By the Pauli exclusion principle no two particles may occupy the same energy state. Mathematically this is captured by the following construction. Let $v_i$ denote a particle in energy level $i$ with $i \in \Z+\frac{1}{2}$ and let $\mathbbm{k}$ be a field of characteristic zero. If a vector $v_{i_m} \wedge v_{i_{m-1}} \wedge \dots$ with $i_m>i_{m-1}>\dots$ denotes a system where only the energy levels $i_m, i_{m-1}, \dots$ are filled, then a \emph{semi-infinite monomial} is a vector of this form satisfying $i_s=i_{s-1}+1$ for all $s\ll 0$. \emph{Fermionic Fock space} is the $\mathbbm{k}$-linear span of all semi-infinite monomials (it is also known as the \emph{semi-infinite wedge} $\Lambda^{\frac{\infty}{2}}$)\cite{Stern,Rios-Zert,Tingley}. This space carries a natural grading by \emph{charge} known as the \emph{principal gradation}. On the other hand, \emph{Bosonic Fock space} $\mathcal{B}$ is defined as a graded sum of infinitely many copies of the ring of symmetric functions $Sym$. Moreover, because Schur functions are a basis for $Sym$ and are indexed by partitions, then $\mathcal{B}$ has a natural basis indexed by charged partitions. As is the case for $\mathcal{F}$, the charge in $\mathcal{B}$ keeps track of the grading. Since to each half infinite monomial there is a canonical way of assigning a charged partition, it can be shown that Fermionic and Bosonic Fock space are isomorphic as graded vector spaces. Thus, Bosonic and Fermionic Fock spaces may be identified \cite{Tingley, Kac1, Rios-Zert}. This isomorphism is the \emph{first part of the Boson-Fermion correspondence }(see Theorem \ref{thm:FockSpaceIso}).

The second part of the Boson-Fermion correspondence concerns the action of the infinite dimensional Clifford and Heisenberg algebras on Fock space. In particular, Fermionic Fock space carries a natural action the infinite dimensional \emph{Clifford algebra} $\mathcal{CL}$ given by a family of creation and annihilation operators $\psi_i$ and $\psi_i^*$ which satisfy the anti-commutation relations
\begin{equation}\label{eq:CliffRelIntro} \psi_i \psi_j^* + \psi_j^* \psi_i = \delta_{i,j} \qquad \psi_i \psi_j + \psi_j \psi_i = 0 \qquad \psi_i^* \psi_j^* + \psi_j^* \psi_1^* = 0.
 \end{equation}
Likewise, Bosonic Fock space has an action of the infinitely generated \emph{Heisenberg algebra} $\mathfrak{h}$, whose generators $\alpha_n$ and $\alpha_{-n}$ with $n \in \Z_{\geq 0}$ satisfy the commutation relations
\[\alpha_n \alpha_{m} - \alpha_{m}\alpha_n = n\delta_{n,-m}.\]
As in \cite{Kh-H, CL-Heis}, we will instead work with alternate generators for $\mathfrak{h}$ denoted by $p^\lambda$ and $q^\lambda$ where $\lambda$ ranges over all row and column partitions. Given the isomorphism between Bosonic and Fermionic Fock space, there is an induced action of the Clifford algebra on $\mathcal{B}$. The \emph{second part of the Boson-Fermion correspondence} states that under this isomorphism the action of the operators $\psi_i$ and $\psi_i^*$ on $\mathcal{B}$ can be expressed in terms of the generators of the Heisenberg algebra \cite{Kac1}. More specifically, identifying $\mathcal{B}$ and $\mathcal{F}$ under this isomorphism and denoting by $\psi_i$ and $\psi_i$ the images of the fermionic operators on $\mathcal{B}$, the formal power series $\psi(z) := \sum_{i \in \Z} \psi_i z^i$ and $\psi_i^*(z):=\sum_{i \in \Z} \psi_i^* z^{-i}$ satisfy the following equalities {\cite[Theorem 14.10]{Kac1}}: 
\begin{align} \psi(z) &= z^{\alpha_0} t \;\; \text{exp}\left( \sum_{m\geq 1} \frac{z^{m}}{m}\alpha_{-m}\right)\text{exp}\left(- \sum_{m\geq 1} \frac{z^{-m}}{m}\alpha_m\right)\label{eq:introPsi}\\
\psi^*(z) &= t^{-1}z^{-\alpha_0} \;\text{exp}\left( -\sum_{m\geq 1} \frac{z^{m}}{m}\alpha_{-m}\right)\text{exp}\left( \sum_{m\geq 1} \frac{z^{-m}}{m}\alpha_m\right)\label{eq:introPsi*}.
\end{align}
In particular, the correspondence allows us to express the induced action of $\mathcal{CL}$ on $\mathcal{B}$ in terms of the action of $\mathfrak{h}$. Operators of the form in \eqref{eq:introPsi} and \eqref{eq:introPsi*} are known as \emph{vertex operators} \cite{Frenkel2000, Frenkel-Kac,Frenkel-BF, Jing-VOA-Spin, Jing-VOA-HL}. They form an integral part in the study of conformal field theories and many other areas of mathematical physics \cite{Miwa-Jimbo-Date, Gab-Intro,Gab-VOA, VOAintro}. The process of constructing the \emph{Fermionic fields} $\psi(z)$ and $\psi^*(z)$ from the \emph{Bosonic generators} $\alpha_n, \alpha_{-n}$ is known as \emph{fermionization}. Since the isomorphism goes both ways, the Heisenberg algebra also acts on $\mathcal{F}$. The reverse process which expresses the action of the Heisenberg algebra on Fermionic Fock space as formal power series in terms of the generators of $\mathcal{CL}$ is known as \emph{bosonization}. It entails the other side of the Boson-Fermion correspondence \cite{Kac1}.

\subsection{The Boson-Fermion Correspondence Categorified}Geissinger \cite{Geissinger} showed that the Hopf algebra of symmetric functions $Sym$ could be categorified by modules over $\mathbbm{k}[S_n]$ for any field $\mathbbm{k}$ of characteristic zero. If $\K_0$ denotes the split Grothendiek ring of the category, then $Sym \cong \bigoplus_{n \geq 0} \K_0(\mathds{k}[S_n]$-{mod}) via the Frobenius character map. In this isomorphism, the algebra and coalgebra structure of $Sym$ correspond to induction and restriction functors on $\mathds{k}[S_n]$-{mod}. 

Building upon these ideas, Khovanov constructed the first categorification of the Heisenberg algebra $\mathfrak{h}$ \cite{Kh-H}. \emph{Khovanov's Heisenberg category} $\H$ is a diagrammatic, additive, monoidal category with generating objects $\P$ and $\Q$ and whose action on $\bigoplus_n \mathbbm{k}[S_n]$-mod is given by induction and restriction functors (see Section \ref{sec:ExtGraphicalCalc}). When taking the Grothendieck group, the generating objects of $\H$ descend to generators of $\mathfrak{h}$ given by $p$ and $q$. Since the Hopf algebra of symmetric functions is categorified by $\bigoplus_n \mathbbm{k}[S_n]$-mod and $\mathcal{B}$ consists of infinitely many copies of $Sym$, it is natural to define \emph{categorical Fock space} $\mathsf{V}_{Fock}$ as the direct sum of copies of $\bigoplus_n \mathbbm{k}[S_n]$-mod indexed by integers. We will encode this in terms of infinite vectors over $\bigoplus_n \mathbbm{k}[S_n]$-mod with finitely many nonzero entries. 

Categorification by nature produces integral structures when decategorified. If an algebra $\mathcal{A}$ arises as the Grothendieck ring of some (reasonable) category, it will usually inherit an integral basis given by the isomorphism classes of indecomposable objects in $\mathcal{C}$. In the Heisenberg algebra, there is an identification between its generators and symmetric functions which precisely exemplifies this fact. When acting on $Sym$, the action of $p$ and $q$ is given by multiplication by elementary and complete symmetric functions, both of which are generators over $\Z$ for $Sym$. However, the action of the generators $\alpha_n$ and $\alpha_{-n}$ is through power sum symmetric functions which generate $Sym$ over $\mathbb{Q}$. This is the fundamental reason behind the use of the $p$ and $q$ generators for $\mathfrak{h}$ in this paper and why, in Khovanov's formulation, these are the generators that lift to indecomposable objects in $\H$. Under this identification, the vertex operators $\psi$ and $\psi^*$ from Kac's formulation in equations \eqref{eq:introPsi} and \eqref{eq:introPsi*} can be rewritten as infinite series in terms of $p$ and $q$. Consequently, the action of the Fermionic fields on $\mathcal{B}$ is given by shifted products of the \emph{Bernstein operators} $B_a$ and $B^*_a$ (see Section \ref{subsec:BFcorrespondence}, \cite{MacDonald, Zelevinsky}).

Cautis and Sussan conjectured a categorical Boson-Fermion correspondence where the generators of the Clifford algebra, when realized as vertex operators, lift to infinite complexes in the homotopy category of Khovanov's Heisenberg category $\K(\H)$ \cite{CS-BosonFermion}. They proposed that these functors satisfy categorical analogues of the Clifford algebra relations \eqref{eq:CliffordRelations} up to homotopy. In this paper we prove a slightly modified but equivalent formulation of their conjecture. 

To this effect, we introduce the \emph{categorical Bernstein operators}, indexed by $a \in \Z$, as the following unbounded complexes in $\K(\H)$ \footnote{The complexes $\B_a, \;\B_a^*$ are equivalent to the complexes $\mathsf{C}_a, \mathsf{C}_a^*$ originally defined by Cautis and Sussan via the involution $a\mapsto -a$ and an internal grading shift of $[a]$.}:
\begin{align*}
\mathsf{B}_a&:= \cdots \to \P^{(x+a)}\Q^{(1^x)}\to \P^{(x+a-1)}\Q^{(1^{x-1})} \to \cdots &\in \K^-(\H)\\
\mathsf{B}^*_a &:= \cdots \to \P^{(1^{y})}\Q^{(y+a)} \to \P^{(1^{y+1})}\Q^{(y+a+1)} \to \cdots &\in \K^+(\H)
\end{align*}

\noindent whose differentials are given by adjunction maps. When taking the Grothendieck ring these functors recover the original Bernstein operators and their properties. In particular, the categorical Bernstein operators have a natural action on $\bigoplus_n \mathbbm{k}[S_n]$-mod induced by the action of $\H$. In Theorems \ref{thm:BernsteinOps1}, \ref{thm:BernsteinOps2}, \ref{thm:CC*=C*C}, and \ref{thm:CC*distingishedTriangle} we show these functors satisfy categorical lifts of the commutation relations fulfilled by the Bernstein operators $B_a$ and $B_a^*$  \cite[Section 1.4]{MacDonald}. In particular, we prove the following theorem (see Theorem \ref{thm:BernsteinSpecht}). 

\begin{theorem}
The categorical Bernstein operators are creation and annihilation functors for Specht modules. That is, for any Specht module $S_\lambda$ associated to partition $\lambda=(\lambda_k, \dots, \lambda_1)\vdash n$ with $\lambda_k\geq \dots \geq \lambda_1 \geq~0$ and $\k$ the trivial module over $\k[S_n]$-mod, then $\B_{\lambda_k}\B_{\lambda_{k-1}}\dots\B_{\lambda_1}(\mathds{k})\simeq S_\lambda$ and $\B^*_{\lambda_1} \dots \B^*_{\lambda_{k-1}}\B^*_{\lambda_k}(S_\lambda)\simeq \mathds{k}$.
\end{theorem}

As evidenced by equations \eqref{eq:introPsi} and \eqref{eq:introPsi*}, the action of the vertex operators on $\mathcal{B}$ does not preserve the charge but instead expresses $\psi_i$ and $\psi_i^*$ in terms of Bernstein operators and a charge function $z^{\alpha_0}$. In their original conjecture, Cautis and Sussan account for the charge operator by incorporating an additional variable $q$ to keep track of the principal gradation. The approach taken here is different. 
Denote by Mat$_{\Z\times\Z}(\K(\H))$ the category of unbounded matrices with certain row-finiteness conditions whose entries are chain complexes in $\K(\H)$ (see Definition \ref{def:MatrixCat}). Within this framework, we define the \emph{charge operators} $\mathcal{Q}, \mathcal{Q}^{-1}$ in Mat$_{\Z\times\Z}(\K(\H))$ that act on $\mathsf{V}_{Fock}$ by raising and lowering the charge of each entry in a vector (see Section \ref{subsec:CatBFcorrespondence}). Employing these charge operators and infinite matrices of categorical Bernstein operators we introduce categorical analogues of the Fermionic vertex operators given in \eqref{eq:introPsi} and \eqref{eq:introPsi*}. The \emph{Fermionic functors} $\Psi_i$ and $\Psi_i^*$ are defined as unbounded matrices in Mat$_{\Z\times\Z}(\K(\H))$ whose entries consist of $\B_a$ and $\B^*_a$ for varying values of $a \in \Z$ and whose action on $\mathsf{V}_{Fock}$ is given by matrix multiplication. This differs from the approach taken by Cautis and Sussan since they defined the action of $\Psi$ and $\Psi^*$ pointwise on each $v \in \bigoplus \mathbbm{k}[S_n]$-mod. In Theorem \ref{thm:catBF} we prove these categorical vertex operators $\Psi_i$ and $\Psi_i^*$ satisfy categorical analogues of the Clifford algebra relations \eqref{eq:CliffRelIntro} and thus lift the Kac-Frenkel formulation of the Boson-Fermion correspondence. 

\begin{theorem}[categorical Boson-Fermion correspondence]\label{thm:catBFintro}
The Fermionic functors $\Psi_i$ and $\Psi_i^*$ satisfy the following relations in \emph{Mat}$_{\Z\times\Z}(\K(\H))$:
\vspace{2mm}

\begin{enumerate}
\item  $(\Psi_i)^2 \cong 0$ \qquad and \qquad  $\Psi_i \Psi_j \cong \begin{cases} 
\Psi_j \Psi_i [-1] &\text{ if } i > j \\
\Psi_j \Psi_i [1] & \text{ if } i < j \end{cases}$
\item  $(\Psi_i^*)^2 \cong 0$ \qquad and \qquad 
$\Psi^*_i \Psi^*_j \cong 
\begin{cases}
\Psi^*_j \Psi^*_i [-1] &\text{ if } i > j \\ \Psi^*_j \Psi^*_i [1] & \text{ if } i < j \end{cases}$
\end{enumerate}
Moreover, the following relations also hold in $\emph{Mat}_{\Z\times\Z}(\K^b(\H))$:
\begin{itemize}
\item[(3)]  $\Psi_i \Psi_j^* \cong \begin{cases}
 \Psi_j^* \Psi_i [1] & \text{ if } i > j \\ \Psi_j^* \Psi_i [-1] & \text{ if } i < j \end{cases}$
\item[(4)]  there exist distinguished triangles $\Psi_i \Psi_i^* \rightarrow \mathbbm{1} \rightarrow \Psi_i^* \Psi_i$ and $ \Psi_i^* \Psi_i  \rightarrow \mathbbm{1} \rightarrow \Psi_i \Psi_i^*$.
\end{itemize}
\end{theorem}
In the original theorem by Kac, the Boson-Fermion correspondence is proven by letting the vertex operators act on Fock space. Since the action is integrable the infinite series become finite. In our categorification this additional finiteness constraint is not needed for proving relations (1) and (2). These isomorphisms hold between unbounded complexes in the homotopy category of $\H$ with no finiteness conditions imposed at all. Consequently, our result is stronger in the sense that these relations hold in full generality for the functors in Mat$_{\Z \times \Z}(\K(\H))$. This is because relation (1) only involves complexes that are bounded below and relation (2) only involves complexes that are bounded above. These boundedness conditions are lacking in relations (3) and (4) and cause very serious convergence issues. Specifically, $\mathsf{B}_a$ is bounded below whereas $\mathsf{B}_a^*$ is bounded above, so relations (3) and (4) consist of homotopy equivalences between chain complexes which are unbounded in both directions. Such complexes do not generally behave well with the usual methods from homological algebra. Things are further complicated by Khovanov's Heisenberg category being neither graded nor Krull-Schmidt. To overcome this obstruction, we prove isomorphisms (3) and (4) in the context of an arbitrary finite quotient of $\H$ (see Definition \ref{def:Hquotient}). This allows us to work over the homotopy category of bounded complexes $\K^b(\H)$, effectively eliminating all convergence issues. Given that the action of $\Psi_i$ and $\Psi_i^*$ on $\mathsf{V}_{Fock}$ is integrable, when the construction is decategorified this additional constraint simply recovers the condition that the correspondence holds when acting on Fock space. 

To prove Theorem~\ref{thm:catBFintro} we augment Khovanov's diagrammatic framework to include the images of Young symmetrizers in the Karoubi envelope and construct explicit diagrammatic isomorphisms realizing certain Littlewood-Richardson decompositions (see Section \ref{subsec:LRBranchingIsos}).  We hope that these isomorphisms are the beginning of a thick calculus, or complete diagrammatic description of the Karoubi envelope of Khovanov's Heisenberg category.  This is of independent interest since the objects of this Karoubi envelope are natural constructions on induction and restriction functors between modules of the symmetric group.  

We also make extensive use of some homological algebra tools for infinite complexes very recently described by Elias and Hogancamp \cite{EH-CatDiag}. As was previously noted, manipulating unbounded complexes introduces many technical complications and subtleties that are nonexistent when the chain complexes are finite. In order to derive the desired homotopy equivalences we work with two different notions of the tensor product of complexes normally absent in usual study of homotopy categories.

With these homological techniques we prove a series of commutation relations for the Bernstein functors $\B_a$ and $\B_a^*$, and from those derive the desired categorical Clifford relations stated in Theorem \ref{thm:catBFintro}. Although we give explicit isomorphisms lifting the Clifford relations, we postpone the study of the Hom spaces between the Clifford generators for future work. It is our hope that this leads to a direct diagrammatic categorification of an infinite dimensional Clifford algebra.  

In recent related work Frenkel, Penkov, and Serganova propose a categorification of the Boson-Fermion correspondence via the representation theory of $sl(\infty)$\cite{FPS}. In their paper they introduce certain chain complexes over the category of tensor modules and prove that when passing to the Grothendieck ring these complexes satisfy the Clifford algebra anticommutation relations. While we anticipate their complexes and ours will be related, the precise connection is unclear. On the one hand, their result is only at the level of the Grothendieck ring. On the other hand, their category not semisimple and their complexes are not biadjoint. Thus their functors can only be defined in one direction.

In another related work \cite{Tian-BF} Tian explores the correspondence from the reverse direction, that is of \emph{bosonization}. His approach uses contact geometry and a modification of Khovanov's Heisenberg category. In this context, he constructs functors in terms of generators of a certain Clifford category and proves they satisfy the commutation relations of a Heisenberg algebra. Once again, his categorical Fock space is not semisimple and his construction does not admit biadjointness. Nonetheless, given its connection to Khovanov's original work, one can surmise a relation with ours.

The structure of the paper is as follows. In Section \ref{sec:DecategorifiedStory} we explain the decategorified picture and define the Heisenberg and Clifford algebras. We also present the vertex operator formulation of the Boson-Fermion correspondence and explain its relation to symmetric functions and Bernstein operators. In Section \ref{sec:HomAlgebra} we establish some notation and describe the tools from homological algebra necessary to manipulate infinite complexes. Section \ref{sec:ExtGraphicalCalc} provides a summary of the diagrammatics of Khovanov's Heisenberg category and expands the calculus by recalling explicit expressions of generalized Young idempotents. This mainly consists of constructing various Littlewood-Richardson branching isomorphisms within the Karoubian envelope of $\H$. The entire categorified construction of the Boson-Fermion correspondence is explained in Section \ref{sec:CatBFCorrespondence}.
This includes the categorical Bernstein operators, their relations, and the fermionic functors $\Psi_i$ and $\Psi_i^*$.
Due to their technical complexity, many of the proofs for the relations of the categorical Bernstein operators are presented in Section \ref{sec:PropCatBernsteinOps}. Finally, in Section \ref{sec:SigmaComplexes} we prove that certain unbounded chain complexes, $\Sigma^\pm$ in $\K(\H)$, satisfy particular properties conjectured by Cautis and Sussan and are thus categorified Fock space idempotents.

\subsection*{Acknowledgments}
I am deeply grateful to Matt Hogancamp for innumerable conversations and my advisor Aaron Lauda for his encouragement, support, and guidance throughout this entire project. Sincere thanks also go to Tony Licata and Josh Sussan for beneficial discussions and Mikhail Khovanov, Alistair Savage, Yin Tian, and Geordie Williamson for detailed suggestions on earlier versions of this work. The author was partially supported by NSF grants DMS-1255334 and DMS-1664240.

\section{Fermionic Vertex Operators and the Boson-Fermion Correspondence}\label{sec:DecategorifiedStory}

Henceforth let $\mathds{k}$ be a field of characteristic zero. In this section we loosely follow \cite[Chapter 14]{Kac1} and refer the reader to this source and \cite{Tingley} for a more detailed exposition of these topics.  

\subsection{Fock Space}\label{subsec:FockSpace}

Given an infinite sequence of integers $i_n \in \Z +\frac{1}{2}$ such that $i_1 > i_2 > \cdots$ and $i_n= i_{n-1}-1$ for $n$ large enough, a vector of the form $v=i_1 \wedge i_2 \wedge i_3 \wedge \cdots$ is called a \emph{semi-infinite monomial}. The $\mathds{k}$-linear span of all such vectors is known as the \emph{semi-infinite wedge} $\Lambda^{\frac{\infty}{2}}$. For any $v \in\Lambda^{\frac{\infty}{2}}$, let $N^+=\#\lbrace i>0| i \wedge v =0 \rbrace$ and $N^-=\#\lbrace i<0| i \wedge v \neq0 \rbrace$. Set the \emph{charge} $c$ of $v$ to be $c:=N^+-N^-$. Then any semi-infinite monomial has a unique charge and thus for any fixed $c\in \Z$, there is a vector $v_0 =\left(c-\frac{1}{2}\right) \wedge \left(c-\frac{3}{2}\right) \wedge \left(c-\frac{5}{2}\right) \wedge \dots$, referred to as the \emph{vacuum vector of charge} $c$.

\emph{Fermionic Fock space} $\mathcal{F}$ is defined as the semi-infinite wedge. This space carries a natural charge decomposition $\mathcal{F}= \bigoplus_{c \in \Z} \mathcal{F}_c$ known as the \emph{principal gradation}. We also define \emph{Bosonic Fock space} as the ring $\mathcal{B} = \mathds{k}[p_1,p_2,\cdots; t,t^{-1}]$.  This space also admits a principal gradation $\mathcal{B}= \bigoplus_{c \in \Z} \mathcal{B}_c$ where each $\mathcal{B}_c \cong \mathds{k}[p_1, p_2, \dots]t^c$. By identifying each $p_n$ with the $n^{th}$ power sum symmetric function then each $\mathcal{B}_c$ is isomorphic to the ring of symmetric functions. 

To any vector $v$ of charge $c$ in $\mathcal{F}_c$ we can assign a partition $\lambda(v)=(\lambda_1, \lambda_2, \dots,\lambda_n)$ given by $\lambda_j=i_j-c+j-\frac{1}{2}$. The ring of symmetric functions has a $\Z$-basis given by Schur functions. These functions are indexed by partitions which allows us to identify $\mathcal{B}_c$ with $\mathcal{F}_c$ for each $c \in \Z$. Since this holds for all charges, this bijection yields the first part of the \emph{Boson-Fermion correspondence} \cite[Section 14]{Kac1}. 

\begin{theorem}[Boson-Fermion correspondence part 1, {\cite[Section 14.10]{Kac1}}]\label{thm:FockSpaceIso}
For any fixed $c \in \Z$, the map $\sigma: \mathcal{F}_c \to \mathcal{B}_{c}$ given by sending a vector $v$ of charge $c$ in $\mathcal{F}_c$ to $t^c s_{\lambda(v)}$ in $\mathcal{B}_c$ is a grading preserving isomorphism of vector spaces. Thus, $\mathcal{F} \cong \mathcal{B}$.
\end{theorem}

\subsection{Heisenberg Algebra}\label{subsec:HeisAlg}

The infinitely generated Heisenberg algebra is a central player in the representation theory of affine Lie algebras and quantum field theories. To any $\Z$-lattice one can associate a Heisenberg algebra; in this article we focus on the case of $\Z$ with pairing $\langle 1,1 \rangle =1$. 

\begin{definition}\label{def:heisenberg-alg}
The Heisenberg algebra $\mathfrak{h}$ is the associative, unital, $\mathds{k}$-algebra with $\Z$-basis $p^{(m)}, q^{(m)}$ for $m\in \N_+$ satisfying the commutation relations:
\begin{itemize}
\item $p^{(n)}p^{(m)}=p^{(m)}p^{(n)}$
\item $q^{(n)}q^{(m)}=q^{(m)}q^{(n)}$
\item $q^{(n)}p^{(m)} = \sum_{k\geq 0} p^{(m-k)}q^{(n-k)}.$ 
\end{itemize}
\end{definition}
Alternatively, $\mathfrak{h}$ could also be presented by the generators $ p^{(1^m)}$ and $ q^{(1^m)}$ for $m\in \N$, satisfying analogous relations to those above. 

\begin{proposition}[{\cite[Sec. 2.1]{Kh-H}} ]\label{prop:heis-relations} The following relations hold in $\mathfrak{h}$:
\begin{equation}
q^{(n)}p^{(1^m)} = p^{(1^m)}q^{(n)} + p^{(1^{m-1})}q^{(n-1)} \;\; and \;\; q^{(1^n)}p^{(m)} = p^{(m)}q^{(1^n)} + p^{({m-1})}q^{(1^{n-1})}. 
\end{equation}
\end{proposition}

If we denote by $\mathfrak{h}^+$ and $\mathfrak{h}^-$ the unital subalgebras of $\mathfrak{h}$ generated by the $q^{(n)}$ and $p^{(n)}$ for $n\in \N_+$ respectively, then there is an isomorphism of vector spaces $\mathfrak{h}\cong \mathfrak{h}^+ \otimes \mathfrak{h}^-$. Furthermore if ${\mathbbm{1}}$ denotes the trivial representation of $\mathfrak{h}^+$, then the induced representation $\Ind_{\mathfrak{h}^+}^\mathfrak{h}({\mathbbm{1}})$ is the one dimensional irreducible representation of $\mathfrak{h}$ known as the \emph{Fock space representation}. This representation is faithful and isomorphic as a vector space to the ring of symmetric functions.

\begin{remark} Those familiar with Heisenberg algebras will notice the previous presentations differ from the usual definition. More generally, the Heisenberg algebra shows up as a collection of oscillators in the following way. Define the \emph{oscillator algebra} as the algebra generated by $\alpha_n, n \in \Z$ and central element $\hbar \in \k$ satisfying the commutation relations: 
\begin{equation}
\alpha_n \alpha_{m} - \alpha_{m}\alpha_n = n\delta_{n,-m}\hbar \qquad \text{and} \qquad \hbar \alpha_n-\alpha_n \hbar=0.
\end{equation}
This algebra can be seen as an algebra of differential operators by defining its action on Bosonic Fock space as follows; For any $f=f(p_1, p_2, \dots)t^c \in \mathcal{B}_c$ set:
\begin{align}\label{eq:HeisGenAlpha}
\alpha_0\left(f\right)&:= \frac{\partial f}{\partial t} &
\alpha_{n}\left(f\right) &:= \frac{\partial f}{\partial p_n} &
\alpha_{-n}\left(f\right) &:=np_nf & \hbar(f):=\hbar f.
\end{align}
In particular, if $\hbar \neq 0$ this action yields an irreducible representation of $\mathcal{B}$. Moreover if we set $\hbar =1$ and consider $\alpha_n$ for $n \in \Z \setminus \lbrace 0 \rbrace$ then we recover the usual formulation for the Heisenberg algebra $\mathfrak{h}$. When seen as an algebra of differential operators, $\mathfrak{h}$ is often referred to as the \emph{Weyl algebra}. Thus, $\mathcal{B}$ is an infinite dimensional irreducible representation of $\mathfrak{h}$, equal to infinitely many copies of the one dimensional Fock space representation. 
\end{remark}

\begin{proposition}[c.f. \cite{Kh-H} ]\label{prop:HeisBasisRel}The following relations hold inside $\mathfrak{h}$:
\begin{align}\label{Heis:Pchangeofbasis}
\sum_{m \in \N} p^{(m)} z^{m} = \emph{exp}\left( \sum_{n\in\N} \frac{
z^n}{n}\alpha_{-n}\right) \;\; \text{and} \;\; 
\sum_{m \in \N} (-1)^{m} p^{(1^m)} z^{m} = \emph{exp}\left( -\sum_{n\in\N} \frac{
z^n}{n}\alpha_{-n}\right) 
\end{align}
\begin{align}\label{Heis:Qchangeofbasis}
\sum_{m \in \N} q^{(m)} z^{m} = \emph{exp}\left( \sum_{n\in\N} \frac{
z^n}{n}\alpha_{n}\right)\;\; \text{and} \;\; 
\sum_{m \in \N} (-1)^m q^{(1^m)} z^{m} = \emph{exp}\left( - \sum_{n\in\N} \frac{
z^n}{n}\alpha_{n}\right). 
\end{align}
\end{proposition}
\subsection{Clifford Algebra}\label{subsec:CliffAlg}
To any vector space $V$ over a field $\mathds{k}$ equipped with a quadratic form $Q(v)$ we can associate an infinite dimensional \emph{Clifford algebra} $\mathcal{CL}(V,Q) := T(V)/J$ where $J$ is its 2-sided ideal generated by terms $v\otimes v-Q(v)1$. Specifically, given operators $\psi_i$ and $\psi_i^*$ for $i\in \Z$, we are interested in the Clifford algebra associated to the vector space $V = \bigoplus_i \mathds{k} \psi_i + \bigoplus_i \mathds{k}\psi_i^*$ with $Q(\psi_i+\psi_j^*)=\delta_{i,j}1$ and zero otherwise \footnote{For completeness, we mention that its unique irreducible module $F_U$, i.e its spin module, is given by the maximal isotropic subspace $U=\bigoplus_{i \leq 0} \mathds{k} \psi_i + \bigoplus_{i>0} \mathds{k}\psi_i^*$ with the property that $U(\overline{0})=0$. This fact will not be used here.}.

\begin{definition}\label{def:clifford-alg} The \emph{Clifford algebra} $\mathcal{CL}$ is the associative, unital, $\mathds{k}$-algebra with generators $\psi_i$ and $\psi_j^*$ for $i,j \in \Z$ satisfying the anticommutation relations,
\begin{equation}\label{eq:CliffordRelations}
\psi_i \psi_j^* + \psi_j^* \psi_i = \delta_{j,i} \qquad \psi_i \psi_j + \psi_j \psi_i = 0 \qquad \psi_i^* \psi_j^* + \psi_j^* \psi_i^* = 0. 
\end{equation}
\end{definition}

\begin{proposition}\label{prop:CliffAction} Fermionic Fock space has an action of the Clifford algebra $\mathcal{CL}$ given by
\begin{align}\label{eq:CliffordGenerators}
\psi_j(i_1 \wedge i_2 \wedge \cdots) &:= \begin{cases} 0 & \text{ if $j-\frac{1}{2}=i_s$ for some s,} \\
(-1)^s i_1 \wedge i_2 \wedge \cdots i_s \wedge (j-\frac{1}{2}) \wedge i_{s+1} \wedge \cdots & \text{ if } i_s > j-\frac{1}{2} > i_{s+1}. \end{cases} \\
\psi^*_j(i_1 \wedge i_2 \wedge \cdots) &:= \begin{cases} 0 & \text{ if $j-\frac{1}{2} \neq i_s$ for all s,} \\
(-1)^{s+1} i_1 \wedge i_2 \wedge \cdots i_{s-1} \wedge i_{s+1} \wedge \cdots & \text{ if } i_s = j-\frac{1}{2}.  \end{cases}
\end{align}
Hence, when acting on $\mathcal{F}$ these operators, known as \emph{free fermions}, are adjoint to each other and act like creation and annihilation operators on $\mathcal{F}$.
\end{proposition}

In essence, $\psi_i$ and $\psi_i^*$ raise and lower the charge of $v$ by creating and eliminating electrons in the $i$th energy state. Moreover, since for any vectors $v,w \in \mathcal{F}$ there is always  a finite sequence of $\psi_i$'s and $\psi_j$'s whose action on $v$ will yield $w$, the representation of the Clifford algebra induced by this action is also irreducible. 

\subsection{Connections with Symmetric Functions}\label{subsec:ConxSymFun}

The ring of symmetric functions $Sym$ has various bases indexed by partitions. We will particularly consider Schur functions $s_\lambda$ and the
elementary $e_{\lambda}$, complete $h_\lambda$, and power sum $p_\lambda$ symmetric functions. This ring has a natural inner product with the property that $\langle s_\lambda, s_\mu\rangle=\delta_{\mu,\lambda}$. Since $Sym$ acts on itself via multiplication, then for each $f \in Sym$ we can define $f^\perp \in \End(Sym)$ as the adjoint operator on $Sym$ with respect to this inner product, that is $\langle f^\perp u, v\rangle = \langle u, fv\rangle$ for any $u,v \in Sym$. In this way, for each integer $n\in \N$ Macdonald considers the operators $e_n^{\perp}, h_n^{\perp}, p_n^{\perp} \in$ End$(Sym)$ and introduces the following generating functions \cite[Section 1.2]{MacDonald}. 
\begin{align}
P(z) &= \sum_{n\geq 0} p_n z^{n-1} & E(z) &= \sum_{n\geq 0} e_n z^n & H(z) &= \sum_{n\geq 0} h_n z^n \label{eq:SymGenSeries}\\
P^{\perp}(z) &= \sum_{n\geq 0} p^{\perp}_n z^{n-1} & E^{\perp}(z) &= \sum_{n\geq 0} e^{\perp}_n z^n 
&H^\perp(z) &= \sum_{n\geq 0} h^\perp z^n \label{eq:SymGenSeries2}
\end{align} 
These functions satisfy the relations $E(z)=$ exp$\left( \sum_{n\geq 1} (-1)^{n-1}\frac{p_n}{n}z^n \right)$ and $H(z)=$ exp$\left( \sum_{n\geq 1}\frac{p_n}{n}z^{n} \right)$. 
\smallskip

The \emph{Bernstein operators} $B_a$ and $B_a^*$ are the coefficients of $z^a$ for each $a \in \Z$ in the following formal power series\footnote{Macdonald uses the notation $B^\perp_a$ instead of $B^*_a$.} \cite{Zelevinsky}: 
\begin{align}\label{eq:BernsteinGenSeries} B(z)&=\sum_{a \in \Z} B_a z^a :=H(z) E^{\perp}\left( -z^{-1}\right) &
B^*(z)&=\sum_{a \in \Z} B^*_a z^a := E\left( -z^{-1}\right)H^\perp(z)&&\\
&=\sum_{a\in \Z} \sum_{\substack{n,m \in \N\\ n-m=a}}(-1)^m h_{n}e_m^\perp &
&=\sum_{a\in\Z}\sum_{\substack{n,m\in \N\\m-n=a}} (-1)^n e_n h_{m}^\perp&&
\end{align} 
The Bernstein operators act like creation and annihilation operators for Schur functions $s_\lambda$ in the sense that for any partition $\lambda=(\lambda_1, \dots, \lambda_n)$, we have $s_\lambda = B_{\lambda_1} \dots B_{\lambda_n}(1)$ and $\B^*_{\lambda_n}\dots \B^*_{\lambda_1}(s_\lambda)=1$ \cite[Section 1.5]{MacDonald}. We prove a categorical analogue of these relations in Theorem \ref{thm:BernsteinSpecht}. 

\begin{proposition}\label{prop:HeisAction}
The Heisenberg algebra $\mathfrak{h}$ acts on the ring of symmetric functions. 
\end{proposition}

\begin{proof}
This is immediate from equation \eqref{eq:HeisGenAlpha}. 
\end{proof}

From the definition of $p_n^\perp$ it can be deduced that $p_n^\perp$ acts on symmetric functions by $n\frac{\partial}{\partial p_n}$ (see \cite[pg. 76]{MacDonald}). Thus, combining the action in equation \eqref{eq:HeisGenAlpha} and Proposition \ref{prop:HeisBasisRel} then we can identify the various generators of $\mathfrak{h}$ and $Sym$ in the following manner:
\begin{equation}\label{eq:T} p^{(n)} \to h_n \qquad q^{(n)} \to h_n^\perp  \qquad \alpha_{-n} \to np_n \end{equation} 
\[p^{(1^n)} \to e_n\qquad q^{(1^n)}\to e_n^\perp  \qquad \alpha_{n} \to p^\perp_n\]

Moreover, since $Sym$ can be expressed as $\Z[h_1,h_2, \dots]$, $\Z[e_1,e_2,\dots]$, and $\mathds{Q}[p_1,p_2,\dots]$ it follows that $p^{(n)}$ and $q^{(n)}$ generate the integral form of the Heisenberg algebra $\mathfrak{h}_\Z$ whereas the $\alpha_n$ are generators for $\mathfrak{h}_{\mathds{Q}}=\mathfrak{h}_\Z \otimes_\Z \mathds{Q}$. 


\subsection{The Boson-Fermion Correspondence}\label{subsec:BFcorrespondence}
Given Theorem \ref{thm:FockSpaceIso} we will henceforth identify Fermionic and Bosonic Fock space and assume all operators are acting on $\mathcal{B}$. This isomorphism allows us to express the action of $\mathcal{CL}$ on $\mathcal{B}$ in such a way that the Clifford algebra action induced from Proposition \ref{prop:CliffAction} can be entirely recovered from the Heisenberg algebra action in Proposition \ref{prop:HeisAction}. In particular, we will construct fermions $\psi_i$, $\psi_i^*$ in terms of bosons $\alpha_n$, by a process known as \emph{fermionization}. To this effect, we introduce the operators on $\mathcal{B}$ 
\[
\Gamma_+(z) := \text{exp} \left( \sum_{n\geq 1} \frac{z^{-n}}{n}\alpha_n \right) \;\; \text{and} \;\; \Gamma_-(z) := \text{exp} \left( \sum_{n\geq 1} \frac{z^{n}}{n}\alpha_{-n} \right)
\]
which satisfy the commutation relation $ \Gamma_-(y) \Gamma_+(z) = (1-\frac{y}{z})\Gamma_+(z)\Gamma_-(y)$. Operators of this form are known as \emph{halves of vertex operators}. By equations (\ref{Heis:Pchangeofbasis}) and (\ref{Heis:Qchangeofbasis}) we immediately obtain the following identities:
\begin{align*}
\Gamma_-(z)&= \sum_{m \in \N} p^{(m)} z^{m} \;\;& \Gamma_-(z)^{-1}&= \sum_{m \in \N} (-1)^{m}p^{(1^m)} z^{m}\\
\Gamma_+\left(\frac{1}{z}\right)&= \sum_{m \in \N} q^{(m)} z^{m} \;\;
& \Gamma_+\left(\frac{1}{z}\right)^{-1}&= \sum_{m \in \N} (-1)^{m}q^{(1^m)} z^{m}
\end{align*}The operators $\Gamma_+(z)$ and $\Gamma_-(z)$ are adjoint operators in the sense that if we consider the generating series $\Gamma_\pm(z)=\sum_{n \geq 0} \Gamma_{n}^{\pm}z^{\mp n}$ then $\Gamma_n^+$ and $\Gamma_n^-$ are adjoint operators on $\mathcal{B}$ \cite[pg 315]{Kac1}. This statement follows immediately from the fact that $p^{(m)}$ and $q^{(m)}$ are adjoint operators in $\mathfrak{h}$. 

Thus, if we define the \emph{Fermionic fields} as $\psi(z) := \sum_{i\in \Z} \psi_i z^i$ and $\psi^*(z) := \sum_{i\in \Z} \psi_i^* z^{-i}$ and the \emph{charge operator} on $\mathcal{B}$ by $z^{\pm\alpha_0}(t^c s_\lambda):= z^{\pm c} s_\lambda$ then the second part of the \emph{Boson-Fermion correspondence} can be stated as follows.

\begin{theorem}[Boson-Fermion correspondence part 2, {\cite[14.10]{Kac1}}]\label{thm:BosonFermionCorrespondence} The operators $\Gamma_-$ and $\Gamma_+$ satisfy, 
\begin{equation*}
\psi(z)=z^{\alpha_0}\;t \;\Gamma_-(z) \Gamma_+(z)^{-1}
\qquad \text{and} \qquad
\psi^*(z)= t^{-1}\\ z^{-\alpha_0}\; \Gamma_-(z)^{-1} \Gamma_+(z).
\end{equation*}
\end{theorem}
Under the identification from \eqref{eq:T}, the generating series from \eqref{eq:SymGenSeries}, \eqref{eq:SymGenSeries2}, and \eqref{eq:BernsteinGenSeries} can be seen as operators on $\mathsf{V}_{Fock}$. Hence we may write $\Gamma_-(z)= H(z)$ and $\Gamma_+(z)^{-1}= E^\perp(-z^{-1})$. Consequently, $\Gamma_-(z) \Gamma_+(z)^{-1}= H(z)E^\perp(-z^{-1})= B(z)$ and $\Gamma_-(z)^{-1} \Gamma_+(z)= E(-z)H^\perp(z^{-1}) = B^*(z^{-1})$ where for any $a \in \Z$ the {Bernstein operators} $B_a$ and $B_a^*$ on $\mathcal{B}$ are given by:
\begin{align}\label{eq:BernsteinOps} 
B_a := \sum_{n-m=a} (-1)^{m}p^{(n)}q^{(1^m)} \;\; \text{and}  \;\; B^*_a := \sum_{m-n=a} (-1)^{n}p^{(1^n)}q^{(m)}.
\end{align}
Thus, for any $t^cs_\lambda \in \mathcal{B}_c$ with charge $c \in \Z$ the action of $\psi(z)$ and $\psi^*(z)$ is given by:
\begin{align*}
\psi(z)(t^c s_{\lambda}) &= z^{\alpha_0}t \sum_{a\in\Z} \mathcal{B}_a (t^c s_\lambda)z^a =z^{c+1}t^{c+1}\sum_{a\in\Z} \mathcal{B}_a (s_\lambda)z^a=\sum_{i\in\Z} t^{c+1} \mathcal{B}_{i-c-1} (s_\lambda)z^i
\end{align*}
\begin{align*}
\psi^*(z)(t^c s_{\lambda}) &= t^{-1}z^{-\alpha_0} \sum_{a\in\Z} \mathcal{B}^*_a (t^c s_\lambda)z^{-a} =t^{c-1}z^{-c}\sum_{a\in\Z} \mathcal{B}^*_a (s_\lambda)z^{-a}=\sum_{i\in\Z} t^{c-1} \mathcal{B}^*_{i-c} (s_\lambda)z^{-i}
\end{align*}
Since $\psi(z)(t^c s_\lambda)=\sum_{i\in\Z} \psi_i(t^c s_\lambda) z^i$ and $\psi^*(z)(t^c s_\lambda)=\sum_{i\in\Z} \psi^*_i(t^c s_\lambda) z^{-i}$ then by equating the sums for any $i \in \Z$ we can express the action of the fermionic generators in terms on the Bernstein operators as follows:
\begin{equation}
\psi_i(t^cs_\lambda) = t^{c+1} B_{i-(c+1)}(s_\lambda) \qquad\qquad \qquad \qquad\qquad \psi^*_i(t^cs_\lambda) = t^{c-1} B^*_{i-c}(s_\lambda).
\end{equation}
Consequently, the condition that $\psi_i$ and $\psi_i^*$ satisfy the Clifford algebra relations from \eqref{eq:CliffordRelations} is equivalent to the Bernstein operators satisfying certain anticommutation relation. Specifically, through a direct computation we find the equations on the left-hand side below hold if and only if the relations on the right-hand side also hold. 
\begin{align}
&\psi_i \psi_j^* + \psi_j^* \psi_i = \delta_{i,j} &&\Leftrightarrow  &&B_{i-c}B_{j-c}^*+B_{j-c-1}^*B_{i-c-1}=\delta_{i,j}\label{eq:cliff-bern-1}\\
&\psi_i \psi_j + \psi_j \psi_i = 0 &&\Leftrightarrow & &B_{i-c-2}B_{j-c-1}+B_{j-c-2}B_{i-c-1}=0\label{eq:cliff-bern-2}\\
&\psi^*_i \psi^*_j + \psi^*_j \psi^*_i = 0 &&
\Leftrightarrow &&B^*_{i-c+1}B^*_{j-c}+B^*_{j-c+1}B^*_{i-c}=0\label{eq:cliff-bern-3}.
\end{align}
The relations on the right hand side involving the Bernstein operators are proven by Macdonald in \cite[Sec. 1.5]{MacDonald} in the context of symmetric functions.


\section{Homological Algebra for Infinite Complexes}\label{sec:HomAlgebra}

We begin by establishing some conventions used throughout the paper.  
Suppose $\mathcal{C}$ is an additive, monoidal category and $\mathcal{A}$ and $\mathcal{B}$ are two chain complexes over $\mathcal{C}$. Then:
\begin{itemize}
\item The $i^{th}$ chain group of $\mathcal{A}$ is denoted by $\mathcal{A}_i$. 
\item The differentials decrease homological degree. That is, they are maps $\d_{\mathcal{A}}: \mathcal{A}_{i+1}\to\mathcal{A}_i$.
\item For any $s \in \Z$ we denote by $\mathcal{A}[s]$ the chain complex $\mathcal{A}$ with a homological degree shift of $s$. That is, $\mathcal{A}[s]$ is the complex whose $i^{th}$ chain group is given by $\mathcal{A}_{i-s}$ and whose differentials are the same as in $\mathcal{A}$ (modulo a sign). Thus, $\mathcal{A}[1]$ indicates a homological shift up by 1 (i.e. to the left) whereas $\mathcal{A}[-1]$ is a shift down by one (i.e. to the right). 
\item If the homological degree of each chain group in a chain complex $\mathcal{A}$ needs to be made explicit, we will write $\mathcal{A}_i[s+i]$ for some $s\in \Z$, to indicate that $\mathcal{A}_i$ lives in homological degree $s+1$. Thus, we will opt for the notation on the right rather than the left. 
\[ \cdots \xrightarrow{\d_{\mathcal{A}}}\mathcal{A}_{i+1} \xrightarrow{\d_{\mathcal{A}}}\underset{s+i}{ \underline{\mathcal{A}_{i}}} \xrightarrow{\d_{\mathcal{A}}}\mathcal{A}_{i-1} \xrightarrow{\d_{\mathcal{A}}} \cdots =
\cdots \xrightarrow{\d_{\mathcal{A}}}\mathcal{A}_{i+1}[s+i+1] \xrightarrow{\d_{\mathcal{A}}} \mathcal{A}_{i}[s+i] \xrightarrow{\d_{\mathcal{A}}}\mathcal{A}_{i-1}[s+i-1] \xrightarrow{\d_{\mathcal{A}}} \cdots  \] 
\item We will also write $\left\lbrace \bigoplus_i \mathcal{A}_i[s+i], \d_{\mathcal{A}} \right\rbrace$ to compactly denote the chain complex above, where $\mathcal{A}_i$ lives in homological degree $s+i$. 
\end{itemize}

We now recall some definitions and present the general machinery needed to perform the upcoming computations on bi-infinite chain complexes. We follow Elias-Hogancamp closely and refer the reader to Section 4 of \cite{EH-CatDiag} for all the statements below in their full generality.

 A \emph{chain map} $f:\mathcal{A}\to\mathcal{B}$ is a collection of maps $f_i:\mathcal{A}_i \to \mathcal{B}_i$ such that $\d_{\mathcal{B}}f_i-f_{i-1}\d_{\mathcal{A}} =0$ for all $i$. We say two chain maps $f,g: \mathcal{A} \to \mathcal{B}$ are \emph{homotopic}, denoted by $f \simeq g$ if there exists a \emph{homotopy}, $\mathsf{H}:\mathcal{A} \to \mathcal{B}[-1]$ with the property that $\mathsf{H}\circ \d_{\mathcal{A}} + \d_{\mathcal{B}}\circ \mathsf{H} = f-g$. If such maps exists, we say that $\mathcal{A}$ and $\mathcal{B}$ are \emph{homotopy equivalent}. In particular, we say a chain complex $\mathcal{A}$ is \emph{nullhomotopic} if there exists a homotopy $\mathsf{H}$ such that $\1_{\mathcal{A}} \simeq 0$. 

The \emph{homotopy category of $\mathcal{C}$} is denoted by $\K(\mathcal{C})$. Its objects are chain complexes over $\mathcal{C}$ and its morphisms are chain maps up to homotopy. Thus two complexes are isomorphic in $\K(\mathcal{C})$, denoted $\mathcal{A} \simeq \mathcal{B}$, if they are homotopy equivalent. Let $\K^+(\mathcal{C}), \K^-{(\mathcal{C})},$ and $\K^b(\mathcal{C})$ denote the full subcategories of $\K(\mathcal{C})$ of chain complexes which are bounded above, bounded below, and bounded above and below respectively. 

For any chain map $f:\mathcal{A}\to\mathcal{B}$, the \emph{mapping cone of $f$} is the chain complex Cone$(f):= \mathcal{A}[1] \oplus \mathcal{B}$ whose differential is given by $\begin{bmatrix}
-\d_{\mathcal{A}} & 0\\
f & \d_{\mathcal{B}}\\
\end{bmatrix}$. In particular, a triangle in $\K(\mathcal{C})$ is distinguished if and only if it is isomorphic to
\[
\mathcal{A} \xrightarrow{f} \mathcal{B} \xrightarrow{\iota} \text{Cone}(f) \xrightarrow{\pi} \mathcal{A}[1]
\]
where both $\iota$ and $\pi$ are canonical inclusion and projection maps associated to mapping cones. Thus, $\K(\mathcal{C})$ (and consequently $\K^+(\mathcal{C})$, $\K^-(\mathcal{C})$, and $\K^b(\H)$) is a triangulated category.

\begin{definition} \label{def:total complex}
Let $\mathcal{T} \subset \K(\mathcal{C})$ be a full triangulated subcategory and $I \subset \Z$ some indexing set. A \emph{bi-complex} in $\mathcal{T}$ is a collection $\lbrace \mathcal{A}^i,\d_{i},\d^i\rbrace_{i \in I}$ with $\mathcal{A}^i \in \mathcal{T}$ such that for each $i \in I$, $\d_i = \d_{\mathcal{A}^i}:\mathcal{A}_j^i \to \mathcal{A}_{j-1}^i$ and $\d^{i}:\mathcal{A}^i_j \to \mathcal{A}^{i-1}_j$ is a morphism of chain complexes satisfying:
\[
{(\d_i)}^2 =0 \text{,\hspace{6mm}} {(\d^{i})}^2 = 0 \text {\hspace{3mm} and \hspace{3mm}}  \d_i\d^i+\d^i\d_i=0.
\]
\end{definition}
\begin{remark}In this work we will only consider $\mathcal{T} = \K(\mathcal{C}),\K^-(\mathcal{C}), \K^+(\mathcal{C})$, or $\K^b(\mathcal{C})$.
\end{remark}

Given a bi-complex $\lbrace \mathcal{A}^i, \d_i,\d^i \rbrace$ in $\mathcal{T}$, we define its \emph{total complex} and its \emph{completed total complex} as the chain complexes given by: 
\begin{equation}
\text{Tot}^{\oplus} \lbrace \mathcal{A}^i, \d_i,\d^i \rbrace := \left\lbrace \bigoplus_i \mathcal{A}^i , \d \right\rbrace
\text{ \hspace{3mm} with differential \hspace{3mm}} \d = \sum \d_i+\d^i
\end{equation}
\begin{equation}
\text{Tot}^{\Pi}\lbrace \mathcal{A}^i, \d_i,\d^i \rbrace := \left\lbrace \prod_i \mathcal{A}^i,\d \right\rbrace
\text{ \hspace{3mm} with differential \hspace{3mm}} \d = \sum \d_i+\d^i
\end{equation}
We say $\lbrace \mathcal{A}^i, \d_i,\d^i \rbrace$ is $\mathcal{T}$-\emph{locally finite} if $\text{Tot}^{\oplus} \lbrace \mathcal{A}^i, \d_i,\d^i \rbrace$ and $ 
\text{Tot}^{\Pi}\lbrace \mathcal{A}^i, \d_i,\d^i \rbrace$ exist in $\mathcal{T}$ and are isomorphic.

\begin{definition}\label{def:ChainTensorProd} Given any additive, monoidal category $\mathcal{C}$ and $\mathcal{A}, \mathcal{B} \in \K(\mathcal{C})$ we define the \emph{tensor product} of $\mathcal{A}$ and $\mathcal{B}$ as the chain complex
\[\mathcal{A}\otimes \mathcal{B}:=\text{Tot}^{\oplus} \left\lbrace \mathcal{A} \otimes \mathcal{B}^j, \d_j,\d^j \right\rbrace \qquad \text{with} \qquad \d_j:= \d_{\mathcal{A}}\otimes \1_{\mathcal{B}^j}\text{ and } \d^j:=\1_{\mathcal{A}} \otimes \d_{\mathcal{B}^{j}} \]
and the \emph{completed tensor product} as the chain complex
\[\mathcal{A} \tilde{\otimes} \mathcal{B}:= \text{Tot}^{\Pi} \left\lbrace \mathcal{A} \otimes \mathcal{B}^j, \d_j,\d^j \right\rbrace \qquad \text{with} \qquad \d_j:= \d_{\mathcal{A}}\otimes \1_{\mathcal{B}^j}\text{ and } \d^j:=\1_{\mathcal{A}} \otimes \d_{\mathcal{B}^{j}}. \]
\end{definition}

When $\mathcal{A}$ is a single term complex, i.e. $\mathcal{A} = \mathcal{A}[s]$ with $\mathcal{A} \in \mathcal{C}$ and $s \in \Z$, we will write $\mathcal{A} \mathcal{B}[s]$ to denote $\mathcal{A}\otimes \mathcal{B}$ with differential $\1_{\mathcal{A}}\otimes \mathsf{d}_{\mathcal{B}}$ and omit any explicit mention of the type of tensor product since the bi-complex $\left\lbrace \mathcal{A}\otimes \mathcal{B}^i, \d_i,\d^i \right\rbrace$ is trivially $\mathcal{T}$-locally finite. 

In light of this new terminology, we pause to make a few technical but important remarks that will be used in Section \ref{sec:PropCatBernsteinOps}.

\begin{remark}\label{rem:totalcomsym}
Due to the symmetry in Definition \ref{def:ChainTensorProd} we can equivalently define the tensor product and the completed tensor product in the following ways:
\[\mathcal{A}\otimes \mathcal{B}:=\text{Tot}^{\oplus} \left\lbrace \mathcal{A}^i \otimes \mathcal{B}, \d_i,\d^i \right\rbrace \qquad \text{with} \qquad \d_i:=\1_{\mathcal{A}^i} \otimes \d_{\mathcal{B}} \text{ and }  \d^i:= \d_{\mathcal{A}^i}\otimes \1_{\mathcal{B}} \]
\[\mathcal{A} \tilde{\otimes} \mathcal{B}:= \text{Tot}^{\Pi} \left\lbrace \mathcal{A}^i \otimes \mathcal{B}, \d_i,\d^i \right\rbrace \qquad \text{with} \qquad \d_i:=\1_{\mathcal{A}^i} \otimes \d_{\mathcal{B}}  \text{ and } \d^i:= \d_{\mathcal{A}^i}\otimes \1_{\mathcal{B}}\]
When the distinction between these formulations is needed for $\mathcal{A}\otimes \mathcal{B}$ (or equivalently for $\mathcal{A}\tilde{\otimes} \mathcal{B}$), we will refer to $\mathcal{A}^i \otimes \mathcal{B}$ as the $i^{th}$ \emph{row} and $\mathcal{A} \otimes \mathcal{B}^j$ as the $j^{th}$ \emph{column} of the bi-complex. 
\end{remark}

\begin{remark}\label{rem:locallyfinite}
While it is clear that for finite complexes these definitions agree, i.e. bounded complexes are $\mathcal{T}$-\emph{locally finite}, the distinction between direct sums and direct products is crucial when "tensoring" infinite complexes. This is because depending on certain boundedness conditions of $\mathcal{A}$ and $\mathcal{B}$ their tensor product and their completed tensor product are often not homotopy equivalent or even well defined. In particular, if $\mathcal{A} \in \K^+(\mathcal{C})$ and $\mathcal{B} \in \K^-(\mathcal{C})$ then $\lbrace \mathcal {A} \otimes \mathcal{B}_j, \d_j,\d^j \rbrace$ is not necessarily $\K(\mathcal{C})$-locally finite. 
\end{remark}

Given a bi-complex $\lbrace \mathcal{A}^i, \d_i,\d^i \rbrace$ where each complex $\mathcal{A}^i = \lbrace \bigoplus \mathcal{A}^i_j,\d_i \rbrace$, we say the bi-complex is \emph{homologically locally finite} if for any $k \in \Z$ the sum $\bigoplus_{i+j=k}\mathcal{A}^i_j$ is finite and thus isomorphic to $\prod_{i+j=k}\mathcal{A}^i_j$. In this case the following proposition holds. 

\begin{proposition}[{\cite[Prop. 4.10]{EH-CatDiag}} ] \label{prop:locallyfinite}
Let $\lbrace \mathcal{A}^i, \d_i,\d^i \rbrace$ be a bi-complex such that each subcomplex $\mathcal{A}^i = \lbrace \bigoplus \mathcal{A}^i_j,\d_i \rbrace \in \mathcal{T}$ and that for each fixed homological degree $j$ as objects inside $\mathcal{C}$ we have,
\[ \bigoplus_{i \in \Z} \mathcal{A}^i_j \cong   \prod_{i \in \Z} \mathcal{A}^i_j .\]
Then $\bigoplus_i \mathcal{A}^i \cong \prod_i\mathcal{A}^i $ and thus $\lbrace \mathcal{A}^i, \d_i,\d^i \rbrace$ is $\mathcal{T}$-locally finite. 
\end{proposition}

 In particular, if $\mathcal{A}, \mathcal{B} \in \K^+(\mathcal{C})$ (resp. $\K^-(\mathcal{C})$) then $\lbrace \mathcal {A} \otimes \mathcal{B}_j, \d_j,\d^j \rbrace$ is homologically locally finite and so, by Proposition \ref{prop:locallyfinite}, also $\K^+(\mathcal{C})$-locally finite (resp. $\K^-(\mathcal{C})$-locally finite). 

Lastly, we present our main tools. We refer the reader to \cite[Lemma 6.1]{CL-Vertex} and \cite[Prop. 4.20]{EH-CatDiag} for their respective proofs. 

\begin{lemma}[Gaussian elimination]\label{lem:gaussian-elimination}
Let $ X, Y, Z, W, U, V$ be six objects in an additive category and consider the complex
$[\dots \rightarrow U \xrightarrow{u} X \oplus Y \xrightarrow{f} Z \oplus W \xrightarrow{v} V \rightarrow \dots]$
where $f = \left( \begin{matrix} C & D \\ A & B \end{matrix} \right)$ and $u,v$ are arbitrary morphisms. If $D: Y \rightarrow Z$ is an isomorphism, then the following chain complexes are homotopic.  
\[ [\dots \rightarrow U \xrightarrow{u} X \oplus Y \xrightarrow{f} Z \oplus W \xrightarrow{v} V \rightarrow \dots] \simeq [\dots \rightarrow U \xrightarrow{u} X \xrightarrow{A-BD^{-1}C} W \xrightarrow{v|_W} V \rightarrow \dots] \]
\end{lemma}

\begin{proposition}[Simultaneous simplifications]\label{prop:SimultSimp} 
Let $I \subset \Z$ and suppose $\lbrace \mathcal{A}^i,\d_{i},\d^i\rbrace_{i\in I}$ and $\lbrace \mathcal{B}^i,{\mathsf{b}}_{i},{\mathsf{b}}^i\rbrace_{i\in I}$ are bi-complexes in $\mathcal{T}$ such that for each $i \in I$, $\mathcal{A}^i$ and $\mathcal{B}^i$ are homotopy equivalent.
\begin{itemize}
\item[i)] If $I$ is bounded below then $\emph{Tot}^{\oplus} \lbrace \mathcal{A}^i, \d_i,\d^i \rbrace_{i \in I} \simeq \emph{Tot}^{\oplus} \lbrace B^i,{\mathsf{b}}_{i},{\mathsf{b}}^i\rbrace_{i\in I}$. 
\item[ii)] If $I$ is bounded above then $\emph{Tot}^{\Pi} \lbrace \mathcal{A}^i, \d_i,\d^i \rbrace_{i\in I} \simeq \emph{Tot}^{\Pi} \lbrace B^i,{\mathsf{b}}_{i},{\mathsf{b}}^i\rbrace_{i\in I}$. 
\end{itemize}
\end{proposition}

In particular, if $\mathcal{A}_i$ and $\mathcal{B}_i$ are homotopic by a sequence of Gaussian eliminations  and $\mathcal{T} = \K^+(\mathcal{C})$ or $\K^-(\mathcal{C})$ then Proposition \ref{prop:SimultSimp} applies and all Gaussian eliminations can be simultaneously performed.


\section{Extended Graphical Calculus for Khovanov's Heisenberg Category}\label{sec:ExtGraphicalCalc}

\subsection{Khovanov's Heisenberg Category}\label{subsec:KhovanovHeisCat}
For any commutative unital ring $\k$, Khovanov introduced a categorical framework for the Heisenberg algebra $\mathfrak{h}$. This framework consists of a monoidal category $\H$ which is the Karoubi envelope of a category $\H'$ whose definition we now sketch (see \cite{Kh-H} for more details).

The additive, $\mathds{k}$-linear, monoidal category $\H'$ is generated by two objects $Q_+$ and $Q_-$. An object of $\H'$ is a finite direct sum of tensor products $Q_{\epsilon}:= Q_{\epsilon_1} \otimes \dots \otimes Q_{\epsilon_k}$ where $\epsilon=(\epsilon_i)$ is a sequence with $\epsilon_i \in \lbrace +,- \rbrace$ for each $i$. The empty sequence $\emptyset$ corresponds to the monoidal unit and is denoted $\1$. If $\epsilon_i=+$ or $\epsilon_i=-$ for all $i$ we denote its corresponding object by $Q_{+k}$ or  $Q_{-k}$ respectively. For any sequences $\epsilon, \epsilon'$, the space of homomorphisms Hom$_{\H'}(Q_{\epsilon},Q_{\epsilon'})$ is the $\mathds{k}$-module described by certain string diagrams with relations. 

Specifically, we denote $Q_+$ as an upward pointing arrow and $Q_-$ as a downward pointing arrow. Monoidal composition of such objects is given by sideways concatenation.  By convention, composition of morphisms is done vertically and read bottom to top. When not presented diagrammatically we will use "$\otimes$" and "$\circ$" to denote horizontal and vertical composition, respectively. The morphisms between such objects are generated by the crossings, caps, and cups displayed below:
\begin{equation}\label{eq:maps}
\hackcenter{
\begin{tikzpicture}[scale=1]
\draw [->](0,0) -- (1,1) [thick];
\draw [->](1,0) -- (0,1) [thick];
\draw (3,.5) arc (180:360:.5)[->] [thick];
\draw (7,.25) arc (0:180:.5) [->][thick];
\draw (9,.5) arc (180:360:.5)[<-] [thick];
\draw (13,.25) arc (0:180:.5) [<-][thick];
\end{tikzpicture}}
\end{equation}
The crossing on the left is a morphism $Q_{++} \to Q_{++}$ whereas the rightmost cap is a morphism $Q_{+-} \to \1$. These morphisms are isotopy invariant and satisfy the relations below (note that \eqref{heis:up-double} and \eqref{heis:up-triple} also hold for downward pointing strands).

\begin{minipage}{0.45\textwidth}
\begin{align}
\label{heis:up-double}
\hackcenter{\begin{tikzpicture}[scale=0.8, yscale=.8]
    \draw[thick] (0,0) .. controls ++(0,.5) and ++(0,-.5) .. (.75,1);
    \draw[thick] (.75,0) .. controls ++(0,.5) and ++(0,-.5) .. (0,1);
    \draw[thick, ->] (0,1 ) .. controls ++(0,.5) and ++(0,-.5) .. (.75,2);
    \draw[thick, ->] (.75,1) .. controls ++(0,.5) and ++(0,-.5) .. (0,2);
    \node at (0,-.25) {$\;$};
    \node at (0,2.25) {$\;$};
\end{tikzpicture}}
&\;\; = \;\;
\hackcenter{\begin{tikzpicture}[scale=0.8, yscale=.8]
    \draw[thick, ->] (0,0) -- (0,2);
    \draw[thick, ->] (.75,0) -- (.75,2);
\end{tikzpicture}}
\\
  \label{heis:downup}
  \hackcenter{\begin{tikzpicture}[scale=0.8, yscale=.8]
    \draw[thick,<-] (0,0) .. controls ++(0,.5) and ++(0,-.5) .. (.75,1);
    \draw[thick] (.75,0) .. controls ++(0,.5) and ++(0,-.5) .. (0,1);
    \draw[thick, ->] (0,1 ) .. controls ++(0,.5) and ++(0,-.5) .. (.75,2);
    \draw[thick] (.75,1) .. controls ++(0,.5) and ++(0,-.5) .. (0,2);
\end{tikzpicture}}
&\;\; = \;\;
\hackcenter{\begin{tikzpicture}[scale=0.8, yscale=.8]
    \draw[thick, <-] (0,0) -- (0,2);
    \draw[thick, ->] (.75,0) -- (.75,2);
\end{tikzpicture}}
\;\; - \;\;
\hackcenter{\begin{tikzpicture}[scale=0.8, yscale=.8]
    \draw[thick, <-] (0,0) .. controls ++(0,.75) and ++(0,.75) ..(.75,0);
    \draw[thick, <-] (.75,2) .. controls ++(0,-.75) and ++(0,-.75) .. (0,2);
\end{tikzpicture}}
  \\
  \label{eq:heis-bub}
\hackcenter{\begin{tikzpicture}[scale=0.8]
\draw [shift={+(0,0)}](0,0) arc (180:360:0.5cm) [thick];
\draw [shift={+(0,0)}][->](1,0) arc (0:180:0.5cm) [thick];
\node at (0,-1) {$\;$};
\node at (0,1) {$\;$};
\end{tikzpicture}}
& \;\; = \;\; 1,
\end{align}
\end{minipage}%
\hfill
\begin{minipage}{0.45\textwidth}
\begin{align} \label{heis:up-triple}
\hackcenter{\begin{tikzpicture}[scale=0.8, yscale=.8]
    \draw[thick, ->] (0,0) .. controls ++(0,1) and ++(0,-1) .. (1.2,2.5);
    \draw[thick, ] (.6,0) .. controls ++(0,.5) and ++(0,-.5) .. (0,1.25);
    \draw[thick, ->] (0,1.25) .. controls ++(0,.5) and ++(0,-.5) .. (0.6,2.5);
    \draw[thick, ->] (1.2,0) .. controls ++(0,1) and ++(0,-1) .. (0,2.5);
\end{tikzpicture}}
&\;\; = \;\;
\hackcenter{\begin{tikzpicture}[scale=0.8, yscale=.8]
    \draw[thick, ->] (0,0) .. controls ++(0,1) and ++(0,-1) .. (1.2,2.5);
    \draw[thick, ] (.6,0) .. controls ++(0,.5) and ++(0,-.5) .. (1.2,1.25);
    \draw[thick, ->] (1.2,1.25) .. controls ++(0,.5) and ++(0,-.5) .. (0.6,2.5);
    \draw[thick, ->] (1.2,0) .. controls ++(0,1) and ++(0,-1) .. (0,2.5);
\end{tikzpicture}}
\\
\label{heis:up down}
  \hackcenter{\begin{tikzpicture}[scale=0.8, yscale=.8]
    \draw[thick] (0,0) .. controls ++(0,.5) and ++(0,-.5) .. (.75,1);
    \draw[thick,<-] (.75,0) .. controls ++(0,.5) and ++(0,-.5) .. (0,1);
    \draw[thick] (0,1 ) .. controls ++(0,.5) and ++(0,-.5) .. (.75,2);
    \draw[thick, ->] (.75,1) .. controls ++(0,.5) and ++(0,-.5) .. (0,2);
\end{tikzpicture}}
&\;\; = \;\;
\hackcenter{\begin{tikzpicture}[scale=0.8, yscale=.8]
    \draw[thick,->] (0,0) -- (0,2);
    \draw[thick, <-] (.75,0) -- (.75,2);
\end{tikzpicture}}
\\
\label{heis: right twist curl}
\hackcenter{
\begin{tikzpicture}[scale=0.8]
\draw  [shift={+(5,0)}](0,0) .. controls (0,.5) and (.7,.5) .. (.9,0) [thick];
\draw  [shift={+(5,0)}](0,0) .. controls (0,-.5) and (.7,-.5) .. (.9,0) [thick];
\draw  [shift={+(5,0)}](1,-1) .. controls (1,-.5) .. (.9,0) [thick];
\draw  [shift={+(5,0)}](.9,0) .. controls (1,.5) .. (1,1) [->] [thick];
\end{tikzpicture}}
&\;\;=\;\; 0
\end{align}
\end{minipage}%

\emph{Khovanov's Heisenberg category} $\H$ is defined as the \emph{Karoubian envelope} of $\H'$. Thus, the objects of $\H$ are given by pairs $(Q_\epsilon,e)$ where $Q_\epsilon$ is an object in $\H'$ and $e \in$ End$_{\H'}(Q_\epsilon)$ an idempotent, i.e. $e^2=e$. The space Hom$_{\H}(Q_\epsilon,e),(Q_{\epsilon'}, e'))$ consists of morphisms $f \in $ Hom$_{\H'}(Q_\epsilon, Q_{\epsilon'})$ such that $e'f=fe$. A map $f \in$ Hom$_{\H}((Q_\epsilon,e),(Q_{\epsilon'}, e'))$ is an isomorphism in $\H$ if there exists $f' \in$ Hom$_{\H}((Q_{\epsilon'},e'),(Q_\epsilon, e))$ such that $ff'=id_{Q_{\epsilon'}}$ and $f'f=id_{Q_\epsilon}$. This category is clearly $\mathds{k}$-linear, additive, and monoidal. Moreover, it satisfies the following theorem. 

\begin{theorem}[{\cite[Section 3.3]{Kh-H}} ]
There is an injective map $\phi \colon \mathfrak{h} \rightarrow \K_0(\H)$.
\end{theorem}

Set $\P:= (Q_+, id)$ and $\Q:=(Q_-,id)$, and denote them diagrammatically by an upward and downward pointing strand respectively. Then by \eqref{heis:up-double} and \eqref{heis:up-triple} there exist canonical homomorphisms:
\[\mathds{k}[S_n]\to \text{End}_{\H'}(Q_{+n}) \qquad \text{and} \qquad \mathds{k}[S_n]\to \text{End}_{\H'}(Q_{-n}).\]
In particular, if $\k$ is a field of characteristic zero then for any partition $\lambda \vdash n$ we denote the set of standard Young tableau by $\SYT(\lambda)$, then for $T \in \SYT(\lambda)$ the image of its Young symmetrizer $e_T$ is an idempotent in $\End_{\H'}(Q_{+n})$ \cite{Fulton-YoungTableux, Garcia}. Consequently, for each $T \in \SYT(\lambda)$ we set:
\[\P^{T}:=(Q_{+n}, e_T) \qquad \text{and} \qquad \Q^{T}:=(Q_{-n}, e_T).\] 
If $T_{row}$ is the standard row ordered filling of $\lambda$, then for any other $T \in \SYT(\lambda)$ there exists a permutation $\sigma \in S_n$ such that $T_{row}=\sigma T$. That is, any standard Young tableau can be transformed into the row ordered filling by applying some permutation of the entries. However, by the inclusions above we have that $S_n \in$ Aut$(Q_{\pm})$ and so the map $\sigma: (Q_{+n},e_{T}) \to (Q_{+n},e_{T_{row}})$ is an isomorphism in $\Kar(\H')$. Consequently, if $T,T' \in \SYT(\lambda)$ then $e^{T}$ and $e^{T'}$ are conjugate and so $\P^{T} \simeq \P^{T'}$ in $\H$. Thus, set 
\[\P^{\lambda} := (Q_{+n},e_{T_{row}})\text{ and }\Q^{\lambda} := (Q_{-n},e_{T_{row}})\]
 as the canonical representatives for the isomorphism classes of $\P^{T}$ and $\Q^{T}$ with $T \in \SYT(\lambda)$. 

\subsection{Diagramatics for Young Idempotents}\label{sec:diag-young-idempotents}
For any partition $\lambda \vdash n$, let $\lambda^t$ denote the transposed partition whose rows are given by the columns of $\lambda$. Denote by $T_{row}$ and $T_{col}$ the standard row and column ordered fillings of $\lambda$ respectively. 

Since $\P^\lambda$ will always denote the idempotent corresponding to the tableau with the standard row filling, if we label the leftmost strand on the top of the diagram by $1$ and the rightmost strand $n$, then the diagram corresponding to $\lambda$ will consist of symmetrizing all strands whose labels lie in the same row in $\lambda$ and antisymmetrizing the strands whose entries lie in the same column of $\lambda$. 
The objects corresponding to $(n)$ and $(n)^t=(1^n)$ in $\H$ are the \textit{complete symmetrizer} and \textit{complete antisymmetrizer} depicted by white and dark gray boxes. In line with Khovanov and Cvitanovi\'{c}, we depict them below. The diagrams for $\Q^{(n)}$ and $\Q^{(n)^t}$ are identical but with downward arrows instead. 
\begin{equation}
\P^{(n)} \;\; :=  \;\;
\hackcenter{
\begin{tikzpicture} [scale=0.6]
  \draw[->] (.2,-1) -- (.2,1);
  \draw[->] (1.3,-1) -- (1.3,1);
  \draw[fill = white] (0,-.35) rectangle (1.5,.35);
  \node at (.75,0) {$n$};
  \node at (.75,-.7) {$\dots$};
    \node at (.75,.7) {$\dots$};
\end{tikzpicture}}
\qquad\qquad
\P^{(n)^t} \;\; :=  \;\;
\hackcenter{
\begin{tikzpicture} [scale=0.6]
  \draw[ ->] (.2,-1) -- (.2,1);
  \draw[ ->] (1.3,-1) -- (1.3,1);
  \draw[fill = gray!60] (0,-.35) rectangle (1.5,.35);
  \node at (.75,0) {$n$};
  \node at (.75,-.7) {$\dots$};
    \node at (.75,.7) {$\dots$};
\end{tikzpicture}}
\end{equation}
More generally, for any $\lambda = (\lambda_1, \dots, \lambda_k)$ with $\lambda^t=(\lambda_1^*, \dots, \lambda_l^*)$ let $\sigma \in S_n$ be a permutation such that $T_{row}=\sigma T_{col}$. Following Cvitanovi\'{c} \cite{Birdtracks} we depict the \emph{Young idempotent} $\P^\lambda$ in $\H$ in the following manner:
\begin{equation}\label{diag:YoungIdempotentExploded}
\P^{\lambda}
=\;\; \alpha_{\lambda} \;\;
\hackcenter{
\begin{tikzpicture} [scale=0.7, yscale=.8]
  \draw[ ->] (.2,-1) -- (.2,1);
  \draw[ ->] (1.3,-1) -- (1.3,1);
  \draw[fill = white] (0,-.35) rectangle (1.5,.35);
  \node at (.75,0) [scale=.75]{$\lambda_1$};
    \node at (.75,.7) {$\dots$};
    \draw[ ->] (1.8,-1) -- (1.8,1);
  \draw[ ->] (2.9,-1) -- (2.9,1);
  \draw[fill = white] (1.6,-.35) rectangle (3.1,.35);
  \node at (2.35,0) [scale=.75]{$\lambda_2$};
    \node at (2.35,.7) {$\dots$};
  \node at (3.5,0) {$\dots$};  
  \draw[ ->] (4.2,-1) -- (4.2,1);
  \draw[ ->] (5.3,-1) -- (5.3,1);
  \draw[fill = white] (4,-.35) rectangle (5.5,.35);
  \node at (4.75,0) [scale=.75]{$\lambda_k$};
    \node at (4.75,.7) {$\dots$};
\begin{scope}[shift={(0,-1.5)}]
  \draw[ ->] (.2,-1) -- (.2,1);
  \draw[ ->] (1.3,-1) -- (1.3,1);
    \node at (.75,.7) {$\dots$};
    \draw[ ->] (1.8,-1) -- (1.8,1);
  \draw[ ->] (2.9,-1) -- (2.9,1);
    \node at (2.35,.7) {$\dots$};
  \node at (3.5,0) {$\dots$};  
  \draw[ ->] (4.2,-1) -- (4.2,1);
  \draw[ ->] (5.3,-1) -- (5.3,1);
   \draw[fill = gray!20] (0,-.35) rectangle (5.5,.35);
   \node at (2.75,0) [scale=.75]{$\sigma$};
    \node at (4.75,.7) {$\dots$};
    \end{scope}
    \begin{scope}[shift={(0,-3)}]
  \draw[ ->] (.2,-1) -- (.2,1);
  \draw[ ->] (1.3,-1) -- (1.3,1);
  \draw[fill = gray!60] (0,-.35) rectangle (1.5,.35);
  \node at (.75,0)[scale=.75] {$\lambda_1^*$};
    \node at (.75,.7) {$\dots$};
    \draw[ ->] (1.8,-1) -- (1.8,1);
  \draw[ ->] (2.9,-1) -- (2.9,1);
  \draw[fill = gray!60] (1.6,-.35) rectangle (3.1,.35);
  \node at (2.35,0) [scale=.75]{$\lambda_2^*$};
    \node at (2.35,.7) {$\dots$};
  \node at (3.5,0) {$\dots$};  
  \draw[ ->] (4.2,-1) -- (4.2,1);
  \draw[ ->] (5.3,-1) -- (5.3,1);
  \draw[fill = gray!60] (4,-.35) rectangle (5.5,.35);
  \node at (4.75,0) [scale=.75]{$\lambda_l^*$};
    \node at (4.75,.7) {$\dots$};
    \end{scope}
    \begin{scope}[shift={(0,-4.5)}]
  \draw[ ->] (.2,-1) -- (.2,1);
  \draw[ ->] (1.3,-1) -- (1.3,1);
  \node at (.75,-.7) {$\dots$};
    \node at (.75,.7) {$\dots$};
    \draw[ ->] (1.8,-1) -- (1.8,1);
  \draw[ ->] (2.9,-1) -- (2.9,1);
  \node at (2.35,-.7) {$\dots$};
    \node at (2.35,.7) {$\dots$};
  \node at (3.5,0) {$\dots$};  
  \draw[ ->] (4.2,-1) -- (4.2,1);
  \draw[ ->] (5.3,-1) -- (5.3,1);
   \draw[fill = gray!20] (0,-.35) rectangle (5.5,.35);
   \node at (2.75,0) [scale=.75]{$\sigma^{-1}$};
  \node at (4.75,-.7) {$\dots$};
    \node at (4.75,.7) {$\dots$};
    \end{scope}
\end{tikzpicture}} \;\; = \;\alpha_{\lambda}\;\;\hackcenter{
\begin{tikzpicture} [scale=0.5, xscale=1.1,scale=.7, every node/.style={scale=1.1}]
\draw[line width=0.15mm, <-] (-5.4,5) to (-5.4,0);
\draw[ gray,<-] (-5,5) to (-5,3)..controls ++(0,-2) and ++(0,1)..(-3.4,0);
\draw[ gray,<-] (-4.6,5) to (-4.6,3)..controls ++(0,-2) and ++(0,1)..(-0.65,0);
\draw[line width=0.15mm, <-] (-3.7,5) to (-3.7,3)..controls ++(0,-2) and ++(0,2)..(5,0);
\draw[ line width=0.15mm,<-] (-2.9,5) to (-2.9,3)..controls ++(0,-1) and ++(0,1)..(-5.1,0);
\draw[gray, <-] (-2.5,5) to (-2.5,3)..controls ++(0,-1) and ++(0,1)..(-3,0);
\draw[ line width=0.15mm,<-] (-1.2,5) to (-1.2,3)..controls ++(0,-2) and ++(0,.75)..(2.5,0);
\draw[ line width=0.15mm,<-] (1.2,5) to (1.2,3)..controls ++(0,-1) and ++(0,2)..(-4.7,0);
\draw[ gray, <-] (2.3,5) to (2.3,3)..controls ++(0,-3) and ++(0,.5)..(0.25,0);
\draw[ line width=0.15mm,<-] (2.9,5) to (2.9,3)..controls ++(0,-1) and ++(0,1)..(3,0);
\draw[line width=0.15mm, <-] (4.2,5) to (4.2,3)..controls ++(0,-2) and ++(0,1)..(-4.2,0);
\draw[ line width=0.15mm,<-] (5.9,5) to (5.9,3)..controls ++(0,-3) and ++(0,.5)..(0.65,0);
\begin{scope}[yshift=-10mm]
\draw[line width=0.15mm, -] (-5.4,-4) to (-5.4,-0);
\draw[ gray,-] (-5,-4) to (-5,-3)..controls ++(0,2) and ++(0,-1)..(-3.4,0);
\draw[ gray,-] (-4.6,-4) to (-4.6,-3)..controls ++(0,2) and ++(0,-1)..(-0.65,0);
\draw[line width=0.15mm, -] (-3.7,-4) to (-3.7,-3)..controls ++(0,2) and ++(0,-2)..(5,0);
\draw[ line width=0.15mm,-] (-2.9,-4) to (-2.9,-3)..controls ++(0,1) and ++(0,-1)..(-5.1,0);
\draw[gray, -] (-2.5,-4) to (-2.5,-3)..controls ++(0,1) and ++(0,-1)..(-3,0);
\draw[ line width=0.15mm,-] (-1.2,-4) to (-1.2,-3)..controls ++(0,2) and ++(0,-.75)..(2.5,0);
\draw[ line width=0.15mm,-] (1.2,-4) to (1.2,-3)..controls ++(0,1) and ++(0,-2)..(-4.7,0);
\draw[ gray, -] (2.3,-4) to (2.3,-3)..controls ++(0,3) and ++(0,-.5)..(0.25,0);
\draw[ line width=0.15mm,-] (2.9,-4) to (2.9,-3)..controls ++(0,1) and ++(0,-1)..(3,0);
\draw[line width=0.15mm, -] (4.2,-4) to (4.2,-3)..controls ++(0,2) and ++(0,-1)..(-4.2,0);
\draw[ line width=0.15mm,-] (5.9,-4) to (5.9,-3)..controls ++(0,3) and ++(0,-.5)..(0.65,0);
\end{scope}
\draw[fill=white] (-5.5,3) rectangle (-3.5,4);
\draw[fill=white] (-3,3) rectangle (-1,4);\draw[fill=white] (1,3) rectangle (3,4);\draw[fill=white] (4,3) rectangle (6,4);
\draw[fill=gray!60] (-5.5,-1) rectangle (-4,0);
\draw[fill=gray!60] (-3.5,-1) rectangle (-2,0);\draw[fill=gray!60] (-0.75,-1) rectangle (.75,0);\draw[fill=gray!60] (2,-1) rectangle (3.5,0);
\draw[fill=gray!60] (4,-1) rectangle (5.5,0);
\node at (-4.5,3.5)[scale=.75]{$\lambda_1$};
\node at (-2,3.5)[scale=.75]{$\lambda_2$};
\node at (2,3.5)[scale=.75]{$\lambda_s$};
\node at (5,3.5)[scale=.75]{$\lambda_k$};
\node at (-4.75,-.5)[scale=.75]{$\lambda_1^*$};
\node at (-2.75,-.5)[scale=.75]{$\lambda_2^*$};
\node at (0,-.5)[scale=.75]{$\lambda_r^*$};
\node at (2.75,-.5)[scale=.75]{$\lambda_t^*$};
\node at (4.75,-.5)[scale=.75]{$\lambda_l^*$};
\node at (-4.5,-5.5) {$\underbrace{\hspace{0.35in}}$};
\node at (-2,-5.5) {$\underbrace{\hspace{0.35in}}$};
\node at (2.1,-5.5) {$\underbrace{\hspace{0.35in}}$};
\node at (5.1,-5.5) {$\underbrace{\hspace{0.35in}}$};
\node at (-4.5,-6)[scale=.7]{$\lambda_1$};
\node at (-2,-6)[scale=.7]{$\lambda_2$};
\node at (2,-6)[scale=.7]{$\lambda_s$};
\node at (5,-6)[scale=.7]{$\lambda_k$};
\node at (0,-5.5){$\dots$};
\node at (0,3.5){$\dots$};
\node at (-1.35,-.5){$\dots$};
\node at (1.45,-.5){$\dots$};
\end{tikzpicture}}
\end{equation}
Then the coefficient $\alpha_{\lambda} $ ensures the object is idempotent and is defined by the formula:
\[\alpha_{\lambda}:= \frac{{\Pi_{i=1}^k |\lambda_i|}\Pi_{j=1}^l |\lambda_j^*!|}{Y}\qquad\text{where} \qquad Y=\frac{n!}{\# \SYT(\lambda)}.\]
Clearly, $\alpha_{\lambda}$ is the same for all $\SYT(\lambda)$ and satisfies $\alpha_{\lambda} = \alpha_{\lambda^t}$. 
Such detail is rarely needed, so henceforth we will instead utilize the following notation to represent \eqref{diag:YoungIdempotentExploded}. Note that the coefficient $\alpha_\lambda$ is built into this notation.  
\begin{equation}\label{diag:YoungIdempotent}
\P^{\lambda}= \;\;
\hackcenter{
\begin{tikzpicture} [scale=0.6]
  \draw[ ->] (.2,-1) -- (.2,1);
  \draw[ ->] (1.8,-1) -- (1.8,1);
  \draw[fill = gray!20] (-0.0,-.35) rectangle (2,.35);
  \node at (1,0) {$\lambda$};
  \node at (1,-.7) {$\dots$};
    \node at (1,.7) {$\dots$};
\end{tikzpicture}}
\hspace{15mm}
\Q^{\lambda^t}=
\hackcenter{
\begin{tikzpicture} [scale=0.6]
  \draw[ <-] (.2,-1) -- (.2,1);
  \draw[ <-] (1.8,-1) -- (1.8,1);
  \draw[fill = gray!20] (-0.0,-.35) rectangle (2,.35);
  \node at (1,0) {$\lambda^t$};
  \node at (1,-.7) {$\dots$};
    \node at (1,.7) {$\dots$};
\end{tikzpicture}}
\end{equation}

Moreover, we will often clump strands together into "thick" lines $\hackcenter{\begin{tikz} \draw [line width= 0.7mm,
    decoration={markings,mark=at position 1 with {\arrow[line width=.3mm]{>}}},
    postaction={decorate},shorten >=0.6pt] (0,-.3)--(0,.3);\end{tikz}}$ to denote more than one strand at a time and reserve "thin" lines $\hackcenter{\begin{tikz} \draw[->] (0,-.3)--(0,.3);\end{tikz}}$ to indicate strands of thickness exactly equal to 1. 
\begin{example}
Suppose $\lambda = (a,b)$, so that $\lambda^t = (2^b, 1^{a-b})$. Then $\P^{\lambda}$ in (\ref{diag:YoungIdempotent}) is depicted by the diagram:
\begin{equation}
 \P^{(a,b)} \;\; = \;\;
\hackcenter{
\begin{tikzpicture} [scale=0.6]
  \draw[ ->] (.2,-1) -- (.2,1);
  \draw[ ->] (1.8,-1) -- (1.8,1);
  \draw[fill = gray!20] (-0.0,-.35) rectangle (2.2,.35);
  \node at (1,0) [scale=.75]{$(a,b)$};
  \node at (1,-.7) [scale=.75]{$\dots$};
    \node at (1,.7) [scale=.75]{$\dots$};
\end{tikzpicture}}
\;\; = \;\; \frac{(2!)^b (a-b+1)}{a+1}\;
 \hackcenter{
\begin{tikzpicture} [scale=0.6, yscale=.8]
 \draw[ <-] (-2.8,2.1) to (-2.8,-3.1);
 \draw[ <-] (-2.2,2.1) to (-2.2,1.4) .. controls ++(0,-1) and ++(0,1) .. (-1.5,-1)
        .. controls ++(0,-1) and ++(0,1) .. (-2.2,-3.1)  ;
\node at (-1.6,2.0) [scale=.75]{$\dots$};
\node at (-1.6,1.0) [scale=.75]{$\dots$};
\node at (-1.6,-2.9) [scale=.75]{$\dots$};
\draw[ <-] (-1,2.1) to (-1,1.4) .. controls ++(0,-1) and ++(0,1) .. (.3,-1)
        .. controls ++(0,-1) and ++(0,1) .. (-1,-3.1) ;
 \draw[ <-] (-.3,2.1) to (-.3,1.4) .. controls ++(0,-1) and ++(0,1) .. (1.7,-1)
        .. controls ++(0,-1) and ++(0,1) .. (-.3,-3.1);
\node at (.3,2.0) [scale=.75]{$\dots$};
\node at (.3,1.0) [scale=.75]{$\dots$};
\node at (.3,-2.9) [scale=.75]{$\dots$};
\draw[ <-] (.9,2.1) to (.9,1.4) .. controls ++(0,-1) and ++(0,1) .. (2.7,-1)
        .. controls ++(0,-1) and ++(0,1) .. (.9,-3.1);
  \draw[ ->] (1.8,-3.1).. controls ++(0,.75) and ++(0,-.75) ..(-2.2,-1.2) to (-2.2,-.7)
            .. controls ++(0,.5) and ++(0,-.5) .. (1.8,1.2) to (1.8,2.1);
  \draw[ ->] (2.4,-3.1).. controls ++(0,1) and ++(0,-.5) ..(-.9,-1.2) to (-.9,-.7) .. controls ++(0,.5) and ++(0,-.85) .. (2.4,1.2) to (2.4,2.1);
   \node at (2.9,2.0)[scale=.75] {$\dots$};
  \draw[ ->] (3.4,-3.1) .. controls ++(0,.85) and ++(0,-.5) ..(.9,-1.2) to (.9,-.7) .. controls ++(0,.5) and ++(0,-.85) .. (3.4,1.2) to (3.4,2.1);
 \begin{scope}[yscale=1.4, shift={(0,-.5)}]
  \draw[fill =white] (-3,1.3) rectangle (1.2,1.8);
  \draw[fill =white] (1.6,1.3) rectangle (3.6,1.8);
  \node at (-1.0,1.55)[scale=.75] {$a$};
    \node at (2.55,1.55)[scale=.75] {$b$};
    \end{scope}
  \begin{scope}[yscale=1.4, shift={(0,.3)}]
  \draw[fill =gray!60] (.1,-1.2) rectangle (1.1,-0.7);
  \draw[fill =gray!60] (-1.7,-1.2) rectangle (-.7,-0.7);
  \draw[fill =gray!60] (-3,-1.2) rectangle (-2,-0.7);
    \node at (-2.5,-.95)[scale=.75] {$2$};
    \node at (-1.2,-.95)[scale=.75] {$2$};
    \node at (.6,-.95) [scale=.75]{$2$};
    \end{scope}
     \node at (-.3,-1) [scale=.75]{$\dots$};
  \node at ( .35,-3.3) [scale=.75]{$\underbrace{\hspace{0.5in}}$};
  \node at (-1.85,-3.3) [scale=.75]{$\underbrace{\hspace{0.65in}}$};
  \node at (2.65,-3.3) [scale=.75]{$\underbrace{\hspace{0.6in}}$};
  \node at ( .35,-3.65) [scale=.75]{$ a-b$};
  \node at (-1.85,-3.65) [scale=.75]{$ b$};
  \node at (2.65,-3.65) [scale=.75]{$ b$};
  \end{tikzpicture}}
\end{equation}
\end{example}
Symmetrizers and antisymmetrizers are elements of the group algebra of the symmetric group and satisfy the following relations (see \cite{Birdtracks, Garcia}):
\begin{equation}\label{eq:sym-anti-absorbsion}
\hackcenter{\begin{tikzpicture}[scale=.3, yscale=-1,xscale=1.2, yscale=.8,every node/.style={scale=0.6}]
\node at (1.3,.8) {$\cdots$};
\node at (1.3,-2.6) {$\cdots$};
\node at (0,.8) {$\cdots$};
\node at (0,-2.6) {$\cdots$};
\node at (-1.3,.8) {$\cdots$};
\node at (-1.3,-2.6) {$\cdots$};
\draw (-1.8, -2.8)--(-1.8,1.2);
\draw (1.8, -2.8)--(1.8,1.2);
\draw (-.8, -2.8)--(-.8,1.2);
\draw (.8, -2.8)--(.8,1.2);
\draw[fill=white] (-2,-.5) rectangle (2,.5);
\draw[fill=white] (-1,-2.2) rectangle (1,-1.2);
\end{tikzpicture}}
\;\;=\;\;
\hackcenter{\begin{tikzpicture}[scale=.3, xscale=1.2, yscale=.8,every node/.style={scale=0.6}]
\node at (1.3,.8) {$\cdots$};
\node at (1.3,-2.6) {$\cdots$};
\node at (0,.8) {$\cdots$};
\node at (0,-2.6) {$\cdots$};
\node at (-1.3,.8) {$\cdots$};
\node at (-1.3,-2.6) {$\cdots$};
\draw (-1.8, -2.8)--(-1.8,1.2);
\draw (1.8, -2.8)--(1.8,1.2);
\draw (-.8, -2.8)--(-.8,1.2);
\draw (.8, -2.8)--(.8,1.2);
\draw[fill=white] (-2,-1.3) rectangle (2,-.3);
\end{tikzpicture}}
\qquad\qquad\text{and}\qquad\qquad
\hackcenter{\begin{tikzpicture}[scale=.3, yscale=-1,xscale=1.2, yscale=.8,every node/.style={scale=0.6}]
\node at (1.3,.8) {$\cdots$};
\node at (1.3,-2.6) {$\cdots$};
\node at (0,.8) {$\cdots$};
\node at (0,-2.6) {$\cdots$};
\node at (-1.3,.8) {$\cdots$};
\node at (-1.3,-2.6) {$\cdots$};
\draw (-1.8, -2.8)--(-1.8,1.2);
\draw (1.8, -2.8)--(1.8,1.2);
\draw (-.8, -2.8)--(-.8,1.2);
\draw (.8, -2.8)--(.8,1.2);
\draw[fill=gray!60] (-2,-.5) rectangle (2,.5);
\draw[fill=gray!60] (-1,-2.2) rectangle (1,-1.2);
\end{tikzpicture}}
\;\;=\;\;
\hackcenter{\begin{tikzpicture}[scale=.3, xscale=1.2, yscale=.8,every node/.style={scale=0.6}]
\node at (1.3,.8) {$\cdots$};
\node at (1.3,-2.6) {$\cdots$};
\node at (0,.8) {$\cdots$};
\node at (0,-2.6) {$\cdots$};
\node at (-1.3,.8) {$\cdots$};
\node at (-1.3,-2.6) {$\cdots$};
\draw (-1.8, -2.8)--(-1.8,1.2);
\draw (1.8, -2.8)--(1.8,1.2);
\draw (-.8, -2.8)--(-.8,1.2);
\draw (.8, -2.8)--(.8,1.2);
\draw[fill=gray!60] (-2,-1.3) rectangle (2,-.3);
\end{tikzpicture}}
\end{equation}

\begin{equation}\label{eq:sym-anti-cross-absorbsion}
\hackcenter{\begin{tikzpicture}[scale=.3, xscale=1.2, yscale=.8,every node/.style={scale=0.5}]
\node at (1.3,.8) {$\cdots$};
\node at (1.3,-2.6) {$\cdots$};
\node at (-1.3,.8) {$\cdots$};
\node at (-1.3,-2.6) {$\cdots$};
\begin{scope}[shift={(1.6,0)}]
\draw (-1.4, -2.8)--(-1.4, 0)--(-1.8,1.2);
\draw (-1.8, -2.8)--(-1.8, 0)--(-1.4,1.2);
\end{scope}
\draw (-1.8, -2.8)--(-1.8,1.2);
\draw (1.8, -2.8)--(1.8,1.2);
\draw (-1, -2.8)--(-1,1.2);
\draw (1, -2.8)--(1,1.2);
\draw[fill=white] (-2,-1.3) rectangle (2,-.3);
\end{tikzpicture}}
\;\;=\;\;
\hackcenter{\begin{tikzpicture}[scale=.3, xscale=1.2, yscale=.8,every node/.style={scale=0.5}]
\node at (1.3,.8) {$\cdots$};
\node at (1.3,-2.6) {$\cdots$};
\node at (0,.8) {$\cdots$};
\node at (0,-2.6) {$\cdots$};
\node at (-1.3,.8) {$\cdots$};
\node at (-1.3,-2.6) {$\cdots$};
\draw (-1.8, -2.8)--(-1.8,1.2);
\draw (1.8, -2.8)--(1.8,1.2);
\draw (-.8, -2.8)--(-.8,1.2);
\draw (.8, -2.8)--(.8,1.2);
\draw[fill=white] (-2,-1.3) rectangle (2,-.3);
\end{tikzpicture}}
\qquad\qquad\text{and}\qquad\qquad
\hackcenter{\begin{tikzpicture}[scale=.3, xscale=1.2, yscale=.8,every node/.style={scale=0.5}]
\node at (1.3,.8) {$\cdots$};
\node at (1.3,-2.6) {$\cdots$};
\node at (-1.3,.8) {$\cdots$};
\node at (-1.3,-2.6) {$\cdots$};
\begin{scope}[shift={(1.6,0)}]
\draw (-1.4, -2.8)--(-1.4, 0)--(-1.8,1.2);
\draw (-1.8, -2.8)--(-1.8, 0)--(-1.4,1.2);
\end{scope}
\draw (-1.8, -2.8)--(-1.8,1.2);
\draw (1.8, -2.8)--(1.8,1.2);
\draw (-1, -2.8)--(-1,1.2);
\draw (1, -2.8)--(1,1.2);
\draw[fill=gray!60] (-2,-1.3) rectangle (2,-.3);
\end{tikzpicture}}
\;\;=\;\;-\;\;
\hackcenter{\begin{tikzpicture}[scale=.3, xscale=1.2, yscale=.8,every node/.style={scale=0.5}]
\node at (1.3,.8) {$\cdots$};
\node at (1.3,-2.6) {$\cdots$};
\node at (0,.8) {$\cdots$};
\node at (0,-2.6) {$\cdots$};
\node at (-1.3,.8) {$\cdots$};
\node at (-1.3,-2.6) {$\cdots$};
\draw (-1.8, -2.8)--(-1.8,1.2);
\draw (1.8, -2.8)--(1.8,1.2);
\draw (-.8, -2.8)--(-.8,1.2);
\draw (.8, -2.8)--(.8,1.2);
\draw[fill=gray!60] (-2,-1.3) rectangle (2,-.3);
\end{tikzpicture}}
\end{equation}

\begin{equation}\label{eq:whiteblackbox}
\hackcenter{\begin{tikzpicture}[scale=.3, xscale=1.2, yscale=.8,every node/.style={scale=0.6}]
\node at (1.3,.8) {$\cdots$};
\node at (1.3,-2.6) {$\cdots$};
\node at (0,.8) {$\cdots$};
\node at (0,-2.6) {$\cdots$};
\node at (-1.3,.8) {$\cdots$};
\node at (-1.3,-2.6) {$\cdots$};
\draw (-1.8, -2.8)--(-1.8,1.2);
\draw (1.8, -2.8)--(1.8,1.2);
\draw (-.8, -2.8)--(-.8,1.2);
\draw (.8, -2.8)--(.8,1.2);
\draw[fill=white] (-2,-.5) rectangle (2,.5);
\draw[fill=gray!60] (-1,-2.2) rectangle (1,-1.2);
\end{tikzpicture}}
\;\;=0=\;\;
\hackcenter{\begin{tikzpicture}[scale=.3, xscale=1.2, yscale=.8,every node/.style={scale=0.6}]
\node at (1.3,.8) {$\cdots$};
\node at (1.3,-2.6) {$\cdots$};
\node at (0,.8) {$\cdots$};
\node at (0,-2.6) {$\cdots$};
\node at (-1.3,.8) {$\cdots$};
\node at (-1.3,-2.6) {$\cdots$};
\draw (-1.8, -2.8)--(-1.8,1.2);
\draw (1.8, -2.8)--(1.8,1.2);
\draw (-.8, -2.8)--(-.8,1.2);
\draw (.8, -2.8)--(.8,1.2);
\draw[fill=gray!60] (-2,-.5) rectangle (2,.5);
\draw[fill=white] (-1,-2.2) rectangle (1,-1.2);
\end{tikzpicture}}
\end{equation}

\begin{align}\label{eq:sym-exploded}
\hackcenter{\begin{tikzpicture}[scale=.3, xscale=1.2, yscale=.8,every node/.style={scale=0.6}]
\node at (1.3,.8) {$\cdots$};
\node at (1.3,-2.6) {$\cdots$};
\node at (0,.8) {$\cdots$};
\node at (0,-2.6) {$\cdots$};
\node at (-1.3,.8) {$\cdots$};
\node at (-1.3,-2.6) {$\cdots$};
\draw (-1.8, -2.8)--(-1.8,1.2);
\draw (1.8, -2.8)--(1.8,1.2);
\draw (-.8, -2.8)--(-.8,1.2);
\draw (.8, -2.8)--(.8,1.2);
\draw[fill=white] (-2,-1.3) rectangle (2,-.3);
\node at (0,-.8){$n$};
\end{tikzpicture}}
&\;\;=\;\;\frac{1}{n}\;\;\left\lbrace\;\;
\hackcenter{\begin{tikzpicture}[scale=.3, xscale=1.2, yscale=.8,every node/.style={scale=0.6}]
\node at (1.3,.8) {$\cdots$};
\node at (1.3,-2.6) {$\cdots$};
\node at (0.3,.8) {$\cdots$};
\node at (0.3,-2.6) {$\cdots$};
\node at (-.8,.8) {$\cdots$};
\node at (-0.8,-2.6) {$\cdots$};
\draw (-1.8, -2.8)--(-1.8,1.2);
\draw (-1.4, -2.8)--(-1.4,1.2);
\draw (1.8, -2.8)--(1.8,1.2);
\draw (-.3, -2.8)--(-.3,1.2);
\draw (.8, -2.8)--(.8,1.2);
\draw[fill=white] (-1.5,-1.3) rectangle (2,-.3);
\node at (.2,-.8){$n-1$};
\end{tikzpicture}}
\;\;+\;\;n-1\;
\hackcenter{\begin{tikzpicture}[scale=.3, xscale=1.2, yscale=.8,every node/.style={scale=0.6}]
\node at (1.3,.8) {$\cdots$};
\node at (1.3,-2.6) {$\cdots$};
\node at (0.3,.8) {$\cdots$};
\node at (0.3,-2.6) {$\cdots$};
\node at (-.8,.8) {$\cdots$};
\node at (-0.8,-2.6) {$\cdots$};
\draw (-1.4, -2.8)--(-1.4, -1.3)--(-1.8,-.5)--(-1.8,1.2);
\draw (-1.8, -2.8)--(-1.8, -1.3)--(-1.4,-.5)--(-1.4,1.2);
\draw (1.8, -2.8)--(1.8,1.2);
\draw (-.3, -2.8)--(-.3,1.2);
\draw (.8, -2.8)--(.8,1.2);
\draw[fill=white] (-1.5,-.3) rectangle (2,.7);
\begin{scope}[shift={(0,-2)}]
\draw[fill=white] (-1.5,-.5) rectangle (2,.5);
\node at (.2,0){$n-1$};
\end{scope}
\node at (.2,0.2){$n-1$};
\end{tikzpicture}}\;\;\right\rbrace
\\
&\;\;=\;\;\frac{1}{n}\;\;\left\lbrace\;\;
\hackcenter{\begin{tikzpicture}[scale=.3, xscale=1.2, yscale=.8,every node/.style={scale=0.6}]
\node at (1.3,.8) {$\cdots$};
\node at (1.3,-2.6) {$\cdots$};
\node at (0.3,.8) {$\cdots$};
\node at (0.3,-2.6) {$\cdots$};
\node at (-.8,.8) {$\cdots$};
\node at (-0.8,-2.6) {$\cdots$};
\draw (-1.8, -2.8)--(-1.8,1.2);
\draw (-1.4, -2.8)--(-1.4,1.2);
\draw (1.8, -2.8)--(1.8,1.2);
\draw (-.3, -2.8)--(-.3,1.2);
\draw (.8, -2.8)--(.8,1.2);
\draw[fill=white] (-1.5,-.5) rectangle (2,.5);
\node at (.2,0){$n-1$};
\end{tikzpicture}}
\;\;+\;\;
\hackcenter{\begin{tikzpicture}[scale=.3, xscale=1.2, yscale=.8,every node/.style={scale=0.6}]
\node at (1.3,.8) {$\cdots$};
\node at (1.3,-2.6) {$\cdots$};
\node at (0.3,.8) {$\cdots$};
\node at (0.3,-2.6) {$\cdots$};
\node at (-.8,.8) {$\cdots$};
\node at (-0.8,-2.6) {$\cdots$};
\draw (-1.4, -2.8)--(-1.8,-.5)--(-1.8,1.2);
\draw (-1.8, -2.8)--(-1.4,-.5)--(-1.4,1.2);
\draw (1.8, -2.8)--(1.8,1.2);
\draw (-.3, -2.8)--(-.3,1.2);
\draw (.8, -2.8)--(.8,1.2);
\draw[fill=white] (-1.5,-.5) rectangle (2,.5);
\node at (.2,0){$n-1$};
\end{tikzpicture}}
\;\;+\;\;
\hackcenter{\begin{tikzpicture}[scale=.3, xscale=1.2, yscale=.8,every node/.style={scale=0.6}]
\node at (1.3,.8) {$\cdots$};
\node at (1.3,-2.6) {$\cdots$};
\node at (0.3,.8) {$\cdots$};
\node at (0.3,-2.6) {$\cdots$};
\node at (-.8,.8) {$\cdots$};
\node at (-0.8,-2.6) {$\cdots$};
\draw (-1.1, -2.8)--(-1.8,-0.5)--(-1.8,1.2);
\draw (-1.8, -2.8)--(-1.4,-.5)--(-1.4,1.2);
\draw (-1.5, -2.8)--(-1.1,-.5)--(-1.1,1.2);
\draw (1.8, -2.8)--(1.8,1.2);
\draw (-.3, -2.8)--(-.3,1.2);
\draw (.8, -2.8)--(.8,1.2);
\draw[fill=white] (-1.5,-.5) rectangle (2,.5);
\node at (.2,0){$n-1$};
\end{tikzpicture}}
\;\;+ \; \cdots \; +\;\;
\hackcenter{\begin{tikzpicture}[scale=.3, xscale=1.2, yscale=.8,every node/.style={scale=0.6}]
\node at (1.3,.8) {$\cdots$};
\node at (1.1,-2.6) {$\cdots$};
\node at (0.3,.8) {$\cdots$};
\node at (0.1,-2.6) {$\cdots$};
\node at (-.6,.8) {$\cdots$};
\node at (-1,-2.6) {$\cdots$};
\draw (2,-2.8)..controls++(0,0) and ++(.6,-0.3)..(1.5, -2.4)--(-1.6,-1)..controls++(-.3,0.2) and ++(0,0)..(-1.8,1.2);
\draw (-1.8, -2.8)--(-1.4,-.5)--(-1.4,1.2);
\draw (-1.5, -2.8)--(-1.1,-.5)--(-1.1,1.2);
\draw (1.5, -2.8)--(1.8,-0.5)--(1.8,1.2);
\draw (-.6, -2.8)--(-.3,-0.5)--(-.3,1.2);
\draw (.5, -2.8)--(.8,-0.5)--(.8,1.2);
\draw[fill=white] (-1.5,-.5) rectangle (2,.5);
\node at (.2,0){$n-1$};
\end{tikzpicture}}\;\;\right\rbrace
\end{align}

\begin{align}\label{eq:anti-exploded}
\hackcenter{\begin{tikzpicture}[scale=.3, xscale=1.2, yscale=.8,every node/.style={scale=0.6}]
\node at (1.3,.8) {$\cdots$};
\node at (1.3,-2.6) {$\cdots$};
\node at (0,.8) {$\cdots$};
\node at (0,-2.6) {$\cdots$};
\node at (-1.3,.8) {$\cdots$};
\node at (-1.3,-2.6) {$\cdots$};
\draw (-1.8, -2.8)--(-1.8,1.2);
\draw (1.8, -2.8)--(1.8,1.2);
\draw (-.8, -2.8)--(-.8,1.2);
\draw (.8, -2.8)--(.8,1.2);
\draw[fill=gray!60] (-2,-1.3) rectangle (2,-.3);
\node at (0,-.8){$n$};
\end{tikzpicture}}
&\;\;=\;\;\frac{1}{n}\;\;\left\lbrace\;\;
\hackcenter{\begin{tikzpicture}[scale=.3, xscale=1.2, yscale=.8,every node/.style={scale=0.6}]
\node at (1.3,.8) {$\cdots$};
\node at (1.3,-2.6) {$\cdots$};
\node at (0.3,.8) {$\cdots$};
\node at (0.3,-2.6) {$\cdots$};
\node at (-.8,.8) {$\cdots$};
\node at (-0.8,-2.6) {$\cdots$};
\draw (-1.8, -2.8)--(-1.8,1.2);
\draw (-1.4, -2.8)--(-1.4,1.2);
\draw (1.8, -2.8)--(1.8,1.2);
\draw (-.3, -2.8)--(-.3,1.2);
\draw (.8, -2.8)--(.8,1.2);
\draw[fill=gray!60] (-1.5,-1.3) rectangle (2,-.3);
\node at (.2,-.8){$n-1$};
\end{tikzpicture}}
\;\;-\;\;n-1\;
\hackcenter{\begin{tikzpicture}[scale=.3, xscale=1.2, yscale=.8,every node/.style={scale=0.6}]
\node at (1.3,.8) {$\cdots$};
\node at (1.3,-2.6) {$\cdots$};
\node at (0.3,.8) {$\cdots$};
\node at (0.3,-2.6) {$\cdots$};
\node at (-.8,.8) {$\cdots$};
\node at (-0.8,-2.6) {$\cdots$};
\draw (-1.4, -2.8)--(-1.4, -1.3)--(-1.8,-.5)--(-1.8,1.2);
\draw (-1.8, -2.8)--(-1.8, -1.3)--(-1.4,-.5)--(-1.4,1.2);
\draw (1.8, -2.8)--(1.8,1.2);
\draw (-.3, -2.8)--(-.3,1.2);
\draw (.8, -2.8)--(.8,1.2);
\draw[fill=gray!60] (-1.5,-.3) rectangle (2,.7);
\begin{scope}[shift={(0,-2)}]
\draw[fill=gray!60] (-1.5,-.5) rectangle (2,.5);
\node at (.2,0){$n-1$};
\end{scope}
\node at (.2,0.2){$n-1$};
\end{tikzpicture}}\;\;\right\rbrace
\\
&\;\;=\;\;\frac{1}{n}\;\;\left\lbrace\;\;
\hackcenter{\begin{tikzpicture}[scale=.3, xscale=1.2, yscale=.8,every node/.style={scale=0.6}]
\node at (1.3,.8) {$\cdots$};
\node at (1.3,-2.6) {$\cdots$};
\node at (0.3,.8) {$\cdots$};
\node at (0.3,-2.6) {$\cdots$};
\node at (-.8,.8) {$\cdots$};
\node at (-0.8,-2.6) {$\cdots$};
\draw (-1.8, -2.8)--(-1.8,1.2);
\draw (-1.4, -2.8)--(-1.4,1.2);
\draw (1.8, -2.8)--(1.8,1.2);
\draw (-.3, -2.8)--(-.3,1.2);
\draw (.8, -2.8)--(.8,1.2);
\draw[fill=gray!60] (-1.5,-.5) rectangle (2,.5);
\node at (.2,0){$n-1$};
\end{tikzpicture}}
\;\;-\;\;
\hackcenter{\begin{tikzpicture}[scale=.3, xscale=1.2, yscale=.8,every node/.style={scale=0.6}]
\node at (1.3,.8) {$\cdots$};
\node at (1.3,-2.6) {$\cdots$};
\node at (0.3,.8) {$\cdots$};
\node at (0.3,-2.6) {$\cdots$};
\node at (-.8,.8) {$\cdots$};
\node at (-0.8,-2.6) {$\cdots$};
\draw (-1.4, -2.8)--(-1.8,-.5)--(-1.8,1.2);
\draw (-1.8, -2.8)--(-1.4,-.5)--(-1.4,1.2);
\draw (1.8, -2.8)--(1.8,1.2);
\draw (-.3, -2.8)--(-.3,1.2);
\draw (.8, -2.8)--(.8,1.2);
\draw[fill=gray!60] (-1.5,-.5) rectangle (2,.5);
\node at (.2,0){$n-1$};
\end{tikzpicture}}
\;\;+\;\;
\hackcenter{\begin{tikzpicture}[scale=.3, xscale=1.2, yscale=.8,every node/.style={scale=0.6}]
\node at (1.3,.8) {$\cdots$};
\node at (1.3,-2.6) {$\cdots$};
\node at (0.3,.8) {$\cdots$};
\node at (0.3,-2.6) {$\cdots$};
\node at (-.8,.8) {$\cdots$};
\node at (-0.8,-2.6) {$\cdots$};
\draw (-1.1, -2.8)--(-1.8,-0.5)--(-1.8,1.2);
\draw (-1.8, -2.8)--(-1.4,-.5)--(-1.4,1.2);
\draw (-1.5, -2.8)--(-1.1,-.5)--(-1.1,1.2);
\draw (1.8, -2.8)--(1.8,1.2);
\draw (-.3, -2.8)--(-.3,1.2);
\draw (.8, -2.8)--(.8,1.2);
\draw[fill=gray!60] (-1.5,-.5) rectangle (2,.5);
\node at (.2,0){$n-1$};
\end{tikzpicture}}
\;\;-\; \cdots \; +(-1)^{n-1}\;\;
\hackcenter{\begin{tikzpicture}[scale=.3, xscale=1.2, yscale=.8,every node/.style={scale=0.6}]
\node at (1.3,.8) {$\cdots$};
\node at (1.1,-2.6) {$\cdots$};
\node at (0.3,.8) {$\cdots$};
\node at (0.1,-2.6) {$\cdots$};
\node at (-.6,.8) {$\cdots$};
\node at (-1,-2.6) {$\cdots$};
\draw (2,-2.8)..controls++(0,0) and ++(.6,-0.3)..(1.5, -2.4)--(-1.6,-1)..controls++(-.3,0.2) and ++(0,0)..(-1.8,1.2);
\draw (-1.8, -2.8)--(-1.4,-.5)--(-1.4,1.2);
\draw (-1.5, -2.8)--(-1.1,-.5)--(-1.1,1.2);
\draw (1.5, -2.8)--(1.8,-0.5)--(1.8,1.2);
\draw (-.6, -2.8)--(-.3,-0.5)--(-.3,1.2);
\draw (.5, -2.8)--(.8,-0.5)--(.8,1.2);
\draw[fill=gray!60] (-1.5,-.5) rectangle (2,.5);
\node at (.2,0){$n-1$};
\end{tikzpicture}}\;\;\right\rbrace
\end{align}

\subsection{Littlewood-Richardson Branching Isomorphisms}\label{subsec:LRBranchingIsos}
We begin with a classic result from the representation theory of the symmetric group. We refer the reader to \cite{Fulton-YoungTableux} for more details and \cite[Section 9.4.2]{Birdtracks} for a proof of the following theorem.

\begin{theorem}\label{thm:CompleteIdemp}
Given partitions $\lambda,\mu \vdash n$, let $\Delta_\lambda= \#STY(\lambda)$ be the dimension of the irreducible representation of $S_n$ corresponding to $\lambda$. For any $T \in$ \emph{SYT}($\lambda$) and $T' \in$ \emph{SYT}($\mu$) the following relations holds in $\H$:
\begin{equation}\label{eq:Pdecomp}\P^{\otimes n} = \bigoplus_{\substack{\lambda \vdash n\\ T \in \emph{SYT}(\lambda)}} \P^{T} \simeq \;\;\bigoplus_{\lambda \vdash n} (\P^{\lambda})^{\oplus \Delta_\lambda}
\end{equation}
\begin{equation}\label{eq:PPidemp}
\1_{\P^{T'}}\circ 1_{\P^{T}} \simeq \begin{cases} 1_{\P^{T}} & \text{if } \lambda = \mu $ and $ T =T',\\
0 & \text{else.}\end{cases}
\end{equation}
\end{theorem}

Given any two standard fillings $T,T'$ of $\lambda,\mu \vdash n$ let $\delta$ denote the Kronecker delta of the corresponding standard Young tableau. Then equations \eqref{eq:Pdecomp} and \eqref{eq:PPidemp} are diagrammatically the following statements. 
\begin{align}
\hackcenter{
\begin{tikzpicture} [scale=0.5, every node/.style={scale=0.75}]
  \draw[, ->] (.2,-1) -- (.2,1);
   \draw[, ->] (1,-1) -- (1,1);
  \draw[, ->] (1.8,-1) -- (1.8,1);
   \node at (.6,0) {$\dots$};
   \node at (1.4,0) {$\dots$};
  \node at (1,-1.3) {$\underbrace{\hspace{20mm}}$};
  \node at (1,-1.7) {$n$};
\end{tikzpicture}}
&= \;\;\bigoplus_{\substack{\lambda \vdash n\\ T \in \emph{SYT}(\lambda)}}\hspace{2mm}
\hackcenter{
\begin{tikzpicture} [scale=0.5, every node/.style={scale=0.75}]
\begin{scope}[shift={(0,1)}]
  \draw[, ->] (.2,-1) -- (.2,1);
  \draw[, ->] (1.8,-1) -- (1.8,1);
  \draw[fill =  gray!20] (-0.0,-.35) rectangle (2,.35);
  \node at (1,0) {$T$};
  \node at (1,-.7) {$\dots$};
    \node at (1,.7) {$\dots$};
 \end{scope}
\end{tikzpicture}}\label{eq:P=sumLambda}
\\
 \hackcenter{
\begin{tikzpicture} [scale=0.5, every node/.style={scale=0.75}]
  \draw[, -] (.2,-1.5) -- (.2,1.5);
  \draw[, -] (.8,-1.5) -- (.8,1.5);
  \draw[, -] (1.8,-1.5) -- (1.8,1.5);
\draw[, -] (2.4,-1.5) -- (2.4,1.5);
  \draw[fill =  gray!20] (-0.2,.3) rectangle (2.8,1);
  \draw[fill =  gray!20] (-0.2,-.3) rectangle (2.8,-1);
  \node at (1.3,.65) {$T$};
   \node at (1.3,-.65) {$T'$};
  \node at (1.3,-1.2) {$\dots$};
    \node at (1.3,1.3) {$\dots$};
\end{tikzpicture}}
\;\; &= \;\; \delta_{\lambda, \mu} \delta_{T,T'}\;
\hackcenter{
\begin{tikzpicture} [scale=0.5, every node/.style={scale=0.75}]
  \draw[, -] (.2,-1) -- (.2,1);
  \draw[, -] (.8,-1) -- (.8,1);
  \draw[, -] (1.8,-1) -- (1.8,1);
  \draw[, -] (2.4,-1) -- (2.4,1);
  \draw[fill =  gray!20] (-0.2,-.35) rectangle (2.8,.35);
  \node at (1.3,0) {$T$};
  \node at (1.3,-.7) {$\dots$};
    \node at (1.3,.7) {$\dots$};
\end{tikzpicture}}\label{eq:PlambdaOrthogonal}
\end{align}

\begin{proposition}\label{prop:QPdecomposition}
Given any $n,m \in \N$, the following relation holds in $\H'$:
\begin{equation}\label{eq:sMatchings}
\hackcenter{\begin{tikzpicture}[scale=0.75,yscale=.6,xscale=1.2 ,every node/.style={scale=0.8, xscale=1.1}]
%
\node at (-1.6,3.1) {$\overbrace{\hspace{29mm}}$};
\node at (1.5,3.1) {$\overbrace{\hspace{30mm}}$};
\node at (-1.6,3.5){$n$};
\node at (1.5,3.5){$m$};
\node at (-1.6,-5.6) {$\underbrace{\hspace{29mm}}$};
\node at (1.5,-5.6) {$\underbrace{\hspace{30mm}}$};
\node at (-1.6,-6){$n$};
\node at (1.5,-6){$m$};
%
\node at (0,2.95){$\dots \dots\dots\dots\dots\dots\dots\dots\dots\dots\dots\dots\dots$};
\node at (0,-5.35){$\dots \dots\dots\dots\dots\dots\dots\dots\dots\dots\dots\dots\dots$};
\node at (-.85,-1){$\dots$};
\node at (-2.2,-1){$\dots$};
\node at (.85,-1){$\dots$};
\node at (2.2,-1){$\dots$};
\draw[>->](-3,3)--(-3,-5.4);
\draw[>->](-2.75,3)--(-2.75,-5.4);
\draw[>->](-1.5,3)--(-1.5,-5.4);
\draw[>->](-.3,3)--(-.3,-5.4);
\draw[<-<](3,3)--(3,-5.4);
\draw[<-<](2.75,3)--(2.75,-5.4);
\draw[<-<](1.5,3)--(1.5,-5.4);
\draw[<-<](.1,3)--(.1,-5.4);
\end{tikzpicture}}
=
\hackcenter{\begin{tikzpicture}
\node at (0,0)[scale=2, yscale=1.2]{$\sum$};
\node at (0,-.75){$\substack{0\leq s \leq \emph{min}(n,m)\\s-\emph{matchings}}$};
\end{tikzpicture}}
\hackcenter{\begin{tikzpicture}[scale=0.75, xscale=1.2, every node/.style={scale=0.8,xscale=1.1}]
%
\node at (-1.6,3.1) {$\overbrace{\hspace{29mm}}$};
\node at (1.5,3.1) {$\overbrace{\hspace{30mm}}$};
\node at (-1.6,3.5){$n$};
\node at (1.5,3.5){$m$};
\node at (-1.6,-5.6) {$\underbrace{\hspace{29mm}}$};
\node at (1.5,-5.6) {$\underbrace{\hspace{30mm}}$};
\node at (-1.6,-6){$n$};
\node at (1.5,-6){$m$};
\node at (-2.2,-1.1) {$\underbrace{\hspace{17mm}}$};
\node at (1.7,-1.1) {$\underbrace{\hspace{17mm}}$};
\node at (-2.2,-1.3){$m-s$};
\node at (1.7,-1.3){$n-s$};
%
\node at (0,2.95){$\dots \dots\dots\dots\dots\dots\dots\dots\dots\dots\dots\dots\dots$};
\node at (-1.9,1.25){$\dots$};
\node at (-1.2,1.25){$\dots$};
\node at (1.5,1.25){$\dots$};
\node at (.6,1.25){$\dots$};
\node at (-2.25,-.95){$\dots\dots\dots$};
\node at (1.7,-.95){$\dots\dots\dots$};
\draw [>->](-3,3)..controls++(0,-1) and ++(0,-1)..(1,3);
\draw [>->](-2.5,3)..controls++(0,-.75) and ++(0,-.75)..(3,3);
\draw [>->](-1,3)..controls++(0,-1.5) and ++(0,-1.5)..(2,3);
\draw [>->](-2,3)..controls++(0,-1.75) and ++(0,-1.75)..(1.25,3);
%
\draw [>->](-2.8,3)..controls++(0,-3) and ++(0,1.5)..(1,-1);
\draw [>->](-2.2,3)..controls++(0,-3) and ++(0,1.5)..(1.6,-1);
\draw [>->](-1.2,3)..controls++(0,-3) and ++(0,1.5)..(2.4,-1);
\begin{scope}[xscale=-1]
\draw [<-<](-2.3,3)..controls++(0,-3) and ++(0,1.5)..(1.5,-1);
\draw [<-<](-1.8,3)..controls++(0,-3) and ++(0,1.5)..(2.2,-1);
\draw [<-<](-.5,3)..controls++(0,-3) and ++(0,1.5)..(2.9,-1);
\end{scope}
\begin{scope}[yscale=-1, shift={(0,2.5)}]
\node at (0,2.95){$\dots \dots\dots\dots\dots\dots\dots\dots\dots\dots\dots\dots\dots$};
\node at (-1.9,1.25){$\dots$};
\node at (-1.2,1.25){$\dots$};
\node at (1.5,1.25){$\dots$};
\node at (.6,1.25){$\dots$};
\node at (-2.25,-.95){$\dots\dots\dots$};
\node at (1.7,-.95){$\dots\dots\dots$};
\draw [<-<](-3,3)..controls++(0,-1) and ++(0,-1)..(1,3);
\draw [<-<](-2.5,3)..controls++(0,-.75) and ++(0,-.75)..(3,3);
\draw [<-<](-1,3)..controls++(0,-1.5) and ++(0,-1.5)..(2,3);
\draw [<-<](-2,3)..controls++(0,-1.75) and ++(0,-1.75)..(1.25,3);
%
\draw [<-<](-2.8,3)..controls++(0,-3) and ++(0,1.5)..(1,-1);
\draw [<-<](-2.2,3)..controls++(0,-3) and ++(0,1.5)..(1.6,-1);
\draw [<-<](-1.2,3)..controls++(0,-3) and ++(0,1.5)..(2.4,-1);
\begin{scope}[xscale=-1]
\draw [>->](-2.3,3)..controls++(0,-3) and ++(0,1.5)..(1.5,-1);
\draw [>->](-1.8,3)..controls++(0,-3) and ++(0,1.5)..(2.2,-1);
\draw [>->](-.5,3)..controls++(0,-3) and ++(0,1.5)..(2.9,-1);
\end{scope}
\end{scope}
\end{tikzpicture}
}\end{equation}
where the sum ranges over all possible ways of selecting $s$ distinct $Q_-$ and $s$ distinct $Q_+$ and pairing them with a cap and cup for all values $0\leq s \leq$ min$(n,m)$. On the left, the map is the identity morphism for $Q_{-n}\otimes Q_{+m}$. On the right, each map is an endomorphism of $Q_{-n}\otimes Q_{+m}$ that factors through $Q_{+(m-s)}\otimes Q_{-(n-s)}$ for each value of $s$.
\medskip

Consequently, the following isomorphism holds in $\H$:
\[\Q^{\otimes n} \P^{\otimes m} \cong \bigoplus_{s=0}^{\emph{min}(n,m)} (s!){{n}\choose{s}}  {{m}\choose{s}} \P^{\otimes (m-s)} \Q^{\otimes (n-s)}\]
\end{proposition}

\begin{proof}
The proof follows from a straightforward inductive argument on $n$ and $m$. The case $n=m=1$ is relation \eqref{heis:downup}. Inducting first on $n$ consists of tensoring $\Q\P$ on the left by $\Q$. In this way, we see that the identity morphism of $\Q^{\otimes n}\P$ equals the sum of $n+1$ maps. The first is the obvious crossing map, which crosses the upward strand across all downward strands to its left and then back. The remaining $n$ maps are given by all $n$ distinct ways of paring $\P$ with a $\Q$ by a cup on top and a cap on the bottom. The second induction follows identically by tensoring $\Q^{\otimes n} \P$ with $\P$ on the right $m$ times and inducting on $m$. 

In $\H$ the top half of the diagrams are morphisms between $\P^{\otimes (m-s)} \Q^{\otimes (n-s)}$ and $\Q^{\otimes n} \P^{\otimes m}$ and the bottom half are maps in the reverse direction. Since for any $0\leq s \leq \text{min}(n,m)$ the number of $s$-matchings between $\P$'s and $\Q$'s is precisely $(s!){{n}\choose{s}}  {{m}\choose{s}}$, the relation in $\H$ follows. 
\end{proof}

\noindent \textbf{Notation:} In the remainder of the paper we make extensive use of the following notation:
\begin{itemize}
\item Given any object $\mathsf{X} \in \H$, denote by $\1_{\mathsf{X}}$ the identity morphism on $\mathsf{X}$ and depict it by vertical oriented lines with no crossings. 
\item For any indexing set $I$ and any two indexes $s,t \in I$, denote by $\delta_{s,t}$ the Kronecker delta function. 
\end{itemize}

With these isomorphisms in hand, we will now present various relations for the tensor products of $\Q^\lambda$ and $\P^\mu$ for varying partitions $\lambda$ and $\mu$. We begin by a Proposition due to Khovanov {\cite[Section 2.2]{Kh-H}}. 

\begin{proposition}\label{prop:Q*Pswap}
For any integers $n,m \geq 1$, define the morphisms $\rho_1, \rho_2, \iota_1, \iota_2$ in $\H$ as follows:
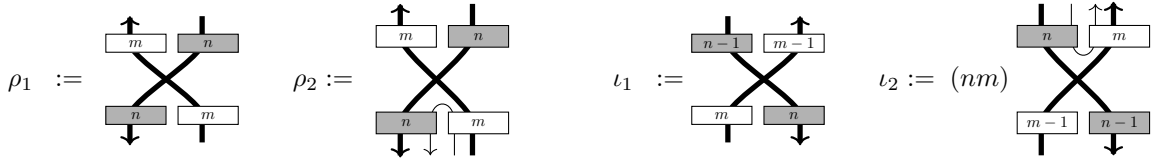
\begin{align*}
\rho_1\;\; := \;\;
\hackcenter{
\begin{tikzpicture}[scale=0.4, every node/.style={scale=0.6}]
 \draw[line width= 0.7mm,
    decoration={markings,mark=at position 1 with {\arrow[line width=.3mm]{>}}},
    postaction={decorate},shorten >=0.6pt] (1,-1.1) to (1,-.2) .. controls ++(0,.5) and ++(0,-.5) .. (-1.4,2.2) to (-1.4,3.1);
  \draw[line width= 0.7mm,
    decoration={markings,mark=at position 0.04 with {\arrow[line width=.3mm]{<}}},
    postaction={decorate}] (-1.4,-1.1) to (-1.4,-.2) .. controls ++(0,.5) and ++(0,-.5) .. (1,2.2) to (1,3.1);
\begin{scope}[shift={(0,-.3)}]
   \draw[fill =white] (-2.2,2.2) rectangle (-0.2,2.8);
    \draw[fill =gray!60] (.2,2.2) rectangle (2.2,2.8);
     \node at (1.2,2.5) {$n$};
      \node at (-1.2,2.5) {$m$};
  \end{scope}
  \begin{scope}[shift={(0,.3)}]
  \draw[fill =white] (.2,-.8) rectangle (2.2,-.2);
  \draw[fill =gray!60] (-2.2,-.8) rectangle (-.2,-.2);
  \node at (1.2,-.5) {$m$}; 
  \node at (-1.2,-.5) {$n$}; 
  \end{scope}
\end{tikzpicture}}
& \qquad\rho_2:=\;\;
\hackcenter{
\begin{tikzpicture}[scale=0.4, every node/.style={scale=0.6}, yscale=1.2]
 \draw[line width= 0.7mm,
    decoration={markings,mark=at position 1 with {\arrow[line width=.3mm]{>}}},
    postaction={decorate},shorten >=0.6pt] (1,-1.1) to (1,-.2) .. controls ++(0,.5) and ++(0,-.5) .. (-1.4,2.2) to (-1.4,3.1);
  \draw[line width= 0.7mm,
    decoration={markings,mark=at position 0.04 with {\arrow[line width=.3mm]{<}}},
    postaction={decorate}] (-1.4,-1.1) to (-1.4,-.2) .. controls ++(0,.5) and ++(0,-.5) .. (1,2.2) to (1,3.1);
  \draw[, <-] (-.4,-1.1) to (-.4,-.2) .. controls ++(0,.7) and ++(0,.7) .. (.4,-.2) to (.4,-1.1);
\begin{scope}[shift={(0,-.3)}]
   \draw[fill =white] (-2.2,2.2) rectangle (-0.2,2.8);
    \draw[fill =gray!60] (.2,2.2) rectangle (2.2,2.8);
     \node at (1.2,2.5) {$n$};
      \node at (-1.2,2.5) {$m$};
  \end{scope}
  \begin{scope}[shift={(0,.3)}]
  \draw[fill =white] (.2,-.8) rectangle (2.2,-.2);
  \draw[fill =gray!60] (-2.2,-.8) rectangle (-.2,-.2);
  \node at (1.2,-.5) {$m$}; 
  \node at (-1.2,-.5) {$n$}; 
  \end{scope}
\end{tikzpicture}}
&\qquad \iota_1\;\; :=
\hackcenter{
\begin{tikzpicture}[scale=0.4, every node/.style={scale=0.6}, xscale =-1]
 \draw[line width= 0.7mm,
    decoration={markings,mark=at position 1 with {\arrow[line width=.3mm]{>}}},
    postaction={decorate},shorten >=0.6pt] (1,-1.1) to (1,-.2) .. controls ++(0,.5) and ++(0,-.5) .. (-1.4,2.2) to (-1.4,3.1);
  \draw[line width= 0.7mm,
    decoration={markings,mark=at position 0.04 with {\arrow[line width=.3mm]{<}}},
    postaction={decorate}] (-1.4,-1.1) to (-1.4,-.2) .. controls ++(0,.5) and ++(0,-.5) .. (1,2.2) to (1,3.1);
\begin{scope}[shift={(0,-.3)}]
   \draw[fill =white] (-2.2,2.2) rectangle (-0.2,2.8);
    \draw[fill =gray!60] (.2,2.2) rectangle (2.2,2.8);
     \node at (1.2,2.5) {$n-1$};
      \node at (-1.2,2.5) {$m-1$};
  \end{scope}
  \begin{scope}[shift={(0,.3)}]
  \draw[fill =white] (.2,-.8) rectangle (2.2,-.2);
  \draw[fill =gray!60] (-2.2,-.8) rectangle (-.2,-.2);
  \node at (1.2,-.5) {$m$}; 
  \node at (-1.2,-.5) {$n$}; 
  \end{scope}
\end{tikzpicture}}
&\qquad
\iota_2:=\;(nm) \;
\hackcenter{
\begin{tikzpicture}[scale=0.4, every node/.style={scale=0.6}, yscale=-1.2]
 \draw[line width= 0.7mm,
    decoration={markings,mark=at position 0.04 with {\arrow[line width=.3mm]{<}}},
    postaction={decorate}] (1,-1.1) to (1,-.2) .. controls ++(0,.5) and ++(0,-.5) .. (-1.4,2.2) to (-1.4,3.1);
  \draw[line width= 0.7mm,
    decoration={markings,mark=at position 1 with {\arrow[line width=.3mm]{>}}},
    postaction={decorate},shorten >=0.6pt] (-1.4,-1.1) to (-1.4,-.2) .. controls ++(0,.5) and ++(0,-.5) .. (1,2.2) to (1,3.1);
  \draw[, ->] (-.4,-1.1) to (-.4,-.2) .. controls ++(0,.7) and ++(0,.7) .. (.4,-.2) to (.4,-1.1);
\begin{scope}[shift={(0,-.3)}]
   \draw[fill =white] (-2.2,2.2) rectangle (-0.2,2.8);
    \draw[fill =gray!60] (.2,2.2) rectangle (2.2,2.8);
     \node at (1.2,2.5) {$n-1$};
      \node at (-1.2,2.5) {$m-1$};
  \end{scope}
  \begin{scope}[shift={(0,.3)}]
  \draw[fill =white] (.2,-.8) rectangle (2.2,-.2);
  \draw[fill =gray!60] (-2.2,-.8) rectangle (-.2,-.2);
  \node at (1.2,-.5) {$m$}; 
  \node at (-1.2,-.5) {$n$}; 
  \end{scope}
\end{tikzpicture}}
\end{align*}
These maps satisfy the relations $\1_{\Q^{(n)^t}\P^{(m)}} = \iota_1 \circ \rho_1 +\iota_2\circ \rho_2$ and $\rho_t \circ \iota_s = \delta_{s,t} \1_s$ in $\H$. 
Consequently, the isomorphism $\Q^{(n)^t}\P^{(m)} \cong \P^{(m)}\Q^{(n)^t} \oplus \P^{(m-1)}\Q^{(n-1)^t}$ holds in $\H$.

\end{proposition}

\begin{proof}
The first relation is obtained via a diagrammatic computation and the use of (\ref{heis:downup}). The second follows from using (\ref{eq:sym-exploded}) and (\ref{eq:anti-exploded}), and then realizing that by (\ref{heis: right twist curl}) the diagrams are nonzero if and only if $s=t$. 
\end{proof}

The analogous statement $\Q^{(n)}\P^{(m)^t} \cong \P^{(m)^t}\Q^{(n)} \oplus \P^{(m-1)^t}\Q^{(n-1)}$ also holds in $\H$.
\medskip

The remainder of the section is devoted  to enhancing the calculus for Khovanov's Heisenberg category by constructing explicit isomorphisms for various Littlewood-Richardson branching rules.

\begin{proposition}\label{prop:QPswap}
For any integers $n,m \geq 1$ and $0 \leq s \leq$ \emph{min}$(m,n)$, define the morphisms $\rho_s$ and $\iota_s$ in $\H$ as follows:
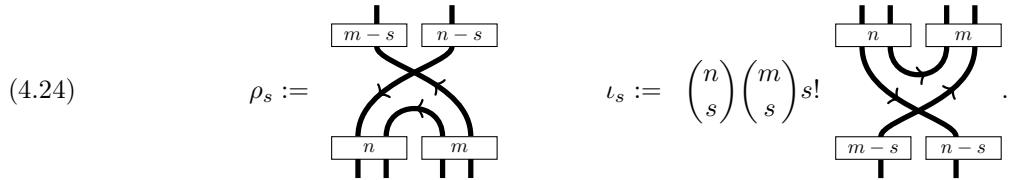
\begin{equation} \label{eqn:rho-iotaQP}
\rho_s:=\;\;
\hackcenter{\begin{tikzpicture} [scale=0.5, every node/.style={scale=0.75}]
\draw[line width= 0.7mm,
    decoration={markings,mark=at position .5 with {\arrow[line width=.3mm]{>}}},
    postaction={decorate} ] (1.5,-1.3) to (1.5,-.2) .. controls ++(0,1.5) and ++(0,-.5) .. (-1,2.2) to (-1,3.3);
\draw[line width= 0.7mm,
    decoration={markings,mark=at position .6 with {\arrow[line width=.3mm]{>}}},
    postaction={decorate}]
(1,3.3)to (1,2.2) .. controls ++(0,-.5) and ++(0,1.5) .. (-1.5,-.2) to(-1.5,-1.3) ;
\draw[line width= 0.7mm,
    decoration={markings,mark=at position .5 with {\arrow[line width=.3mm]{>}}},
    postaction={decorate}](.75,-1.3) to (.75,-.2).. controls ++(0,1) and ++(0,1) .. (-.75,-.2) to (-.75,-1.3) ;
   \draw[fill =white] (-2.2,2.2) rectangle (-0.2,2.8);
  \draw[fill =white] (-2.2,-.8) rectangle (-.2,-.2);
  \draw[fill =white] (.2,2.2) rectangle (2.2,2.8);
  \draw[fill =white] (.2,-.8) rectangle (2.2,-.2);
  \node at (1.2,-.5) {$m$};
  \node at (1.2,2.5) {$n-s$};
  \node at (-1.2,-.5) {$n$};
  \node at (-1.2,2.5) {$m-s$};
\end{tikzpicture}}
\qquad\qquad
\iota_s:=\;\; {{n}\choose{s}}{{m}\choose{s}} s!\;\;
\hackcenter{\begin{tikzpicture} [scale=0.5, every node/.style={scale=0.75}, yscale=-1, xscale=-1]
\draw[line width= 0.7mm,
    decoration={markings,mark=at position .5 with {\arrow[line width=.3mm]{>}}},
    postaction={decorate}] (1.5,-1.3) to (1.5,-.2) .. controls ++(0,1.5) and ++(0,-.5) .. (-1,2.2) to (-1,3.3);
\draw[line width= 0.7mm,
    decoration={markings,mark=at position .6 with {\arrow[line width=.3mm]{>}}},
    postaction={decorate}]
(1,3.3)to (1,2.2) .. controls ++(0,-.5) and ++(0,1.5) .. (-1.5,-.2) to(-1.5,-1.3) ;
\draw[line width= 0.7mm,
    decoration={markings,mark=at position .55 with {\arrow[line width=.3mm]{>}}},
    postaction={decorate}](.75,-1.3) to (.75,-.2).. controls ++(0,1) and ++(0,1) .. (-.75,-.2) to (-.75,-1.3) ;
   \draw[fill =white] (-2.2,2.2) rectangle (-0.2,2.8);
  \draw[fill =white] (-2.2,-.8) rectangle (-.2,-.2);
  \draw[fill =white] (.2,2.2) rectangle (2.2,2.8);
  \draw[fill =white] (.2,-.8) rectangle (2.2,-.2);
  \node at (1.2,-.5) {$n$};
  \node at (1.2,2.5) {$m-s$};
  \node at (-1.2,-.5) {$m$};
  \node at (-1.2,2.5) {$n-s$};
\end{tikzpicture}}.
\end{equation}
These morphisms satisfy $\rho_s \circ \iota_t = \delta_{s,t} \cdot \1_{\P^{(n-t)}\Q^{(m-t)}}$ and  $
\1_{\Q^{(m)}\P^{(n)}} = \sum_s \iota_s \circ \rho_s$. Consequently, the following isomorphisms hold in $\H$:
\[\Q^{(m)}\P^{(n)} \cong \bigoplus_{s=0}^{\textup{min}(m,n)} \P^{(n-s)}\Q^{(m-s)}
\qquad \qquad \text{and} \qquad \qquad \Q^{(m)}\P^{(n)} \cong \P^{(n)}\Q^{(m)} \oplus \Q^{(m-1)}\P^{(n-1)}.\]
\end{proposition}

\begin{proof}
The proof consists of composing the diagrams in (\ref{eqn:rho-iotaQP}) and computing the result. The relation $\rho_s \circ \iota_t = \delta_{s,t} \1_{\P^{(n-t)}\Q^{(m-t)}}$ follows from exploding the symmetrizers $n$ and $m$ as in (\ref{eq:sym-exploded}) and noting that either all terms (when $s \neq t$) or all but one term (when $s=t$) will contain a right twist curl. Thus, by (\ref{heis:up down}) and (\ref{heis: right twist curl}) the result will be either zero or the identity.  Now, if we compose the morphisms in equation \eqref{eq:sMatchings} with the symmetrizers for $n$ and $m$ on top and for $m-s$ and $n-s$ on the bottom we obtain $\rho_s$. The reverse composition yields $\iota_s$. Hence, the second relation $
\1_{\Q^{(m)}\P^{(n)}} = \sum_s \iota_s \circ \rho_s$ follows directly from Proposition~ \ref{prop:QPdecomposition}.
\end{proof}
Proposition \ref{prop:QPswap} holds in far more generality. We refer the reader to Theorem 9.2 in \cite{Savage-WreathProd} for details. 

\begin{proposition} \label{prop:PP*merge}
For any integers $n \geq 1$ and $m>1$  define the following morphisms in $\H$: 
\begin{align}\label{eq:iotaPP*}
\iota_{m}^{n,m}&:=\;\hackcenter{
\begin{tikzpicture} [scale=0.5, yscale = -1,every node/.style={scale=0.6}]
  \draw[line width=.7mm, -] (.85,-1.2) -- (.85,1.2);
  \draw[line width=.7mm , -] (2.15,-1.2) -- (2.15,1.2);
  \draw[fill = gray!20] (.5,.3) rectangle (2.5,1);
  \draw[fill = white] (0.5,-.3) rectangle (1.25,-1);
\draw[fill = gray!60] (1.75,-.3) rectangle (2.5,-1);
  \node at (1.5,.65) {$(n,1^{m})$};
   \node at (.85,-.65) {$n$};
   \node at (2.15,-.65) {$m$};
\end{tikzpicture}}
\;\;=\;\frac{(n)(m+1)}{n+m}\;\;
\hackcenter{\begin{tikzpicture} [scale=.25,every node/.style={scale=0.8}, yscale=.85]
\begin{scope}[shift={(0,-6)}]
\draw[line width=.7mm,-] (4.5,2) to (4.5,9);
\draw[-] (4,5.5) to (4,4.5)..controls ++(0,-0.3) and ++(1,0.3)..(2.5,3.5)..controls ++(-1,-0.3) and ++(0,0)..(1.5,2.5)--(1.5,2);
\draw[line width=.7mm, -](3.5, 2)..controls ++(0,0) and ++(1,-0.5)..(2.5, 4)..controls ++(-1,0.5) and ++(0,0)..(1.5,5.5) to (1.5, 9);
\draw  (2,7.5) to (3.5, 5.5);
\draw[fill =white] (2.5,6.75) rectangle (0.4,8.25);
\draw[fill =gray!60] (5.7,4.25) rectangle (3.2,5.75);
\node at (4.5,5) [scale=.75] {$m+1$};
\node at (1.5,7.5)  [scale=.75]{$n$};
\end{scope}
\end{tikzpicture}}
\qquad\qquad
\\
\iota_{m-1}^{n,m}&:=\;\;(m)\;\;\hackcenter{
\begin{tikzpicture} [scale=0.5,yscale=-1, every node/.style={scale=0.6}]
  \draw[line width=.7mm, -] (.85,-1.2) -- (.85,1.2);
  \draw[line width=.7mm , -] (2.15,-1.2) -- (2.15,1.2);
  \draw[fill = gray!20] (0,.3) rectangle (3,1);
  \draw[fill = white] (0.5,-.3) rectangle (1.25,-1);
\draw[fill = gray!60] (1.75,-.3) rectangle (2.5,-1);
  \node at (1.5,.65) {$(n+1,1^{m-1})$};
   \node at (.85,-.65) {$n$};
   \node at (2.15,-.65) {$m$};
\end{tikzpicture}}
\;\;=\;\;\frac{(n+1)(m)^2}{n+m}\;\;
\hackcenter{\begin{tikzpicture} [scale=.25,every node/.style={scale=0.8}, yscale=.85]
\begin{scope}[shift={(0,-6)}]
\draw[line width=.7mm,-] (4.5,2) to (4.5,11.5);
\draw[-] (4,5.5) to (4,4.5)..controls ++(0,-0.3) and ++(1,0.3)..(2.5,3.5)..controls ++(-1,-0.3) and ++(0,0)..(1.5,2.5)--(1.5,2);
\draw[line width=.7mm, -](3.5, 2)..controls ++(0,0) and ++(1,-0.5)..(2.5, 4)..controls ++(-1,0.5) and ++(0,0)..(1.5,5.5) to (1.5, 11.5);
\draw  (2,7.5) to (3.5, 5.5);
\begin{scope}[shift={(0,15.1)}, yscale=-1]
\draw  (2,7.5) to (3.5, 5.5);
\end{scope}
\draw[fill =white] (2.5,6.75) rectangle (0.4,8.25);
\draw[fill =gray!60] (5.7,4.25) rectangle (3.2,5.75);
\draw[fill =gray!60] (5.7,9.5) rectangle (3.2,11);
\node at (4.5,5) [scale=.75] {$m$};
\node at (1.5,7.5)  [scale=.75]{$n+1$};
\node at (4.5,10.25) [scale=.75] {$m$};
\end{scope}
\end{tikzpicture}}
\\
\rho_{n,m}^{m}&:=\;\hackcenter{
\begin{tikzpicture} [scale=0.5, every node/.style={scale=0.6}]
  \draw[line width=.7mm, -] (.85,-1.2) -- (.85,1.2);
  \draw[line width=.7mm , -] (2.15,-1.2) -- (2.15,1.2);
  \draw[fill = gray!20] (.5,.3) rectangle (2.5,1);
  \draw[fill = white] (0.5,-.3) rectangle (1.25,-1);
\draw[fill = gray!60] (1.75,-.3) rectangle (2.5,-1);
  \node at (1.5,.65) {$(n,1^{m})$};
   \node at (.85,-.65) {$n$};
   \node at (2.15,-.65) {$m$};
\end{tikzpicture}}
\;\;=\; \frac{(n)(m+1)}{n+m}\;\;
\hackcenter{\begin{tikzpicture} [scale=.35,every node/.style={scale=0.8}, xscale=.75]
\begin{scope}[shift={(0,6)}]
\draw[line width=.7mm, -] (4.5,4) to (4.5,9);
\draw[line width=.7mm, -] (1.5,4) to (1.5, 9);
\draw (2.4,7.5) to (3.4, 7);
\draw[fill =white] (2.5,7.5) rectangle (0.4,8.5);
\draw[fill =gray!60] (5.6,6) rectangle (3.5,7);
\node at (4.5,6.5) [scale=.75] {$m+1$};
\node at (1.5,8)  [scale=.75]{$n$};
\end{scope}
\begin{scope}[shift={(0,6)}]
\draw[line width=.7mm, -] (4.5,5) to (4.5,6);
\draw[line width=.7mm, -] (1.5,4) to (1.5, 4);
\draw  (4.2,6.5) to (1.2,4.7);
\draw[fill =gray!60] (5.6,6) rectangle (3.3,7);
\draw[fill =white] (2.5,4.5) rectangle (0.4,5.5);
\node at (1.5,5) [scale=.75]{$n$};
\node at (4.5,6.5) [scale=.75]{$m+1$};
\end{scope}
\end{tikzpicture}}
\\
\rho_{n,m}^{m-1}&:= \;\hackcenter{
\begin{tikzpicture} [scale=0.5, every node/.style={scale=0.6}]
  \draw[line width=.7mm, -] (.85,-1.2) -- (.85,1.2);
  \draw[line width=.7mm , -] (2.15,-1.2) -- (2.15,1.2);
  \draw[fill = gray!20] (0,.3) rectangle (3,1);
  \draw[fill = white] (0.5,-.3) rectangle (1.25,-1);
\draw[fill = gray!60] (1.75,-.3) rectangle (2.5,-1);
  \node at (1.5,.65) {$(n+1,1^{m-1})$};
   \node at (.85,-.65) {$n$};
   \node at (2.15,-.65) {$m$};
\end{tikzpicture}}
\;\;=\; \frac{(n+1)(m)}{n+m}\;\;
\hackcenter{\begin{tikzpicture} [scale=.35,every node/.style={scale=0.8}, xscale=.75]
\begin{scope}[shift={(0,6)}]
\draw[line width=.7mm, -] (4.5,3.5) to (4.5,9);
\draw[line width=.7mm, -] (1.5,3.5) to (1.5, 9);
\draw (2.4,7.5) to (3.4, 7);
\draw (2,7.5) to (3.4, 5);
\draw (3.4,6) to (2.4, 5);
\draw[fill =white] (2.5,7.5) rectangle (0.4,8.5);
\draw[fill =gray!60] (5.6,6) rectangle (3.3,7);
\draw[fill =white] (2.5,5) rectangle (0.4,4);
\draw[fill =gray!60] (5.6,5) rectangle (3.3,4);
\node at (4.5,6.5) [scale=.75] {$m$};
\node at (1.5,8)  [scale=.75]{$n+1$};
\node at (4.5,4.5) [scale=.75] {$m$};
\node at (1.5,4.5)  [scale=.75]{$n$};
\end{scope}
\end{tikzpicture}}
\end{align}
These morphisms satisfy the relations
\begin{itemize}
\item[(i)] $\rho_{n,m}^m \circ \iota_{m}^{n,m}= \1_{\P^{(n,1^m)}}$ and $\rho_{n,m}^{m-1} \circ \iota_{m-1}^{n,m}=\1_{\P^{(n+1,1^{m-1})}}$
\item[(ii)]$\rho_{n,m}^m \circ \iota_{m-1}^{n,m}=0$ and $\rho_{n,m}^{m-1} \circ \iota_m^{n,m}=0$ 
\item[(iii)]$\1_{\P^{(n)}\P^{(m)^t}}=\iota_m^{n,m} \circ \rho_{n,m}^m + \iota_{m-1}^{n,m} \circ {\rho_{n,m}^{m-1}}.$ 
\end{itemize}
Consequently, the isomorphism $\P^{(n)}\P^{(m)^t} \cong \P^{(n,1^m)} \oplus \P^{(n+1,1^{m-1})}$ holds in $\H$. 
\end{proposition}

\begin{proof}
Relation (i) follows directly from composing the diagrams and Theorem \ref{thm:CompleteIdemp}. 
Relation (ii) follows from equation \eqref{eq:whiteblackbox}, since composing the respective morphisms will connect a symmetrizer and an antisymmetrizer by more than one strand. Relation (iii) is a consequence of pre and post composing the diagrams in Proposition \ref{thm:CompleteIdemp} by symmetrizers for $n$ and antisymmetrizers for $m$. In particular, given any $\lambda=(\lambda_1, \dots, \lambda_k) \vdash (n+m)$ suppose we post-compose the diagram of $\P^\lambda$ with the symmetrizer for $n$ on the rightmost $n$ strands and the antisymmetrizer for $m$ on the leftmost $m$ strands. If $\lambda_1 <n$ then all the strands of $\lambda_1$ and at least one strand of $\lambda_2$ will be be connected to the symmetrizer for $n$ above them. However, since every strand in $\lambda_2$ is antisymmetrized with one strand from $\lambda_1$, then there will exist an antisymmetrizer than is connected to the symmetrizer for $n$ by more than one strand. By relation \eqref{eq:whiteblackbox} the diagram would then be zero. Likewise, if $\lambda_1>n+1$ or $\lambda_i>1$ for any $i>1$ then the antisymmetrizer for $m$ would be connected to a symmetrizer for some $\lambda_i$ by more than one strand. Thus, $\lambda_1=n$ or $n+1$ and $\lambda_i=1$ for all $i>1$. Consequently, the only partitions that are not zero are $\lambda=(n,1^m)$ and $(n+1,1^{m-1})$. 
\end{proof}

\begin{proposition}\label{prop:PPmerge}
For any $0\leq s\leq$ \emph{min}$(n,m)$ define the morphisms $\iota_s^m:\P^{(m+n-s,s)} \to \P^{(m)}\P^{(n)}$ and $\rho_m^s: \P^{(m)}\P^{(n)} \to \P^{(m+n-s,s)} $ in $\H$ as follows:\begin{equation}\label{eq:rho-iotaPP}
\iota_{s}^{m}:=\;\;\begin{cases} (m)\; \hackcenter{
\begin{tikzpicture} [scale=0.5, yscale = -1,every node/.style={scale=0.6}]
  \draw[line width=.7mm,-] (2.15,-2.6)--(2.15,-1.7)--(.85,-1) -- (.85,1.2);
  \draw[line width=.7mm , -] (.85,-2.6)--(.85,-1.7)--(2.15,-1) -- (2.15,1.2);
  \draw[fill = gray!20] (0,.3) rectangle (3,1);
  \draw[fill = white] (0.5,-.3) rectangle (1.25,-1);
\draw[fill = white] (1.75,-.3) rectangle (2.5,-1);
 \draw[fill = white] (0.5,-1.7) rectangle (1.25,-2.4);
\draw[fill = white] (1.75,-1.7) rectangle (2.5,-2.4);
  \node at (1.5,.65) {$(m+n-s,s)$};
   \node at (.85,-.65) {$n$};
   \node at (2.15,-.65) {$m$};
   \node at (.85,-2) {$m$};
   \node at (2.15,-2) {$n$};
\end{tikzpicture}}& m<n \\
\\
(n) \;\; \hackcenter{
\begin{tikzpicture} [scale=0.5, yscale = -1,every node/.style={scale=0.6}]
  \draw[line width=.7mm, -] (.85,-1.2) -- (.85,1.2);
  \draw[line width=.7mm , -] (2.15,-1.2) -- (2.15,1.2);
  \draw[fill = gray!20] (0,.3) rectangle (3,1);
  \draw[fill = white] (0.5,-.3) rectangle (1.25,-1);
\draw[fill = white] (1.75,-.3) rectangle (2.5,-1);
  \node at (1.5,.65) {$(m+n-s,s)$};
   \node at (.85,-.65) {$m$};
   \node at (2.15,-.65) {$n$};
\end{tikzpicture}}& m \geq n
\end{cases}
\qquad \qquad
\rho_{m}^{s}:=\;\;\begin{cases} \hackcenter{
\begin{tikzpicture} [scale=0.5, yscale = 1,every node/.style={scale=0.6}]
  \draw[line width=.7mm,-] (2.15,-2.6)--(2.15,-1.7)--(.85,-1) -- (.85,1.2);
  \draw[line width=.7mm , -] (.85,-2.6)--(.85,-1.7)--(2.15,-1) -- (2.15,1.2);
  \draw[fill = gray!20] (0,.3) rectangle (3,1);
  \draw[fill = white] (0.5,-.3) rectangle (1.25,-1);
\draw[fill = white] (1.75,-.3) rectangle (2.5,-1);
 \draw[fill = white] (0.5,-1.7) rectangle (1.25,-2.4);
\draw[fill = white] (1.75,-1.7) rectangle (2.5,-2.4);
  \node at (1.5,.65) {$(m+n-s,s)$};
   \node at (.85,-.65) {$n$};
   \node at (2.15,-.65) {$m$};
   \node at (.85,-2) {$m$};
   \node at (2.15,-2) {$n$};
\end{tikzpicture}}& m<n \\
\\
\hackcenter{
\begin{tikzpicture} [scale=0.5, yscale = 1,every node/.style={scale=0.6}]
  \draw[line width=.7mm, -] (.85,-1.2) -- (.85,1.2);
  \draw[line width=.7mm , -] (2.15,-1.2) -- (2.15,1.2);
  \draw[fill = gray!20] (0,.3) rectangle (3,1);
  \draw[fill = white] (0.5,-.3) rectangle (1.25,-1);
\draw[fill = white] (1.75,-.3) rectangle (2.5,-1);
  \node at (1.5,.65) {$(m+n-s,s)$};
   \node at (.85,-.65) {$m$};
   \node at (2.15,-.65) {$n$};
\end{tikzpicture}}& m \geq n
\end{cases}
\end{equation}
These maps satisfy the relations $\rho_m^t \circ \iota_s^m = \delta_{s,t} \cdot \1_{\P^{(m+n-s,s)}} $ and $1_{\P^{(m)}\P^{(n)}} = \sum_{s=0}^{\emph{min}(n,m)} \iota_s^m \circ \rho_m^s$. Hence, the following isomorphism holds in $\H$. 
\[\P^{(m)}\P^{(n)} \cong \bigoplus_{s=0}^{\emph{min}(m,n)} \P^{(m+n-s,s)}\]
\end{proposition}

\begin{proof}
The relation $\rho_m^t \circ \iota_s^m = \delta_{s,t} \cdot \1_{\P^{(m+n-s,s)}} $ follows from \eqref{eq:PlambdaOrthogonal}. Likewise, $1_{\P^{(m)}\P^{(n)}} = \sum_{s=0}^{\emph{min}(n,m)} \iota_s^m \circ \rho_m^s$ is a consequence of pre and post composing the diagrams from equation \eqref{eq:P=sumLambda} with symmetrizers for $m$ and $n$ and then noting that the only possible partitions that are nonzero have shape $\lambda=(n+m-s,s)$ for $0\leq s \leq$ min($n,m)$. Hence, by Theorem \ref{thm:CompleteIdemp} the result follows. The final statement, $\P^{(m)}\P^{(n)} \cong \bigoplus_{s=0}^{\emph{min}(m,n)} \P^{(m+n-s,s)}$ is immediate from Theorem \ref{thm:CompleteIdemp}. 
\end{proof}

\noindent \textbf{Notation:} Given partitions $\mu \vdash n$ and $\lambda \vdash n-1$, we write $\lambda=\mu-\square$ if $\lambda$ can be obtained from $\mu$ by removing a box (equivalently, $\mu=\lambda+\square$ if $\mu$ can be obtained from $\lambda$ by adding a box). Hence, if $\mu=(\mu_1, \dots, \mu_m)$ and $\lambda=(\lambda_1, \dots, \lambda_k)$ then $m=k$ or $k+1$ and there will exist an integer $1\leq s\leq$ min$(m,k)$ such that $\mu_s=\lambda_{s}+1$ and $\mu_i=\lambda_i$ for all $i\neq s$ (if $m=k+1$ set $\lambda_{k+1}=0$). In this case, we write $\lambda=\mu(s^-)$ or equivalently $\mu=\lambda(s^+)$. 

\begin{proposition}\label{prop:QlambdaPswap}
For any $\mu=(\mu_1, \dots, \mu_m) \vdash n$ and for each $\mu(s^-)=\mu-\square$, define the following morphisms in $\H$:
\begin{align}\label{eq:rhoQlambdaP}
\rho_0:=\hackcenter{
\begin{tikzpicture} [scale=0.65, yscale = -1,every node/.style={scale=1}]
  \draw[line width= 0.7mm,
    decoration={markings,mark=at position 1 with {\arrow[line width=.3mm]{>}}},
    postaction={decorate},shorten >=0.6pt] (2,-1.2) ..controls++(0,.5) and ++(0,-.5)..(1,.4)-- (1,1.2);
  \begin{scope}[xscale=-1, shift={(-3.2,0)}]
  \draw[ <-] (2.3,-1.2) ..controls++(0,.5) and ++(0,-.5)..(1,.4)-- (1,1.2);
  \end{scope}
  \draw[fill = gray!20] (0.2,.0) rectangle (1.7,.7);
  \node at (1,.4) {$\mu$};
\end{tikzpicture}}
&
\qquad 
\rho_s:=\;
\hackcenter{\begin{tikzpicture} [scale=0.65, every node/.style={scale=1}]
\draw[ line width= 0.7mm,
    decoration={markings,mark=at position 1 with {\arrow[line width=.3mm]{>}}},
    postaction={decorate},shorten >=0.6pt] (-.75,2)to (-.75,-1.7) ;
\draw[line width= 0.7mm,
    decoration={markings,mark=at position 1 with {\arrow[line width=.3mm]{>}}},
    postaction={decorate},shorten >=0.6pt] (-1.5,2)to (-1.5,-1.7) ;
\draw[directed=.25](.75,-1.7) to (.75,-.2).. controls ++(0,1) and ++(0,1) .. (-.5,-.2)  ;
\draw[line width= 0.7mm,
    decoration={markings,mark=at position 1 with {\arrow[line width=.3mm]{>}}},
    postaction={decorate},shorten >=0.6pt] (.1,2)to (.1,-1.7) ;
   \draw[fill =gray!20] (-2.2,1.1) rectangle (.4,1.8);
  \draw[fill =white] (-1.2,-.4) rectangle (-.2,.2);
   \draw[fill =gray!20] (-2.2,-1.2) rectangle (.4,-.6);
  \node at (-.6,-.1) {$\mu_s$};
  \node at (-.8,1.45) {$\mu(s^-)$};
   \node at (-.8,-.9) {$\mu$};
\end{tikzpicture}}
&
\qquad
\iota_0:=\hackcenter{
\begin{tikzpicture} [scale=0.65, yscale=-1,xscale = -1,every node/.style={scale=1}]
  \draw[line width= 0.7mm,
    decoration={markings,mark=at position 1 with {\arrow[line width=.3mm]{>}}},
    postaction={decorate},shorten >=0.6pt] (2,-1.2) ..controls++(0,.5) and ++(0,-.5)..(1,.4)-- (1,1.2);
  \begin{scope}[xscale=-1, shift={(-3.2,0)}]
  \draw[ <-] (2.3,-1.2) ..controls++(0,.5) and ++(0,-.5)..(1,.4)-- (1,1.2);
  \end{scope}
  \draw[fill = gray!20] (0.2,.0) rectangle (1.7,.7);
  \node at (1,.4) {$\mu$};
\end{tikzpicture}}
&
\qquad
\iota_s:=\;
\hackcenter{\begin{tikzpicture} [yscale=-1,scale=0.65, every node/.style={scale=1}]
\draw[ line width= 0.7mm,
    decoration={markings,mark=at position 0.04 with {\arrow[line width=.3mm]{<}}},
    postaction={decorate}] (-.75,2)to (-.75,-1.7) ;
\draw[line width= 0.7mm,
    decoration={markings,mark=at position 0.04 with {\arrow[line width=.3mm]{<}}},
    postaction={decorate}] (-1.5,2)to (-1.5,-1.7) ;
\draw[ directed=.8] (-.5,-1).. controls ++(0,1) and ++(0,1) .. (.75,-1) to (.75,-1.7) ;
\draw[line width= 0.7mm,
    decoration={markings,mark=at position 0.04 with {\arrow[line width=.3mm]{<}}},
    postaction={decorate}] (.1,2)to (.1,-1.7) ;
\begin{scope}[shift={(0,-.3)}]
   \draw[fill =gray!20] (-2.2,1.1) rectangle (.4,1.8);
  \draw[fill =white] (-1.2,0.2) rectangle (-.2,.8);
   \draw[fill =gray!20] (-2.2,-1.2) rectangle (.4,-.6);
  \node at (-.7,.5)[scale=.65]{$\mu_s -1$};
  \node at (-.8,1.45) {$\mu(s^-)$};
   \node at (-.8,-.9) {$\mu$};
   \end{scope}
\end{tikzpicture}}
\end{align}
\noindent These morphisms satisfy the relations:
\[ \rho_s \circ \iota_r= \delta_{s,r} \1_{\mu(s^-)}\qquad \text{and} \qquad \1_{\Q^{\mu}\P}=\sum_{\mu-\square}\iota_s\circ \rho_s\ .\]
Consequently, the following isomorphism holds in $\H$:
\[\Q^{\mu}\P \cong \P\Q^{\mu} \oplus \bigoplus_{\lambda=\mu - \square}\Q^{\lambda}.\]
\end{proposition}

\begin{proof}
The result follows from an explicit computation of the composition of the diagrams. For $s=0$, the identity $\rho_s \circ \iota_r= \delta_{s,r} \1_{\mu(s^-)}$ is immediate from \eqref{heis:up down}. For any other values of $s,t$ we expand the idempotent corresponding to $\mu$ and then use \eqref{heis: right twist curl} to observe that the only nonzero diagram occurs precisely when $s=t$, in which case by \eqref{eq:PlambdaOrthogonal} from Theorem \ref{thm:CompleteIdemp} the result follows.  The identity $\1_{\Q^{\lambda}\P}=\sum_{s=0}^m\iota_s\circ \rho_s $ is a consequence of once again composing the diagrams, expanding the idempotent for $\mu(s+)$ and then using Proposition \ref{prop:QPdecomposition} to deduce the desired equality in $\H'$ and the resulting isomorphism in $\H$. 

\end{proof}

\begin{proposition}\label{prop:PlambdaPmerge}
For any $\lambda = (\lambda_1, \dots, \lambda_k) \vdash n-1$ and each $\lambda(s^+)=\lambda +\square$, define the morphisms $\iota_s: \P^{\lambda(s^+)} \to \P^{\lambda}\P$ and $\rho_s: \P^{\lambda}\P \to \P^{\lambda(s^+)}$ in $\H$ as follows:
\begin{align}\label{eq:PlambdaPmerge}
\iota_s:=\;
\hackcenter{\begin{tikzpicture} [scale=0.65, every node/.style={scale=1}]
\draw[ line width= 0.7mm,
    decoration={markings,mark=at position 0.04 with {\arrow[line width=.3mm]{<}}},
    postaction={decorate}] (-.75,2)to (-.75,-1.7) ;
\draw[line width= 0.7mm,
    decoration={markings,mark=at position 0.04 with {\arrow[line width=.3mm]{<}}},
    postaction={decorate}] (-1.5,2)to (-1.5,-1.7) ;
\draw[directed=.8] (-.5,-.2).. controls ++(0,1) and ++(0,-.5) .. (.75,.8) to (.75,2);
\draw[line width= 0.7mm,
    decoration={markings,mark=at position 0.04 with {\arrow[line width=.3mm]{<}}},
    postaction={decorate}] (.1,2)to (.1,-1.7) ;
\begin{scope}[shift={(0,-.3)}]
   \draw[fill =gray!20] (-2.2,1.1) rectangle (.4,1.8);
  \draw[fill =white] (-1.2,-.4) rectangle (-.2,.2);
   \draw[fill =gray!20] (-2.2,-1.2) rectangle (.4,-.6);
  \node at (-.7,-.1)  [scale=.6]{$\lambda_s+1$};
  \node at (-.8,1.4){$\lambda$};
   \node at (-.8,-.95) {$\lambda(s^+)$};
   \end{scope}
\end{tikzpicture}}
\qquad \qquad\qquad
\rho_s:=\;
\hackcenter{\begin{tikzpicture} [scale=0.65, every node/.style={scale=1}]
\draw[ line width= 0.7mm,
    decoration={markings,mark=at position 0.04 with {\arrow[line width=.3mm]{<}}},
    postaction={decorate}] (-.75,2)to (-.75,-1.7) ;
\draw[line width= 0.7mm,
    decoration={markings,mark=at position 0.04 with {\arrow[line width=.3mm]{<}}},
    postaction={decorate}] (-1.5,2)to (-1.5,-1.7) ;
\draw[directed=.25](.75,-1.7)to(.75,-1.3)  .. controls ++(0,2) and ++(0,-1) ..(-.5,.7) to (-.5,1) ;
\draw[line width= 0.7mm,
    decoration={markings,mark=at position 0.04 with {\arrow[line width=.3mm]{<}}},
    postaction={decorate}] (.1,2)to (.1,-1.7) ;
\begin{scope}[shift={(0,-.3)}]
   \draw[fill =gray!20] (-2.2,1.1) rectangle (.4,1.8);
  \draw[fill =white] (-1.2,.2) rectangle (-.2,-.4);
   \draw[fill =gray!20] (-2.2,-1.2) rectangle (.4,-.6);
  \node at (-.7,-.1)  [scale=1]{$\lambda_s$};
  \node at (-.8,1.4){$\lambda(s^+)$};
   \node at (-.8,-.95) {$\lambda$};
   \end{scope}
\end{tikzpicture}}
\end{align}
These morphism satisfy the relations: $\rho_s \circ \iota_r= \delta_{s,r} \1_{\lambda(s^+)} 
$ and $\1_{\P^{(\lambda)}\P}=\sum_{\lambda+\square}\iota_s\circ \rho_s $. 
Thus, the following isomorphism holds in $\H$:
\[\P^{\lambda}\P \cong \bigoplus_{\mu =\lambda+ \square}\P^{\mu}.\]
\end{proposition}

\begin{proof}
By composing the diagrams and expanding the idempotent for $\lambda$ by relation \eqref{eq:PlambdaOrthogonal} in Theorem \ref{thm:CompleteIdemp} we find that $\rho_s \circ \iota_r$ is either zero for $s \neq t$ or the identity for $s=t$. When composing in the other direction, we instead expand the idempotent corresponding to $\lambda(s^+)$ with \eqref{eq:sym-exploded} and \eqref{eq:anti-exploded}. We then repeatedly apply \eqref{eq:whiteblackbox} and \eqref{eq:sym-anti-absorbsion} to simplify and cancel the summands. Finally, we use Theorem \ref{thm:CompleteIdemp} to ascertain the relation $\1_{\P^{(\lambda)}\P}=\sum_{\lambda+\square}\iota_s\circ \rho_s$ holds in $\H'$. The final isomorphism in $\H$ follows immediate from these results by Theorem \ref{thm:CompleteIdemp}.
\end{proof}

\section{A Categorical Boson-Fermion Correspondence}\label{sec:CatBFCorrespondence}

\begin{definition}\label{def:MatrixCat}
For any category $\mathcal{C}$ denote by Mat$_{\Z \times \Z}(\mathcal{C})$ the category with \emph{objects} given by formal $\Z \times \Z$ matrices $\mathsf{X}=(\mathsf{X}_{i,j})_{i,j \in \Z}$ with $\mathsf{X}_{i,j} \in \mathcal{C}$ satisfying the condition that  $\mathsf{X}_{i,j}=0$ for $|i-j|\gg0$. Its \emph{morphisms} also consist of infinite matrices and are given by $(f_{i,j})_{i,j\in\Z} \in$ Hom$_{Mat(\mathcal{C})}(\mathsf{X},\mathsf{Y})$ with entries $f_{i,j} \in$ Hom$_\mathcal{C}(\mathsf{X}_{i,j},\mathsf{Y}_{i,j})$. Composition of morphisms is given by entry-wise composition of the matrices, so that $(f\circ g)_{i,j}:=f_{i,j}\circ g_{i,j}$.

Likewise, define Mat$_{\Z \times \1}(\mathcal{C})$ by restricting the conditions in the definition above to column vectors. Thus, Mat$_{\Z \times \1}(\mathcal{C})$ consists of infinite dimensional vectors on $\mathcal{C}$ with finitely many nonzero entries. 

If $\mathcal{C}$ is additive and monoidal then Mat$_{\Z \times \Z}(\mathcal{C})$ inherits the additive and monoidal structure by taking entry-wise direct sums and the usual multiplication of matrices. Moreover, if $\mathcal{C}$ acts an additive category $\mathcal{A}$, there is an induced action of Mat$_{\Z \times \Z}(\mathcal{C})$ on Mat$_{\Z \times 1}(\mathcal{A})$ given by multiplying the matrix with the column vector in the usual way and then taking the direct sums of the resulting actions inside $\mathcal{A}$. That is, for any row $[\dots \mathsf{X}_{i,j_1} \dots \mathsf{X}_{i,j_k}\dots]$ of $\mathsf{X} \in$ Mat$_{\Z \times \Z}(\mathcal{C})$ and $[\dots \mathsf{v}_{j_1}\dots \mathsf{v}_{j_k}\dots]^T \in$ Mat$_{\Z \times \1}(\mathcal{A})$ then $[\dots \mathsf{X}_{i,j_1} \dots \mathsf{X}_{i,j_k}\dots] \cdot [\dots \mathsf{v}_{j_1}\dots \mathsf{v}_{j_k}\dots]^T = \bigoplus_{j\in\Z} \mathsf{X}_{i,j}(\mathsf{v}_{j})$ where $\mathsf{X}_{i,j}(\mathsf{v}_{j})$ is given by the action of $\mathcal{C}$ on $\mathcal{A}$. Since only finitely many $\mathsf{v}_{j}$ are non-zero the sum is finite and thus this action is well defined. Thus, for an additive category $\mathcal{C}$, the categories Mat$_{\Z\times 1}(\mathcal{C})$ and Mat$_{\Z\times \Z}(\mathcal{C})$ are isomorphic (equivalent) to a direct sum of $\N$ copies of $\mathcal{C}$. 

\end{definition}

In particular, if $\mathcal{C}=\K(\H)$ then $\mathsf{X}, \mathsf{Y}\in$ Mat$_{\Z \times \Z}(\mathcal{\K(\H)})$ are infinite matrices whose entries $\mathsf{X}_{i,j}, \mathsf{Y}_{i,j}$ are infinite chain complexes in $\K(\H)$. Consequently their product, denoted by $\mathsf{X}\mathsf{Y}$, is given by multiplying the matrices in the usual way and has entries of the form $\bigoplus_j \mathsf{X}_{i,j}\otimes \mathsf{Y}_{j,k}$. The term $\mathsf{X}_{i,j}\otimes\mathsf{Y}_{j,k}$ is the usual tensor product of the chain complexes $\mathsf{X}_{i,j},\mathsf{Y}_{j,k}$ up to homotopy. As above, the condition $\mathsf{X}_{i,j}=0$ for $|i-j|\gg0$ ensures that the direct sum of complexes $\bigoplus_j \mathsf{X}_{i,j}\mathsf{Y}_{j,k}$ is always finite. Thus, the entries of $\mathsf{X}\mathsf{Y}$ consist of finite direct sums of infinite complexes up to homotopy.

In fact, many of the properties of object in $\K(\H)$ make sense in Mat$_{\Z \times \Z}(\mathcal{\K(\H)})$. For instance, given $\mathsf{X},\mathsf{Y},\mathsf{Z} \in$ Mat$_{\Z \times \Z}(\mathcal{\K(\H)})$ and a family of morphisms 
\[\mathsf{d}_{i,j}:\mathsf{X}_{i,j} \to \mathsf{Y}_{i,j} \qquad\text{ and }\qquad \mathsf{d}'_{i,j}:\mathsf{Y}_{i,j} \to \mathsf{Z}_{i,j}\]
with the property that $\mathsf{d}'_{i,j}\mathsf{d}_{i,j}=0$ for all $i,j \in \Z$, then the morphisms $\mathsf{d}:=(\mathsf{d}_{i,j}) \in$ Hom$_{\text{Mat}(\mathcal{\K(\H)})}(\mathsf{X},\mathsf{Y})$ and $\mathsf{d'}:=(\mathsf{d'}_{i,j}) \in$ Hom$_{\text{Mat}(\mathcal{\K(\H)})}(\mathsf{Y},\mathsf{Z})$ are well defined and satisfy $\d' \circ \d =0$. Hence, it makes sense to consider the following sequence in Mat$_{\Z \times \Z}(\mathcal{\K(\H)})$:
\[\mathsf{X} \xrightarrow{\d} \mathsf{Y} \xrightarrow{\d'} \mathsf{Z}\]
Thus, given appropriate maps $\d_{ij}$ we can consider chain complexes on Mat$_{\Z \times \Z}(\mathcal{\K(\H)})$. Moreover, if $\mathsf{X}_{i,j} \simeq \mathsf{Y}_{i,j}$ in $\K(\H)$ for all $i,j\in\Z$, then we say $\mathsf{X}$ is homotopy equivalent to $\mathsf{Y}$ in Mat$_{\Z \times \Z}(\mathcal{\K(\H)})$ and write $\mathsf{X} \simeq \mathsf{Y}$. Likewise, for any $s \in \Z$ we write $\mathsf{X}[s]$ to denote the matrix with entries $\mathsf{X}_{i,j}[s]$ for each $i,j \in \Z$. 

At the decategorified level, recall that by Proposition \ref{prop:HeisAction} the Fock space representation of $\mathfrak{h}$ is isomorphic to $Sym$ and thus Bosonic Fock space as a module over $\mathfrak{h}$ is isomorphic to the direct sum of infinitely many copies of the Fock space representation. Motivated by this define \emph{categorical Fock space} as the additive category \[\mathsf{V}_{Fock} := \text{Mat}_{\Z \times 1}\left(\bigoplus_n \mathds{k}[S_n]\text{-mod}\right),\] 
where any $\mathsf{v} =(\mathsf{v}_c)_{c\in\Z}  \in \textup{Ob}(\mathsf{V}_{Fock})$ has only finitely many nonzero entries, and whose morphisms consist of vectors with entries given by $\bigoplus_n \k[S_n]-$module maps. Moreover, we make the following generalization. 



\begin{definition}\label{def:Hquotient}
Given any $n \in \N$, define the category \begin{equation}
\H_n:= \bigoplus_{m\in \Z} \left(\mathds{k}[S_m],\mathds{k}[S_n]\right)\text{-bimod}.
\end{equation}
So that, $\H_0 = \bigoplus_n \mathds{k}[S_n]$-mod and $\mathsf{V}_{Fock} =$ Mat$_{\Z \times 1}(\H_0)$.
\end{definition}

\begin{remark}Since $\P$ and $\Q$ are induction and restriction functors on $\bigoplus_n \mathds{k}[S_n$-mod, by considering the left action of $\P$ and the right action of $\Q$ on the trivial module $\k$ we see that $\H_n$ can equivalently be defined as the category whose objects and morphisms are the same as in $\H$, but subject to the condition that any morphism that factors through $\mathsf{M}\otimes \Q^{\otimes n}$ for any $\mathsf{M} \in$ Mor$(\H)$ is zero. This category is the universal categorical Fock space since it maps to any other categorification of Fock space. For more details see \cite[Section 7]{CS-BosonFermion}. 
\end{remark}

\subsection{Categorical Bernstein Operators}\label{subsec:CatBernsteinOps}
Motivated by the Bernstein operators defined in \eqref{eq:BernsteinOps}, we make the following definition. 

\begin{definition}\label{def:CatBernsteinOps} The \emph{categorical Bernstein operators} are infinite chain complexes in $\K(\H)$ given by:
\begin{equation}\label{eq:Ca}
\B_a:=\left\lbrace \bigoplus_{x-y=a} \P^{(x)} \Q^{(y)^t}[y], \mathsf{d}_a \right\rbrace 
\qquad \text{with} \qquad
\mathsf{d}_a := \;\;
\hackcenter{
\begin{tikzpicture} [scale=0.6]
  \draw[<-] (.4,-1) -- (.4,0);
  \draw[line width= 0.7mm,
    decoration={markings,mark=at position 0.07 with {\arrow[line width=.3mm]{<}}},
    postaction={decorate},shorten >=1pt] (1,-1) -- (1,1);
  \draw[line width= 0.7mm,
    decoration={markings,mark=at position 1 with {\arrow[line width=.3mm]{>}}},
    postaction={decorate},shorten >=1pt] (-1,-1) -- (-1,1);
  \draw[ ] (-.4,-1) -- (-.4,0);
    \draw[ directed=.8]  (-.4,.3) .. controls ++(0,.4) and ++(0,.4) .. (.4,.3);
        \draw[fill =white] (-1.4,-.3) rectangle (-.2,.3);
        \draw[fill =gray!60] (.2,-.3) rectangle (1.4,.3);
  \node at (.85,-0) [scale=.5]{$y$};
  \node at (-.85,0) [scale=.75]{$x$};
  \end{tikzpicture}} \qquad \text{ and } \qquad a \in \Z,
\end{equation}
so that $\B_a=[ \dots\to \P^{(x)}\Q^{(y)^t}[y] \to \P^{(x-1)}\Q^{(y-1)^t}[y-1]\to\dots]$, and 
\begin{equation}\label{eq:Ca*}
\B_a^*:=\left\lbrace \bigoplus_{y-x=a} \P^{(x)^t} \Q^{(y)}[-x], \mathsf{d}_a^* \right\rbrace \qquad \text{with} \qquad
\mathsf{d}_a^*:= \;\;
\hackcenter{
\begin{tikzpicture} [scale=0.6, yscale=-1, xscale=-1]
  \draw[<-] (.4,-1) -- (.4,0);
  \draw[line width= 0.7mm,
    decoration={markings,mark=at position 0.07 with {\arrow[line width=.3mm]{<}}},
    postaction={decorate},shorten >=0.6pt] (1,-1) -- (1,1);
  \draw[line width= 0.7mm,
    decoration={markings,mark=at position 1 with {\arrow[line width=.3mm]{>}}},
    postaction={decorate},shorten >=1pt] (-1,-1) -- (-1,1);
  \draw[ ] (-.4,-1) -- (-.4,0);
    \draw[ directed=.8]  (-.4,.3) .. controls ++(0,.4) and ++(0,.4) .. (.4,.3);
        \draw[fill =white] (-1.4,-.3) rectangle (-.2,.3);
        \draw[fill =gray!60] (.2,-.3) rectangle (1.4,.3);
  \node at (.85,-0) [scale=.75]{$x$};
  \node at (-.85,0) [scale=.75]{$y$};
  \end{tikzpicture}} \qquad \text{and} \qquad a \in \Z,
  \end{equation}
so that $\B_a^*=[ \dots\to \P^{(x)^t}\Q^{(y)}[-x] \to \P^{(x+1)^t}\Q^{(y+1)}[-x-1]\to\dots]$. 
  \end{definition}
Since $x,y \geq 0$, the complex $\B_a \in \K^-(\H)$ is unbounded above and supported in non-negative homological degrees whereas $\B_a^* \in \K^+(\H)$ is unbounded below and supported in non-positive homological degrees. These functors are biadjoint in $\K(\H)$ 
since $\P$ and $\Q$ are biadjoint in $\H$. 

As in the decategorified picture, their action on $\mathsf{V}_{Fock}$ is integrable. That is, for any $\mathsf{v} \in \bigoplus_n \mathds{k}[S_n]$-mod, since $\Q$ is the restriction functor there is an $N \in \N$ such that $\Q^{\otimes n}(\mathsf{v})=0$ for all $n>N$. Hence for any $a \in \Z$, the complexes $\B_a(\mathsf{v})$ and $\B_a^*(\mathsf{v})$ are finite and belong to Mat$_{\Z\otimes 1}(\K^b(\H))$. 

\begin{theorem}\label{thm:BernsteinSpecht}
The categorical Bernstein operators are creation and annihilation functors for Specht modules. That is, for any Specht module $S_\lambda$ associated to partition $\lambda=(\lambda_k, \dots, \lambda_1)\vdash n$ with $\lambda_k\geq \dots \geq \lambda_1 \geq~0$ and $\k$ the trivial module over $\k[S_n]$-mod, then $\B_{\lambda_k}\B_{\lambda_{k-1}}\dots\B_{\lambda_1}(\mathds{k})\simeq S_\lambda$ and $\B^*_{\lambda_1} \dots \B^*_{\lambda_{k-1}}\B^*_{\lambda_k}(S_\lambda)\simeq \mathds{k}$.
\end{theorem}

\begin{proof}
Suppose $n \in \N$ and $\lambda = (\lambda_k, \dots, \lambda_1)\vdash n$. Since $\Q$ acts by zero on $\k$, by Proposition~\ref{prop:QPdecomposition} for any $y \geq 0$
\[\Q^{\otimes y} \P^{\otimes n} (\k) \simeq \bigoplus_{s=0}^{\text{min}(y,n)} \P^{\otimes (n-s)}\Q^{\otimes (y-s)}(\k) \simeq \begin{cases} \P^{\otimes (n-y)}(\k) & 0 \leq y \leq n \\ 0 & n<y \end{cases}\]
Composing with the canonical projection and inclusion maps for any $y\leq n$
\[\Q^{(y)^t} \P^{\lambda} \hookrightarrow \Q^{\otimes y} \P^{n} \twoheadrightarrow \P^{\otimes(n-y)} \twoheadrightarrow \P^{\mu}\]
then, when acting on $\k$, there is an isomorphism $\Q^{(y)^t} \P^{\lambda}(\k) \simeq \bigoplus_{\lambda\setminus\mu = 1^y}\P^{\mu}(\k)$ where the sum is taken over all ways of removing $y$ boxes in distinct rows from $\lambda$. Moreover, we also know that $\P^{(y+a)}\P^{\mu} = \bigoplus_{\gamma\setminus\mu=(y+a)} \P^{\gamma}$ where the sum ranges over all ways of adding $y+a$ boxes in distinct columns to $\mu$. Hence, for any $a\geq \lambda_n$ the action of $\B_{a}$ on $\P^{\lambda}(\k)$ reduces to following bounded complex
\begin{align*}
\B_{a}\P^{\lambda}(\k)&\simeq \bigoplus_{\lambda\setminus \mu = 1^n}\P^{(n+a)}\P^{\mu}(\k)
\rightarrow \dots \rightarrow\bigoplus_{\lambda\setminus \mu = 1^y}\P^{(y+a)}\P^{\mu}(\k)
\rightarrow \bigoplus_{\lambda\setminus \mu = 1^{y-1}}\P^{(y+a-1)}\P^{\mu}(\k) \rightarrow \dots \rightarrow 
\P^{(a)}\P^{\mu}(\k)\\
&\simeq \bigoplus_{\lambda\setminus \mu = 1^n} \bigoplus_{\gamma\setminus\mu=(n+a)}  \P^{\gamma}(\k)
\rightarrow \dots \rightarrow\bigoplus_{\lambda\setminus \mu = 1^y}\bigoplus_{\gamma\setminus\mu=(y+a)}  \P^{\gamma}(\k)
\rightarrow\dots \rightarrow  \bigoplus_{\gamma\setminus\mu=(a)}  \P^{\gamma}(\k).
\end{align*}
Composing the differential of $\B_a$ with the projection and injection maps above we find that the differential from $\bigoplus_{\lambda\setminus \nu = 1^y}\bigoplus_{\gamma\setminus\nu=(y+a)}  \P^{\gamma}(\k) \to  \bigoplus_{\lambda\setminus \mu = 1^{y-1}}\bigoplus_{\beta\setminus\mu=(y-1+a)}  \P^{\beta}(\k)$ for each $\gamma$ with $\gamma\setminus \nu = (y+a)$ and $\nu$ with $\lambda\setminus \nu = 1^y$ is given by the leftmost diagram below. The subsequent equivalences are obtained by the following steps. 
\begin{itemize}
\item[(1)] Expanding the idempotent for $\lambda$ the map is nonzero if and only if $\mu\setminus\nu=1$. That is, $\mu=\nu+\square$ so that $\nu_{i_s}=\mu_{i_s}-1$ for some index $i_s$ and $\nu_j=\mu_j$ for all other $j \neq i_s$. After some computations the differential reduces to a nonzero multiple of the second diagram.
\item[(2)] By fully expanding the symmetrizers and anti-symmetrizers for $y+a-1$, $\mu$, $\nu$ and $y+a$ the map is equivalent to a sum of multiples $c_{\sigma}$ of $\1_{\P^{\beta}}\circ \sigma\circ \1_{\P^{\gamma}}$ over some subset $ J \in S_{|\gamma|}$. 
\item[(3)] Since $\1_{\P^{\beta}}\circ \sigma\circ \1_{\P^{\gamma}} = \1_{\P^{T_{row}}}\circ \1_{\P^{T'}}$ with $T_{row} \in \SYT(\beta)$ and $T' = \sigma T_{row} \in \SYT(\gamma)$ then by equation \eqref{eq:PlambdaOrthogonal} from Theorem \ref{thm:CompleteIdemp} the sum is nonzero only when $\beta = \gamma$ and $\sigma=id$. Thus, the third diagram is equivalent to a scalar multiple of the last diagram below. 
\end{itemize}
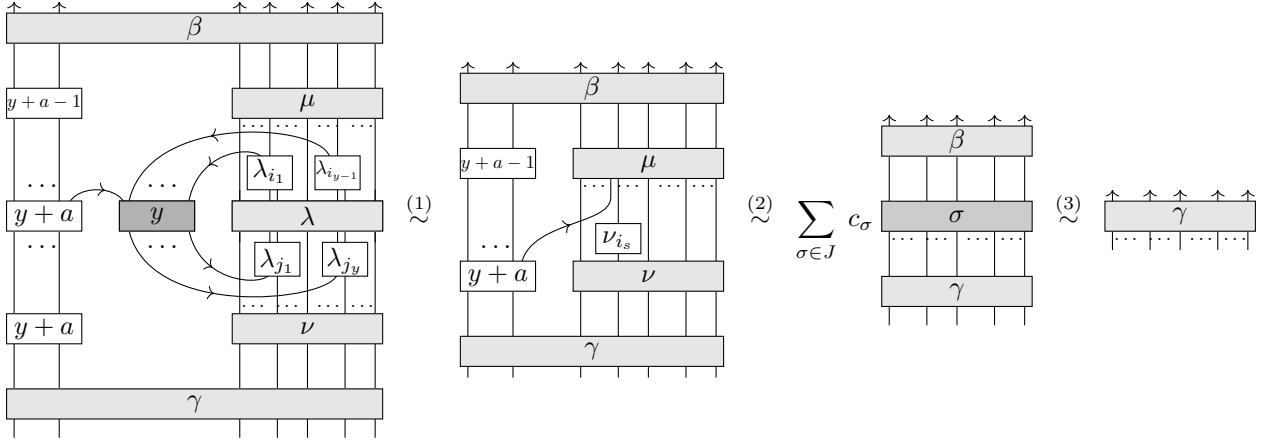
\[
\hackcenter{
\begin{tikzpicture}
\draw[->] (-2.9,-2.8) -- (-2.9,3);
\draw [->](-2.3,-2.8) -- (-2.3,3);
\draw [directed=.6](-2.2,0.2)..controls++(0,.4) and ++(0,.4)..(-1.4,.2);
\draw [fill=white]  (-3,1.45) rectangle (-2,1.85);
\draw [fill=white]  (-3,-.05) rectangle (-2,.35);
\node at (-2.5,0.15){$y+a$};
\node at (-2.5,1.65)[scale=.7]{$y+a-1$};
\node at (-2.5,.55){$\dots$};
%
\draw [->](0.1,-.2) -- (0.1,3);
\draw [->](0.5,-.2) -- (0.5,3);
\draw [->](1,-.2) -- (1,3);
\draw [->](1.4,-.2) -- (1.4,3);
\draw [->](1.9,-.2) -- (1.9,3);
\draw [directed=.5](1.3,.9)..controls++(0,.5) and ++(0,1.5)..(-1.4,.1);
\draw [directed=.5](.4,.9)..controls++(0,.2) and ++(0,1)..(-.6,.1);
\draw [fill=gray!20] (-3,2.45) rectangle (2,2.85);
\draw [fill=gray!60]  (-1.5,-.05) rectangle (-.5,.35);
\draw [fill=gray!20] (0,-.05) rectangle (2,.35);
\draw [fill=gray!20]  (0,1.45) rectangle (2,1.85);
\draw [fill=white]  (.2,.5) rectangle (.8,.95);
\draw [fill=white]  (1.1,.5) rectangle (1.7,.95);
\node at (-.5,2.65){$\beta$};
\node at (-1,0.15){$y$};
\node at (1,0.15){$\lambda$};
\node at (1,1.65){$\mu$};
\node at (.5,0.75){$\lambda_{i_1}$};
\node at (1.4,0.75)[scale=.75]{$\lambda_{i_{y-1}}$};
\node at (-.9,.55){$\dots$};
\node at (.75,1.35)[scale=.75]{$\dots$};
\node at (.3,1.35)[scale=.75]{$\dots$};
\node at (1.3,1.35)[scale=.75]{$\dots$};
\node at (1.75,1.35)[scale=.75]{$\dots$};
\begin{scope}[shift={(0,.3)},yscale=-1]
\draw (0.1,-.2) -- (0.1,3.1);
\draw (0.6,-.2) -- (0.6,3.1);
\draw (1,-.2) -- (1,3.1);
\draw (1.5,-.2) -- (1.5,3.1);
\draw (1.9,-.2) -- (1.9,3.1);
\draw [directed=.5] (-1.4,0.1)..controls++(0,1.5) and ++(0,.4)..(1.4,.9);
\draw [directed=.5](-.6,0.1)..controls++(0,1) and ++(0,.2)..(.5,.9);
\draw [fill=gray!60]  (-1.5,-.05) rectangle (-.5,.35);
\draw [fill=gray!20] (0,-.05) rectangle (2,.35);
\draw [fill=gray!20]  (0,1.45) rectangle (2,1.85);
\draw [fill=white]  (.3,.5) rectangle (.9,.95);
\draw [fill=white]  (1.2,.5) rectangle (1.8,.95);
\node at (-1,0.15){$y$};
\node at (1,0.15){$\lambda$};
\node at (1,1.65){$\nu$};
\node at (.6,0.75){$\lambda_{j_1}$};
\node at (1.5,0.75){$\lambda_{j_y}$};
\node at (-.9,.55){$\dots$};
\node at (.75,1.35)[scale=.75]{$\dots$};
\node at (.3,1.35)[scale=.75]{$\dots$};
\node at (1.3,1.35)[scale=.75]{$\dots$};
\node at (1.75,1.35)[scale=.75]{$\dots$};
\end{scope}
\begin{scope}[shift={(0,.3)},yscale=-1]
\draw [fill=white]  (-3,1.45) rectangle (-2,1.85);
\draw [fill=gray!20] (-3,2.45) rectangle (2,2.85);
\node at (-2.5,1.65){$y+a$};
\node at (-2.5,.55){$\dots$};
\node at (-.5,2.65){$\gamma$};
\end{scope}
\end{tikzpicture}}
\;\;\overset{(1)}{\sim}\;\;
\hackcenter{\begin{tikzpicture}
\draw[->] (-2.9,-1.2) -- (-2.9,3);
\draw [->](-2.3,-1.2) -- (-2.3,3);
\draw [directed=.5](-2.2,0.2)..controls++(0,.7) and ++(0,-.4)..(-1,1.2) to (-1,1.5);
\draw [fill=white]  (-3,1.45) rectangle (-2,1.85);
\draw [fill=white]  (-3,-.05) rectangle (-2,.35);
\node at (-2.5,0.15){$y+a$};
\node at (-2.5,1.65)[scale=.7]{$y+a-1$};
\node at (-2.5,.55){$\dots$};
%
%
\begin{scope}[shift={(-1.5,0)}]
\draw (0.1,-1.2) -- (0.1,3)[->];
\draw (0.6,-1.2) -- (0.6,3)[->];
\draw (1,-1.2) -- (1,3)[->];
\draw (1.5,-1.2) -- (1.5,3)[->];
\draw (1.9,-1.2) -- (1.9,3)[->];
\draw [fill=gray!20]  (-1.5,2.45) rectangle (2,2.85);
\draw [fill=gray!20]  (0,1.45) rectangle (2,1.85);
\draw [fill=gray!20] (0,-.05) rectangle (2,.35);
\draw [fill=white]  (.3,.45) rectangle (.9,.85);
\draw [fill=gray!20] (-1.5,-1.05) rectangle (2,-.65);
\node at (.25,2.65){$\beta$};
\node at (.25,-.85){$\gamma$};
\node at (.6,.625){$\nu_{i_s}$};
\node at (1,0.15){$\nu$};
\node at (1,1.65){$\mu$};
\node at (.75,1.35)[scale=.75]{$\dots$};
\node at (.3,1.35)[scale=.75]{$\dots$};
\node at (1.3,1.35)[scale=.75]{$\dots$};
\node at (1.75,1.35)[scale=.75]{$\dots$};
\end{scope}
\end{tikzpicture}}
\;\;\overset{(2)}{\sim}\;\;
\sum_{\sigma \in J}\;c_{\sigma}\;
\hackcenter{\begin{tikzpicture}
\begin{scope}[shift={(-1.5,0)}]
\draw (0.1,.2) -- (0.1,3)[->];
\draw (0.6,.2) -- (0.6,3)[->];
\draw (1,.2) -- (1,3)[->];
\draw (1.5,.2) -- (1.5,3)[->];
\draw (1.9,.2) -- (1.9,3)[->];
\draw [fill=gray!20]  (0,2.45) rectangle (2,2.85);
\draw [fill=gray!40]  (0,1.45) rectangle (2,1.85);
\draw [fill=gray!20] (0,.45) rectangle (2,.85);
\node at (1,2.65){$\beta$};
\node at (1,1.65){$\sigma$};
\node at (1,.65){$\gamma$};
\node at (.75,1.35)[scale=.75]{$\dots$};
\node at (.3,1.35)[scale=.75]{$\dots$};
\node at (1.3,1.35)[scale=.75]{$\dots$};
\node at (1.75,1.35)[scale=.75]{$\dots$};
\end{scope}
\end{tikzpicture}}
\;\;\overset{(3)}{\sim}\;\;
\hackcenter{\begin{tikzpicture}
\begin{scope}[shift={(-1.5,0)}]
\draw (0.1,2.2) -- (0.1,3)[->];
\draw (0.6,2.2) -- (0.6,3)[->];
\draw (1,2.2) -- (1,3)[->];
\draw (1.5,2.2) -- (1.5,3)[->];
\draw (1.9,2.2) -- (1.9,3)[->];
\draw [fill=gray!20]  (0,2.45) rectangle (2,2.85);
\node at (1,2.65){$\gamma$};
\node at (.75,2.35)[scale=.75]{$\dots$};
\node at (.3,2.35)[scale=.75]{$\dots$};
\node at (1.3,2.35)[scale=.75]{$\dots$};
\node at (1.75,2.35)[scale=.75]{$\dots$};
\end{scope}
\end{tikzpicture}}
\]
Thus, these maps are isomorphisms whenever $\beta = \gamma$ and zero otherwise. By Lemma \ref{lem:gaussian-elimination} there is a homotopy equivalence
\[
\B_{a}\P^{\lambda}(\k)\simeq 0
\rightarrow \dots \rightarrow 0
\rightarrow0 \dots
\rightarrow \P^{(a,\lambda)}(\k) \simeq \P^{(a,\lambda_n, \dots, \lambda_1)}(\k).
 \]
Inductively, we obtain $\B_{\lambda_{n}}\B_{\lambda_{n-1}}\dots\B_{\lambda_1}(\k)\simeq\B_{\lambda_n}\P^{(\lambda_{n-1},\dots,\lambda_1)}(\k) \simeq \P^\lambda(\k) \simeq S_{\lambda}$ for any partition $\lambda$. The dual statement $\B^*_{\lambda_{1}}\dots\B^*_{\lambda_{n-1}}\B^*_{\lambda_n}(S_{\lambda})\simeq\k$ follows from the biadjointness of $\B_a$ and $\B^*_a$. 
\end{proof}

Recall that by \eqref{eq:cliff-bern-1}, \eqref{eq:cliff-bern-2}, and \eqref{eq:cliff-bern-3} the condition that the fermionic vertex operators satisfy the Clifford algebra relations in the Boson-Fermion correspondence is equivalent to the Bernstein operators satisfying certain anticommutation relations. Thus, in order to categorify this construction we prove analogous relations for the categorical Bernstein operators. We now state these relations but the proofs,due to their technical complexity, are deferred until Section \ref{sec:PropCatBernsteinOps}. 

\begin{theorem}\label{thm:CatBernstein1} Given any $a,b \in \Z$, the categorical Bernstein operators satisfy the following chain homotopy relations in $\K(\H)$:
\begin{equation*}
\B_{a-1}\otimes \B_b \cong \begin{cases}
\B_{b-1}\otimes \B_a[-1] & a>b\\
\B_{b-1}\otimes \B_a[1] & a<b \\
0 & a=b\\
\end{cases}
\qquad \text{and} \qquad
\B^*_{a+1}\otimes \B^*_{b} \cong \begin{cases}
\B^*_{b+1}\otimes \B^*_{a}[-1] & a>b\\
\B^*_{b+1}\otimes \B^*_{a}[1] & a<b \\
0 & a=b.\\
\end{cases}
\end{equation*}
\end{theorem}

\begin{theorem}\label{thm:CatBernstein2}
Given any $a,b \in \Z$ and $n \in \N$, the categorical Bernstein operators satisfy the following chain homotopy relations in $\K(\H_n)$:
\[\B_{a+1}\otimes\B_{b+1}^* \cong \begin{cases}
\B_{b}^*\otimes\B_{a}[-1] &a<b \\
\B_{b}^*\otimes\B_{a}[1]&a>b.
\end{cases}\]
Moreover if $a\geq 0$ then $\B_{a}^*\otimes\B_{a} \simeq$ \emph{Cone }$(\B_{a+1}\otimes\B_{a+1}^* \to \1)$ and if $a<0$ then $\B_{a+1}\otimes\B_{a+1}^*[1] \simeq$ \emph{Cone} $(\1 \to \B_{a}^*\otimes\B_{a})$. Hence, there are distinguished triangles
\[\B_{a+1}\otimes\B_{a+1}^* \rightarrow \1 \rightarrow \B_{a}^*\otimes\B_{a} \;\; and \;\; \B_{a}^*\otimes\B_{a} \rightarrow \1 \rightarrow \B_{a+1}\otimes\B_{a+1}^*.\]
\end{theorem}

\subsection{A Categorical Boson-Fermion Correspondence}\label{subsec:CatBFcorrespondence} We begin by defining certain categorical operators on Mat$_{\Z\times \Z}(\K(\H))$. 

\begin{corollary}\label{cor:BernsteinMatrices}
For integers $i\in \Z$ define the operators $\emph{\textbf{B}}_i$ and $\emph{\textbf{B}}^*_i$ in \emph{Mat}$_{\Z \times \Z}(\K(\H))$ 
as infinite diagonal matrices with entries given by:
\begin{equation} \label{def:BernsteinMatrices}
(\emph{\textbf{B}}_i)_{n,m}:=\begin{cases} \B_{i-n} &  m=n\\ 0 & \text{ else}\end{cases}
\qquad\qquad
(\emph{\textbf{B}}_i^*)_{n,m}:=\begin{cases} \B^*_{i-n} & m=n\\ 0 & \text{ else.}\end{cases}
\end{equation}
These functors satisfy the following relations in \emph{Mat}$_{\Z \times \Z}(\K(\H))$:
\begin{equation}\label{eq:BernsteinMatRel1}
\mathbf{B}_{i-1} \mathbf{B}_j \cong \begin{cases}
\mathbf{B}_{j-1}\mathbf{B}_i[-1] & i>j\\
\mathbf{B}_{j-1}\mathbf{B}_i[1] & i<j \\
0 & i=j\\
\end{cases}
\qquad \text{and} \qquad
\mathbf{B}^*_{i+1}\mathbf{B}^*_{j} \cong \begin{cases}
\mathbf{B}^*_{j+1}\mathbf{B}^*_{i}[-1] & i>j\\
\mathbf{B}^*_{j+1}\mathbf{B}^*_{i}[1] & i<j \\
0 & i=j.\\
\end{cases}
\end{equation}
Additionally, in \emph{Mat}$_{\Z \times \Z}(\K(\H_n))$ they also satisfy the relations: 
\begin{align}
&\mathbf{B}_{i+1}\mathbf{B}_{j+1}^* \cong \begin{cases}
\mathbf{B}_{j}^*\mathbf{B}_{i}[-1] &i<j \\
\mathbf{B}_{j}^*\mathbf{B}_{i}[1]&i>j
\end{cases}\label{eq:BernsteinMatRel2}\\
&\mathbf{B}_{i+1}\mathbf{B}_{i+1}^* \rightarrow \mathbf{Id} \rightarrow \mathbf{B}_{i}^*\mathbf{B}_{i}\;\; \text{ and } \;\; \mathbf{B}_{i+1}^*\mathbf{B}_{i+1} \rightarrow \mathbf{Id} \rightarrow \mathbf{B}_{i}\mathbf{B}_{i}^* \;\; \text{ are distinguished triangles.}\label{eq:BernsteinMatRel3}
\end{align}
\end{corollary}

\begin{proof}
Since by Definition \ref{def:MatrixCat} morphisms in Mat$_{\Z\times \Z}(\K(\H))$ are given by matrices whose entries are chain homomorphisms and whose composition is given by composing the morphisms in $\K(\H)$ entry-wise, then the relations are an immediate consequence of Theorems \ref{thm:CatBernstein1} and \ref{thm:CatBernstein2}. Moreover, since any $\mathsf{v}=(\mathsf{v}_c)_{c\in\Z}$ only finitely many nonzero entries, the action of $\mathbf{B}_i$ and $\mathbf{B}_i^*$ on $\mathsf{V}_{Fock}$ is integrable and thus well defined.
\end{proof}

We define the \emph{charge functor} $\mathcal{Q} \in$ Mat$_{\Z\times \Z}(\K(\H))$ as the infinite dimensional matrix with $\1$'s on the lower off-diagonal and zeroes elsewhere. Likewise, we define $\mathcal{Q}^{-1} \in$ Mat$_{\Z\times \Z}(\K(\H))$ as the infinite dimensional matrix with $\1$'s on the upper off-diagonal and zeroes everywhere else. Functor $\mathcal{Q}$ raises the charge by shifting the indexes of $\mathsf{v} = (\dots, \mathsf{v}_{c-1},\mathsf{v}_c, \mathsf{v}_{c+1}, \dots)^T \in \mathsf{V}_{Fock}$ down by one. Analogously, $\mathcal{Q}^{-1}$ does the opposite. These functors are mutual inverses since $\mathcal{Q}\mathcal{Q}^{-1}\simeq\mathbf{Id}\simeq \mathcal{Q}^{-1}\mathcal{Q}$. 

A straightforward computation shows that for any integer $i \in \Z$, the functors $\mathbf{B}_i$, $\mathbf{B}_i^*$, and $\mathcal{Q}$ satisfy the following commutation relations in Mat$_{\Z\times \Z}(\K(\H))$:  
\begin{equation}\label{eq:BQcommutation}
\mathbf{B}_{i}\simeq \mathcal{Q} \mathbf{B}_{i-1} \mathcal{Q}^{-1}
\qquad\qquad
\mathbf{B}_{i}^*\simeq \mathcal{Q}\mathbf{B}_{i-1}^* \mathcal{Q}^{-1}
\end{equation}

We can now introduce the categorical Fermionic creation and annihilation operators as functors in Mat$_{\Z\times \Z}(\K(\H))$ and present the main result, a categorification of Theorem \ref{thm:BosonFermionCorrespondence}.

\begin{definition}
Given any $i\in \Z$, the \emph{ Fermionic functors} are the operators in \emph{Mat}$_{\Z\times \Z}(\K(\H))$ given by
\begin{equation}
\Psi_i:=\mathbf{B}_i\mathcal{Q} \qquad \text{and} \qquad \Psi_i^*:=\mathcal{Q}^{-1}\mathbf{B}^*_i. 
\end{equation}
\end{definition}

\begin{theorem}[categorical Boson-Fermion correspondence]\label{thm:catBF}
For any $i,j \in \Z$ the Fermionic functors satisfy the following relations in \emph{Mat}$_{\Z\times \Z}(\K(\H))$:
\vspace{2mm}
\begin{enumerate}
\item \label{R1} $(\Psi_i)^2 \simeq 0$ \qquad and \qquad  $\Psi_i \Psi_j \cong \begin{cases} 
\Psi_j \Psi_i [-1] &\text{ if } i > j \\
\Psi_j \Psi_i [1] & \text{ if } i < j \end{cases}$
\item \label{R2} $ (\Psi_i^*)^2 \simeq 0$ \qquad and \qquad 
$\Psi^*_i \Psi^*_j \cong 
\begin{cases}
\Psi^*_j \Psi^*_i [-1] &\text{ if } i > j \\ \Psi^*_j \Psi^*_i [1] & \text{ if } i < j \end{cases}$
\end{enumerate}
Moreover, for any $n \in \N$ the following relations also hold in \emph{Mat}$_{\Z \times \Z}(\K(\H_n))$:
\begin{itemize}
\item[(3)] \label{R3} $\Psi_i \Psi_j^* \cong \begin{cases}
 \Psi_j^* \Psi_i [-1] & \text{ if } i < j \\ \Psi_j^* \Psi_i [1] & \text{ if } i > j \end{cases}$
\item[(4)] \label{R4}  $\Psi_i \Psi_i^* \rightarrow \;$\emph{\textbf{Id}}$ \;\rightarrow \Psi_i^* \Psi_i$ and $\Psi_i^* \Psi_i \rightarrow \;$\emph{\textbf{Id}}$ \;\rightarrow \Psi_i \Psi_i^*$ are distinguished triangles.
\end{itemize}
\end{theorem}

\begin{proof}
The result follows easily from Corollary \ref{cor:BernsteinMatrices}. Given any $i,j \in \Z$ we use the commutation relations from \eqref{eq:BQcommutation} and compute the required compositions.
\begin{align*}
\Psi_i\Psi_j &= \mathbf{B}_i \mathcal{Q} \mathbf{B}_j \mathcal{Q} \simeq \mathcal{Q}\mathbf{B}_{i-1} \mathbf{B}_j \mathcal{Q} & \Psi_i^*\Psi_j^* &= \mathcal{Q}^{-1}\mathbf{B}^*_i \mathcal{Q}^{-1} \mathbf{B}_j^* \simeq \mathcal{Q}^{-1}\mathcal{Q}^{-1}\mathbf{B}^*_{i+1} \mathbf{B}_j^*
\\
\Psi_j\Psi_i &= \mathbf{B}_j \mathcal{Q} \mathbf{B}_i \mathcal{Q} \simeq \mathcal{Q}\mathbf{B}_{j-1} \mathbf{B}_i \mathcal{Q}
&
\Psi^*_j\Psi^*_i &= \mathcal{Q}^{-1}\mathbf{B}^*_j \mathcal{Q}^{-1} \mathbf{B}_i^* \simeq \mathcal{Q}^{-1}\mathcal{Q}^{-1} \mathbf{B}^*_{j+1} \mathbf{B}_i^*
\end{align*}
Since the charge operators $\mathcal{Q}$ and $\mathcal{Q}^{-1}$ are invertible, then relations (1) and (2) of the theorem follow directly from \eqref{eq:BernsteinMatRel1} in Corollary \ref{cor:BernsteinMatrices}. 
Similarly we have, 
\begin{align*}
&\Psi_i \Psi_j^* = \mathbf{B}_i \mathcal{Q} \mathcal{Q}^{-1}\mathbf{B}_j^* \simeq \mathbf{B}_i\mathbf{B}_j^*\\
&\Psi_j^* \Psi_i = \mathcal{Q}^{-1}\mathbf{B}_j^* \mathbf{B}_i \mathcal{Q} \simeq \mathcal{Q}^{-1}\mathbf{B}_j^* \mathcal{Q}\mathbf{B}_{i-1} \simeq \mathcal{Q}^{-1} \mathcal{Q}\mathbf{B}_{j-1}^*\mathbf{B}_{i-1} \simeq \mathbf{B}_{j-1}^*\mathbf{B}_{i-1}
\end{align*}
Once again, relations (3) and (4) follow directly from \eqref{eq:BernsteinMatRel2} and \eqref{eq:BernsteinMatRel3} in Corollary \ref{cor:BernsteinMatrices}. 
\end{proof}


\section{Properties of Categorical Bernstein Operators}\label{sec:PropCatBernsteinOps}

\noindent \textbf{Notation:} We will often sum over indexes $t\geq$ max$(a_1, \dots, a_n)$ or min$(a_1,\dots,a_n) \geq t$. For notational simplicity, we will simply denote this by $t\geq (a_1, \dots, a_n)$ or $(a_1,\dots,a_n)\geq t$.

\subsection{Proof of Theorem \ref{thm:CatBernstein1}}\hfill
\medskip

In this section we will compute the tensor products $\B_{b-1}\otimes \B_a$ and $\B_{a-1}\otimes \B_b$ for integers $a,b$. In order to do this we appeal to Definition \ref{def:ChainTensorProd} and compute the total complex of certain bi-complexes. Since $\B_{b-1}\otimes \B_a =$ Tot$^{\oplus}\left\lbrace \P^{(x)}\Q^{(y)^t}\B_a, \d_{b-1}\otimes \1_{\B_a}, (-1)^x\1_{\B_{b-1}}\otimes \d_{a}\right\rbrace_{x-y=a}$ we first prove certain homotopy equivalences for the subcomplexes $\P^{(x)}\Q^{(y)^t}\B_{a}$ for fixed integers $x,y$. With these equivalences, Lemma \ref{lem:gaussian-elimination}, and Proposition \ref{prop:SimultSimp} we derive the desired categorical commutation relations for $\B_{b-1} \otimes \B_a$ and $\B_{a-1}\otimes \B_b$ stated in Theorem \ref{thm:CatBernstein1}. Thus, we begin with a sequence of technical lemmas. 
\begin{lemma}\label{lem:QC-reduction}
Given any fixed $a \in \Z$ and $n \in \N$, the chain complex $\left\lbrace \Q^{(n)^t}\B_a, \1_{n}\otimes \d_a \right\rbrace$ is homotopy equivalent to \emph{Cone ($\mathsf{D}$)} where $\mathsf{D}$ is the chain map:
\begin{equation*}
\left\lbrace \bigoplus_{x\geq (n,-a,0)}\P^{(x+a)}\Q^{(x,n)}[x-1], \mathsf{d}_x =\hackcenter{\begin{tikzpicture} [scale=.25]
\draw[line width= 0.7mm,
    decoration={markings,mark=at position 0.04 with {\arrow[line width=.3mm]{<}}},
    postaction={decorate},shorten >=.6pt] (4.5,2) to (4.5,8.5);
\draw[line width= 0.7mm,
    decoration={markings,mark=at position 1 with {\arrow[line width=.3mm]{>}}},
    postaction={decorate},shorten >=1pt] (-1.5,2) to (-1.5, 8.5);
\draw[line width= 0.7mm,
    decoration={markings,mark=at position 0.04 with {\arrow[line width=.3mm]{<}}},
    postaction={decorate},shorten >=.6pt] (1.5,2) to (1.5, 8.5);
\draw[directed=.6] (-.8,5.5)..controls ++(0,1) and ++(0,1)..(4.2, 5.5);
\draw[fill =white] (-2.7,4.5) rectangle (-0.4,5.5);
\draw[fill =gray!20] (0.4,7) rectangle (5.6,8);
\draw[fill =gray!60] (5.6,4.5) rectangle (3.3,5.5);
\draw[fill =gray!20] (0.4,4) rectangle (5.6,3);
\node at (4.5,5) [scale=.75]{$x$};
\end{tikzpicture}}\right\rbrace \xrightarrow{\mathsf{D}} \left\lbrace\bigoplus_{x' \geq (0,-a+1)}^{n-1}\P^{(x'+a-1)}\Q^{(n-1,x')}[x'],\mathsf{d}_{x'}=\hackcenter{\begin{tikzpicture} [scale=.25]
\draw[line width= 0.7mm,
    decoration={markings,mark=at position 0.04 with {\arrow[line width=.3mm]{<}}},
    postaction={decorate},shorten >=.6pt] (4.5,2) to (4.5,8.5);
\draw[line width= 0.7mm,
    decoration={markings,mark=at position 1 with {\arrow[line width=.3mm]{>}}},
    postaction={decorate},shorten >=1pt] (-1.5,2) to (-1.5, 8.5);
\draw[line width= 0.7mm,
    decoration={markings,mark=at position 0.04 with {\arrow[line width=.3mm]{<}}},
    postaction={decorate},shorten >=.6pt] (1.5,2) to (1.5, 8.5);
\draw[directed=.6] (-.8,5.5)..controls ++(0,1) and ++(0,1)..(4.2, 5.5);
\draw[fill =white] (-2.7,4.5) rectangle (-0.4,5.5);
\draw[fill =gray!20] (0.4,7) rectangle (5.6,8);
\draw[fill =gray!60] (5.6,4.5) rectangle (3.3,5.5);
\draw[fill =gray!20] (0.4,4) rectangle (5.6,3);
\node at (4.5,5) [scale=.75]{$x'$};
\end{tikzpicture}} \right\rbrace
\end{equation*}
given by $\mathsf{D}= \hackcenter{\begin{tikzpicture} [scale=.25]
\draw[line width= 0.7mm,
    decoration={markings,mark=at position 0.04 with {\arrow[line width=.3mm]{<}}},
    postaction={decorate},shorten >=.6pt] (4.5,2) to (4.5,8.5);
\draw[line width=0.7mm, -] (-1.5,2) to (-1.5, 8.5);
\draw[line width= 0.7mm,
    decoration={markings,mark=at position 0.04 with {\arrow[line width=.3mm]{<}}},
    postaction={decorate},shorten >=.6pt] (1.5,2) to (1.5, 8.5);
\draw[directed=.6] (-.8,5.5)..controls ++(0,1) and ++(0,1)..(4.2, 5.5);
\draw[fill =white] (-3,4.5) rectangle (-0,5.5);
\draw[fill =gray!20] (0.4,7) rectangle (5.6,8);
\draw[fill =gray!60] (5.6,4.5) rectangle (3.3,5.5);
\draw[fill =gray!20] (0.4,4) rectangle (5.6,3);
\node at (-1.5,5) [scale=.5]{$n+a$};
\node at (3,3.5) [scale=.5]{$(n,n)$};
\node at (4.5,5) [scale=.5]{$n$};
\node at (3,7.5) [scale=.5]{$(x+a-1,y)$};
\begin{scope}[shift={(0,4)}]
\draw[line width=0.7mm, -] (4.5,2) to (4.5,8.5);
\draw[line width= 0.7mm,
    decoration={markings,mark=at position 1 with {\arrow[line width=.3mm]{>}}},
    postaction={decorate},shorten >=1pt] (-1.5,2) to (-1.5, 8.5);
\draw[line width=0.7mm, -] (1.5,2) to (1.5, 8.5);
\draw[directed=.6] (-.8,5.5)..controls ++(0,1) and ++(0,1)..(4.2, 5.5);
\draw[fill =white] (-3,4.5) rectangle (-0,5.5);
\draw[fill =gray!20] (0.4,7) rectangle (5.6,8);
\draw[fill =gray!60] (5.6,4.5) rectangle (3.3,5.5);
\draw[fill =gray!20] (0.4,4) rectangle (5.6,3);
\node at (-1.5,5) [scale=.5]{$n+a-1$};
\node at (3,3.5) [scale=.5]{$(n,n-1)$};
\node at (4.5,5) [scale=.5]{$n$};
\node at (3,7.5) [scale=.5]{$(n-1,n-1)$};
\end{scope}
\end{tikzpicture}} 
$ \;\; for $x=n$ and zero otherwise. 
\end{lemma}
\begin{proof}
We begin by applying the isomorphism from Propositions (\ref{prop:Q*Pswap}) and  (\ref{prop:PPmerge}) to each chain group of $\Q^{(n)^t} \otimes \B_a$. 
\begin{align*}
\Q^{(n)^t} \otimes \B_a& \simeq [ \dots \rightarrow \P^{(x+a)}\Q^{(n)^t}\Q^{(x)^t}[x] \oplus \P^{(x+a-1)}\Q^{(n-1)^t}\Q^{(x)^t}[x] \rightarrow \dots ]\\
& \simeq [\dots \rightarrow \bigoplus_{s=0}^{(n,x)} \P^{(x+a)}\Q^{(n+x-s,s)^t}[x] \oplus  \bigoplus_{s=0}^{(n-1,x)}\P^{(x+a-1)}\Q^{(n-1+x-s,s)^t}[x] \rightarrow \dots]
\end{align*}
Pre and post composing with these isomorphisms, for each fixed index $s$ in homological degree $x$, the differential on each summand of the resulting chain complex is a morphism
\[
\begin{tikzcd}[row sep=.3mm, column sep=.3mm]
& \P^{(x+a)}\Q^{(n+x-s,s)^t}[x] && (\delta_{s',s}+\delta_{s',s-1})\P^{(x+a-1)}\Q^{(n+x-1-s',s')^t}[x-1]\\
\mathsf{d}:& \oplus & \to &\oplus\\
&\P^{(x+a-1)}\Q^{(n+x-1-s,s)^t}[x] && (\delta_{s',s}+\delta_{s',s-1})\P^{(x+a-2)}\Q^{(n+x-2-s',s')^t}[x-1]
\end{tikzcd}\]
given explicitly by the matrix
\begin{align}\label{eq:s-differential}\d=
\begin{pmatrix} (\delta_{s',s}+\delta_{s',s-1})
\hackcenter{\begin{tikzpicture} [scale=.30]
\draw[line width= 0.7mm,
    decoration={markings,mark=at position 0.04 with {\arrow[line width=.3mm]{<}}},
    postaction={decorate},shorten >=.6pt] (4.5,2) to (4.5,8.5);
\draw[line width= 0.7mm,
    decoration={markings,mark=at position 1 with {\arrow[line width=.3mm]{>}}},
    postaction={decorate},shorten >=.6pt] (-1.5,2) to (-1.5, 8.5);
\draw[line width= 0.7mm,
    decoration={markings,mark=at position 0.04 with {\arrow[line width=.3mm]{<}}},
    postaction={decorate},shorten >=.6pt] (1.5,2) to (1.5, 8.5);
\draw[directed=.6] (-.8,5.5)..controls ++(0,1) and ++(0,1)..(4.2, 5.5);
\draw[fill =white] (-2.7,4.5) rectangle (-0.4,5.5);
\draw[fill =gray!20] (0.1,7) rectangle (5.9,8);
\draw[fill =gray!60] (5.6,4.5) rectangle (3.3,5.5);
\draw[fill =gray!60] (2.5,4.5) rectangle (0.4,5.5);
\draw[fill =gray!20] (0.1,4) rectangle (5.8,3);
\node at (-1.5,5) [scale=.5]{$x+a$};
\node at (3,3.5) [scale=.5]{$(n+x-s,s)$};
\node at (1.5,5)[scale=.5] {$n$};
\node at (4.5,5) [scale=.5]{$x$};
\node at (3,7.5) [scale=.5]{$(n+x-1-s',s')$};
\end{tikzpicture}} & \delta_{s,s'}\hackcenter{\begin{tikzpicture} [scale=.27]
\draw[line width= 0.7mm,
    decoration={markings,mark=at position 0.04 with {\arrow[line width=.3mm]{<}}},
    postaction={decorate},shorten >=.6pt] (4.5,2) to (4.5,9.5);
\draw[line width= 0.7mm,
    decoration={markings,mark=at position 1 with {\arrow[line width=.3mm]{>}}},
    postaction={decorate},shorten >=1pt] (-1.5,2) to (-1.5, 9.5);
\draw[line width= 0.7mm,
    decoration={markings,mark=at position 0.04 with {\arrow[line width=.3mm]{<}}},
    postaction={decorate},shorten >=.6pt] (1.5,2) to (1.5, 9.5);
\draw[ directed= .6]  (2.2,5.5) to (4.2, 6.5);
\draw[fill =white] (-3,4.5) rectangle (-0.1,5.5);
\draw[fill =gray!20] (-0.2,8) rectangle (6.2,9);
\draw[fill =gray!60] (6,6.5) rectangle (3,7.5);
\draw[fill =gray!60] (2.7,4.5) rectangle (0.2,5.5);
\draw[fill =gray!20] (-0.1,4) rectangle (6.2,3);
\node at (-1.5,5) [scale=.5]{$x+a-1$};
\node at (3,3.5) [scale=.5]{$(n+x-1-s,s)$};
\node at (1.5,5)[scale=.5] {$n-1$};
\node at (4.5,7) [scale=.5]{$x-1$};
\node at (3,8.5) [scale=.5]{$(n+x-1-s',s')$};
\end{tikzpicture}} \\
0&  (\delta_{s',s}+\delta_{s',s-1})\hackcenter{\begin{tikzpicture} [scale=.27]
\draw[line width= 0.7mm,
    decoration={markings,mark=at position 0.04 with {\arrow[line width=.3mm]{<}}},
    postaction={decorate},shorten >=.6pt] (4.5,2) to (4.5,8.5);
\draw[line width= 0.7mm,
    decoration={markings,mark=at position 1 with {\arrow[line width=.3mm]{>}}},
    postaction={decorate},shorten >=1pt] (-1.5,2) to (-1.5, 8.5);
\draw[line width= 0.7mm,
    decoration={markings,mark=at position 0.04 with {\arrow[line width=.3mm]{<}}},
    postaction={decorate},shorten >=.6pt] (1.5,2) to (1.5, 8.5);
\draw[directed=.6] (-.8,5.5)..controls ++(0,1) and ++(0,1)..(4.2, 5.5);
\draw[fill =white] (-3,4.5) rectangle (-0.1,5.5);
\draw[fill =gray!20] (-.2,7) rectangle (6.2,8);
\draw[fill =gray!60] (5.6,4.5) rectangle (3.3,5.5);
\draw[fill =gray!20] (-0.1,4) rectangle (6,3);
\node at (-1.5,5) [scale=.5]{$x+a-1$};
\node at (3,3.5) [scale=.5]{$(n+x-1-s,s)$};
\node at (4.5,5) [scale=.5]{$x$};
\node at (3,7.5) [scale=.5]{$(n+x-2-s',s')$};
\end{tikzpicture}} \\
\end{pmatrix}.
\end{align}
The differential from $ \P^{(x+a-1)}\Q^{(n-1+x-s,s)^t}[x]\rightarrow \P^{(x+a-1)}\Q^{(n+x-1-s',s')^t}[x-1]$ on the top right corner of \eqref{eq:s-differential} equals $\rho_{n-2}^{s'}\circ \iota_s^{n-2}$. By Proposition \ref{prop:PPmerge} it is an isomorphism if and only if $s=s'$ and is zero otherwise. Thus, the map between homological degrees $x$ and $x-1$ given below is zero whenever $s\neq s'$. 
\[ 
\d: \bigoplus_{s=0}^{(n-1,x)}\P^{(x+a-1)}\Q^{(n-1+x-s,s)^t}[x]\rightarrow \bigoplus_{s'=0}^{(n,x-1)} \P^{(x+a-1)}\Q^{(n+x-1-s',s')^t}[x-1].\]
Specifically, for $n>x$ this map becomes
\[
\d: \bigoplus_{s=0}^{x} \P^{(x+a-1)}\Q^{(n-1+x-s,s)^t}[x] \to \bigoplus_{s'=0}^{x-1}\P^{(x+a-1)}\Q^{(n+x-1-s',s')^t}[x-1].\]
Since the terms with $0\leq s \leq x-1$ in homological degree $x$ are bijectively mapped onto the terms with $0\leq s'\leq x-1$ in homological degree $x-1$, then for each homological degree $x<n$ the only terms that are not canceled by Lemma \ref{lem:gaussian-elimination} are $\P^{(x+a-1)}\Q^{(n-1,x)^t}[x]$. Moreover, since the maps
\[
\P^{(x+a-1)}\Q^{(n-1,x)^t}[x] \to \bigoplus_{s'=0}^{x-1}\P^{(x+a-1)}\Q^{(n+x-1-s',s')^t}[x-1]\]
are zero, then the Gaussian elimination will not alter the existing arrow between $\P^{(x+a-1)}\Q^{(n-1,x)^t}[x] \to \P^{(x+a-2)}\Q^{(n-1,x-1)^t}[x-1]$. As a result, there is a chain homotopy equivalence
\[\left\lbrace \Q^{(n)^t}\B_a, \d_x \right\rbrace_{n>x\geq (0,-a+1)} \simeq \left\lbrace  \P^{(x+a-1)}\Q^{(n-1,x)^t}[x], \d_x \right\rbrace_{n>x\geq (0,-a+1)}.\] 
Likewise, if $x\geq n$ then
\[
\d: \bigoplus_{s=0}^{n-1} \P^{(x+a)}\Q^{(n+x-s,s)^t}[x+1] \to \bigoplus_{s'=0}^{n}\P^{(x+a)}\Q^{(n+x-s',s')^t}[x].\]
By the same argument as before we obtain that $\lbrace \Q^{(n)^t}\B_a, \d_x \rbrace_{x\geq n} \simeq \lbrace  \P^{(x+a)}\Q^{(x,n)^t}[x], \d_x \rbrace_{x\geq n}$.

When considering $\d:(\Q^{(n)^t}\B_a)_n \to (\Q^{(n)^t}\B_a)_{n-1}$, however, the situation is different since we have an isomorphism between all terms in the source and target as follows
\[
\d: \bigoplus_{s=0}^{n-1} \P^{(n+a-1)}\Q^{(2n-1-s,s)^t}[n] \to \bigoplus_{s'=0}^{n-1}\P^{(n+a-1)}\Q^{(2n-1-s',s')^t}[n-1].\]
Therefore, the Gaussian elimination in this homological degree alters the differential from $\P^{(n+a)}\Q^{(n,n)^t}[n] \to \P^{(n+a-1)}\Q^{(n-1,n)^t}[n-1]$, which by Lemma \ref{lem:gaussian-elimination} is given by the diagram for $\mathsf{D}$ at $x=n$ in the statement of the Lemma \ref{lem:QC-reduction}.
Putting all this together we see that
\[\Q^{(n)^t}\B_{a} \simeq \dots \xrightarrow{\mathsf{d}_x} \P^{(n+a)}\Q^{(n,n)^t}[n] \xrightarrow{\mathsf{D}}\P^{(n+a-2)}\Q^{(n-1,n-1)^t}[n-1]\xrightarrow{\mathsf{d}_x} \dots.\] 
Define the chain map $\mathsf{D}:\lbrace  \P^{(x+a)}\Q^{(x,n)^t}[x-1], \d_x \rbrace_{x\geq n} \to \lbrace  \P^{(x+a-1)}\Q^{(n-1,x)^t}[x], \d_x \rbrace_{n>x\geq (0,-a+1)}$ by setting $\mathsf{D}_x =0$ for all $x\neq n$ and $\mathsf{D}_n$ equal to the previous diagram. It then follows that $\lbrace \Q^{(n)^t}\B_a, \d_x \rbrace_{x\geq(0,-a+1)}$ is homotopy equivalent to Cone$(\mathsf{D})$.
\end{proof}

\begin{lemma}\label{lem:CP-reduction}
Given any $b \in \Z$ and $n \in \N$, the chain complex $\lbrace \B_{b-1}\P^{(n)},\mathsf{d}_{b-1} \otimes \1_n \rbrace $ is homotopy equivalent to \emph{Cone ($\mathsf{D}$)}, where $\mathsf{D}$ is the chain map:
\begin{equation*}
\left\lbrace \bigoplus_{y\geq (n,0,b-1)}\P^{(y,n)}\Q^{(y-b+1)^t}[y-b], \d_y \right\rbrace \xrightarrow{\mathsf{D}} \left\lbrace\bigoplus^{n-1}_{y\geq (0,b)}\P^{(n-1,y)}\Q^{(y-b)^t}[y-b+1],\mathsf{d}_y \right\rbrace
\end{equation*}
where $\mathsf{d}_y= \hackcenter{\begin{tikzpicture} [xscale=-1,scale=.25]
\draw[line width= 0.7mm,
    decoration={markings,mark=at position 1 with {\arrow[line width=.3mm]{>}}},
    postaction={decorate},shorten >=1pt] (4.5,2) to (4.5,8.9);
\draw[line width= 0.7mm,
    decoration={markings,mark=at position 0.04 with {\arrow[line width=.3mm]{<}}},
    postaction={decorate},shorten >=.6pt] (-1.5,2) to (-1.5, 8.9);
\draw[line width= 0.7mm,
    decoration={markings,mark=at position 1 with {\arrow[line width=.3mm]{>}}},
    postaction={decorate},shorten >=1pt] (1.5,2) to (1.5, 8.9);
\draw[directed=.4] (4.2, 5.5)..controls ++(0,1) and ++(0,1)..(-.8,5.5);
\draw[fill =gray!60] (-3,4.5) rectangle (-0,5.5);
\draw[fill =gray!20] (0.4,7) rectangle (5.6,8);
\draw[fill =white] (5.6,4.5) rectangle (3.3,5.5);
\draw[fill =gray!20] (0.4,4) rectangle (5.6,3);
\node at (-1.5,5) [scale=.5]{$y-b+1$};
\end{tikzpicture}}\;\;$ and $\;\;
\mathsf{D}= \hackcenter{\begin{tikzpicture} [xscale=-1,scale=.25]
\draw[line width=0.7mm, -] (4.5,1.7) to (4.5,8.5);
\draw[line width= 0.7mm,
    decoration={markings,mark=at position 0.04 with {\arrow[line width=.3mm]{<}}},
    postaction={decorate},shorten >=.6pt] (-1.5,1.7) to (-1.5, 8.5);
\draw[line width=0.7mm, -] (1.5,1.7) to (1.5, 8.5);
\draw[directed=.4] (4.2, 5.5)..controls ++(0,1) and ++(0,1)..(-.8,5.5);
\draw[fill =gray!60] (-3,4.5) rectangle (-0,5.5);
\draw[fill =gray!20] (0.4,7) rectangle (5.6,8);
\draw[fill =white] (5.6,4.5) rectangle (3.3,5.5);
\draw[fill =gray!20] (0.4,4) rectangle (5.6,3);
\node at (-1.5,5) [scale=.5]{$n-b+1$};
\node at (3,3.5) [scale=.5]{$(n,n)$};
\node at (4.5,5) [scale=.5]{$n$};
\begin{scope}[shift={(0,4)}]
\draw[line width= 0.7mm,
    decoration={markings,mark=at position 1 with {\arrow[line width=.3mm]{>}}},
    postaction={decorate},shorten >=1pt] (4.5,2) to (4.5,8.9);
\draw[line width=0.7mm, -] (-1.5,2) to (-1.5, 8.9);
\draw[line width= 0.7mm,
    decoration={markings,mark=at position 1 with {\arrow[line width=.3mm]{>}}},
    postaction={decorate},shorten >=1pt] (1.5,2) to (1.5, 8.9);
\draw[directed=.4] (4.2, 5.5)..controls ++(0,1) and ++(0,1)..(-.8,5.5);
\draw[fill =gray!60] (-3,4.5) rectangle (-0,5.5);
\draw[fill =gray!20] (0.4,7) rectangle (5.6,8);
\draw[fill =white] (5.6,4.5) rectangle (3.3,5.5);
\draw[fill =gray!20] (0.4,4) rectangle (5.6,3);
\node at (-1.5,5) [scale=.5]{$n-b$};
\node at (3,3.5) [scale=.5]{$(n,n-1)$};
\node at (4.5,5) [scale=.5]{$n$};
\node at (3,7.5) [scale=.5]{$(n-1,n-1)$};
\end{scope}
\end{tikzpicture}}\;\;$ for y=n and zero otherwise. 
\end{lemma}

\begin{proof}
The proof is identical to Lemma \ref{lem:QC-reduction}.
\end{proof}

Since $\B_{b-1} \otimes \B_a$ is the total complex of the bi-complex with columns indexed by the homological degrees of $\B_a$ and whose rows are indexed by the homological degrees of $\B_{b-1}$, then combining  \eqref{eq:Ca} and Definition \ref{def:ChainTensorProd} we have: 
\begin{equation*}\label{eq:Cb+1Ca-original}
\B_{b-1}\otimes\B_a = \text{Tot}^{\oplus}\left\lbrace \B_{b-1} \P^{(x)}\Q^{(x-a)^t} , \d_{b-1}\otimes \1_x,(-1)^x\1_{\B_{b-1}} \otimes \d_a \right\rbrace_{x\geq (a,0)}.
\end{equation*}
Since the indexing set of the bi-complex is $I=\lbrace x\geq$ max$(0,a)\rbrace$ which is bounded below, by Proposition \ref{prop:SimultSimp} we can apply Lemma \ref{lem:CP-reduction} to each column in $\B_{b-1}\otimes \B_a$ and simultaneously reduce all the columns of the bi-complex. The simultaneous chain homotopies, however, alter the differentials along the rows in a nontrivial manner which Lemma \ref{lem:CP-reduction} does not address. In the following lemma we investigate the bi-complex further. 
 
\begin{proposition}\label{prop:F} Given any $a,b \in \Z$, the chain complex $\B_{b-1}\otimes\B_a  \in \K(\H)$ is homotopy equivalent to \emph{Tot}$^\oplus\left\lbrace\bigoplus_{y\geq (0,b-1)}\mathcal{F}_{x,y}^{\mathbb{I}} \oplus \mathcal{F}_{x,y}^{\mathbb{II}}, \mathsf{d}_y,\mathsf{d}^x\right\rbrace_{x\geq (a,0)}$ defined by:
\[
\mathcal{F}_{x,y}^{\mathbb{I}} := \bigoplus_{s=0}^{(y,x)} \bigoplus_{r=0}^{(y-b+1,x-a)} \P^{(x+y-s,s)}\Q^{(x+y+1-a-b-r,r)^t}[x+y-a-b+1]
\]
\[
\mathcal{F}_{x,y}^{\mathbb{II}} := \bigoplus_{s=0}^{(y,x-1)} \bigoplus_{r=0}^{(y-b,x-a)} \P^{(x+y-1-s,s)}\Q^{(x+y-a-b-r,r)^t}[x+y-a-b+1]
\]
where $\mathsf{d}_y:\mathcal{F}_{x,y}^{\mathbb{I}} \oplus \mathcal{F}_{x,y}^{\mathbb{II}} \to \mathcal{F}_{x,y-1}^{\mathbb{I}} \oplus \mathcal{F}_{x,y-1}^{\mathbb{II}}$ and  $\mathsf{d}^x:\mathcal{F}_{x,y}^{\mathbb{I}} \oplus \mathcal{F}_{x,y}^{\mathbb{II}} \to \mathcal{F}_{x-1,y}^{\mathbb{I}} \oplus \mathcal{F}_{x-1,y}^{\mathbb{II}}$ on each summand $s,r \mapsto s',r'$ is given by:
\begin{equation*}
{\mathsf{d}_{y}|_{s,r}} =  
\left(\begin{matrix}
\hackcenter{\begin{tikzpicture} [scale=.35, xscale=.7, yscale=1]
\draw[line width= 0.7mm,
    decoration={markings,mark=at position 1 with {\arrow[line width=.3mm]{>}}},
    postaction={decorate},shorten >=1pt] (-4.5,2.5) to (-4.5,8);
\draw[line width= 0.7mm,
    decoration={markings,mark=at position 0.06 with {\arrow[line width=.3mm]{<}}},
    postaction={decorate},shorten >=.6pt] (4.5,2.5) to (4.5,8);
\draw[line width= 0.7mm,
    decoration={markings,mark=at position 1 with {\arrow[line width=.3mm]{>}}},
    postaction={decorate},shorten >=1pt] (-1.5,2.5) to (-1.5, 8);
\draw[line width= 0.7mm,
    decoration={markings,mark=at position 0.06 with {\arrow[line width=.3mm]{<}}},
    postaction={decorate},shorten >=.6pt] (1.5,2.5) to (1.5, 8);
\draw[directed=.8] (-3.8,5.5)..controls ++(0,1) and ++(0,1)..(1.2, 5.5);
\draw[fill =gray!20] (-5.8,6.5) rectangle (-0.4,7.5);
\draw[fill =gray!20] (-5.8,3) rectangle (-0.4,4);
\draw[fill =gray!20] (5.8,6.5) rectangle (0.4,7.5);
\draw[fill =gray!20] (5.8,3) rectangle (0.4,4);
\draw[fill =white] (-5.8,4.5) rectangle (-3.3,5.5);
\draw[fill =white] (-2.7,4.5) rectangle (-0.4,5.5);
\draw[fill =gray!60] (5.8,4.5) rectangle (3.2,5.5);
\draw[fill =gray!60] (2.9,4.5) rectangle (0,5.5);
\node at (-3,7) [scale=.75]{$(\cdots,s')$};
\node at (-3,3.5) [scale=.75]{$(\cdots,s)$};
\node at (-4.5,5) [scale=.75]{$y$};
\node at (-1.5,5) [scale=.75]{$x$};
\node at (3,7) [scale=.75]{$(\cdots,r')$};
\node at (3,3.5) [scale=.75]{$(\cdots,r)$};
\node at (1.5,5)[scale=.5] {$y-b+1$};
\node at (4.5,5) [scale=.75]{$x-a$};
\end{tikzpicture}} 
 & 
\delta_{s,s'}\delta_{r,r'} \1
 \\
0
&
\hackcenter{\begin{tikzpicture} [scale=.35, xscale=.7, yscale=1]
\draw[line width= 0.7mm,
    decoration={markings,mark=at position 1 with {\arrow[line width=.3mm]{>}}},
    postaction={decorate},shorten >=1pt] (-4.5,2.5) to (-4.5,8);
\draw[line width= 0.7mm,
    decoration={markings,mark=at position 0.06 with {\arrow[line width=.3mm]{<}}},
    postaction={decorate},shorten >=.6pt] (4.5,2.5) to (4.5,8);
\draw[line width= 0.7mm,
    decoration={markings,mark=at position 1 with {\arrow[line width=.3mm]{>}}},
    postaction={decorate},shorten >=1pt] (-1.5,2.5) to (-1.5, 8);
\draw[line width= 0.7mm,
    decoration={markings,mark=at position 0.06 with {\arrow[line width=.3mm]{<}}},
    postaction={decorate},shorten >=.6pt] (1.5,2.5) to (1.5, 8);
\draw[directed=.8] (-3.8,5.5)..controls ++(0,1) and ++(0,1)..(1.2, 5.5);
\draw[fill =gray!20] (-5.8,6.5) rectangle (-0.4,7.5);
\draw[fill =gray!20] (-5.8,3) rectangle (-0.4,4);
\draw[fill =gray!20] (5.8,6.5) rectangle (0.4,7.5);
\draw[fill =gray!20] (5.8,3) rectangle (0.4,4);
\draw[fill =white] (-5.8,4.5) rectangle (-3.3,5.5);
\draw[fill =white] (-2.7,4.5) rectangle (-0.4,5.5);
\draw[fill =gray!60] (5.8,4.5) rectangle (3.2,5.5);
\draw[fill =gray!60] (2.9,4.5) rectangle (0,5.5);
\node at (-3,7) [scale=.75]{$(\cdots,s')$};
\node at (-3,3.5) [scale=.75]{$(\cdots,s)$};
\node at (-4.5,5) [scale=.75]{$y$};
\node at (-1.5,5) [scale=.65]{$x-1$};
\node at (3,7) [scale=.75]{$(\cdots,r')$};
\node at (3,3.5) [scale=.75]{$(\cdots,r)$};
\node at (1.5,5)[scale=.75] {$y-b$};
\node at (4.5,5) [scale=.75]{$x-a$};
\end{tikzpicture}}
\\
\end{matrix}\right)\;\;
{\mathsf{d}^{x}|_{s,r}} = 
\begin{pmatrix}
\;\;
\hackcenter{\begin{tikzpicture} [scale=.35, xscale=.7]
\draw[line width= 0.7mm,
    decoration={markings,mark=at position 1 with {\arrow[line width=.3mm]{>}}},
    postaction={decorate},shorten >=1pt] (-4.5,2.5) to (-4.5,8);
\draw[line width= 0.7mm,
    decoration={markings,mark=at position 0.06 with {\arrow[line width=.3mm]{<}}},
    postaction={decorate},shorten >=.6pt] (4.5,2.5) to (4.5,8);
\draw[line width= 0.7mm,
    decoration={markings,mark=at position 1 with {\arrow[line width=.3mm]{>}}},
    postaction={decorate},shorten >=1pt] (-1.5,2.5) to (-1.5, 8);
\draw[line width= 0.7mm,
    decoration={markings,mark=at position 0.06 with {\arrow[line width=.3mm]{<}}},
    postaction={decorate},shorten >=.6pt] (1.5,2.5) to (1.5, 8);
\draw[directed=.6] (-.8,5.5)..controls ++(0,1) and ++(0,1)..(4.2, 5.5);
\draw[fill =gray!20] (-5.8,6.5) rectangle (-0.4,7.5);
\draw[fill =gray!20] (-5.8,3) rectangle (-0.4,4);
\draw[fill =gray!20] (5.8,6.5) rectangle (0.4,7.5);
\draw[fill =gray!20] (5.8,3) rectangle (0.4,4);
\draw[fill =white] (-5.8,4.5) rectangle (-3.3,5.5);
\draw[fill =white] (-2.7,4.5) rectangle (-0.4,5.5);
\draw[fill =gray!60] (5.8,4.5) rectangle (3.2,5.5);
\draw[fill =gray!60] (2.9,4.5) rectangle (0,5.5);
\node at (-3,7) [scale=.75]{$(\cdots,s')$};
\node at (-3,3.5) [scale=.75]{$(\cdots,s)$};
\node at (-4.5,5) [scale=.75]{$y$};
\node at (-1.5,5) [scale=.75]{$x$};
\node at (3,7) [scale=.75]{$(\cdots,r')$};
\node at (3,3.5) [scale=.75]{$(\cdots,r)$};
\node at (1.5,5)[scale=.5] {$y-b+1$};
\node at (4.5,5) [scale=.75]{$x-a$};
\end{tikzpicture}}
 & 
\delta_{s,s'}\delta_{r,r'} \1 \\
0
&
\hackcenter{\begin{tikzpicture} [scale=.35, xscale=.7, yscale=1]
\draw[line width= 0.7mm,
    decoration={markings,mark=at position 1 with {\arrow[line width=.3mm]{>}}},
    postaction={decorate},shorten >=1pt] (-4.5,2.5) to (-4.5,8);
\draw[line width= 0.7mm,
    decoration={markings,mark=at position 0.06 with {\arrow[line width=.3mm]{<}}},
    postaction={decorate},shorten >=.6pt] (4.5,2.5) to (4.5,8);
\draw[line width= 0.7mm,
    decoration={markings,mark=at position 1 with {\arrow[line width=.3mm]{>}}},
    postaction={decorate},shorten >=1pt] (-1.5,2.5) to (-1.5, 8);
\draw[line width= 0.7mm,
    decoration={markings,mark=at position 0.06 with {\arrow[line width=.3mm]{<}}},
    postaction={decorate},shorten >=.6pt] (1.5,2.5) to (1.5, 8);
\draw[directed=.6] (-.8,5.5)..controls ++(0,1) and ++(0,1)..(4.2, 5.5);
\draw[fill =gray!20] (-5.8,6.5) rectangle (-0.4,7.5);
\draw[fill =gray!20] (-5.8,3) rectangle (-0.4,4);
\draw[fill =gray!20] (5.8,6.5) rectangle (0.4,7.5);
\draw[fill =gray!20] (5.8,3) rectangle (0.4,4);
\draw[fill =white] (-5.8,4.5) rectangle (-3.3,5.5);
\draw[fill =white] (-2.7,4.5) rectangle (-0.4,5.5);
\draw[fill =gray!60] (5.8,4.5) rectangle (3.2,5.5);
\draw[fill =gray!60] (2.9,4.5) rectangle (0,5.5);
\node at (-3,7) [scale=.75]{$(\cdots,s')$};
\node at (-3,3.5) [scale=.75]{$(\cdots,s)$};
\node at (-4.5,5) [scale=.75]{$y$};
\node at (-1.5,5) [scale=.65]{$x-1$};
\node at (3,7) [scale=.75]{$(\cdots,r')$};
\node at (3,3.5) [scale=.75]{$(\cdots,r)$};
\node at (1.5,5)[scale=.75] {$y-b$};
\node at (4.5,5) [scale=.75]{$x-a$};
\end{tikzpicture}}
\\
\end{pmatrix}.
\end{equation*}
In particular, the restriction $\mathsf{d}^x: \mathcal{F}_{x,y}^{\mathbb{II}} {\color{blue}\to} \mathcal{F}_{x-1,y}^{\mathbb{I}}$ is injective for $y-b+1<x-a$, surjective for $y-b+1>x-a$, and bijective for $y-b+1=x-a$. Likewise, the restriction $\mathsf{d}_y:\mathcal{F}_{x,y}^{\mathbb{II}} {\color{black!40!green}\to} \mathcal{F}_{x,y-1}^{\mathbb{I}}$ is surjective for $x>y$, injective for $x<y$, and bijective for $x=y$. 
\end{proposition} 

\begin{proof}
Applying the chain homotopy given by the isomorphism $\Q^{(y-b+1)^t}\P^{(x)} \simeq \P^{(x)}\Q^{(y-b+1)^t} \oplus \P^{(x-1)}\Q^{(y-b)^t} $ from Proposition \ref{prop:QPswap} to $\B_{b-1}\otimes\B_a $, 
\begin{align*}
\B_{b-1}\otimes\B_a &= \text{Tot}^{\oplus}\left\lbrace \bigoplus_{y\geq (0,b-1)} \B_{b-1}\P^{(x)}\Q^{(x-a)^t} , \d_{b-1}\otimes \1_x,(-1)^x\1_{\B_{b-1}} \otimes \d_a \right\rbrace_{x\geq (a,0)}\\
&\simeq \text{Tot}^{\oplus}\left\lbrace \bigoplus_{y\geq (0,b-1)} \mathcal{\tilde{F}}_{x,y}^{\mathbb{I}} \oplus  \mathcal{\tilde{F}}_{x,y}^{\mathbb{II}}, \d_y,(-1)^x\d^x \right\rbrace_{x\geq (a,0)}
\end{align*}
with chain groups and differentials $\mathsf{d}_y': \mathcal{\tilde{F}}_{x,y}^{\mathbb{I}} \oplus \mathcal{\tilde{F}}_{x,y}^{\mathbb{II}} \to \mathcal{\tilde{F}}_{x,y-1}^{\mathbb{I}} \oplus \mathcal{\tilde{F}}_{x,y-1}^{\mathbb{II}}$ and $\mathsf{d}'^x: \mathcal{\tilde{F}}_{x,y}^{\mathbb{I}} \oplus \mathcal{\tilde{F}}_{x,y}^{\mathbb{II}} \to \mathcal{\tilde{F}}_{x-1,y}^{\mathbb{I}} \oplus \mathcal{\tilde{F}}_{x-1,y}^{\mathbb{II}}$ given by,
\[
\mathcal{\tilde{F}}_{x,y}^{\mathbb{I}} := \P^{(y)}\P^{(x)}\Q^{(y-b+1)^t}\Q^{(x-a)^t}[x+y-a-b+1]
\]
\[
\mathcal{\tilde{F}}_{x,y}^{\mathbb{II}} := \P^{(y)}\P^{(x-1)}\Q^{(y-b)^t}\Q^{(x-a)^t}[x+y-a-b+1] \]

\[{\mathsf{d}_{y}'} =
\hackcenter{\begin{pmatrix}
\begin{tikzpicture} [scale=.35, xscale=.7]
\draw[line width= 0.7mm,
    decoration={markings,mark=at position 1 with {\arrow[line width=.3mm]{>}}},
    postaction={decorate},shorten >=1pt] (-4.5,3.5) to (-4.5,7);
\draw[line width= 0.7mm,
    decoration={markings,mark=at position 0.06 with {\arrow[line width=.3mm]{<}}},
    postaction={decorate},shorten >=.6pt] (4.5,3.5) to (4.5,7);
\draw[line width= 0.7mm,
    decoration={markings,mark=at position 1 with {\arrow[line width=.3mm]{>}}},
    postaction={decorate},shorten >=1pt] (-1.5,3.5) to (-1.5, 7);
\draw[line width= 0.7mm,
    decoration={markings,mark=at position 0.06 with {\arrow[line width=.3mm]{<}}},
    postaction={decorate},shorten >=.6pt] (1.5,3.5) to (1.5, 7);
\draw[directed=.8] (-3.8,5.5)..controls ++(0,1) and ++(0,1)..(1.2, 5.5);
\draw[fill =white] (-5.8,4.5) rectangle (-3.3,5.5);
\draw[fill =white] (-2.7,4.5) rectangle (-0.4,5.5);
\draw[fill =gray!60] (5.8,4.5) rectangle (3.2,5.5);
\draw[fill =gray!60] (2.9,4.5) rectangle (0,5.5);
\node at (-4.5,5) [scale=.75]{$y$};
\node at (-1.5,5) [scale=.75]{$x$};
\node at (1.5,5)[scale=.5] {$y-b+1$};
\node at (4.5,5) [scale=.75]{$x-a$};
\end{tikzpicture} 
 & 
\begin{tikzpicture} [scale=.35, xscale=.7, yscale=.7]
\draw[line width= 0.7mm,
    decoration={markings,mark=at position 0.06 with {\arrow[line width=.3mm]{<}}},
    postaction={decorate},shorten >=.6pt] (4.5,3.5) to (4.5,8.5);
\draw[line width= 0.7mm,
    decoration={markings,mark=at position 1 with {\arrow[line width=.3mm]{>}}},
    postaction={decorate},shorten >=1pt] (-4.5,3.5) to (-4.5,8.5);
\draw[line width= 0.7mm,
    decoration={markings,mark=at position 0.06 with {\arrow[line width=.3mm]{<}}},
    postaction={decorate},shorten >=.6pt] (1.5,3.5) to (1.5, 8.5);
\draw[line width= 0.7mm,
    decoration={markings,mark=at position 1 with {\arrow[line width=.3mm]{>}}},
    postaction={decorate},shorten >=1pt] (-1.5,3.5) to (-1.5, 8.5);
\draw[directed= .8]  (-4.2, 5.5) to (-1.8,6.5);
\draw[fill =gray!60] (5.8,4.5) rectangle (3.3,5.5);
\draw[fill =gray!60] (3,4.5) rectangle (0.2,5.5);
\draw[fill =white] (-2.7,6.5) rectangle (-0.4,7.5);
\draw[fill =white] (-5.8,4.5) rectangle (-3.1,5.5);
\node at (4.5,5)[scale=.75]  {$x-a$};
\node at (1.5,5) [scale=.75] {$y-b$};
\node at (-4.5,5)[scale=.75]  {$y$};
\node at (-1.5,7)[scale=.75]  {$x$};
\end{tikzpicture} \\
0
&
\hackcenter{\begin{tikzpicture} [scale=.35, xscale=.7, yscale=.8]
\draw[line width= 0.7mm,
    decoration={markings,mark=at position 1 with {\arrow[line width=.3mm]{>}}},
    postaction={decorate},shorten >=1pt] (-4.5,3.5) to (-4.5,7);
\draw[line width= 0.7mm,
    decoration={markings,mark=at position 0.06 with {\arrow[line width=.3mm]{<}}},
    postaction={decorate},shorten >=.6pt] (4.5,3.5) to (4.5,7);
\draw[line width= 0.7mm,
    decoration={markings,mark=at position 1 with {\arrow[line width=.3mm]{>}}},
    postaction={decorate},shorten >=1pt] (-1.5,3.5) to (-1.5, 7);
\draw[line width= 0.7mm,
    decoration={markings,mark=at position 0.06 with {\arrow[line width=.3mm]{<}}},
    postaction={decorate},shorten >=.6pt] (1.5,3.5) to (1.5, 7);
\draw[directed=.8] (-3.8,5.5)..controls ++(0,1) and ++(0,1)..(1.2, 5.5);
\draw[fill =white] (-5.8,4.5) rectangle (-3.3,5.5);
\draw[fill =white] (-2.9,4.5) rectangle (-0.2,5.5);
\draw[fill =gray!60] (5.8,4.5) rectangle (3.2,5.5);
\draw[fill =gray!60] (2.7,4.5) rectangle (0.2,5.5);
\node at (-4.5,5) [scale=.75]{$y$};
\node at (-1.5,5) [scale=.75]{$x-1$};
\node at (1.5,5)[scale=.75] {$y-b$};
\node at (4.5,5) [scale=.75]{$x-a$};
\end{tikzpicture}}
\end{pmatrix}}
\;\;
{\mathsf{d}'^{x}} =  \hackcenter{
\begin{pmatrix}
\begin{tikzpicture} [scale=.35, xscale=.7]
\draw[line width= 0.7mm,
    decoration={markings,mark=at position 1 with {\arrow[line width=.3mm]{>}}},
    postaction={decorate},shorten >=1pt] (-4.5,3.5) to (-4.5,7);
\draw[line width= 0.7mm,
    decoration={markings,mark=at position 0.06 with {\arrow[line width=.3mm]{<}}},
    postaction={decorate},shorten >=.6pt] (4.5,3.5) to (4.5,7);
\draw[line width= 0.7mm,
    decoration={markings,mark=at position 1 with {\arrow[line width=.3mm]{>}}},
    postaction={decorate},shorten >=1pt] (-1.5,3.5) to (-1.5, 7);
\draw[line width= 0.7mm,
    decoration={markings,mark=at position 0.06 with {\arrow[line width=.3mm]{<}}},
    postaction={decorate},shorten >=.6pt] (1.5,3.5) to (1.5, 7);
\draw[directed=.6] (-.8,5.5)..controls ++(0,1) and ++(0,1)..(4.2, 5.5);
\draw[fill =white] (-5.8,4.5) rectangle (-3.3,5.5);
\draw[fill =white] (-2.7,4.5) rectangle (-0.4,5.5);
\draw[fill =gray!60] (5.8,4.5) rectangle (3.2,5.5);
\draw[fill =gray!60] (2.9,4.5) rectangle (0,5.5);
\node at (-4.5,5) [scale=.75]{$y$};
\node at (-1.5,5) [scale=.75]{$x$};
\node at (1.5,5)[scale=.5] {$y-b+1$};
\node at (4.5,5) [scale=.75]{$x-a$};
\end{tikzpicture} 
 & 
\begin{tikzpicture} [scale=.35, xscale = .7, yscale=.8]
\draw[line width= 0.7mm,
    decoration={markings,mark=at position 1 with {\arrow[line width=.3mm]{>}}},
    postaction={decorate},shorten >=1pt] (-4.5,3.5) to (-4.5,8.5);
\draw[line width= 0.7mm,
    decoration={markings,mark=at position 0.06 with {\arrow[line width=.3mm]{<}}},
    postaction={decorate},shorten >=.6pt] (4.5,3.5) to (4.5,8.5);
\draw[line width= 0.7mm,
    decoration={markings,mark=at position 1 with {\arrow[line width=.3mm]{>}}},
    postaction={decorate},shorten >=1pt] (-1.5,3.5) to (-1.5, 8.5);
\draw[line width= 0.7mm,
    decoration={markings,mark=at position 0.06 with {\arrow[line width=.3mm]{<}}},
    postaction={decorate},shorten >=.6pt] (1.5,3.5) to (1.5, 8.5);
\draw[directed= .6]  (1.8,6.5) to (4.2, 5.5);
\draw[fill =white] (-5.8,4.5) rectangle (-3.3,5.5);
\draw[fill =white] (-3,4.5) rectangle (-0.2,5.5);
\draw[fill =gray!60] (3,6.5) rectangle (0,7.5);
\draw[fill =gray!60] (5.8,4.5) rectangle (3.1,5.5);
\node at (-4.5,5)[scale=.75]  {$y$};
\node at (-1.5,5) [scale=.75] {$x-1$};
\node at (4.5,5)[scale=.75]  {$x-a$};
\node at (1.5,7)[scale=.5]  {$y-b+1$};
\end{tikzpicture} \\
0
&
\hackcenter{\begin{tikzpicture} [scale=.35, xscale=.7, yscale=.8]
\draw[line width= 0.7mm,
    decoration={markings,mark=at position 1 with {\arrow[line width=.3mm]{>}}},
    postaction={decorate},shorten >=1pt] (-4.5,3.5) to (-4.5,7);
\draw[line width= 0.7mm,
    decoration={markings,mark=at position 0.06 with {\arrow[line width=.3mm]{<}}},
    postaction={decorate},shorten >=.6pt] (4.5,3.5) to (4.5,7);
\draw[line width= 0.7mm,
    decoration={markings,mark=at position 1 with {\arrow[line width=.3mm]{>}}},
    postaction={decorate},shorten >=1pt] (-1.5,3.5) to (-1.5, 7);
\draw[line width= 0.7mm,
    decoration={markings,mark=at position 0.06 with {\arrow[line width=.3mm]{<}}},
    postaction={decorate},shorten >=.6pt] (1.5,3.5) to (1.5, 7);
\draw[directed=.6] (-.8,5.5)..controls ++(0,1) and ++(0,1)..(4.2, 5.5);
\draw[fill =white] (-5.8,4.5) rectangle (-3.3,5.5);
\draw[fill =white] (-2.9,4.5) rectangle (-0.2,5.5);
\draw[fill =gray!60] (5.8,4.5) rectangle (3.2,5.5);
\draw[fill =gray!60] (2.7,4.5) rectangle (0.2,5.5);
\node at (-4.5,5) [scale=.75]{$y$};
\node at (-1.5,5) [scale=.75]{$x-1$};
\node at (1.5,5)[scale=.75] {$y-b$};
\node at (4.5,5) [scale=.75]{$x-a$};
\end{tikzpicture}}
\end{pmatrix}}.\]

Thus, if we apply the isomorphisms $\P^{(y)}\P^{(x)} \simeq \bigoplus_{s=0}^{(x,y)} \P^{(y+x-s,s)}$ and $\Q^{(u)^t}\Q^{(w)^t} \simeq \bigoplus_{r=0}^{(u,w)} \Q^{(u+w-r,r)^t}$ from Proposition \ref{prop:PPmerge} to each summand $\tilde{\mathcal{F}}^{\mathbb{I}}_{x,y}$ and $\tilde{\mathcal{F}}^{\mathbb{II}}_{x,y}$ and differential $\d_y'$ and $\d_x'$ then the result follows by denoting the resulting summands by $\mathcal{F}^{\mathbb{I}}_{x,y}$ and $\mathcal{F}^{\mathbb{II}}_{x,y}$ and setting:
\begin{align*}
\d_y&:= (\rho_{y-1}^s\otimes\rho_{x-a}^r)\circ \d_y'\circ (\iota_{s'}^y \otimes \iota_{r'}^{x-a}) &
\d^x&:= (\rho_{y-1}^s\otimes\rho_{x-a}^r)\circ \d'^x\circ (\iota_{s'}^y \otimes \iota_{r'}^{x-a}).
\end{align*}
\end{proof}

\begin{figure}[h]
\begin{tikzpicture}[xscale=1.5,yscale=1, every node/.style={scale=1}]

\node at (-5.2,2){$\B_{b-1}\P^{(x)}\Q^{(x-a)^t} \simeq$};

\node at (-2.5,2.35){$\d_y$};
\node at (-1.5,2.35){$\d_y$};
\node at (-1,3){y=x};
\node at (-0.5,2.35){$\mathsf{D}$};
\node at (0,3){y=x-1};
\node at (0.5,2.35){$\d_y$};
\node at (1.5,2.35){$\d_y$};

\draw[->](-1,2.7)--(-1,2.3);
\draw[->](0,2.7)--(0,2.3);

\node at (1.4,1.5){$\underbrace{\hspace{40mm}}$};
\node at (-2.4,1.5){$\underbrace{\hspace{40mm}}$};

\node at (1.4,1){$\bigoplus_{y}\mathcal{F}_{x,y}^{\mathbb{II}}$};
\node at (-2.4,1){$\bigoplus_y\mathcal{F}_{x,y}^{\mathbb{I}}$};

\node at (-1,2) {$\bullet $ };

\node at (-2,2) {$\bullet $ };

\node at (-3,2) {$\bullet $ };

\node at (0,2) {$\bullet$ };

\node at (1,2) {$\bullet$ };

\node at (2,2) {$\bullet $ };



\draw[black!40!green, dotted, directed =.6] (-4,2) to (-3,2);

\draw[black!40!green,  directed =.6] (-3,2) to (-2,2);

\draw[black!40!green,  directed =.6] (-2,2) to (-1,2);

\draw[black!60!green, very thick, directed =.6] (-1,2) to (0,2);

\draw[black!40!green,  directed =.6] (0,2) to (1,2);

\draw[black!40!green,  directed =.6] (1,2) to (2,2);

\draw[black!40!green,  dotted, directed =.6] (2,2) to (3,2);

\end{tikzpicture}
\caption{The $x^{th}$ row of the bi-complex  $\B_{b-1} \otimes \B_a$ with $a,b \in \Z$ and $x\geq$ max$(0,a)$.}\label{fig:ConeD}
\end{figure}
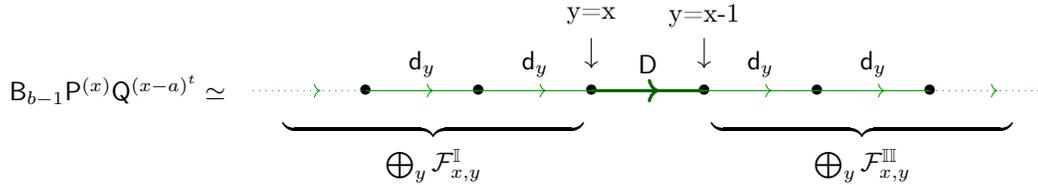

A first glance, it may seem like the chain complex has become more complicated after applying these isomorphisms. However, the bijections in Proposition \ref{prop:F} between summands the of $\mathcal{F}_{x,y}^{\mathbb{II}}$ and $\mathcal{F}_{x,y}^{\mathbb{I}}$ is what will allow the bi-complex of $\B_{b-1}\otimes\B_a$ to be further reduced. 

\begin{theorem}\label{thm:BBa<b} Let $a,b \in \Z$ with $a<b$ and for each integer $k\geq 2-b-a$ define $\mathcal{G}_{k}$ as the following object in $\H$:
\begin{equation}\label{eq:Gk} \mathcal{G}_{k}:= \bigoplus_{\substack{x+y=k+a+b\\x > y\geq (b,0)}}\P^{(x-1,y)} \left(\bigoplus_{r=y-b+1}^{(x-b,y-a)}\Q^{(x+y-a-b-r,r)^t}\right)
\oplus \bigoplus_{\substack{x+y=k+a+b\\x \geq b>y\geq (a,0)}} \P^{(x-1,y)}\Q^{(x-b)^t}\Q^{(y-a)^t} 
\end{equation}
and let $\mathsf{d}_{k}$ be given by:
\begin{align}
\d_k:=&\sum_{\substack{x+y=k+a+b\\x>y\geq (b,0)}}\sum_{r={y-b+1}}^{(x-b,y-a)} \hackcenter{\begin{tikzpicture} [scale=.35, xscale=.7, yscale=1]
\draw[line width= 0.7mm,
    decoration={markings,mark=at position 1 with {\arrow[line width=.3mm]{>}}},
    postaction={decorate},shorten >=1pt] (-4.5,2.5) to (-4.5,8);
\draw[line width= 0.7mm,
    decoration={markings,mark=at position 0.06 with {\arrow[line width=.3mm]{<}}},
    postaction={decorate},shorten >=.6pt] (4.5,2.5) to (4.5,8);
\draw[line width= 0.7mm,
    decoration={markings,mark=at position 1 with {\arrow[line width=.3mm]{>}}},
    postaction={decorate},shorten >=1pt] (-1.5,2.5) to (-1.5, 8);
\draw[line width= 0.7mm,
    decoration={markings,mark=at position 0.06 with {\arrow[line width=.3mm]{<}}},
    postaction={decorate},shorten >=.6pt] (1.5,2.5) to (1.5, 8);
\draw[directed=.8] (-3.8,5.5)..controls ++(0,1) and ++(0,1)..(1.2, 5.5);
\draw[fill =gray!20] (-5.8,6.5) rectangle (-0.4,7.5);
\draw[fill =gray!20] (-5.8,3) rectangle (-0.4,4);
\draw[fill =gray!20] (5.8,6.5) rectangle (0.4,7.5);
\draw[fill =gray!20] (5.8,3) rectangle (0.4,4);
\draw[fill =white] (-5.8,4.5) rectangle (-3.3,5.5);
\draw[fill =white] (-2.7,4.5) rectangle (-0.4,5.5);
\draw[fill =gray!60] (5.8,4.5) rectangle (3.2,5.5);
\draw[fill =gray!60] (2.9,4.5) rectangle (0,5.5);
\node at (-3,7) [scale=.75]{$(x-2,y)$};
\node at (-3,3.5) [scale=.75]{$(x-1,y)$};
\node at (-4.5,5) [scale=.75]{$x-1$};
\node at (-1.5,5) [scale=.75]{$y$};
\node at (3,7) [scale=.75]{$(\cdots,r')$};
\node at (3,3.5) [scale=.75]{$(\cdots,r)$};
\node at (1.5,5)[scale=.75] {$x-b$};
\node at (4.5,5) [scale=.75]{$y-a$};
\end{tikzpicture}}\;\;+ \;\;(-1)^y\;\;
\hackcenter{\begin{tikzpicture} [scale=.35, xscale=.7]
\draw[line width= 0.7mm,
    decoration={markings,mark=at position 1 with {\arrow[line width=.3mm]{>}}},
    postaction={decorate},shorten >=1pt] (-4.5,2.5) to (-4.5,8);
\draw[line width= 0.7mm,
    decoration={markings,mark=at position 0.06 with {\arrow[line width=.3mm]{<}}},
    postaction={decorate},shorten >=.6pt] (4.5,2.5) to (4.5,8);
\draw[line width= 0.7mm,
    decoration={markings,mark=at position 1 with {\arrow[line width=.3mm]{>}}},
    postaction={decorate},shorten >=1pt] (-1.5,2.5) to (-1.5, 8);
\draw[line width= 0.7mm,
    decoration={markings,mark=at position 0.06 with {\arrow[line width=.3mm]{<}}},
    postaction={decorate},shorten >=.6pt] (1.5,2.5) to (1.5, 8);
\draw[directed=.6] (-.8,5.5)..controls ++(0,1) and ++(0,1)..(4.2, 5.5);
\draw[fill =gray!20] (-5.8,6.5) rectangle (-0.4,7.5);
\draw[fill =gray!20] (-5.8,3) rectangle (-0.4,4);
\draw[fill =gray!20] (5.8,6.5) rectangle (0.4,7.5);
\draw[fill =gray!20] (5.8,3) rectangle (0.4,4);
\draw[fill =white] (-5.8,4.5) rectangle (-3.3,5.5);
\draw[fill =white] (-2.7,4.5) rectangle (-0.4,5.5);
\draw[fill =gray!60] (5.8,4.5) rectangle (3.2,5.5);
\draw[fill =gray!60] (2.9,4.5) rectangle (0,5.5);
\node at (-3,7) [scale=.65]{$(x-1,y-1)$};
\node at (-3,3.5) [scale=.75]{$(x-1,y)$};
\node at (-4.5,5) [scale=.75]{$x-1$};
\node at (-1.5,5) [scale=.75]{$y$};
\node at (3,7) [scale=.75]{$(\cdots,r')$};
\node at (3,3.5) [scale=.75]{$(\cdots,r)$};
\node at (1.5,5)[scale=.75] {$x-b$};
\node at (4.5,5) [scale=.75]{$y-a$};
\end{tikzpicture}}
\\
&+
\sum_{\substack{x+y=k+a+b\\ x\geq b>y\geq(0,a)}}\hackcenter{\begin{tikzpicture} [scale=.35, xscale=.7, yscale=1]
\draw[line width= 0.7mm,
    decoration={markings,mark=at position 1 with {\arrow[line width=.3mm]{>}}},
    postaction={decorate},shorten >=1pt] (-4.5,2.5) to (-4.5,8);
\draw[line width= 0.7mm,
    decoration={markings,mark=at position 0.06 with {\arrow[line width=.3mm]{<}}},
    postaction={decorate},shorten >=.6pt] (4.5,2.5) to (4.5,8);
\draw[line width= 0.7mm,
    decoration={markings,mark=at position 1 with {\arrow[line width=.3mm]{>}}},
    postaction={decorate},shorten >=1pt] (-1.5,2.5) to (-1.5, 8);
\draw[line width= 0.7mm,
    decoration={markings,mark=at position 0.06 with {\arrow[line width=.3mm]{<}}},
    postaction={decorate},shorten >=.6pt] (1.5,2.5) to (1.5, 8);
\draw[directed=.8] (-3.8,5.5)..controls ++(0,1) and ++(0,1)..(1.2, 5.5);
\draw[fill =gray!20] (-5.8,6.5) rectangle (-0.4,7.5);
\draw[fill =gray!20] (-5.8,3) rectangle (-0.4,4);
\draw[fill =white] (-5.8,4.5) rectangle (-3.3,5.5);
\draw[fill =white] (-2.7,4.5) rectangle (-0.4,5.5);
\draw[fill =gray!60] (5.8,4.5) rectangle (3.2,5.5);
\draw[fill =gray!60] (2.9,4.5) rectangle (0,5.5);
\node at (-3,7) [scale=.75]{$(x-2,y)$};
\node at (-3,3.5) [scale=.75]{$(x-1,y)$};
\node at (-4.5,5) [scale=.75]{$x-1$};
\node at (-1.5,5) [scale=.75]{$y$};
\node at (1.5,5)[scale=.75] {$x-b$};
\node at (4.5,5) [scale=.75]{$y-a$};
\end{tikzpicture}}  
\;\;+\;\;(-1)^y\;\;
\hackcenter{\begin{tikzpicture} [scale=.35, xscale=.7]
\draw[line width= 0.7mm,
    decoration={markings,mark=at position 1 with {\arrow[line width=.3mm]{>}}},
    postaction={decorate},shorten >=1pt] (-4.5,2.5) to (-4.5,8);
\draw[line width= 0.7mm,
    decoration={markings,mark=at position 0.06 with {\arrow[line width=.3mm]{<}}},
    postaction={decorate},shorten >=.6pt] (4.5,2.5) to (4.5,8);
\draw[line width= 0.7mm,
    decoration={markings,mark=at position 1 with {\arrow[line width=.3mm]{>}}},
    postaction={decorate},shorten >=1pt] (-1.5,2.5) to (-1.5, 8);
\draw[line width= 0.7mm,
    decoration={markings,mark=at position 0.06 with {\arrow[line width=.3mm]{<}}},
    postaction={decorate},shorten >=.6pt] (1.5,2.5) to (1.5, 8);
\draw[directed=.6] (-.8,5.5)..controls ++(0,1) and ++(0,1)..(4.2, 5.5);
\draw[fill =gray!20] (-5.8,6.5) rectangle (-0.4,7.5);
\draw[fill =gray!20] (-5.8,3) rectangle (-0.4,4);
\draw[fill =white] (-5.8,4.5) rectangle (-3.3,5.5);
\draw[fill =white] (-2.7,4.5) rectangle (-0.4,5.5);
\draw[fill =gray!60] (5.8,4.5) rectangle (3.2,5.5);
\draw[fill =gray!60] (2.9,4.5) rectangle (0,5.5);
\node at (-3,7) [scale=.65]{$(x-1,y-1)$};
\node at (-3,3.5) [scale=.75]{$(x-1,y)$};
\node at (-4.5,5) [scale=.75]{$x-1$};
\node at (-1.5,5) [scale=.75]{$y$};
\node at (1.5,5)[scale=.75] {$x-b$};
\node at (4.5,5) [scale=.75]{$y-a$};
\end{tikzpicture}}.
\end{align}
Then, the following homotopy equivalence holds in $\K(\H)$:
\begin{equation}\label{Cb+1Ca-reduced-case1}
\B_{b-1}\otimes \B_a \simeq \dots \xrightarrow{\mathsf{d}_{k+2}} \mathcal{G}_{k+1}[k+1]\xrightarrow{\mathsf{d}_{k+1}} \mathcal{G}_{k}[k]\xrightarrow{\mathsf{d}_{k-1}} \dots
\end{equation} 
\end{theorem}

\begin{proof}
 Let $x>y$ and $\mathcal{F}_{x,y}^{\mathbb{I}}$, $\mathcal{F}_{x,y}^{\mathbb{II}}$, $\mathsf{d}_y:\mathcal{F}_{x,y}^{\mathbb{II}} \to \mathcal{F}_{x,y-1}^{\mathbb{I}}$,and $\mathsf{d}^x:\mathcal{F}_{x,y}^{\mathbb{II}} \to \mathcal{F}_{x-1,y}^{\mathbb{I}}$ be defined as in Proposition \ref{prop:F}. Then, for any integer $1 \leq j <x-y-a+b$ consider the following restrictions of the differentials
 \begin{align*}
 \mathsf{d}_{y+j}:& \mathcal{F}_{x-j,y+j}^{\mathbb{II}} {\color{black!40!green}\to} \mathcal{F}_{x-j,y-1+j}^{\mathbb{I}} &
 \mathsf{d}^{x-j}:& \mathcal{F}_{x-j,y+j}^{\mathbb{II}} {\color{blue}\to} \mathcal{F}_{x-1-j,y-j}^{\mathbb{I}}.
 \end{align*} 
By Proposition \ref{prop:F}, $\mathsf{d}_{y+j}$ is surjective for $1\leq j \leq \frac{x-y}{2}$ and injective for $\frac{x-y}{2}\leq j< x-y$.  Consequently, if we restrict $\d_{y+j}$  to the summands corresponding to $0\leq s \leq y$ for each $j$, then $\mathsf{d}_{y+j}: \mathcal{F}_{x-j,y+j}^{\mathbb{II}} {\color{black!40!green}\to} \mathcal{F}_{x-j,y+j-1}^{\mathbb{I}}$ is an isomorphism for all $1 \leq j <x-y$.  Likewise, $\mathsf{d}^{x-j}$ is injective for $0\leq j \leq \frac{x-y-a+b}{2}$ and surjective for $\frac{x-y-a+b}{2}\leq j< x-y-a+b$. Thus, if we restrict $\d^{x-j}$ to the summands with $0\leq r \leq y-b$ for each $j$,  then the map $\mathsf{d}^{x-j}: \mathcal{F}_{x-j,y+j}^{\mathbb{II}} {\color{blue}\to} \mathcal{F}_{x-1-j,y+j}^{\mathbb{I}}$ is an isomorphism for all $0\leq j <x-y-a+b$. 

Now for each $1 \leq j <x-y$ and $0\leq s \leq y$, the map $\mathsf{d}_{y+j}: \mathcal{F}_{x-j,y+j}^{\mathbb{II}} {\color{black!40!green}\to} \mathcal{F}_{x-j,y+j-1}^{\mathbb{I}}$ is invertible. Moreover since $x-y < x-y-a+b$, then $\mathsf{d}^{x-j}: \mathcal{F}_{x-j,y+j}^{\mathbb{II}} {\color{blue}\to} \mathcal{F}_{x-1-j,y+j}^{\mathbb{I}}$ is a bijection for all $0\leq r \leq y-b$ and $1 \leq j <x-y$. Thus, the composition:
\begin{equation}\label{eq:djRED}
\d_j:=\d^{x-j}\circ (\d_{y+j-1})^{-1} \circ \d^{x-j-1} \cdots (\d_{y+1})^{-1}\circ \d^{x}
\end{equation}
is well defined and exists for all $1\leq j < x-y$ and $0\leq s \leq y$. Graphically, these maps correspond to following the darker blue and green lines in an \emph{upward} zigzag in Figure \ref{fig:F+Fbicomplex}.

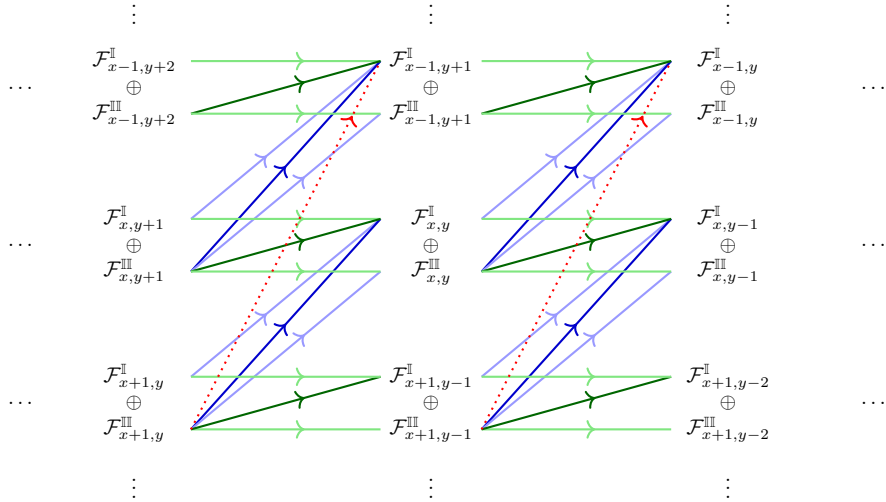
\begin{figure}[ht]
\begin{tikzpicture}[xscale=1.2, scale=.7, every node/.style={scale=0.8}]

\node at (0,4.5) {$\vdots$};

\node at (0,3.5) {$\mathcal{F}_{x-1,y+1}^{\mathbb{I}}$ };
\node at (0,3) {$\oplus$ };
\node at (0,2.5) {$\mathcal{F}_{x-1,y+1}^{\mathbb{II}}$ };

\node at (0,.5) {$\mathcal{F}_{x,y}^{\mathbb{I}}$ };
\node at (0,0) {$\oplus$ };
\node at (0,-.5) {$\mathcal{F}_{x,y}^{\mathbb{II}}$ };

\node at (0,-3.5) {$\mathcal{F}_{x+1,y-1}^{\mathbb{II}}$ };
\node at (0,-3) {$\oplus$ };
\node at (0,-2.5) {$\mathcal{F}_{x+1,y-1}^{\mathbb{I}}$ };

\node at (0,-4.5) {$\vdots$};

\node at (4.7,4.5) {$\vdots$};

\node at (4.7,3.5) {$\mathcal{F}_{x-1,y}^{\mathbb{I}}$ };
\node at (4.7,3) {$\oplus$ };
\node at (4.7,2.5) {$\mathcal{F}_{x-1,y}^{\mathbb{II}}$ };

\node at (4.7,.5) {$\mathcal{F}_{x,y-1}^{\mathbb{I}}$ };
\node at (4.7,0) {$\oplus$ };
\node at (4.7,-.5) {$\mathcal{F}_{x,y-1}^{\mathbb{II}}$ };

\node at (4.7,-3.5) {$\mathcal{F}_{x+1,y-2}^{\mathbb{II}}$ };
\node at (4.7,-3) {$\oplus$ };
\node at (4.7,-2.5) {$\mathcal{F}_{x+1,y-2}^{\mathbb{I}}$ };

\node at (4.7,-4.5) {$\vdots$};

\node at (7,3) {$\hdots$};
\node at (7,0) {$\hdots$};
\node at (7,-3) {$\hdots$};


\draw[blue!40, thick, directed = .4] (0.8,0.5) to (3.8,3.5);
\draw[black!20!blue, thick, directed = .5] (0.8,-0.5) to (3.8,3.5);
\draw[blue!40, thick, directed = .6] (0.8,-0.5) to (3.8,2.5);

\draw[blue!40, thick, directed = .4] (0.8,-2.5) to (3.8,0.5);
\draw[black!20!blue, thick, directed = .5] (0.8,-3.5) to (3.8,0.5);
\draw[blue!40, thick, directed = .6] (0.8,-3.5) to (3.8,-.5);


\draw[black!20!green!50, thick, directed =.6] (0.8,3.5) to (3.8,3.5);
\draw[black!60!green, thick, directed =.6] (0.8,2.5) to (3.8,3.5);
\draw[black!20!green!50, thick, directed =.6] (0.8,2.5) to (3.8,2.5);

\draw[black!20!green!50, thick, directed =.6] (0.8,0.5) to (3.8,0.5);
\draw[black!60!green, thick, directed =.6] (0.8,-0.5) to (3.8,0.5);
\draw[black!20!green!50, thick, directed =.6] (0.8,-0.5) to (3.8,-0.5);

\draw[black!20!green!50, thick, directed =.6] (0.8,-3.5) to (3.8,-3.5);
\draw[black!60!green, thick, directed =.6] (0.8,-3.5) to (3.8,-2.5);
\draw[black!20!green!50, thick, directed =.6] (0.8,-2.5) to (3.8,-2.5);


\draw[red, thick, dotted, directed=.85](.8,-3.5) to (3.8,3.5);

\node at (-4.7,4.5) {$\vdots$};

\node at (-4.7,3.5) {$\mathcal{F}_{x-1,y+2}^{\mathbb{I}}$ };
\node at (-4.7,3) {$\oplus$ };
\node at (-4.7,2.5) {$\mathcal{F}_{x-1,y+2}^{\mathbb{II}}$ };

\node at (-4.7,.5) {$\mathcal{F}_{x,y+1}^{\mathbb{I}}$ };
\node at (-4.7,0) {$\oplus$ };
\node at (-4.7,-.5) {$\mathcal{F}_{x,y+1}^{\mathbb{II}}$ };

\node at (-4.7,-3.5) {$\mathcal{F}_{x+1,y}^{\mathbb{II}}$ };
\node at (-4.7,-3) {$\oplus$ };
\node at (-4.7,-2.5) {$\mathcal{F}_{x+1,y}^{\mathbb{I}}$ };

\node at (-4.7,-4.5) {$\vdots$};


\node at (-6.5,3) {$ \hdots$};
\node at (-6.5,0) {$\hdots$};
\node at (-6.5,-3) {$ \hdots$};



\draw[blue!40, thick, directed = .4] (-3.8,0.5) to (-0.8,3.5);
\draw[black!20!blue, thick, directed = .5] (-3.8,-0.5) to (-0.8,3.5);
\draw[blue!40, thick, directed = .6] (-3.8,-0.5) to (-0.8,2.5);

\draw[blue!40, thick, directed = .4] (-3.8,-2.5) to (-0.8,0.5);
\draw[black!20!blue, thick, directed = .5] (-3.8,-3.5) to (-0.8,0.5);
\draw[blue!40, thick, directed = .6] (-3.8,-3.5) to (-0.8,-.5);


\draw[black!20!green!50, thick, directed =.6] (-3.8,3.5) to (-.8,3.5);
\draw[black!60!green, thick, directed =.6] (-3.8,2.5) to (-.8,3.5);
\draw[black!20!green!50, thick, directed =.6] (-3.8,2.5) to (-.8,2.5);

\draw[black!20!green!50, thick, directed =.6] (-3.8,0.5) to (-.8,0.5);
\draw[black!60!green, thick, directed =.6] (-3.8,-0.5) to (-.8,0.5);
\draw[black!20!green!50, thick, directed =.6] (-3.8,-0.5) to (-.8,-0.5);

\draw[black!20!green!50, thick, directed =.6] (-3.8,-3.5) to (-.8,-3.5);
\draw[black!60!green, thick, directed =.6] (-3.8,-3.5) to (-.8,-2.5);
\draw[black!20!green!50, thick, directed =.6] (-3.8,-2.5) to (-.8,-2.5);


\draw[red, thick, dotted, directed=.85](-3.8,-3.5) to (-.8,3.5);
\end{tikzpicture}
\caption{The bi-complex for $\B_{b-1}\otimes\B_a$ with $\d_y$ denoted in green and $\d^x$ denoted in blue. Terms in the same homological degree lie in the same column.} \label{fig:F+Fbicomplex}
\end{figure}
This map can also be envisioned as the following sequence of compositions:
\[
\begin{tikzpicture}
\node at (0,-1) {${\color{red}\mathsf{d}_{j}}$};
\node at (0,0){$\mathcal{F}_{x,y}^{\mathbb{II}} {\color{blue}{\to}} \mathcal{F}_{x-1,y}^{\mathbb{I}} {\color{black!40!green}\to} \mathcal{F}_{x-1,y+1}^{\mathbb{II}}{\color{blue}\to} \dots {\color{blue}\to} \mathcal{F}_{x-j,y-1+j}^{\mathbb{I}} {\color{black!40!green}\to} \mathcal{F}_{x-j,y+j}^{\mathbb{II}} {\color{blue}\to} \mathcal{F}_{x-(j+1),y+j}^{\mathbb{I}}$};
\begin{scope}[xshift=-10, yshift = -8, xscale=.9, yscale =.2]
\draw (-5,0) arc (180:360:5)[->] [thick, dotted, red];
\end{scope}
\end{tikzpicture}\]
By Proposition \ref{prop:F} we know that $\B_{b-1}\otimes \B_a \simeq$ Tot$^\oplus\left\lbrace\bigoplus_{y\geq (0,b-1)}\mathcal{F}_{x,y}^{\mathbb{I}} \oplus \mathcal{F}_{x,y}^{\mathbb{II}}, \mathsf{d}_y,\mathsf{d}^x\right\rbrace_{x\geq (a,0)}$. Moreover, by Lemma \ref{lem:CP-reduction} each subcomplex $\left\lbrace \bigoplus_{y\geq (0,b-1)}\mathcal{F}_{x,y}^{\mathbb{I}} \oplus \mathcal{F}_{x,y}^{\mathbb{II}}, \mathsf{d}_y\right\rbrace $ is homotopy equivalent to Cone$(\mathsf{D}\otimes \1)$ where $\mathsf{D}\otimes \1:\mathcal{F}_{x}^{\mathbb{I}} \to \mathcal{F}_{x}^{\mathbb{II}}$ is the chain map between the chain complexes below with $\mathsf{D}$ defined as in Lemma \ref{lem:CP-reduction}.
 \begin{align*}
\mathcal{F}_{x}^{\mathbb{I}}&:= \left\lbrace \bigoplus_{y \geq (0,b-1,x)}\P^{(y,x)} \Q^{(y+1-b)^t} \Q^{(x-a)^t}[x+y-a-b],\hackcenter{\begin{tikzpicture} [scale=.3, xscale=.7, yscale=1,every node/.style={scale=.8}]
\draw[line width= 0.7mm,
    decoration={markings,mark=at position 1 with {\arrow[line width=.3mm]{>}}},
    postaction={decorate},shorten >=1pt] (-4.5,2.5) to (-4.5,8);
\draw[line width= 0.7mm,
    decoration={markings,mark=at position 0.06 with {\arrow[line width=.3mm]{<}}},
    postaction={decorate},shorten >=.6pt] (4.5,2.5) to (4.5,8);
\draw[line width= 0.7mm,
    decoration={markings,mark=at position 1 with {\arrow[line width=.3mm]{>}}},
    postaction={decorate},shorten >=1pt] (-1.5,2.5) to (-1.5, 8);
\draw[line width= 0.7mm,
    decoration={markings,mark=at position 0.06 with {\arrow[line width=.3mm]{<}}},
    postaction={decorate},shorten >=.6pt] (1.5,2.5) to (1.5, 8);
\draw[directed=.8] (-3.8,5.5)..controls ++(0,1) and ++(0,1)..(1.2, 5.5);
\draw[fill =gray!20] (-5.8,6.5) rectangle (-0.4,7.5);
\draw[fill =gray!20] (-5.8,3) rectangle (-0.4,4);
\draw[fill =white] (-5.8,4.5) rectangle (-3.3,5.5);
\draw[fill =white] (-2.7,4.5) rectangle (-0.4,5.5);
\draw[fill =gray!60] (5.8,4.5) rectangle (3.2,5.5);
\draw[fill =gray!60] (2.9,4.5) rectangle (0,5.5);
\node at (-3,7) [scale=.75]{$(y-1,x)$};
\node at (-3,3.5) [scale=.75]{$(y,x)$};
\node at (-4.5,5) [scale=.75]{$y$};
\node at (-1.5,5) [scale=.75]{$x$};
\node at (1.5,5)[scale=.5] {$y-b+1$};
\node at (4.5,5) [scale=.75]{$x-a$};
\end{tikzpicture}}  \right\rbrace
&\overset{\ref{prop:PPmerge}}{\simeq}&\left\lbrace \bigoplus_{y \geq (0,b-1,x)} \mathcal{F}_{x,y}^{\mathbb{I}}|_{s=x}[-1],\d_y\right\rbrace\\
\mathcal{F}_{x}^{\mathbb{II}}&:=\left\lbrace\bigoplus_{ y \geq (0,b)}^{x-1} \P^{(x-1,y)} \Q^{(y-b)^t}\Q^{(x-a)^t}[x+y-a-b+1],\hackcenter{\begin{tikzpicture} [scale=.3, xscale=.7, yscale=1,every node/.style={scale=.8}]
\draw[line width= 0.7mm,
    decoration={markings,mark=at position 1 with {\arrow[line width=.3mm]{>}}},
    postaction={decorate},shorten >=1pt] (-4.5,2.5) to (-4.5,8);
\draw[line width= 0.7mm,
    decoration={markings,mark=at position 0.06 with {\arrow[line width=.3mm]{<}}},
    postaction={decorate},shorten >=.6pt] (4.5,2.5) to (4.5,8);
\draw[line width= 0.7mm,
    decoration={markings,mark=at position 1 with {\arrow[line width=.3mm]{>}}},
    postaction={decorate},shorten >=1pt] (-1.5,2.5) to (-1.5, 8);
\draw[line width= 0.7mm,
    decoration={markings,mark=at position 0.06 with {\arrow[line width=.3mm]{<}}},
    postaction={decorate},shorten >=.6pt] (1.5,2.5) to (1.5, 8);
\draw[directed=.8] (-3.8,5.5)..controls ++(0,1) and ++(0,1)..(1.2, 5.5);
\draw[fill =gray!20] (-5.8,6.5) rectangle (-0.4,7.5);
\draw[fill =gray!20] (-5.8,3) rectangle (-0.4,4);
\draw[fill =white] (-5.8,4.5) rectangle (-3.3,5.5);
\draw[fill =white] (-2.7,4.5) rectangle (-0.4,5.5);
\draw[fill =gray!60] (5.8,4.5) rectangle (3.2,5.5);
\draw[fill =gray!60] (2.9,4.5) rectangle (0,5.5);
\node at (-3,7) [scale=.65]{$(x-1,y-1)$};
\node at (-3,3.5) [scale=.75]{$(x-1,y)$};
\node at (-4.5,5) [scale=.75]{$y$};
\node at (-1.5,5) [scale=.65]{$x-1$};
\node at (1.5,5)[scale=.75] {$y-b$};
\node at (4.5,5) [scale=.75]{$x-a$};
\end{tikzpicture}}  \right\rbrace&\overset{\ref{prop:PPmerge}}{\simeq}&\left\lbrace \bigoplus_{ y \geq (0,b)}^{x-1}   \mathcal{F}_{x,y}^{\mathbb{II}}|_{s=y},\d_y\right\rbrace
\end{align*}
These homotopies affect the existing differentials $\d^x$ along the rows of the bi-complex in the following way: 
\begin{itemize}
\item If $y\geq x$ then $\d^x:\mathcal{F}_x^{\mathbb{I}} \to \mathcal{F}_{x-1}^{\mathbb{I}}$ is given by $\d^x:\mathcal{F}_{x,y}^{\mathbb{I}} \to \mathcal{F}_{x-1,y}^{\mathbb{I}}$ defined in Proposition \ref{prop:F}. Since $\mathcal{F}_x^{\mathbb{II}}$ does not exist for $y\geq x$ then all the previous arrows $\d^x:\mathcal{F}_{x,y}^{\mathbb{II}} \to \mathcal{F}_{x-1,y}^{\mathbb{I}}\oplus \mathcal{F}_{x-1,y}^{\mathbb{II}}$ disappear. 
\item If $x>y$ then the differentials $\d^x:\mathcal{F}_x^{\mathbb{II}} \to \mathcal{F}_{x-1}^{\mathbb{II}}$ are given by $\d^x:\mathcal{F}_{x,y}^{\mathbb{II}}\to \mathcal{F}_{x-1,y}^{\mathbb{II}}$ from Proposition \ref{prop:F}. If also $y\neq x-1$ then the terms contained in the image of $\d^x:\mathcal{F}_{x,y}^{\mathbb{II}}|_{s=y} \to \mathcal{F}_{x-1,y}^{\mathbb{I}}$ are all canceled under Lemma \ref{lem:CP-reduction}. Hence, the image of $\mathcal{F}_x^{\mathbb{II}}$ in $\mathcal{F}_{x-1}^{\mathbb{I}}$ under $\d^x$ is zero except for when $y=x-1$, in which case by Proposition \ref{prop:F} $\d^x: \P^{(x-1,x-1)}\Q^{(x-a)^t}\Q^{(x-1-b)^t} \to \P^{(x-1,x-1)}\Q^{(x-1-a)^t}\Q^{(x-b)^t}$ is an isomorphism for all but one summand in the source. 
\end{itemize}
When $x>y$ the process also creates certain new arrows which we now describe (see Figure \ref{fig:F+Fbicomplex}). Since the only value of $j$ for which $\mathcal{F}_{x-j-1,y+j}^{\mathbb{I}}|_{s=x-j-1}$ is not eliminated by the homotopies in Lemma \ref{lem:CP-reduction} is $j=x-y-1$, then the simultaneous reductions generate the maps ${\color{red}\mathsf{d}_{x-y-1}}:\mathcal{F}_{x,y}^{\mathbb{II}}|_{s=y} {\color{red}\to} \mathcal{F}_{y,x-1}^{\mathbb{I}}|_{s=y}$ given by the iterated compositions of $\d^{x-j}$ and $(\d_{y+j})^{-1}$ in \eqref{eq:djRED} (see Figure \ref{fig:FF-reduced}). In particular, each ${\color{red}\mathsf{d}_{x-y-1}}$ is the composition of maps which are either isomorphisms or zero. Hence, for all $x\geq$ max(1,b+1) and $x\geq y \geq$ max$(0,a)$ there is an isomorphism between $\mathcal{F}_{x,y}^{\mathbb{II}}|_{s=y}$ and $\mathcal{F}_{x',y'}^{\mathbb{I}}|_{s=x'}$ by sending $x'\mapsto y$ and $y'\mapsto x-1$.

Taking the total complex of the bi-complex we see that, for each integer $k \geq 2-a-b$, the complex $\B_{b-1}\otimes \B_{a}$ is homotopy equivalent to a complex whose chain group in homological degree $k$ is given by
\begin{align*}
(\B_{b-1}\otimes\B_a)_{k}\;\;&\simeq
\bigoplus_{\substack{x'+y'=k+a+b-1\\ x'\geq (0,a)\\ y'\geq (0,b-1,x')}} \mathcal{F}_{x',y'}^{\mathbb{I}}|_{s=x'}
\oplus
\bigoplus_{\substack{x+y=k+a+b-1\\ x\geq (1,b+1)\\x-1\geq y\geq (0,b)}}  \mathcal{F}_{x,y}^{\mathbb{II}}|_{s=y} \\
&\simeq
\bigoplus_{\substack{x+y=k+a+b\\ y\geq (0,a)\\ x\geq (1,b,y+1)}} \mathcal{F}_{y,x-1}^{\mathbb{I}}|_{s=y}
\oplus
\bigoplus_{\substack{x+y=k+a+b-1\\ x\geq (1,b+1)\\x-1\geq y\geq (0,b)}}  \mathcal{F}_{x,y}^{\mathbb{II}}|_{s=y} .
\end{align*}
In particular, for any fixed $k$ such that $x+y=k+a+b-1$ and $x>y$ the new arrows from \eqref{eq:djRED} are morphisms from
\begin{equation}\label{eq:d-red}
\bigoplus_{\substack{ x\geq (1,b+1)\\ x>y\geq (0,b)}} \mathcal{F}_{x,y}^{\mathbb{II}}|_{s=y}[k]\;\; {\color{red}\xrightarrow{[\d_{x-y-1}]}} \bigoplus_{\substack{x\geq (1,b)\\ x>y\geq (0,a)}} \mathcal{F}_{y,x-1}^{\mathbb{I}}|_{s=y}[k-1]
\end{equation}
which are either multiples of the identity or zero.  Since $b>a$, these maps are injective. Performing Gaussian elimination on these isomorphisms $\mathcal{F}_{x}^{\mathbb{II}}$ is eliminated for all $x$. Moreover, since the source of these morphisms consists only of summands of $\mathcal{F}_x^{\mathbb{II}}$ and the target only of summands of $\mathcal{F}_y^\mathbb{I}$, then these reductions can be applied iteratively along increasing homological degrees. Therefore, the only terms that are left are certain summands of $\mathcal{F}_y^{\mathbb{I}}$ for $x>y$ described below. By comparing them with the source and target of \eqref{eq:d-red} it is straight forward to see that these homotopies do not alter the differentials $\d^{x'}\mapsto\d^y$ and $\d_{y'}\mapsto\d_{x-1}$ on $\mathcal{F}_{y,x-1}^{\mathbb{I}}$ for $x>y$. Hence, each chain group
\begin{align*}
\bigoplus_{\substack{x+y=k+a+b\\y\geq (0,a)\\x\geq(1,b,y+1)}}\mathcal{F}_{y,x-1}^{\mathbb{I}}|_{s=y}[k] \oplus \bigoplus_{\substack{x+y=k+a+b-1\\x\geq (1,b+1)\\x-1\geq y\geq(0,b)}}\mathcal{F}_{x,y}^{\mathbb{II}}|_{s=y}[k]
\end{align*}
can be reduced to the following object in $\H$
\begin{align*}
\mathcal{G}_k:=\bigoplus_{\substack{x+y=k+a+b\\x > y\geq (b,0)}}\P^{(x-1,y)} \left(\bigoplus_{r=y-b+1}^{(x-b,y-a)}\Q^{(x+y-a-b-r,r)^t}\right)[k]\;\;
\oplus \bigoplus_{\substack{x+y=k+a+b\\x \geq b>y\geq (a,0)}} \P^{(x-1,y)}\Q^{(x-b)^t}\Q^{(y-a)^t}[k]
\end{align*}
with differential $\mathsf{d}_{k}:=\sum_{x+y=k+a+b}\mathsf{d}_{x-1} + (-1)^y \mathsf{d}^y$ given by
\begin{align}
\d_{x-1}= \sum_{r={y-b+1}}^{(x-b,y-a)} \hackcenter{\begin{tikzpicture} [scale=.35, xscale=.7, yscale=1]
\draw[line width= 0.7mm,
    decoration={markings,mark=at position 1 with {\arrow[line width=.3mm]{>}}},
    postaction={decorate},shorten >=1pt] (-4.5,2.5) to (-4.5,8);
\draw[line width= 0.7mm,
    decoration={markings,mark=at position 0.06 with {\arrow[line width=.3mm]{<}}},
    postaction={decorate},shorten >=.6pt] (4.5,2.5) to (4.5,8);
\draw[line width= 0.7mm,
    decoration={markings,mark=at position 1 with {\arrow[line width=.3mm]{>}}},
    postaction={decorate},shorten >=1pt] (-1.5,2.5) to (-1.5, 8);
\draw[line width= 0.7mm,
    decoration={markings,mark=at position 0.06 with {\arrow[line width=.3mm]{<}}},
    postaction={decorate},shorten >=.6pt] (1.5,2.5) to (1.5, 8);
\draw[directed=.8] (-3.8,5.5)..controls ++(0,1) and ++(0,1)..(1.2, 5.5);
\draw[fill =gray!20] (-5.8,6.5) rectangle (-0.4,7.5);
\draw[fill =gray!20] (-5.8,3) rectangle (-0.4,4);
\draw[fill =gray!20] (5.8,6.5) rectangle (0.4,7.5);
\draw[fill =gray!20] (5.8,3) rectangle (0.4,4);
\draw[fill =white] (-5.8,4.5) rectangle (-3.3,5.5);
\draw[fill =white] (-2.7,4.5) rectangle (-0.4,5.5);
\draw[fill =gray!60] (5.8,4.5) rectangle (3.2,5.5);
\draw[fill =gray!60] (2.9,4.5) rectangle (0,5.5);
\node at (-3,7) [scale=.75]{$(x-2,y)$};
\node at (-3,3.5) [scale=.75]{$(x-1,y)$};
\node at (-4.5,5) [scale=.75]{$x-1$};
\node at (-1.5,5) [scale=.75]{$y$};
\node at (3,7) [scale=.75]{$(\cdots,r')$};
\node at (3,3.5) [scale=.75]{$(\cdots,r)$};
\node at (1.5,5)[scale=.75] {$x-b$};
\node at (4.5,5) [scale=.75]{$y-a$};
\end{tikzpicture}} \qquad \text{and} \qquad
\d^{y}=
\hackcenter{\begin{tikzpicture} [scale=.35, xscale=.7]
\draw[line width= 0.7mm,
    decoration={markings,mark=at position 1 with {\arrow[line width=.3mm]{>}}},
    postaction={decorate},shorten >=1pt] (-4.5,2.5) to (-4.5,8);
\draw[line width= 0.7mm,
    decoration={markings,mark=at position 0.06 with {\arrow[line width=.3mm]{<}}},
    postaction={decorate},shorten >=.6pt] (4.5,2.5) to (4.5,8);
\draw[line width= 0.7mm,
    decoration={markings,mark=at position 1 with {\arrow[line width=.3mm]{>}}},
    postaction={decorate},shorten >=1pt] (-1.5,2.5) to (-1.5, 8);
\draw[line width= 0.7mm,
    decoration={markings,mark=at position 0.06 with {\arrow[line width=.3mm]{<}}},
    postaction={decorate},shorten >=.6pt] (1.5,2.5) to (1.5, 8);
\draw[directed=.6] (-.8,5.5)..controls ++(0,1) and ++(0,1)..(4.2, 5.5);
\draw[fill =gray!20] (-5.8,6.5) rectangle (-0.4,7.5);
\draw[fill =gray!20] (-5.8,3) rectangle (-0.4,4);
\draw[fill =gray!20] (5.8,6.5) rectangle (0.4,7.5);
\draw[fill =gray!20] (5.8,3) rectangle (0.4,4);
\draw[fill =white] (-5.8,4.5) rectangle (-3.3,5.5);
\draw[fill =white] (-2.7,4.5) rectangle (-0.4,5.5);
\draw[fill =gray!60] (5.8,4.5) rectangle (3.2,5.5);
\draw[fill =gray!60] (2.9,4.5) rectangle (0,5.5);
\node at (-3,7) [scale=.65]{$(x-1,y-1)$};
\node at (-3,3.5) [scale=.75]{$(x-1,y)$};
\node at (-4.5,5) [scale=.75]{$x-1$};
\node at (-1.5,5) [scale=.75]{$y$};
\node at (3,7) [scale=.75]{$(\cdots,r')$};
\node at (3,3.5) [scale=.75]{$(\cdots,r)$};
\node at (1.5,5)[scale=.75] {$x-b$};
\node at (4.5,5) [scale=.75]{$y-a$};
\end{tikzpicture}}
\qquad \text{for}\;\; x>y\geq(b,0)
\end{align}
\begin{align}
\d_{x-1}=
\hackcenter{\begin{tikzpicture} [scale=.35, xscale=.7, yscale=1]
\draw[line width= 0.7mm,
    decoration={markings,mark=at position 1 with {\arrow[line width=.3mm]{>}}},
    postaction={decorate},shorten >=1pt] (-4.5,2.5) to (-4.5,8);
\draw[line width= 0.7mm,
    decoration={markings,mark=at position 0.06 with {\arrow[line width=.3mm]{<}}},
    postaction={decorate},shorten >=.6pt] (4.5,2.5) to (4.5,8);
\draw[line width= 0.7mm,
    decoration={markings,mark=at position 1 with {\arrow[line width=.3mm]{>}}},
    postaction={decorate},shorten >=1pt] (-1.5,2.5) to (-1.5, 8);
\draw[line width= 0.7mm,
    decoration={markings,mark=at position 0.06 with {\arrow[line width=.3mm]{<}}},
    postaction={decorate},shorten >=.6pt] (1.5,2.5) to (1.5, 8);
\draw[directed=.8] (-3.8,5.5)..controls ++(0,1) and ++(0,1)..(1.2, 5.5);
\draw[fill =gray!20] (-5.8,6.5) rectangle (-0.4,7.5);
\draw[fill =gray!20] (-5.8,3) rectangle (-0.4,4);
\draw[fill =white] (-5.8,4.5) rectangle (-3.3,5.5);
\draw[fill =white] (-2.7,4.5) rectangle (-0.4,5.5);
\draw[fill =gray!60] (5.8,4.5) rectangle (3.2,5.5);
\draw[fill =gray!60] (2.9,4.5) rectangle (0,5.5);
\node at (-3,7) [scale=.75]{$(x-2,y)$};
\node at (-3,3.5) [scale=.75]{$(x-1,y)$};
\node at (-4.5,5) [scale=.75]{$x-1$};
\node at (-1.5,5) [scale=.75]{$y$};
\node at (1.5,5)[scale=.75] {$x-b$};
\node at (4.5,5) [scale=.75]{$y-a$};
\end{tikzpicture}}
\qquad \text{and}\qquad
\d^y=
\hackcenter{\begin{tikzpicture} [scale=.35, xscale=.7]
\draw[line width= 0.7mm,
    decoration={markings,mark=at position 1 with {\arrow[line width=.3mm]{>}}},
    postaction={decorate},shorten >=1pt] (-4.5,2.5) to (-4.5,8);
\draw[line width= 0.7mm,
    decoration={markings,mark=at position 0.06 with {\arrow[line width=.3mm]{<}}},
    postaction={decorate},shorten >=.6pt] (4.5,2.5) to (4.5,8);
\draw[line width= 0.7mm,
    decoration={markings,mark=at position 1 with {\arrow[line width=.3mm]{>}}},
    postaction={decorate},shorten >=1pt] (-1.5,2.5) to (-1.5, 8);
\draw[line width= 0.7mm,
    decoration={markings,mark=at position 0.06 with {\arrow[line width=.3mm]{<}}},
    postaction={decorate},shorten >=.6pt] (1.5,2.5) to (1.5, 8);
\draw[directed=.6] (-.8,5.5)..controls ++(0,1) and ++(0,1)..(4.2, 5.5);
\draw[fill =gray!20] (-5.8,6.5) rectangle (-0.4,7.5);
\draw[fill =gray!20] (-5.8,3) rectangle (-0.4,4);
\draw[fill =white] (-5.8,4.5) rectangle (-3.3,5.5);
\draw[fill =white] (-2.7,4.5) rectangle (-0.4,5.5);
\draw[fill =gray!60] (5.8,4.5) rectangle (3.2,5.5);
\draw[fill =gray!60] (2.9,4.5) rectangle (0,5.5);
\node at (-3,7) [scale=.65]{$(x-1,y-1)$};
\node at (-3,3.5) [scale=.75]{$(x-1,y)$};
\node at (-4.5,5) [scale=.75]{$x-1$};
\node at (-1.5,5) [scale=.75]{$y$};
\node at (1.5,5)[scale=.75] {$x-b$};
\node at (4.5,5) [scale=.75]{$y-a$};
\end{tikzpicture}}
\qquad \text{for}\;\; x\geq b>y\geq(0,a).
\end{align}
Thus,  $\B_{b-1}\otimes \B_a $ is homotopy equivalent to the following chain complex:
\begin{equation}
\B_{b-1}\otimes \B_a \simeq \dots \xrightarrow{\mathsf{d}_{k+2}} \mathcal{G}_{k+1}[k+1]\xrightarrow{\mathsf{d}_{k+1}} \mathcal{G}_{k}[k]\xrightarrow{\mathsf{d}_{k-1}} \dots
\end{equation}
\end{proof}

\begin{figure}[ht]
\begin{tikzpicture}[xscale=.8,yscale=.8, every node/.style={scale=.8}]

\node at (-3,2.5){6};
\node at (-2,2.5){5};
\node at (-1,2.5){4};
\node at (0,2.5){3};
\node at (1,2.5){2};
\node at (2,2.5){1};
\node at (3,2.5){0};

\node at (-4.8,2) {$x=0$};
\node at (-4.8,1) {$x=1$};
\node at (-4.8,0) {$x=2$};
\node at (-4.8,-1) {$x=3$};
\node at (-4.8,-2) {$x=4$};

\node at (-1,2) {$\bullet $ };
\node at (-1,1) {$\bullet $ };
\node at (-1,0) {$\bullet $ };
\node at (-1,-1) {$\bullet $ };
\node at (-1,-2) {$\bullet $ };

\node at (-2,2) {$\bullet $ };
\node at (-2,1) {$\bullet $ };
\node at (-2,0) {$\bullet $ };
\node at (-2,-1) {$\bullet $ };
\node at (-2,-2) {$\bullet $ };

\node at (-3,2) {$\bullet $ };
\node at (-3,1) {$\bullet $ };
\node at (-3,0) {$\bullet $ };
\node at (-3,-1) {$\bullet $ };
\node at (-3,-2) {$\bullet $ };

\node at (0,2) {$\bullet$ };
\node at (0,1) {$\bullet$ };
\node at (0,0) {$\bullet $ };
\node at (0,-1) {$\bullet$ };

\node at (1,2) {$\bullet$ };
\node at (1,1) {$\bullet$ };
\node at (1,0) {$\bullet $ };

\node at (2,2) {$\bullet $ };
\node at (2,1) {$\bullet $ };

\node at (3,2) {$\bullet$ };


\draw[black!40!green, dotted,  directed =.6] (-4,-2) to (-3,-2);
\draw[black!40!green, dotted, directed =.6] (-4,-1) to (-3,-1);
\draw[black!40!green, dotted, directed =.6] (-4,0) to (-3,0);
\draw[black!40!green, dotted, directed =.6] (-4,1) to (-3,1);
\draw[black!40!green, dotted, directed =.6] (-4,2) to (-3,2);

\draw[black!40!green,  directed =.6] (-3,-2) to (-2,-2);
\draw[black!60!green, very thick,  directed =.6] (-3,-1) to (-2,-1);
\draw[black!40!green,  directed =.6] (-3,0) to (-2,0);
\draw[black!40!green,  directed =.6] (-3,1) to (-2,1);
\draw[black!40!green,  directed =.6] (-3,2) to (-2,2);

\draw[black!40!green,  directed =.6] (-2,-2) to (-1,-2);
\draw[black!40!green,  directed =.6] (-2,-1) to (-1,-1);
\draw[black!40!green,  directed =.6] (-2,-0) to (-1,-0);
\draw[black!40!green,  directed =.6] (-2,1) to (-1,1);
\draw[black!40!green,  directed =.6] (-2,2) to (-1,2);

\draw[black!40!green,  directed =.6] (-1,-1) to (0,-1);
\draw[black!60!green, very thick,  directed =.6] (-1,-0) to (0,-0);
\draw[black!40!green,  directed =.6] (-1,1) to (0,1);
\draw[black!40!green,  directed =.6] (-1,2) to (0,2);

\draw[black!40!green,  directed =.6] (0,-0) to (1,-0);
\draw[black!40!green,  directed =.6] (0,1) to (1,1);
\draw[black!40!green,  directed =.6] (0,2) to (1,2);

\draw[black!60!green, very thick, directed =.6] (1,1) to (2,1);
\draw[black!40!green,  directed =.6] (1,2) to (2,2);

\draw[black!40!green,  directed =.6] (2,2) to (3,2);



\draw[blue!60, thick,dotted,  directed =.6] (-4,-3) to (-3,-3+1);
\draw[blue!60, thick, dotted,directed =.6] (-4,-2) to (-3,-2+1);
\draw[blue!60, thick, dotted,directed =.6] (-4,-1) to (-3,0);
\draw[blue!60, thick, dotted, directed =.6] (-4,0) to (-3,-0+1);
\draw[blue!60, thick, dotted,directed =.6] (-4,1) to (-3,1+1);

\draw[blue!60, thick,dotted,  directed =.6] (-3,-3) to (-2,-3+1);
\draw[blue!60, thick, directed =.6] (-3,-2) to (-2,-2+1);
\draw[blue!60, thick, directed =.6] (-3,-1) to (-2,0);
\draw[blue!60, thick, directed =.6] (-3,0) to (-2,1);
\draw[blue!60, thick, directed =.6] (-3,1) to (-2,1+1);

\draw[blue!60, thick,dotted,  directed =.6] (-2,-3) to (-1,-3+1);
\draw[blue!60, thick, directed =.6] (-2,-2) to (-1,-2+1);
\draw[blue!60, thick, directed =.6] (-2,-1) to (-1,0);
\draw[blue!60, thick, directed =.6] (-2,0) to (-1,1);
\draw[blue!60, thick, directed =.6] (-2,1) to (-1,1+1);

\draw[blue!60, thick, directed =.6] (-1,-2) to (0,-2+1);
\draw[blue!60, thick, directed =.6] (-1,-1) to (0,0);
\draw[blue!60, thick, directed =.6] (-1,0) to (0,1);
\draw[blue!60, thick, directed =.6] (-1,1) to (0,1+1);

\draw[blue!60, thick, directed =.6] (0,-1) to (1,0);
\draw[blue!60, thick, directed =.6] (0,0) to (1,1);
\draw[blue!60, thick, directed =.6] (0,1) to (1,1+1);

\draw[blue!60, thick, directed =.6] (1,0) to (2,1);
\draw[blue!60, thick, directed =.6] (1,1) to (2,1+1);

\draw[blue!60, thick, directed =.6] (2,1) to (3,1+1);





\draw[red!60, very thick, directed =.6](2,1) to (3,2);
\draw[red!60, very thick, directed =.75](1,0) to (2,2);
\draw[red!60, very thick, directed =.8](0,-1) to (1,2);
\draw[red!60, very thick, directed =.85](-1,-2) to (0,2);

\draw[red!60, very thick, directed =.87](-1.8,-2) to (-1,2);
\draw[red!60, very thick, dotted](-2,-3) to (-1.8,-2);

\draw[red!60, very thick, directed =.89](-2.68,-2) to (-2,2);
\draw[red!60, very thick,dotted ](-2.85,-3) to (-2.68,-2);

\draw[red!60, very thick,dotted, directed =.91](-3.57,-2) to (-3,2);
\draw[red!60, very thick, dotted](-3.7,-3) to (-3.57,-2);

\end{tikzpicture}
\begin{tikzpicture}[xscale=.8,yscale=.8, every node/.style={scale=.8}]

\node at (-3,2.5){6};
\node at (-2,2.5){5};
\node at (-1,2.5){4};
\node at (0,2.5){3};
\node at (1,2.5){2};
\node at (2,2.5){1};
\node at (3,2.5){0};

\node at (-1,2) {$\bullet $ };
\node at (-1,1) {$\bullet $ };
\node at (-1,0) {$\bullet $ };
\node at (-1,-1) {$\bullet $ };
\node at (-1,-2) {$\bullet $ };

\node at (-2,2) {$\bullet $ };
\node at (-2,1) {$\bullet $ };
\node at (-2,0) {$\bullet $ };
\node at (-2,-1) {$\bullet $ };
\node at (-2,-2) {$\bullet $ };

\node at (-3,2) {$\bullet $ };
\node at (-3,1) {$\bullet $ };
\node at (-3,0) {$\bullet $ };
\node at (-3,-1) {$\bullet $ };
\node at (-3,-2) {$\bullet $ };

\node at (0,2) {$\bullet$ };
\node at (0,1) {$\bullet$ };
\node at (0,0) {$\bullet $ };
\node at (0,-1) {$\bullet$ };

\node at (1,2) {$\bullet$ };
\node at (1,1) {$\bullet$ };
\node at (1,0) {$\bullet $ };

\node at (2,2) {$\bullet $ };
\node at (2,1) {$\bullet $ };

\node at (3,2) {$\bullet$ };


\draw[black!40!green, dotted,  directed =.6] (-4,-2) to (-3,-2);
\draw[black!40!green, dotted, directed =.6] (-4,-1) to (-3,-1);
\draw[black!40!green, dotted, directed =.6] (-4,0) to (-3,0);
\draw[black!40!green, dotted, directed =.6] (-4,1) to (-3,1);
\draw[black!40!green, dotted, directed =.6] (-4,2) to (-3,2);

\draw[black!40!green,  directed =.6] (-3,-2) to (-2,-2);
\draw[black!60!green, very thick,  directed =.6] (-3,-1) to (-2,-1);
\draw[black!40!green,  directed =.6] (-3,0) to (-2,0);
\draw[black!40!green,  directed =.6] (-3,1) to (-2,1);
\draw[black!40!green,  directed =.6] (-3,2) to (-2,2);

\draw[black!40!green,  directed =.6] (-2,-2) to (-1,-2);
\draw[black!40!green,  directed =.6] (-2,-1) to (-1,-1);
\draw[black!40!green,  directed =.6] (-2,-0) to (-1,-0);
\draw[black!40!green,  directed =.6] (-2,1) to (-1,1);
\draw[black!40!green,  directed =.6] (-2,2) to (-1,2);

\draw[black!40!green,  directed =.6] (-1,-1) to (0,-1);
\draw[black!60!green, very thick,  directed =.6] (-1,-0) to (0,-0);
\draw[black!40!green,  directed =.6] (-1,1) to (0,1);
\draw[black!40!green,  directed =.6] (-1,2) to (0,2);

\draw[black!40!green,  directed =.6] (0,-0) to (1,-0);
\draw[black!40!green,  directed =.6] (0,1) to (1,1);
\draw[black!40!green,  directed =.6] (0,2) to (1,2);

\draw[black!60!green, very thick, directed =.6] (1,1) to (2,1);
\draw[black!40!green,  directed =.6] (1,2) to (2,2);

\draw[black!40!green,  directed =.6] (2,2) to (3,2);



\draw[blue!60, thick,dotted,  directed =.6] (-4,-3) to (-3,-3+1);
\draw[blue!60, thick, dotted,directed =.6] (-4,-2) to (-3,-2+1);
\draw[blue!60, thick, dotted,directed =.6] (-4,-1) to (-3,0);
\draw[blue!60, thick, dotted, directed =.6] (-4,0) to (-3,-0+1);
\draw[blue!60, thick, dotted,directed =.6] (-4,1) to (-3,1+1);

\draw[blue!60, thick,dotted,  directed =.6] (-3,-3) to (-2,-3+1);
\draw[blue!60, thick, directed =.6] (-3,-2) to (-2,-2+1);
\draw[blue!60, thick, directed =.6] (-3,-1) to (-2,0);
\draw[blue!60, thick, directed =.6] (-3,0) to (-2,1);
\draw[blue!60, thick, directed =.6] (-3,1) to (-2,1+1);

\draw[blue!60, thick,dotted,  directed =.6] (-2,-3) to (-1,-3+1);
\draw[blue!60, thick, directed =.6] (-2,-2) to (-1,-2+1);
\draw[blue!60, thick, directed =.6] (-2,-1) to (-1,0);
\draw[blue!60, thick, directed =.6] (-2,0) to (-1,1);
\draw[blue!60, thick, directed =.6] (-2,1) to (-1,1+1);

\draw[blue!60, thick, directed =.6] (-1,-2) to (0,-2+1);
\draw[blue!60, thick, directed =.6] (-1,-1) to (0,0);
\draw[blue!60, thick, directed =.6] (-1,0) to (0,1);
\draw[blue!60, thick, directed =.6] (-1,1) to (0,1+1);

\draw[blue!60, thick, directed =.6] (0,-1) to (1,0);
\draw[blue!60, thick, directed =.6] (0,0) to (1,1);
\draw[blue!60, thick, directed =.6] (0,1) to (1,1+1);

\draw[blue!60, thick, directed =.6] (1,0) to (2,1);
\draw[blue!60, thick, directed =.6] (1,1) to (2,1+1);

\draw[blue!60, thick, directed =.6] (2,1) to (3,1+1);





\draw[red!60, very thick, directed =.6](0,0) to (1,1);
\draw[red!60, very thick, directed =.75](-1,-1) to (0,1);
\draw[red!60, very thick, directed =.85](-2,-2) to (-1,1);


\draw[red!60, very thick, directed =.87](-2.75,-2) to (-2,1);
\draw[red!60, very thick, dotted](-3,-3) to (-2.75,-2);

\draw[red!60, very thick, ,dotted,directed =.89](-3.85,-3) to (-3,1);


\end{tikzpicture}
\caption{The bi-complex $\B_{b-1}\otimes\B_a$ for $a<b \leq 0$. The differentials ${\color{red}\d_{x-y}}$ for $y=0$ (\emph{left}) and $y=1$ (\emph{right}) are presented in upward red arrows. The differential $\d_x$ are the diagonal blue arrows, $\d_y$  the thin horizontal green arrows, and $\mathsf{D}$ the thick dark green arrows. Terms in the same column have equal homological degree, which is denoted by the number above shifted by $a+b-1$.}\label{fig:FF-reduced}
\end{figure}
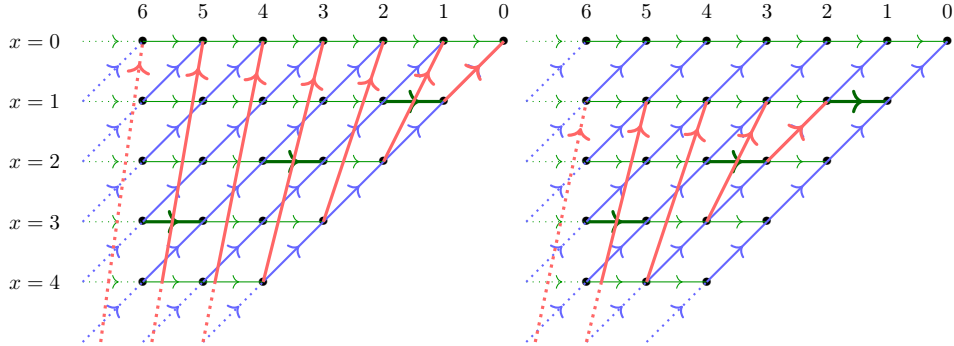

\begin{theorem}\label{thm:BBa>b} Let $a,b \in \Z$ with $a>b$ and for each integer $k \geq 1-a-b$ define $\mathcal{G}_{k}' \in \H$,
\begin{equation}\label{eq:Gk'} \mathcal{G}_{k}':= \bigoplus_{\substack{x+y=k+a+b\\y > x\geq (a,0)}}\P^{(y-1,x)} \left(\bigoplus_{r=x-a+1}^{(x-b,y-a)}\Q^{(x+y-a-b-r,r)^t}\right)
\oplus \bigoplus_{\substack{x+y=k+a+b\\y \geq a>x\geq (b,0)}} \P^{(y-1,x)}\Q^{(x-b)^t}\Q^{(y-a)^t}
\end{equation}
and let $\mathsf{d}'_{k}$ be given by:
\begin{align}
\d'_k:=&\sum_{\substack{x+y=k+a+b\\y>x\geq (a,0)}}\sum_{r={x-a+1}}^{(x-b,y-a)} (-1)^x 
\hackcenter{\begin{tikzpicture} [scale=.35, xscale=.7]
\draw[line width= 0.7mm,
    decoration={markings,mark=at position 1 with {\arrow[line width=.3mm]{>}}},
    postaction={decorate},shorten >=1pt] (-4.5,2.5) to (-4.5,8);
\draw[line width= 0.7mm,
    decoration={markings,mark=at position 0.06 with {\arrow[line width=.3mm]{<}}},
    postaction={decorate},shorten >=.6pt] (4.5,2.5) to (4.5,8);
\draw[line width= 0.7mm,
    decoration={markings,mark=at position 1 with {\arrow[line width=.3mm]{>}}},
    postaction={decorate},shorten >=1pt] (-1.5,2.5) to (-1.5, 8);
\draw[line width= 0.7mm,
    decoration={markings,mark=at position 0.06 with {\arrow[line width=.3mm]{<}}},
    postaction={decorate},shorten >=.6pt] (1.5,2.5) to (1.5, 8);
\draw[directed=.6] (-.8,5.5)..controls ++(0,1) and ++(0,1)..(4.2, 5.5);
\draw[fill =gray!20] (-5.8,6.5) rectangle (-0.4,7.5);
\draw[fill =gray!20] (-5.8,3) rectangle (-0.4,4);
\draw[fill =gray!20] (5.8,6.5) rectangle (0.4,7.5);
\draw[fill =gray!20] (5.8,3) rectangle (0.4,4);
\draw[fill =white] (-5.8,4.5) rectangle (-3.3,5.5);
\draw[fill =white] (-2.7,4.5) rectangle (-0.4,5.5);
\draw[fill =gray!60] (5.8,4.5) rectangle (3.2,5.5);
\draw[fill =gray!60] (2.9,4.5) rectangle (0,5.5);
\node at (-3,7) [scale=.65]{$(y-1,x-1)$};
\node at (-3,3.5) [scale=.75]{$(y-1,x)$};
\node at (-4.5,5) [scale=.75]{$y-1$};
\node at (-1.5,5) [scale=.75]{$x$};
\node at (3,7) [scale=.75]{$(\cdots,r')$};
\node at (3,3.5) [scale=.75]{$(\cdots,r)$};
\node at (1.5,5)[scale=.75] {$y-a$};
\node at (4.5,5) [scale=.75]{$x-b$};
\end{tikzpicture}}\;\;+\;\;
\hackcenter{\begin{tikzpicture} [scale=.35, xscale=.7, yscale=1]
\draw[line width= 0.7mm,
    decoration={markings,mark=at position 1 with {\arrow[line width=.3mm]{>}}},
    postaction={decorate},shorten >=1pt] (-4.5,2.5) to (-4.5,8);
\draw[line width= 0.7mm,
    decoration={markings,mark=at position 0.06 with {\arrow[line width=.3mm]{<}}},
    postaction={decorate},shorten >=.6pt] (4.5,2.5) to (4.5,8);
\draw[line width= 0.7mm,
    decoration={markings,mark=at position 1 with {\arrow[line width=.3mm]{>}}},
    postaction={decorate},shorten >=1pt] (-1.5,2.5) to (-1.5, 8);
\draw[line width= 0.7mm,
    decoration={markings,mark=at position 0.06 with {\arrow[line width=.3mm]{<}}},
    postaction={decorate},shorten >=.6pt] (1.5,2.5) to (1.5, 8);
\draw[directed=.8] (-3.8,5.5)..controls ++(0,1) and ++(0,1)..(1.2, 5.5);
\draw[fill =gray!20] (-5.8,6.5) rectangle (-0.4,7.5);
\draw[fill =gray!20] (-5.8,3) rectangle (-0.4,4);
\draw[fill =gray!20] (5.8,6.5) rectangle (0.4,7.5);
\draw[fill =gray!20] (5.8,3) rectangle (0.4,4);
\draw[fill =white] (-5.8,4.5) rectangle (-3.3,5.5);
\draw[fill =white] (-2.7,4.5) rectangle (-0.4,5.5);
\draw[fill =gray!60] (5.8,4.5) rectangle (3.2,5.5);
\draw[fill =gray!60] (2.9,4.5) rectangle (0,5.5);
\node at (-3,7) [scale=.75]{$(y-2,x)$};
\node at (-3,3.5) [scale=.75]{$(y-1,x)$};
\node at (-4.5,5) [scale=.75]{$y-1$};
\node at (-1.5,5) [scale=.75]{$x$};
\node at (3,7) [scale=.75]{$(\cdots,r')$};
\node at (3,3.5) [scale=.75]{$(\cdots,r)$};
\node at (1.5,5)[scale=.75] {$y-a$};
\node at (4.5,5) [scale=.75]{$x-b$};
\end{tikzpicture}}
\\
&+ 
\sum_{\substack{x+y=k+a+b\\ y\geq a>x\geq(0,b)}} (-1)^x 
\hackcenter{\begin{tikzpicture} [scale=.35, xscale=.7]
\draw[line width= 0.7mm,
    decoration={markings,mark=at position 1 with {\arrow[line width=.3mm]{>}}},
    postaction={decorate},shorten >=1pt] (-4.5,2.5) to (-4.5,8);
\draw[line width= 0.7mm,
    decoration={markings,mark=at position 0.06 with {\arrow[line width=.3mm]{<}}},
    postaction={decorate},shorten >=.6pt] (4.5,2.5) to (4.5,8);
\draw[line width= 0.7mm,
    decoration={markings,mark=at position 1 with {\arrow[line width=.3mm]{>}}},
    postaction={decorate},shorten >=1pt] (-1.5,2.5) to (-1.5, 8);
\draw[line width= 0.7mm,
    decoration={markings,mark=at position 0.06 with {\arrow[line width=.3mm]{<}}},
    postaction={decorate},shorten >=.6pt] (1.5,2.5) to (1.5, 8);
\draw[directed=.6] (-.8,5.5)..controls ++(0,1) and ++(0,1)..(4.2, 5.5);
\draw[fill =gray!20] (-5.8,6.5) rectangle (-0.4,7.5);
\draw[fill =gray!20] (-5.8,3) rectangle (-0.4,4);
\draw[fill =white] (-5.8,4.5) rectangle (-3.3,5.5);
\draw[fill =white] (-2.7,4.5) rectangle (-0.4,5.5);
\draw[fill =gray!60] (5.8,4.5) rectangle (3.2,5.5);
\draw[fill =gray!60] (2.9,4.5) rectangle (0,5.5);
\node at (-3,7) [scale=.65]{$(y-1,x-1)$};
\node at (-3,3.5) [scale=.75]{$(y-1,x)$};
\node at (-4.5,5) [scale=.75]{$y-1$};
\node at (-1.5,5) [scale=.75]{$x$};
\node at (1.5,5)[scale=.75] {$y-a$};
\node at (4.5,5) [scale=.75]{$x-b$};
\end{tikzpicture}}\;\;+\;\;
\hackcenter{\begin{tikzpicture} [scale=.35, xscale=.7, yscale=1]
\draw[line width= 0.7mm,
    decoration={markings,mark=at position 1 with {\arrow[line width=.3mm]{>}}},
    postaction={decorate},shorten >=1pt] (-4.5,2.5) to (-4.5,8);
\draw[line width= 0.7mm,
    decoration={markings,mark=at position 0.06 with {\arrow[line width=.3mm]{<}}},
    postaction={decorate},shorten >=.6pt] (4.5,2.5) to (4.5,8);
\draw[line width= 0.7mm,
    decoration={markings,mark=at position 1 with {\arrow[line width=.3mm]{>}}},
    postaction={decorate},shorten >=1pt] (-1.5,2.5) to (-1.5, 8);
\draw[line width= 0.7mm,
    decoration={markings,mark=at position 0.06 with {\arrow[line width=.3mm]{<}}},
    postaction={decorate},shorten >=.6pt] (1.5,2.5) to (1.5, 8);
\draw[directed=.8] (-3.8,5.5)..controls ++(0,1) and ++(0,1)..(1.2, 5.5);
\draw[fill =gray!20] (-5.8,6.5) rectangle (-0.4,7.5);
\draw[fill =gray!20] (-5.8,3) rectangle (-0.4,4);
\draw[fill =white] (-5.8,4.5) rectangle (-3.3,5.5);
\draw[fill =white] (-2.7,4.5) rectangle (-0.4,5.5);
\draw[fill =gray!60] (5.8,4.5) rectangle (3.2,5.5);
\draw[fill =gray!60] (2.9,4.5) rectangle (0,5.5);
\node at (-3,7) [scale=.75]{$(y-2,x)$};
\node at (-3,3.5) [scale=.75]{$(y-1,x)$};
\node at (-4.5,5) [scale=.75]{$y-1$};
\node at (-1.5,5) [scale=.75]{$x$};
\node at (1.5,5)[scale=.75] {$y-a$};
\node at (4.5,5) [scale=.75]{$x-b$};
\end{tikzpicture}}.
\end{align}
Then the following homotopy equivalence holds in $\K(\H)$:
\begin{equation}\label{eq:Cb+1Ca-reduced-case2}
\B_{b-1}\otimes \B_a \simeq \dots \xrightarrow{\mathsf{d}_{k+1}'} \mathcal{G}_{k}'[k+1] \xrightarrow{\mathsf{d}_k'} \mathcal{G}_{k-1}'[k] \xrightarrow{\mathsf{d}_{k-1}'} \dots
\end{equation}
\end{theorem}

\begin{proof}
This proof is analogous to Theorem \ref{thm:BBa<b} with some minor differences which we now sketch. Let $x < y$ and for integers $1< j \leq y-x-b+a$ and $\mathcal{F}_{x,y}^{\mathbb{I}}$,$\mathcal{F}_{x,y}^{\mathbb{II}}$, $\d_y$, and $\d^x$ defined as in Proposition \ref{prop:F} consider the restrictions of the maps
\[\mathsf{d}_{y-j}:\mathcal{F}_{x+j,y-j}^{\mathbb{II}} {\color{black!40!green}\to} \mathcal{F}_{x+j,y-j-1}^{\mathbb{I}}\qquad \text{and}\qquad\mathsf{d}^{x+j}: \mathcal{F}_{x+j,y-j}^{\mathbb{II}} {\color{blue}\to}\mathcal{F}_{x+j-1,y-j}^{\mathbb{I}}.\]
As before, restricting the summands to $0\leq s\leq x$ and $0\leq r <y-x+b-a$, we obtain that $\mathsf{d}_y$ and $\mathsf{d}^x$ are bijections for $1< j \leq y-x$. Since $b<a$ then $y-x<y-x-b+a$. Thus, the composition
\begin{equation}\label{eq:djRED2}
\d_j:=\d^{x-j}\circ (\d_{y+j-1})^{-1} \circ \d^{x-j-1} \cdots (\d_{y+1})^{-1}\circ \d^{x}
\end{equation}
is well defined and exists for all exists for all $1< j \leq y-x$ and $0\leq s \leq x$.  These maps correspond to following the darker blue and green lines in an downward zigzag in Figure \ref{fig:F+Fbicomplex}. Equivalently, we can envision them as the following sequence of compositions:
\begin{align*}
\begin{tikzpicture}
\node at (0,-1) {${\color{red}\mathsf{d}_{j}}$};
\node at (0,0){$\mathcal{F}_{x+j,y-j}^{\mathbb{II}} {\color{blue}{\to}} \mathcal{F}_{x+j-1,y-j}^{\mathbb{I}} {\color{black!40!green}\to} \mathcal{F}_{x+j-1,y-j+1}^{\mathbb{II}}{\color{blue}\to} \dots {\color{blue}\to} \mathcal{F}_{x+1,y-2}^{\mathbb{I}} {\color{black!40!green}\to} \mathcal{F}_{x+1,y-1}^{\mathbb{II}} {\color{blue}\to} \mathcal{F}_{x,y-1}^{\mathbb{I}}$};
\begin{scope}[xshift=-10, yshift = -8, xscale=1.1, yscale =.2]
\draw (-5,0) arc (180:360:5)[->] [thick, dotted, red];
\end{scope}
\end{tikzpicture}
\end{align*}
Once again by Proposition \ref{prop:F} $\B_{b-1}\otimes \B_a \simeq$ Tot$^\oplus\left\lbrace\bigoplus_{y\geq (0,b-1)}\mathcal{F}_{x,y}^{\mathbb{I}} \oplus \mathcal{F}_{x,y}^{\mathbb{II}}, \mathsf{d}_y,\mathsf{d}^x\right\rbrace_{x\geq (a,0)}$ with each subcomplex $\left\lbrace \bigoplus_{y\geq (0,b-1)}\mathcal{F}_{x,y}^{\mathbb{I}} \oplus \mathcal{F}_{x,y}^{\mathbb{II}}, \mathsf{d}_y\right\rbrace $ is homotopy equivalent to the mapping cone of the chain map $\mathsf{D}\otimes \1:\mathcal{F}_{x}^{\mathbb{I}} \to \mathcal{F}_{x}^{\mathbb{II}}$ where, as before:
\begin{align*}
&\mathcal{F}_{x}^{\mathbb{I}}\;\;{\simeq}\;\;\left\lbrace \bigoplus_{y \geq (0,b-1,x)} \mathcal{F}_{x,y}^{\mathbb{I}}|_{s=x}[-1],\d_y\right\rbrace &\text{and}&&
\mathcal{F}_{x}^{\mathbb{II}}\;\;{\simeq}\;\;\left\lbrace \bigoplus_{ y \geq (0,b)}^{x-1}   \mathcal{F}_{x,y}^{\mathbb{II}}|_{s=y},\d_y\right\rbrace
\end{align*}
Since $a>b$, this time if we apply Lemma \ref{lem:CP-reduction} the new arrows arise for $j=y-x$. Hence, for all $x\geq$ max(1,b+1) and $x\geq y \geq$ max$(0,a)$ there is an isomorphism between $\mathcal{F}_{x',y'}^{\mathbb{II}}|_{s=y'}$ and $\mathcal{F}_{x,y-1}^{\mathbb{I}}|_{s=x}$ by sending $x'\mapsto y$ and $y'\mapsto x$.

Thus, if we take the total complex of the reduced bi-complex then for each integer $k \geq 2-a-b$, the complex $\B_{b-1}\otimes \B_{a}$ is homotopy equivalent to a complex whose chain group in homological degree $k$ is
\begin{align*}
(\B_{b-1}\otimes\B_a)_{k}\;\;&\simeq
\bigoplus_{\substack{x+y=k+a+b-1\\ x\geq (0,a)\\ y\geq (0,b-1,x)}} \mathcal{F}_{x,y}^{\mathbb{I}}|_{s=x}
\oplus
\bigoplus_{\substack{x'+y'=k+a+b-1\\ x'\geq (1,b+1)\\x'-1\geq y'\geq (0,b)}}  \mathcal{F}_{x',y'}^{\mathbb{II}}|_{s=y'} \\
&\simeq
\bigoplus_{\substack{x+y=k+a+b-1\\ x\geq (0,a)\\ y\geq (0,b-1,x)}} \mathcal{F}_{x,y}^{\mathbb{I}}|_{s=x}
\oplus
\bigoplus_{\substack{y+x=k+a+b-1\\ y\geq (1,b+1)\\y-1\geq x\geq (0,b)}}  \mathcal{F}_{y,x}^{\mathbb{II}}|_{s=x}.
\end{align*}
In particular, for any fixed $k$ such that $x+y=k+a+b-1$ and $y>x$ the new arrows from \eqref{eq:djRED2} analogous to the morphism from \eqref{eq:d-red} are given by:
\begin{equation}\label{eq:d-red2}
\bigoplus_{\substack{y\geq (1,a)\\y-1\geq x\geq (0,b)}} \mathcal{F}_{y,x}^{\mathbb{II}}|_{s=x} {\color{red}\xrightarrow{[\d_{y-x}]}} \bigoplus_{\substack{y\geq (1,a+1)\\ y-1\geq x\geq (0,a)}} \mathcal{F}_{x,y-1}^{\mathbb{I}}|_{s=x}
\end{equation}
Since $a>b$ these maps are surjective. Performing Gaussian elimination across all homological degrees $\mathcal{F}_{x,y}^{\mathbb{I}}$ is eliminated for all $y\geq(0,b-1)$ with $y>x\geq (0,a)$. Hence, each chain group
\[
\bigoplus_{\substack{x+y=k+a+b-1\\ x\geq (0,a)\\ y\geq (0,b-1,x)}} \mathcal{F}_{x,y}^{\mathbb{I}}|_{s=x}[k]
\oplus
\bigoplus_{\substack{y+x=k+a+b-1\\ y\geq (1,b+1)\\y-1\geq x\geq (0,b)}}  \mathcal{F}_{y,x}^{\mathbb{II}}|_{s=x}[k]
\]
can be reduced to the following
\[ \mathcal{G}_{k-1}':= \bigoplus_{\substack{x+y=k-1+a+b\\y > x\geq (a,0)}}\P^{(y-1,x)} \left(\bigoplus_{r=x-a+1}^{(x-b,y-a)}\Q^{(x+y-a-b-r,r)^t}\right)[k]
\oplus \bigoplus_{\substack{x+y=k-1+a+b\\y \geq a>x\geq (b,0)}} \P^{(y-1,x)}\Q^{(x-b)^t}\Q^{(y-a)^t}[k].
\]
Thus, we have the following homotopy equivalence
\begin{equation}
\B_{b-1}\otimes \B_a \simeq \dots \xrightarrow{\mathsf{d}_{k+1}'} \mathcal{G}_{k}'[k+1] \xrightarrow{\mathsf{d}_{k}'} \mathcal{G}_{k-1}'[k] \xrightarrow{\mathsf{d}_{k-1}'} \dots
\end{equation}
with $\mathsf{d}'_k = \sum_{x+y=k+a+b} (-1)^x  \mathsf{d}^x + \mathsf{d}_y$ defined analogously as for Theorem \ref{thm:BBa<b}. 
\end{proof}

Combining the last two theorems we prove the first of the categorical Bernstein operator relations. 

\begin{theorem}\label{thm:BernsteinOps1}
For any integers $a,b \in \Z$, the categorical Bernstein operators satisfy the following homotopy relations in $\K(\H)$:
\[\B_{a-1}\otimes \B_a \simeq 0 \qquad\text{and}\qquad \B_{a-1}\otimes \B_b \simeq \begin{cases}
\B_{b-1}\otimes \B_a[+1] & a<b\\
\B_{b-1}\otimes \B_a[-1] & a>b. \\
\end{cases}\]
\end{theorem}

\begin{proof}
The first homotopy, $\B_{a-1}\otimes \B_a \simeq 0$ follows from setting $a=b$ in equation \eqref{eq:d-red} within the proof of Theorem \ref{thm:BBa<b}. That is, the isomorphisms between homological degrees are given by
\begin{equation}
\bigoplus_{\substack{ x\geq (1,a+1)\\ x>y\geq (0,a)}} \mathcal{F}_{x,y}^{\mathbb{II}}|_{s=y}[k]\;\; {\color{red}\xrightarrow{[\d_{x-y-1}]}} \bigoplus_{\substack{x\geq (1,a+1)\\ x>y\geq (0,a)}} \mathcal{F}_{y,x-1}^{\mathbb{I}}|_{s=y}[k-1]
\end{equation}
Since $\Q^{\nu}\otimes\Q^{\lambda} \simeq \Q^{\lambda}\otimes \Q^{\nu}$ for any partitions $\nu, \lambda$, performing Gaussian eliminations across all homological degrees, all the chain groups cancel one another. Thus, $\B_{a-1} \otimes \B_a$ is contractible.

Now, suppose $a<b$. Then by Theorem \ref{thm:BBa>b}, $\B_{a-1}\otimes \B_b \simeq \left\lbrace \bigoplus_{k\in\Z}  \mathcal{G}_{k}'[k+1], \d_k \right\rbrace$. Relabeling, so that $x\mapsto y$ and $y\mapsto x$ in \eqref{eq:Gk'} and comparing with \ref{eq:Gk} we immediately see $\mathcal{G}_{k}'\mapsto \mathcal{G}_{k}$ and
\[\mathsf{d}'_k = \sum_{x+y=k+a+b}(-1)^x \mathsf{d}^x + \mathsf{d}_{y-1} \mapsto \sum_{x+y=k+a+b} (-1)^y \mathsf{d}^y + \mathsf{d}_{x-1}=\mathsf{d}_{k}.\]
Since by Theorem \ref{thm:BBa<b} we know $\B_{b-1}\otimes \B_a \simeq \left\lbrace \bigoplus_{k\in\Z}  \mathcal{G}_{k}[k], \d_k \right\rbrace$, it immediately follows that $\B_{b-1}\otimes \B_a[+1] \simeq \B_{a-1}\otimes \B_b$. 

If instead $a>b$, then in an identical manner by exchanging $x,y$ in equation \eqref{eq:Gk} of Theorem \ref{thm:BBa<b} we see that $\B_{a-1}\otimes \B_b \simeq \left\lbrace \bigoplus_{k\in\Z}  \mathcal{G}_{k}[k], \d_k \right\rbrace$. Since by Theorem \ref{thm:BBa>b}, we have $\B_{b-1}\otimes \B_a \simeq \left\lbrace \bigoplus_{k\in\Z}  \mathcal{G}_{k}'[k+1], \d'_k \right\rbrace$  then again we obtain the homotopy equivalence $\B_{b-1}\otimes \B_a[-1] \simeq \B_{a-1}\otimes \B_b$.
\end{proof}

\begin{remark}
Since $\H$ is not Krull-Schmidt none of the presentations above are necessarily minimal. In particular, via an analogous argument using Lemma \ref{lem:QC-reduction} instead of Lemma \ref{lem:CP-reduction}, one can reduce the bi-complex and reprove Theorem \ref{thm:BernsteinOps1} using completely different but still homotopy equivalent presentations of the chain complex $\B_{b-1} \otimes \B_a$. 
\end{remark}

In an identical manner, we also show the analogous statements for all the Lemma, Propositions and Theorems for the adjoint categorical Bernstein operators. In particular, we can derive the following theorem. 

\begin{theorem}\label{thm:BernsteinOps2}
For any integers $a,b \in \Z$, the following categorical Bernstein relations hold in $\K(\H)$.
\[\B^*_{a+1} \otimes \B^*_a \simeq 0 \qquad \text{and} \qquad 
\B^*_{a+1}\otimes \B^*_b \simeq \begin{cases}
\B^*_{b+1}\otimes \B^*_a[+1] & a<b\\
\B^*_{b+1}\otimes \B^*_a[-1] & a>b. \\
\end{cases}\]
\end{theorem}

\begin{proof}
The result follows from an identical proof  to that of Theorem \ref{thm:BernsteinOps1} where symmetrizers and antisymmetrizers are exchanged and the dual statements of the isomorphisms are applied instead. 
\end{proof}

\subsection{Proof of Theorem \ref{thm:CatBernstein2}}\hfill
\medskip

In this section we present the proofs for the remaining two relations for the categorical Bernstein operators. We consider the bi-complexes for the various tensor products of $\B_a$ and $\B_b^*$. This case differs vastly from Theorem \ref{thm:BernsteinOps1} in that $\B_a \in \K^-(\H)$ whereas $\B_b^* \in \K^+(\H)$. Thus, when considering the corresponding bi-complexes we will see that they are not $\K(\H)$-locally finite. In particular, we will show that $\B_{a+1}\otimes \B_{b+1}^*$ and $\B_{b}\tilde{\otimes} \B_a^*$ arise as the total complex and the completed total complex of the same bi-complex. Unfortunately, the bi-complex is not homologically locally-finite and thus Proposition \ref{prop:locallyfinite} does not hold. By setting certain additional finiteness conditions, however, the tensor product and the completed tensor product do agree and the desired result is obtained.  

We begin by considering the two complexes $\Q^{(m)^t}\B_b^*$ and $\Q^{(n)}\B_a$ for integers $m,n \in \N$.

\begin{proposition} \label{lem:q-reduction}
Given any $k \in \N$, the chain complex $\left\lbrace \bigoplus_{x=A}^B \Q^{(k-x)^t} \Q^{(x)}[-x],
\mathsf{d}=\hackcenter{
\begin{tikzpicture} [xscale=-1,scale=.35]
\draw[line width=0.7mm, decoration={markings,mark=at position .03 with {\arrow[line width=.3mm]{<}}},
    postaction={decorate},shorten >=0.6pt] (4.5,3.5) to (4.5,9.5);
\draw[line width=0.7mm,decoration={markings,mark=at position .03 with {\arrow[line width=.3mm]{<}}},
    postaction={decorate},shorten >=0.6pt] (1.5,3.5) to (1.5, 9.5);
\draw[directed= .6]  (1.8,7.5) to (4.2, 5.5);
\draw[fill =gray!60] (5.8,7.5) rectangle (3.1,8.5);
\draw[fill =white] (2.5,7.5) rectangle (0.4,8.5);
\draw[fill =gray!60] (5.6,4.5) rectangle (3.3,5.5);
\draw[fill =white] (2.5,4.5) rectangle (0.4,5.5);
\node at (1.5,5) [scale=.75]{$x$};
\node at (4.5,5) [scale=.75] {$k-x$};
\node at (1.5,8)  [scale=.75]{$x+1$};
\node at (4.5,8) [scale=.5]{$k-x-1$};
\end{tikzpicture}} \right\rbrace
$ is homotopy equivalent to one of the following single term complexes, 
\[\bigoplus_{x=A}^B \Q^{(k-x)^t} \Q^{(x)}[-x]\simeq
\begin{cases}
\Q^{(A,1^{k-A})}[-A] \oplus \Q^{(B+1, 1^{k-B-1})}[-B] &, 0<A \leq B<k\\
\Q^{(A,1^{k-A})}[-A] &, 0<A<B=k\\
\Q^{(B+1, 1^{k-B-1})}[-B] &, 0=A<B<k\\
0 &, 0=A<B=k.
\end{cases}\]
\end{proposition}

\begin{proof}
Applying Proposition \ref{prop:PP*merge} to each chain group will force the differentials of the resulting chain complex to be either isomorphisms or zero. In particular, the composition of the map $\d:\Q^{(x)^t}\Q^{(y-1)}\to \Q^{(x-1)^t}\Q^{(y)}$ defined as above with the projection and injection morphisms from Proposition \ref{prop:PP*merge} satisfy
\[\rho_{y,x-1}^{x-1}\circ \d \circ \iota_{x-1}^{y-1,x}=\1_{\Q^{(y,1^x-1)}}, \rho_{y,x-1}^{x-2}\circ \d \circ \iota_{x-1}^{y-1,x}=0, \rho_{y,x-1}^{x-1}\circ \d \circ \iota_{x}^{y-1,x}=0,\text{ and }\rho_{y,x-1}^{x-2}\circ \d \circ \iota_{x}^{y-1,x}=0.\]

Since the complex is finite we can use Lemma \ref{lem:gaussian-elimination} across each homological degree so that only terms on the endpoints of the bi-complex are left. When $A=0$ or $B=k$ the terms at $x=A$ and $x=B$, respectively, do not decompose further. Hence, these terms cancel completely when applying Gaussian elimination. Consequently, the complex is contractible unless $A>0$ or $B<K$ in which case only the point with $x=A$ or $x=B$ remain. 
\end{proof}

\noindent \textbf{Notation:} For any set $J \subset \Z$ let $\mathsf{1}^{z}_{J}:= \begin{cases}1 &; z\in J\\0&; $ else.$\end{cases}$ denote its characteristic function and $[0,N]$ denote the set of all integers $0\leq z \leq N$. 

\begin{lemma} \label{lem:QC-reductions}
For any $n,m \in \N$ let $\mathsf{D}_m:=\hackcenter{
\begin{tikzpicture} [scale=0.6, yscale=-1, xscale=-.7]
 \draw[line width=0.7mm,decoration={markings,mark=at position .06 with {\arrow[line width=.3mm]{<}}},
    postaction={decorate},shorten >=0.6pt] (1,-1) -- (1,1);
  \draw[line width=.7mm] (-1,-1) -- (-1,0);
    \draw[line width=.7mm,decoration={markings,mark=at position 0.6 with {\arrow[line width=.3mm]{>}}},
    postaction={decorate},shorten >=0.6pt]  (-1,.3) .. controls ++(0,.5) and ++(0,.5) .. (.5,.3);
        \draw[fill =gray!60] (-1.5,-.3) rectangle (-.2,.3);
        \draw[fill =gray!60] (.2,-.3) rectangle (1.5,.3);
  \node at (.85,-0) [scale=.5]{$-b+1$};
  \node at (-.85,0) [scale=.5]{$m+1$};
  \end{tikzpicture}}$ and $\mathsf{D}_n:=\hackcenter{
\begin{tikzpicture} [scale=0.6, yscale=1, xscale=-.7]
  \draw[->, line width=.7mm] (1,-1) -- (1,0);
 \draw[line width=0.7mm,decoration={markings,mark=at position 1 with {\arrow[line width=.3mm]{>}}},
    postaction={decorate},shorten >=0.6pt] (1,-1) -- (1,1);
    postaction={decorate},shorten >=0.6pt] (-1,-1) -- (-1,1);
  \draw[line width=.7mm] (-1,-1) -- (-1,0);
    \draw[line width=0.7mm,decoration={markings,mark=at position .6 with {\arrow[line width=.3mm]{>}}},
    postaction={decorate},shorten >=0.6pt]  (.5,.3) .. controls ++(0,.5) and ++(0,.5) .. (-1,.3);
        \draw[fill =white] (-1.5,-.3) rectangle (-.2,.3);
        \draw[fill =white] (.2,-.3) rectangle (1.5,.3);
  \node at (.85,-0) [scale=.5]{$a+1$};
  \node at (-.85,0) [scale=.5]{$n+1$};
  \end{tikzpicture}}$. Then, the following chain homotopies hold in $\K(\H)$:
\begin{align*}
&\left\lbrace \Q^{(m)^t}\B_b^*, \1_m \otimes \mathsf{d^*_b} \right\rbrace \simeq \emph{\text{Cone}}\left(
 \mathsf{1}^m_{[0,-b]}\P^{(-b-m)^t}[b-1] \xrightarrow{\mathsf{D}_m} \left\lbrace\bigoplus_{R\geq(-b+1,0)} \P^{(R)^t}\Q^{(R+b,1^m)}[-R],\mathsf{d}=\hackcenter{\begin{tikzpicture} [yscale=-1.2,scale=.25]
\draw[line width=0.7mm,decoration={markings,mark=at position 1 with {\arrow[line width=.3mm]{>}}},
    postaction={decorate},shorten >=0.6pt] (4.5,2) to (4.5,8.5);
\draw[line width=0.7mm,decoration={markings,mark=at position .03 with {\arrow[line width=.3mm]{<}}},
    postaction={decorate},shorten >=0.6pt] (-1.5,2) to (-1.5, 8.5);
\draw[line width=0.7mm,decoration={markings,mark=at position 1 with {\arrow[line width=.3mm]{>}}},
    postaction={decorate},shorten >=0.6pt] (1.5,2) to (1.5, 8.5);
\draw[directed=.8] (4.2,5.5)..controls ++(0,1) and ++(0,1)..(-.8, 5.5);
\draw[fill =gray!60] (-2.7,4.5) rectangle (-0.4,5.5);
\draw[fill =gray!20] (0.4,6.7) rectangle (5.6,7.7);
\draw[fill =white] (6,4.5) rectangle (3,5.5);
\draw[fill =gray!20] (0.4,4) rectangle (5.6,3);
\end{tikzpicture}}
\right\rbrace \right)
\\
&\left\lbrace \Q^{(n)}\B_{a}, \1_n \otimes \mathsf{d}_{a}
\right\rbrace \simeq \emph{\text{Cone}}\left(\left\lbrace  \bigoplus_{R\geq (a+1,0)}\P^{(R)}\Q^{(n+1,1^{R-a-1})}[R-a-1], \mathsf{d}=\hackcenter{\begin{tikzpicture} [yscale=1.2,scale=.25]
\draw[line width=0.7mm,decoration={markings,mark=at position .03 with {\arrow[line width=.3mm]{<}}},
    postaction={decorate},shorten >=0.6pt] (4.5,2) to (4.5,8.5);
\draw[line width=0.7mm,decoration={markings,mark=at position 1 with {\arrow[line width=.3mm]{>}}},
    postaction={decorate},shorten >=0.6pt] (-1.5,2) to (-1.5, 8.5);
\draw[line width=0.7mm,decoration={markings,mark=at position .03 with {\arrow[line width=.3mm]{<}}},
    postaction={decorate},shorten >=0.6pt] (1.5,2) to (1.5, 8.5);
\draw[directed=.8] (-.8,5.5)..controls ++(0,1) and ++(0,1)..(4.2, 5.5);
\draw[fill =white] (-2.7,4.5) rectangle (-0.4,5.5);
\draw[fill =gray!20] (0.4,6.7) rectangle (5.6,7.7);
\draw[fill =gray!60] (5.6,4.5) rectangle (3.3,5.5);
\draw[fill =gray!20] (0.4,4) rectangle (5.6,3);
\end{tikzpicture}}
\right\rbrace \xrightarrow{\mathsf{D}_n} \mathsf{1}_{[0,a]}^{n} \P^{(a-n)}[0]\right).
\end{align*}
\end{lemma}

\begin{proof}
Since both claims are proved similarly we provide details only for the first homotopy equivalence. Given $m \in \N$ we apply Proposition (\ref{prop:PP*merge}) to $\Q^{(m)^t}\B_b$. At the level of chain groups we have the isomorphisms
\[
\Q^{(m)^t}\B_b^* \simeq \left\lbrace \bigoplus_{s=0}^{(m,w)} \P^{(w-s)^t}\Q^{(m-s)^t}\Q^{(v)} [-w],\1\otimes\d_b^*\right\rbrace_{v-w=b}
.\]
Letting $R:=w-s$ we have
\begin{align*}
\Q^{(m)^t}\B_b^* &\simeq \left\lbrace \bigoplus_{R\geq (0,w-m)}^{w} \P^{(R)^t}\Q^{(m+R-w)^t}\Q^{(b+w)} [-w], 1\otimes \d_b^* \right\rbrace_{w \geq (0,-b)}.
\end{align*}
In particular, the differentials are now given by:
\begin{align*}
\d_w&:  \P^{(R)^t}\Q^{(m+R-w)^t}\Q^{(b+w)} [-w]\to  \P^{(R)^t}\Q^{(m+R-w-1)^t}\Q^{(b+w+1)} [-w-1]\\ 
\d^R&:  \P^{(R)^t}\Q^{(m+R-w)^t}\Q^{(b+w)} [-w] \to  \P^{(R+1)^t}\Q^{(m+R-w)^t}\Q^{(b+w+1)} [-w-1]
\end{align*}
\begin{equation} \label{eq:ARdiff}
\mathsf{d}_w=\;-\;\hackcenter{
\begin{tikzpicture} [xscale = -1, scale=.35]
\draw[line width=.7mm,decoration={markings,mark=at position 1 with {\arrow[line width=.3mm]{>}}},
    postaction={decorate},shorten >=0.6pt] (7,3.5) to (7, 8);
\draw[line width=.7mm,decoration={markings,mark=at position 0.04 with {\arrow[line width=.3mm]{<}}},
    postaction={decorate},shorten >=0.6pt] (4.5,3.5) to (4.5,8);
\draw[line width=.7mm,decoration={markings,mark=at position 0.04 with {\arrow[line width=.3mm]{<}}},
    postaction={decorate},shorten >=0.6pt] (1.5,3.5) to (1.5, 8);
\draw[directed= .6]  (2,6.5) to (3.8, 5);
\draw[fill =white] (2.8,6.5) rectangle (0.2,7.5);
\draw[fill =gray!60] (5.8,4) rectangle (3.1,5);
\draw[fill =gray!60] (8.1,5.5) rectangle (5.8,6.5);
\node at (4.5,4.5) [scale=.5] {$m+R-w$};
\node at (1.5,7)  [scale=.5]{$b+w+1$};
\node at (7,6) [scale=.75] {$R$};
\end{tikzpicture}} 
\qquad\qquad\qquad
\mathsf{d}^R=\hackcenter{
\begin{tikzpicture} [xscale = -1, scale=.35]
\draw[line width=.7mm,decoration={markings,mark=at position 1 with {\arrow[line width=.3mm]{>}}},
    postaction={decorate},shorten >=0.6pt] (7.5,5) to (7.5, 9.5);
\draw[line width=.7mm,decoration={markings,mark=at position 0.04 with {\arrow[line width=.3mm]{<}}},
    postaction={decorate},shorten >=0.6pt] (4.5,5) to (4.5,9.5);
\draw[line width=.7mm,decoration={markings,mark=at position 0.04 with {\arrow[line width=.3mm]{<}}},
    postaction={decorate},shorten >=0.6pt] (1.5,5) to (1.5, 9.5);
\draw[ directed=.6]  (1.8,7.5) .. controls ++(0,-1.5) and ++(0,-1.5) .. (7.2,7.5);
\draw[fill =gray!60] (5.8,7.5) rectangle (3.2,8.5);
\draw[fill =white] (2.8,7.5) rectangle (0.2,8.5);
\draw[fill =gray!60] (8.6,7.5) rectangle (6.3,8.5);
\node at (1.5,8)  [scale=.5]{$b+w+1$};
\node at (4.5,8) [scale=.5]{$m+R-w$};
\node at (7.5,8) [scale=.75] {$R+1$};
\end{tikzpicture}}. 
\end{equation}
Thus for each fixed $R \geq (-m-b,0)$ we may identify subcomplexes $\lbrace \mathcal{A}_{R}(m),\mathsf{d}_w \rbrace \in \K^b(\H)$ defined by:
\[\mathcal{A}_{R}(m) := \bigoplus_{w\geq (-b,R)}^{R+m} \P^{(R)^t} \Q^{(m+R-w)^t} \Q^{(b+w)}[-w].\]
Since $\d^R: \mathcal{A}_R(m)[1] \to \mathcal{A}_{R+1}(m)$ is a chain map such that $(\d^R)^2=0$, then $\Q^{(m)}\B_b^*$ is the total complex of bi-complex:
\begin{align}\label{eq:Tot}
\Q^{(m)^t}\B_b^* \simeq \text{Tot}^{\oplus} \lbrace \mathcal{A}_{R}(m),\mathsf{d}_R, \mathsf{d}^R\rbrace_{R\geq(0-m-b)}
\end{align}
Since each $\mathcal{A}_R(m)$ is finite we can apply Proposition \ref{lem:q-reduction} to each row of the bi-complex. Specifically, $\Q^{(m+R-w)^t}\Q^{(b+w)}\simeq \Q^{(b+w,1^{m+R-w})}\oplus \Q^{(b+w+1,1^{m+R-w-1})}$, then if we take $A=$max$(-b,R)$ and $B=k=R+m$ in Proposition \ref{lem:q-reduction}, each $\mathcal{A}_R(m)$ will be contractible whenever $-b\geq R >-b-m$ and isomorphic to a single term with $w=R$ whenever $R>-b$ and $w=-b$ when $R=-b-m$. Because $R=w-s$ these terms correspond to $s=0$ for $R>-b$ and $s=m$ when $R=-b-m$. Consequently, each subcomplex $\mathcal{A}_R(m)$ is homotopy equivalent to one of the following:
\begin{equation}\label{eq:Ar-reduced}
\mathcal{A}_{R}(m) \simeq \begin{cases}
\P^{(R)^t}\Q^{(R+b,1^m)}[-R] &, R>-b \;(w=R) \\\P^{(-b-m)^t}[b] &, R=-b-m \;(w=-b)\\
0 &, -b\geq R>-b-m.
\end{cases}
\end{equation}
Since $\mathcal{A}_R(m)$ is bounded for all $R,m$ then the bi-complex is homologically locally finite and by Proposition \ref{prop:locallyfinite}:
\[\Q^{(m)^t}\B_b^* \simeq \text{Tot}^{\oplus} \lbrace \mathcal{A}_{R}(m),\mathsf{d}_R, \mathsf{d}^R\rbrace_{R\geq(0,-m-b)}\simeq\text{Tot}^{\Pi} \lbrace \mathcal{A}_{R}(m),\mathsf{d}_R, \mathsf{d}^R\rbrace_{R\geq(0,-m-b)}.\]
Since $R \geq (0,-m-b)$, then the indexing set of the homological degrees is bounded above, so by Proposition \ref{prop:SimultSimp} we can simultaneously simplify each row in $\text{Tot}^{\Pi} \lbrace \mathcal{A}_{R}(m),\mathsf{d}_R, \mathsf{d}^R\rbrace_{R\geq(0,-m-b)}$. 

The simultaneous reductions will affect the affect the arrows in the following manner. If $R=-b-m\geq0$ then $\mathcal{A}_{-b-m}(m)$ will consist of a single term $\P^{(-b-m)^t}[b]$, denoted by the top most node in Figure \ref{fig:QCb}, whose new differential is obtained by following the zigzag down $\mathsf{d}^R$ (vertical blue arrows) and back along $\mathsf{d}_R$ (horizontal red arrows) until reaching the circled node in $\mathcal{A}_{R+1}(m)$. Following the diagram in Figure \ref{fig:QCb}, we can see that since any new morphism must point downward and to the left, this zig-zag will eventually terminate and thus no other arrows are created.
\begin{figure}[ht]
\begin{tikzpicture}[scale=.8, every node/.style={scale=.7}]
\begin{scope}[shift={(0,5)}, yscale=-1]
\node at (-4.3,0) {$\vdots$};
\node at (-4.3,-1) {$\mathcal{A}_{-b-m}(m)$};
%
\begin{scope}[shift={(4,0)}]
\node at (-3,1) {$\bullet$};
%
\node at (-2,0) {$\bullet$};
\node at (-2,1) {$\bullet$};
%
\node at (-1,-1) {$\bullet$};
\node at (-1,0) {$\bullet$};
\node at (-1,1) {$\bullet$};
\draw (-1,-1)[black!20!green] circle (2mm);
\draw[red,->] (-2.8,1)--(-2.2,1);
%
\draw[red,->] (-1.8,1)--(-1.2,1);
\draw[red,->] (-1.8,0)--(-1.2,0);
%
\draw [blue,->] (-2,.2)--(-2,.8);
%
\draw [blue,->] (-1,.2)--(-1,.8);
\draw [blue,->] (-1,-.8)--(-1,-.2);
\end{scope}
\end{scope}
\begin{scope}[shift={(0,2)}, yscale=-1]
\node at (-4.3,-1) {$\vdots$};
\node at (-4.3,0) {$\mathcal{A}_{R-1}(m)$};
\node at (-4.3,1) {$\mathcal{A}_R(m)$};
%
\node at (-3,1) {$\bullet$};
%
\node at (-2,0) {$\bullet$};
\node at (-2,1) {$\bullet$};
%
\node at (-1,-1) {$\bullet$};
\node at (-1,0) {$\bullet$};
\node at (-1,1) {$\bullet$};
%
\node at (0,-1) {$\cdots$};
\node at (0,0) {$\cdots$};
\node at (0,1) {$\cdots$};
\node at (3,-1.5) {$\vdots$};
\node at (3,-1) {$\bullet$};
\node at (3,0) {$\bullet$};
\node at (3,1) {$\bullet$};
%
\node at (2,-1.5) {$\vdots$};
\node at (2,-1) {$\bullet$};
\node at (2,0) {$\bullet$};
\node at (2,1) {$\bullet$};
%
\node at (1,-1.5) {$\vdots$};
\node at (1,-1) {$\bullet$};
\node at (1,0) {$\bullet$};
\node at (1,1) {$\bullet$};
\draw[red,->] (-2.8,1)--(-2.2,1);
%
\draw[red,->] (-1.8,1)--(-1.2,1);
\draw[red,->] (-1.8,0)--(-1.2,0);
%
\draw[red,<-] (2.8,1)--(2.2,1);
\draw[red,<-] (2.8,0)--(2.2,0);
\draw[red,<-] (2.8,-1)--(2.2,-1);
\draw[red,<-] (1.8,1)--(1.2,1);
\draw[red,<-] (1.8,0)--(1.2,0);
\draw[red,<-] (1.8,-1)--(1.2,-1);
\begin{scope}[yscale=1]
%
\draw [blue,->] (-2,.2)--(-2,.8);
%
\draw [blue,->] (-1,.2)--(-1,.8);
\draw [blue,->] (-1,-.8)--(-1,-.2);
\draw [blue,->] (1,.2)--(1,.8);
\draw [blue,->] (1,-.8)--(1,-.2);
\draw [blue,->] (2,.2)--(2,.8);
\draw [blue,->] (2,-.8)--(2,-.2);
\draw [blue,->] (3,.2)--(3,.8);
\draw [blue,->] (3,-.8)--(3,-.2);
\end{scope}
\draw [thick,decorate,decoration={brace,amplitude=5pt}] (3.5,1.3) -- (3.5,3.5);
\node at (4.7,2.4){$R>-b>0$};
\draw [thick,decorate,decoration={brace,amplitude=5pt}] (3.5,-4) -- (3.5,.7);
\node at (4.7,-1.6){$0 \leq R<-b$};
\node at (4.7,1){$R=\text{max}(-b,0)$};
\end{scope}
\begin{scope}[yscale=-1]
\node at (-4.3,0) {$\mathcal{A}_{R+1}(m)$};
\node at (-4.3,1) {$\vdots$};
%
\draw (-3,0)[black!20!green] circle (2mm);
\draw (-3,1)[black!20!green] circle (2mm);
\node at (-3,-1) {$\bullet$};
\node at (-3,0) {$\bullet$};
\node at (-3,1) {$\bullet$};
\node at (-3,1.5) {$\vdots$};
%
\node at (-2,-1) {$\bullet$};
\node at (-2,0) {$\bullet$};
\node at (-2,1) {$\bullet$};
\node at (-2,1.5) {$\vdots$};
%
\node at (-1,-1) {$\bullet$};
\node at (-1,0) {$\bullet$};
\node at (-1,1) {$\bullet$};
\node at (-1,1.5) {$\vdots$};
\node at (0,-1) {$\cdots$};
\node at (0,0) {$\cdots$};
\node at (0,1) {$\cdots$};
%
\node at (3,-1) {$\bullet$};
\node at (3,0) {$\bullet$};
\node at (3,1) {$\bullet$};
\node at (3,1.5) {$\vdots$};
%
\node at (2,-1) {$\bullet$};
\node at (2,0) {$\bullet$};
\node at (2,1) {$\bullet$};
\node at (2,1.5) {$\vdots$};
%
\node at (1,-1) {$\bullet$};
\node at (1,0) {$\bullet$};
\node at (1,1) {$\bullet$};
\node at (1,1.5) {$\vdots$};
\draw[red,->] (-2.8,1)--(-2.2,1);
\draw[red,->] (-2.8,0)--(-2.2,0);
\draw[red,->] (-2.8,-1)--(-2.2,-1);
\draw[red,->] (-1.8,1)--(-1.2,1);
\draw[red,->] (-1.8,0)--(-1.2,0);
\draw[red,->] (-1.8,-1)--(-1.2,-1);
\draw[red,<-] (2.8,1)--(2.2,1);
\draw[red,<-] (2.8,0)--(2.2,0);
\draw[red,<-] (2.8,-1)--(2.2,-1);
\draw[red,<-] (1.8,1)--(1.2,1);
\draw[red,<-] (1.8,0)--(1.2,0);
\draw[red,<-] (1.8,-1)--(1.2,-1);
\draw [blue,->] (-3,.2)--(-3,.8);
\draw [blue,->] (-3,-.8)--(-3,-.2);
\draw [blue,->] (-2,.2)--(-2,.8);
\draw [blue,->] (-2,-.8)--(-2,-.2);
\draw [blue,->] (-1,.2)--(-1,.8);
\draw [blue,->] (-1,-.8)--(-1,-.2);
\draw [blue,->] (1,.2)--(1,.8);
\draw [blue,->] (1,-.8)--(1,-.2);
\draw [blue,->] (2,.2)--(2,.8);
\draw [blue,->] (2,-.8)--(2,-.2);
\draw [blue,->] (3,.2)--(3,.8);
\draw [blue,->] (3,-.8)--(3,-.2);
\end{scope}
\end{tikzpicture}
\caption{
}\label{fig:QCb} 
\end{figure}
Consequently, the differential $\d^R$ is not modified but simply restricted to the terms that do not cancel under the Gaussian elimination. When $0\leq R<-b$ we also have a new map from $\mathcal{A}_{-b-m}(m) \to \mathcal{A}_{R+1}(m)$ corresponding to $\mathsf{D}: \P^{(-m-b)^t}[b] \to \P^{(-b+1)^t}\Q^{(m+1)^t}[b-1]$ and given by
$
\mathsf{D}=\hackcenter{
\begin{tikzpicture} [scale=0.6, yscale=-1, xscale=-.7]
 \draw[line width=0.7mm,decoration={markings,mark=at position .06 with {\arrow[line width=.3mm]{<}}},
    postaction={decorate},shorten >=0.6pt] (1,-1) -- (1,1);
  \draw[line width=.7mm] (-1,-1) -- (-1,0);
    \draw[line width=.7mm,decoration={markings,mark=at position 0.6 with {\arrow[line width=.3mm]{>}}},
    postaction={decorate},shorten >=0.6pt]  (-1,.3) .. controls ++(0,.5) and ++(0,.5) .. (.5,.3);
        \draw[fill =gray!60] (-1.5,-.3) rectangle (-.2,.3);
        \draw[fill =gray!60] (.2,-.3) rectangle (1.5,.3);
  \node at (.85,-0) [scale=.5]{$-b+1$};
  \node at (-.85,0) [scale=.5]{$m+1$};
  \end{tikzpicture}}$.
Therefore, we have the equivalences
\begin{align*}
\Q^{(m)^t}\B_b^* &= \text{Tot}^{\oplus} \lbrace \mathcal{A}_{R}(m), \mathsf{d}_w, \mathsf{d}^R \rbrace \\
&\simeq \text{Cone}\left(
 \mathsf{1}^m_{[0,-b]}\P^{(-b-m)^t}[b-1] \xrightarrow{\mathsf{D}} \left\lbrace\bigoplus_{R \geq (-b+1,0)} \P^{(R)^t}\Q^{(R+b,1^m)}[-R], \hackcenter{\begin{tikzpicture} [yscale=-1.2,scale=.25]
\draw[line width=0.7mm,decoration={markings,mark=at position 1 with {\arrow[line width=.3mm]{>}}},
    postaction={decorate},shorten >=0.6pt] (4.5,2) to (4.5,8.5);
\draw[line width=0.7mm,decoration={markings,mark=at position .03 with {\arrow[line width=.3mm]{<}}},
    postaction={decorate},shorten >=0.6pt] (-1.5,2) to (-1.5, 8.5);
\draw[line width=0.7mm,decoration={markings,mark=at position 1 with {\arrow[line width=.3mm]{>}}},
    postaction={decorate},shorten >=0.6pt] (1.5,2) to (1.5, 8.5);
\draw[directed=.8] (4.2,5.5)..controls ++(0,1) and ++(0,1)..(-.8, 5.5);
\draw[fill =gray!60] (-2.7,4.5) rectangle (-0.4,5.5);
\draw[fill =gray!20] (0.4,6.7) rectangle (5.6,7.7);
\draw[fill =white] (6,4.5) rectangle (3,5.5);
\draw[fill =gray!20] (0.4,4) rectangle (5.6,3);
\node at (-1.5,5) [scale=.5]{$R+1$};
\node at (4.5,5) [scale=.5]{$R+b+1$};
\end{tikzpicture}} \right\rbrace\right).
\end{align*}
\end{proof}

\begin{theorem}\label{thm:CC*directsum}
For any $a,b \in \Z$ the homotopy equivalence
$\B_{a+1} \otimes \B_{b+1}^* \simeq \emph{Tot}^{\oplus} \left\lbrace \mathcal{A}^y, \mathsf{d}_y,\mathsf{d}^y \right\rbrace_{y\geq(0,-a-1)} 
$ holds in $\K(\H)$ where the chain complex $\left\lbrace\mathcal{A}^y,\mathsf{d}_y \right\rbrace \simeq $ \emph{Cone}$(\mathsf{D_y})$ with
\begin{equation*} 
 \mathsf{1}^y_{[0,-b-1]}\P^{(y+a+1)}\P^{(-b-y-1)^t}[b+y+1] \xrightarrow{\mathsf{D}_y} \left\lbrace\bigoplus_{R\geq(-b,0)} \P^{(y+a+1)}\P^{(R)^t}\Q^{(R+b,1^y)}[y-R],\mathsf{d}_y
\right\rbrace \end{equation*}
with   $\mathsf{D}_y=\hackcenter
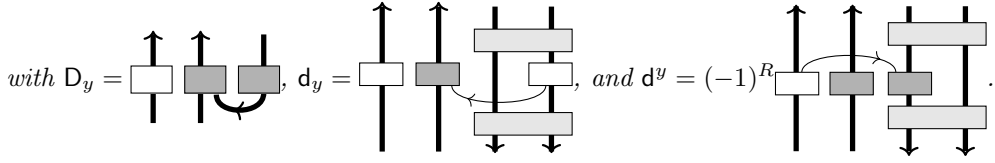
{
\begin{tikzpicture} [scale=0.6, yscale=-1, xscale=-.7]
 \draw[line width=0.7mm,decoration={markings,mark=at position .06 with {\arrow[line width=.3mm]{<}}},
    postaction={decorate},shorten >=0.6pt] (1,-1) -- (1,1);
 \draw[line width=0.7mm,decoration={markings,mark=at position .06 with {\arrow[line width=.3mm]{<}}},
    postaction={decorate},shorten >=0.6pt] (2.5,-1) -- (2.5,1);
  \draw[line width=.7mm] (-1,-1) -- (-1,0);
    \draw[line width=.7mm,decoration={markings,mark=at position 0.6 with {\arrow[line width=.3mm]{>}}},
    postaction={decorate},shorten >=0.6pt]  (-1,.3) .. controls ++(0,.5) and ++(0,.5) .. (.5,.3);
      \draw[fill =white] (3.2,-.3) rectangle (1.9,.3);
        \draw[fill =gray!60] (-1.5,-.3) rectangle (-.2,.3);
        \draw[fill =gray!60] (.2,-.3) rectangle (1.5,.3);
  \end{tikzpicture}}$, 
  $\mathsf{d}_y= \hackcenter{\begin{tikzpicture} [yscale=-1.2,scale=.25]
\draw[line width=0.7mm,decoration={markings,mark=at position 1 with {\arrow[line width=.3mm]{>}}},
    postaction={decorate},shorten >=0.6pt] (4.5,2) to (4.5,8.5);
\draw[line width=0.7mm,decoration={markings,mark=at position .03 with {\arrow[line width=.3mm]{<}}},
    postaction={decorate},shorten >=0.6pt] (-1.5,2) to (-1.5, 8.5);
\draw[line width=0.7mm,decoration={markings,mark=at position .03 with {\arrow[line width=.3mm]{<}}},
    postaction={decorate},shorten >=0.6pt] (-4.5,2) to (-4.5, 8.5);
\draw[line width=0.7mm,decoration={markings,mark=at position 1 with {\arrow[line width=.3mm]{>}}},
    postaction={decorate},shorten >=0.6pt] (1.5,2) to (1.5, 8.5);
\draw[directed=.8] (4.2,5.5)..controls ++(0,1) and ++(0,1)..(-.8, 5.5);
\draw[fill =gray!60] (-2.7,4.5) rectangle (-0.4,5.5);
\draw[fill =white] (-5.7,4.5) rectangle (-3.4,5.5);
\draw[fill =gray!20] (0.4,6.7) rectangle (5.6,7.7);
\draw[fill =white] (5.6,4.5) rectangle (3.3,5.5);
\draw[fill =gray!20] (0.4,4) rectangle (5.6,3);
\end{tikzpicture}}$, and  $\mathsf{d}^y =(-1)^R \hackcenter{\begin{tikzpicture} [yscale=1.2,xscale=-1,scale=.25]
\draw[line width=0.7mm,decoration={markings,mark=at position 1 with {\arrow[line width=.3mm]{>}}},
    postaction={decorate},shorten >=0.6pt] (4.5,2) to (4.5,8.5);
\draw[line width=0.7mm,decoration={markings,mark=at position .03 with {\arrow[line width=.3mm]{<}}},
    postaction={decorate},shorten >=0.6pt] (-1.5,2) to (-1.5, 8.5);
\draw[line width=0.7mm,decoration={markings,mark=at position .03 with {\arrow[line width=.3mm]{<}}},
    postaction={decorate},shorten >=0.6pt] (-4.5,2) to (-4.5, 8.5);
\draw[line width=0.7mm,decoration={markings,mark=at position 1 with {\arrow[line width=.3mm]{>}}},
    postaction={decorate},shorten >=0.6pt] (1.5,2) to (1.5, 8.5);
\draw[directed=.8] (4.2,5.5)..controls ++(0,1) and ++(0,1)..(-.8, 5.5);
\draw[fill =gray!60] (-2.7,4.5) rectangle (-0.4,5.5);
\draw[fill =gray!60] (2.7,4.5) rectangle (0.4,5.5);
\draw[fill =gray!20] (-0.4,6.7) rectangle (-5.6,7.7);
\draw[fill =white] (5.6,4.5) rectangle (3.3,5.5);
\draw[fill =gray!20] (-0.4,4) rectangle (-5.6,3);
\end{tikzpicture}}
$.
\end{theorem}

\begin{proof}
Once again, we consider $\B_{a+1} \otimes \B_b^*$ as a bi-complex so that, 
\[
\B_{a+1}\otimes \B_{b+1}^*=\text{Tot}^{\oplus} \left\lbrace \P^{(y+a+1)} \Q^{(y)^t}\B_{b+1}^*, \1_y \otimes \mathsf{d}_{b+1}^*, (-1)^R\mathsf{d}_{a+1} \otimes \1_{b+1} \right\rbrace_{y\geq \text{max}(0,-a-1)}\]
By tensoring $\Q^{(y)^t}\B_{b+1}^*$ with $\P^{(y+a+1)}$ on the right we can apply Lemma \ref{lem:QC-reductions} to each row of the bi-complex and obtain: $\P^{(y+a+1)}\Q^{(y)^t}\B_{b+1}^* \simeq$ Cone$(\mathsf{D}_y)$, where $\mathsf{D}_y$ is the chain map:
\[\mathsf{1}^y_{[0,-b-1]}\P^{(y+a+1)}\P^{(-b-y-1)^t}[b+y] \xrightarrow{\mathsf{D}_y} \left\lbrace \bigoplus_{R\geq(-b,0)} \P^{(y+a+1)}\P^{(R)^t}\Q^{(R+b+1,1^y)}[y-R], \mathsf{d}_y=\1_{\P^{(y+a+1)}} \otimes \mathsf{d}\right\rbrace
\]
where $\mathsf{d}$ is defined as in Lemma \ref{lem:QC-reductions}. 
\medskip

As always, we must address the effect of the simultaneous homotopies on $\d^y$. We claim that $\mathsf{d}^y$ is given by $(-1)^R\mathsf{d}_{a+1}\otimes 1_{b+1}$ restricted to the terms in the reduced complexes. This is equivalent to saying that the reductions do not interfere with the existing differentials. Indeed, if such a map were generated it would necessarily originate in a term that is not canceled under Lemma \ref{lem:QC-reductions}. Thus, it must have the form $\P^{(y+a+1)}\P^{(R)^t}\Q^{(R+b+1,1^y)}$ for $R>$max$(0,-b-1)$. We will show that the set of all its images under $\d_{a+1}\otimes \1_{b+1}$ have no preimages under the isomorphisms in $\1_y\otimes \d_{b+1}$. Since any new morphism would necessarily have to trace back through these isomorphisms, the previous statement would prove that zigzags originating in $\P^{(y+a+1)}\P^{(R)^t}\Q^{(R+b+1,1^y)}$ cannot exists. 

\medskip

Since $\P^{(y+a+1)}\P^{(R)^t}\Q^{(R+b+1,1^y)} \hookrightarrow \P^{(y+a+1)}\P^{(R)^t}\Q^{(y)^t}\Q^{(R+b+1)}$, consider the compositions:
\[
\begin{tikzcd}
\P^{(y+a+1)}\P^{(R)^t}\Q^{(y)^t}\Q^{(R+b+1)} \arrow[d,hook,"\iota"]&\P^{(y+a)}\P^{(R-s)^t}\Q^{(y-1-s)^t}\Q^{(R+b+1)}\\
\P^{(y+a+1)}\Q^{(y)^t}\P^{(R)^t}\Q^{(R+b+1)} \arrow[r,"\mathsf{d}_{a+1}\otimes \1_{b+1}"]& \P^{(y+a)}\Q^{(y-1)^t}\P^{(R)^t}\Q^{(R+b+1)}\arrow[u,two heads, "\rho"]
\end{tikzcd}
\]
Given diagrammatically by, 
\[\hackcenter{\begin{tikzpicture} [scale=0.5,yscale=1, every node/.style={scale=0.5}]
\draw [line width=.7mm,decoration={markings,mark=at position 0.02 with {\arrow[line width=.3mm]{<}}},
    postaction={decorate},shorten >=0.6pt] (-3.5,3.3) to (-3.5,-7.3);
\draw[line width=.7mm,decoration={markings,mark=at position 0.5 with {\arrow[line width=.3mm]{>}}},
    postaction={decorate},shorten >=0.6pt] (-1,-7.3)to (-1,-6.2).. controls ++(0,.5) and ++(0,-1.5).. (1.5,-3.8) to (1.5,-.2) .. controls ++(0,1.5) and ++(0,-.5) .. (-1,2.2) to (-1,3.3);
\begin{scope}[shift={(-.75,0)}]
\draw[line width=.7mm,decoration={markings,mark=at position 0.5 with {\arrow[line width=.3mm]{>}}},
    postaction={decorate},shorten >=0.6pt] (-1,-7.3)to (-1,-6.2).. controls ++(0,.5) and ++(0,-1.5).. (1.5,-3.8) to (1.5,-.2);
\end{scope}
\begin{scope}[xscale=-1,yscale=-1, shift={(-0.75,4)}]
\draw[line width=.7mm,decoration={markings,mark=at position 0.5 with {\arrow[line width=.3mm]{>}}},
    postaction={decorate},shorten >=0.6pt](1.5,-3.8) to (1.5,-.2) .. controls ++(0,1.5) and ++(0,-.5) .. (-1,2.2) to (-1,3.3);
\end{scope}
\begin{scope}[xscale=-1,yscale=-1, shift={(0,4)}]
\draw[line width=.7mm,decoration={markings,mark=at position 0.5 with {\arrow[line width=.3mm]{>}}},
    postaction={decorate},shorten >=0.6pt] (-1,-7.3)to (-1,-6.2).. controls ++(0,.5) and ++(0,-1.5).. (1.5,-3.8) to (1.5,-.2) .. controls ++(0,1.5) and ++(0,-.5) .. (-1,2.2) to (-1,3.3);
\end{scope}
\draw[line width=.7mm,decoration={markings,mark=at position 0.55 with {\arrow[line width=.3mm]{>}}},
    postaction={decorate},shorten >=0.6pt](.75,-1.3) to (.75,-.2).. controls ++(0,1) and ++(0,1) .. (-.75,-.2) to (-.75,-1.3) ;
\begin{scope}[shift={(-2.5,-3)}]
\draw[ directed=.5] (-.75,-.2).. controls ++(0,1) and ++(0,1) .. (.75,-.2)  ;
\end{scope}
\draw [line width=.7mm,decoration={markings,mark=at position 1 with {\arrow[line width=.3mm]{>}}},
    postaction={decorate},shorten >=0.6pt] (3.5,3.3) to (3.5,-7.3);
   \draw[fill =gray!60] (-2.2,2.2) rectangle (-0.2,2.8);
  \draw[fill =gray!60] (-2.2,-.8) rectangle (-.2,-.2);
  \draw[fill =gray!60] (.2,2.2) rectangle (2.2,2.8);
  \draw[fill =gray!60] (.2,-.8) rectangle (2.2,-.2);
  \draw[fill =white] (-4.6,2.2) rectangle (-2.6,2.8);
   \draw[fill =white] (4.6,2.2) rectangle (2.6,2.8);
  \begin{scope}[shift={(0,-6)}]
  \draw[fill =gray!60] (-2.2,2.2) rectangle (-0.2,2.8);
  \draw[fill =gray!60] (-2.2,-.8) rectangle (-.2,-.2);
  \draw[fill =gray!60] (.2,2.2) rectangle (2.2,2.8);
  \draw[fill =gray!60] (.2,-.8) rectangle (2.2,-.2);
  \draw[fill =white] (-4.6,2.2) rectangle (-2.6,2.8);
  \draw[fill =white] (-4.6,-.8) rectangle (-2.6,-.2);
  \draw[fill =white] (4.6,-.8) rectangle (2.6,-.2);
  \end{scope}
  \node at (1.2,-.5) {$R$};
  \node at (1.2,2.5) {$y-1-s$};
  \node at (-1.2,-.5) {$y-1$};
  \node at (-1.2,2.5) {$R-s$};
  \node at (3.7,2.5) {$R+b+1$};
  \node at (-3.7,2.5) {$y+a$};
  \begin{scope}[shift={(0,-6)}]
  \node at (1.2,-.5) {$y$};
  \node at (1.2,2.5) {$R$};
  \node at (-1.2,-.5) {$R$};
  \node at (-1.2,2.5) {$y$};
  \node at (3.7,-.5) {$R+b+1$};
   \node at (-3.7,-.5) {$y+a+1$};
  \node at (-3.7,2.5) {$y+a+1$};
  \end{scope}
\end{tikzpicture}}
\;\;=\;\;
\hackcenter{\begin{tikzpicture} [scale=0.5,yscale=1, every node/.style={scale=0.5}]
\draw [line width=.7mm,decoration={markings,mark=at position 0.03 with {\arrow[line width=.3mm]{<}}},
    postaction={decorate},shorten >=0.6pt] (-3.5,3.3) to (-3.5,-3.9);
\draw[line width=.7mm,decoration={markings,mark=at position 0.5 with {\arrow[line width=.3mm]{>}}},
    postaction={decorate},shorten >=0.6pt] (-1,-3.9)to (-1,-2.8).. controls ++(0,.5) and ++(0,-1.5).. (1.5,-.4) to (1.5,-.2) .. controls ++(0,1.5) and ++(0,-.5) .. (-1,2.2) to (-1,3.3);
\begin{scope}[shift={(-.75,0)}]
\draw[line width=.7mm] (-1,-3.9)to (-1,-2.8).. controls ++(0,.5) and ++(0,-1.5).. (1.5,-.4) to (1.5,-.2);
\end{scope}
\begin{scope}[xscale=-1,yscale=-1, shift={(-0.75,.6)}]
\draw[line width=.7mm](1.5,-.4) to (1.5,-.2) .. controls ++(0,1.5) and ++(0,-.5) .. (-1,2.2) to (-1,3.3);
\end{scope}
\begin{scope}[xscale=-1,yscale=-1, shift={(0,.6)}]
\draw[line width=.7mm,decoration={markings,mark=at position 0.5 with {\arrow[line width=.3mm]{<}}},
    postaction={decorate},shorten >=0.6pt] (-1,-3.9)to (-1,-2.8).. controls ++(0,.5) and ++(0,-1.5).. (1.5,-.4) to (1.5,-.2) .. controls ++(0,1.5) and ++(0,-.5) .. (-1,2.2) to (-1,3.3);
\end{scope}
\draw[line width=.7mm,decoration={markings,mark=at position 0.5 with {\arrow[line width=.3mm]{>}}},
    postaction={decorate}](.75,-.2).. controls ++(0,1) and ++(0,1) .. (-.75,-.2)  ;
\begin{scope}[shift={(-2.5,-2.8)}]
\draw[ directed=.3] (-.75,-.2).. controls ++(0,1) and ++(0,1) .. (3.2,-.2)  ;
\end{scope}
\draw [line width=.7mm,decoration={markings,mark=at position 1 with {\arrow[line width=.3mm]{>}}},
    postaction={decorate},shorten >=0.6pt] (3.5,3.3) to (3.5,-3.9);
   \draw[fill =gray!60] (-2.2,2.2) rectangle (-0.2,2.8);
  \draw[fill =gray!60] (.2,2.2) rectangle (2.2,2.8);
  \draw[fill =white] (-4.6,2.2) rectangle (-2.6,2.8);
   \draw[fill =white] (4.6,2.2) rectangle (2.6,2.8);
  \begin{scope}[shift={(0,-2.6)}]
  \draw[fill =gray!60] (-2.2,-.8) rectangle (-.2,-.2);
  \draw[fill =gray!60] (.2,-.8) rectangle (2.2,-.2);
  \draw[fill =white] (-4.6,-.8) rectangle (-2.6,-.2);
  \draw[fill =white] (4.6,-.8) rectangle (2.6,-.2);
  \end{scope}
  \node at (1.2,2.5) {$y-1-s$};
  \node at (-1.2,2.5) {$R-s$};
  \node at (3.7,2.5) {$R+b+1$};
  \node at (-3.7,2.5) {$y+a$};
  \begin{scope}[shift={(0,-2.6)}]
  \node at (1.2,-.5) {$y$};
  \node at (-1.2,-.5) {$R$};
  \node at (3.7,-.5) {$R+b+1$};
   \node at (-3.7,-.5) {$y+a+1$};
  \end{scope}
\end{tikzpicture}}
\;\;=\;\;\delta_{s,0}\;\;
\hackcenter{\begin{tikzpicture} [yscale=1.2,xscale=-1,scale=.25]
\draw[line width=0.7mm,decoration={markings,mark=at position 1 with {\arrow[line width=.3mm]{>}}},
    postaction={decorate},shorten >=0.6pt] (4.5,3) to (4.5,7);
\draw[line width=0.7mm,decoration={markings,mark=at position .06 with {\arrow[line width=.3mm]{<}}},
    postaction={decorate},shorten >=0.6pt] (-1.5,3) to (-1.5, 7);
\draw[line width=0.7mm,decoration={markings,mark=at position .06 with {\arrow[line width=.3mm]{<}}},
    postaction={decorate},shorten >=0.6pt] (-4.5,3) to (-4.5, 7);
\draw[line width=0.7mm,decoration={markings,mark=at position 1 with {\arrow[line width=.3mm]{>}}},
    postaction={decorate},shorten >=0.6pt] (1.5,3) to (1.5, 7);
\draw[directed=.8] (4.2,5.3)..controls ++(0,1) and ++(0,1)..(-.8, 5.3);
\draw[fill =gray!60] (-2.7,4.5) rectangle (-0.4,5.5);
\draw[fill =gray!60] (2.7,4.5) rectangle (0.4,5.5);
\draw[fill =white] (6,4.5) rectangle (3.1,5.5);
\draw[fill =white] (-6,4.5) rectangle (-3.1,5.5);
\node at (-4.5,5) [scale=.5]{$R+b+1$};
\node at (-1.5,5) [scale=.5]{$y$};
\node at (1.5,5)[scale=.5] {$R$};
\node at (4.5,5) [scale=.5]{$y+a+1$};
\end{tikzpicture}}.
\]
Consequently, the map $\P^{(y+a+1)}\P^{(R)^t}\Q^{(R+b+1,1^y)} \to \P^{(y+a)}\P^{(R-s)^t}\Q^{(y-1-s)^t}\Q^{(R+b+1)}$ is nonzero if and only if $s=0$. From Proposition \ref{lem:q-reduction} we know that the composition below is an isomorphism if and only if $s'=s-1$, but since $s'\geq 0$ then when $s=0$ no such $s'$ exists. 
\[
\begin{tikzcd}
\P^{(y+a)}\P^{(R-1-s')^t}\Q^{(y-1-s')^t}\Q^{(R+b-1)} \arrow[d,hook,"\iota"]&\P^{(y+a)}\P^{(R-s)^t}\Q^{(y-1-s)^t}\Q^{(R+b+1)}\\
\P^{(y+a)}\Q^{(y-1)^t}\P^{(R-1)^t}\Q^{(R+b)} \arrow[r,"\1_y \otimes \mathsf{d}_{b+1}^*"]& \P^{(y+a)}\Q^{(y-1)^t}\P^{(R)^t}\Q^{(R+b+1)}\arrow[u,two heads, "\rho"]
\end{tikzcd}
\]
This is precisely saying that image of $\P^{(y+a+1)}\P^{(R)^t}\Q^{(R+b+1,1^y)}$ in $\P^{(y+a)}\Q^{(y-1)^t}\P^{(R)^t}\Q^{(R+b+1)}$ under $\mathsf{d}_{a+1} \otimes \1_b$ cannot be pulled back isomorphically onto any summand of $\P^{(y+a)}\Q^{(y-1)^t}\P^{(R-1)^t}\Q^{(R+b)}$ under $\1_y \otimes \mathsf{d}_{b+1}^*$, i.e. no new arrows are created under the simultaneous simplification of Lemma \ref{lem:QC-reductions}. 

Consequently, the differential $\mathsf{d}^y$ is $\mathsf{d}_{a+1}\otimes \1_{b+1}$ composed and precomposed with the isomorphisms that exchange $\P^t$ and $\Q^t$ (Proposition \ref{prop:QPswap}) and merge the $\Q$'s (Proposition \ref{prop:PP*merge}), then restricted to the remaining terms under the reductions from Lemma \ref{lem:QC-reductions}. After some diagrammatic computations this differential simplified to the desired diagram $\d^y$ given in the theorem. 
\end{proof}

\begin{theorem} \label{thm:C*Cdirectprod}
For any $a,b \in \Z$ the homotopy equivalence
$\B_{b}^* \tilde{\otimes} \B_{a} \simeq \emph{Tot}^{\Pi} \left\lbrace \mathcal{A}^v, \mathsf{d}_v,\mathsf{d}^v \right\rbrace_{v\geq(0,b)} 
$ holds in $\K(\H)$ where the chain complex $\left\lbrace \mathcal{A}^v, \mathsf{d}_v\right\rbrace \simeq$ \emph{Cone}$(\mathsf{D}_v)$ with $\mathsf{D}_v$ the chain map: 
\begin{equation*}
\left\lbrace  \bigoplus_{R\geq(0,a+1)}\P^{(v-b)^t}\P^{(R)}\Q^{(v+1,1^{R-a-1})}[R-a+b-v-1],
\mathsf{d}_v
\right\rbrace \xrightarrow{\mathsf{D}_v} \mathsf{1}_{[0,a]}^{v} \P^{(v-b)^t} \P^{(a-v)}[b-v]
\end{equation*}with  $\mathsf{D}_v=\hackcenter
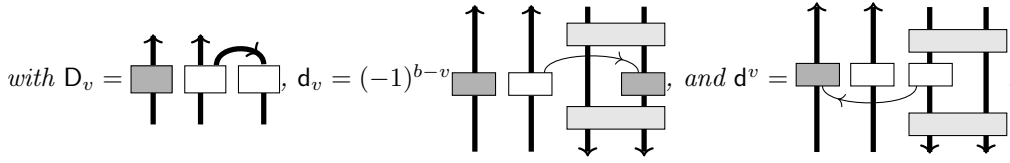
{
\begin{tikzpicture} [scale=0.6, yscale=1, xscale=-.7]
 \draw[line width=0.7mm,decoration={markings,mark=at position 1 with {\arrow[line width=.3mm]{>}}},
    postaction={decorate},shorten >=0.6pt] (1,-1) -- (1,1);
    postaction={decorate},shorten >=0.6pt] (-1,-1) -- (-1,1);
  \draw[line width=.7mm] (-1,-1) -- (-1,0);
    \draw[line width=.7mm,decoration={markings,mark=at position 0.8 with {\arrow[line width=.3mm]{>}}},
    postaction={decorate},shorten >=0.6pt]  (.5,.3) .. controls ++(0,.5) and ++(0,.5) .. (-1,.3);
     \draw[line width=0.7mm,decoration={markings,mark=at position 1 with {\arrow[line width=.3mm]{>}}},
    postaction={decorate},shorten >=0.6pt] (2.5,-1) -- (2.5,1);
  \draw[fill =gray!60] (3.2,-.3) rectangle (1.9,.3);
        \draw[fill =white] (-1.5,-.3) rectangle (-.2,.3);
        \draw[fill =white] (.2,-.3) rectangle (1.5,.3);
  \end{tikzpicture}}$, $\mathsf{d}_v=(-1)^{b-v}\hackcenter{\begin{tikzpicture} [yscale=1.2,xscale=1,scale=.25]
\draw[line width=0.7mm,decoration={markings,mark=at position .03 with {\arrow[line width=.3mm]{<}}},
    postaction={decorate},shorten >=0.6pt] (4.5,2) to (4.5,8.5);
\draw[line width=0.7mm,decoration={markings,mark=at position 1 with {\arrow[line width=.3mm]{>}}},
    postaction={decorate},shorten >=0.6pt] (-1.5,2) to (-1.5, 8.5);
\draw[line width=0.7mm,decoration={markings,mark=at position 1 with {\arrow[line width=.3mm]{>}}},
    postaction={decorate},shorten >=0.6pt] (-4.5,2) to (-4.5, 8.5);
\draw[line width=0.7mm,decoration={markings,mark=at position .03 with {\arrow[line width=.3mm]{<}}},
    postaction={decorate},shorten >=0.6pt] (1.5,2) to (1.5, 8.5);
\draw[directed=.8] (-.8,5.5)..controls ++(0,1) and ++(0,1)..(4.2, 5.5);
\draw[fill =white] (-2.7,4.5) rectangle (-0.4,5.5);
\draw[fill =gray!60] (-5.7,4.5) rectangle (-3.4,5.5);
\draw[fill =gray!20] (0.4,6.7) rectangle (5.6,7.7);
\draw[fill =gray!60] (5.6,4.5) rectangle (3.3,5.5);
\draw[fill =gray!20] (0.4,4) rectangle (5.6,3);
\end{tikzpicture}}$, 
 and $\mathsf{d}^v = \hackcenter{\begin{tikzpicture} [yscale=-1.2,xscale=-1,scale=.25]
\draw[line width=0.7mm,decoration={markings,mark=at position .03 with {\arrow[line width=.3mm]{<}}},
    postaction={decorate},shorten >=0.6pt] (4.5,2) to (4.5,8.5);
\draw[line width=0.7mm,decoration={markings,mark=at position 1 with {\arrow[line width=.3mm]{>}}},
    postaction={decorate},shorten >=0.6pt] (-1.5,2) to (-1.5, 8.5);
\draw[line width=0.7mm,decoration={markings,mark=at position 1 with {\arrow[line width=.3mm]{>}}},
    postaction={decorate},shorten >=0.6pt] (-4.5,2) to (-4.5, 8.5);
\draw[line width=0.7mm,decoration={markings,mark=at position .03 with {\arrow[line width=.3mm]{<}}},
    postaction={decorate},shorten >=0.6pt] (1.5,2) to (1.5, 8.5);
\draw[directed=.8] (-.8,5.5)..controls ++(0,1) and ++(0,1)..(4.2, 5.5);
\draw[fill =white] (-2.7,4.5) rectangle (-0.4,5.5);
\draw[fill =white] (2.7,4.5) rectangle (0.4,5.5);
\draw[fill =gray!20] (-0.4,6.7) rectangle (-5.6,7.7);
\draw[fill =gray!60] (5.6,4.5) rectangle (3.3,5.5);
\draw[fill =gray!20] (-0.4,4) rectangle (-5.6,3);
\end{tikzpicture}}
$.
\end{theorem}

\begin{proof}
As before, we write $\B_{b}^* \tilde{\otimes} \B_{a} = \text{Tot}^{\Pi} \left\lbrace \P^{(v-b)^t}\Q^{(v)}\B_{a},(-1)^{b-v} \1_v \otimes \mathsf{d}_{a}, \mathsf{d}^*_{b}\otimes \1 \right\rbrace_{v\geq(0,b)}$. The result follows from an identical argument to Theorem \ref{thm:CC*directsum} with the following minor variations. Since $\P^{(v-b)^t}\Q^{(v)}$ sits in homological degree $b-v$ an increase in $v$ corresponds to a decrease in the homological degree of the term. Consequently, $v$ is bounded below by max$(0,b)$ and hence the bi-complex is bounded above. Since we are dealing with the \emph{completed} tensor product, by Proposition \ref{prop:SimultSimp} we can simultaneously apply Lemma \ref{lem:QC-reductions} to each subcomplex $\P^{(v-b)^t}\Q^{(v)}\B_{a}$ and arrive at the desired homotopy equivalence. 
\end{proof}

The previous result would not hold if we swapped the completed tensor product with the usual tensor product. This is because the usual total complex of this bi-complex would not be well defined after performing the simultaneous reductions on each subcomplex. 

\begin{remark}
Via identical arguments, the dual statements for Lemma \ref{lem:QC-reductions} and Theorems \ref{thm:C*Cdirectprod} and \ref{thm:CC*directsum} also hold. 
\end{remark}

Now, as in Remark \ref{rem:locallyfinite}, when considering tensor products of complexes in $\K^-(\H)$ and $\K^+(\H)$, the bi-complexes that arise are generally not $\K(\H)$-locally finite. To remedy this we shift our paradigm from $\K(\H)$ to $\K(\H_n)$ for $n \in \N$ arbitrary. Recall that $\H_n$ is equivalent to $\H$ quotiented by the maximal ideal generated by $\mathsf{R}\Q^{\otimes n}$ for $\mathsf{R} \in \H$. This additional condition has two immediate effects:
\begin{enumerate}
\item The chain complexes $\B_a$ and $\B_b^*$ become finite over $\K(\H_n)$ and are consequently contained in $\K^b(\H)$. 
\item The bi-complexes are bounded and thus homologically locally finite so Proposition \ref{prop:locallyfinite} applies and the usual tensor product and completed tensor product are the same. 
\end{enumerate}
Thus, we have the following theorem. 

\begin{theorem}\label{thm:CC*=C*C}
For any $n \in \N$, the following homotopy equivalences hold in $\K(\H_n)$. 
\[ \B_{a+1}\otimes\B_{b+1}^* \simeq \begin{cases}
\B_{b}^*\otimes\B_{a}[-1] &a<b \\
\B_{b}^*\otimes\B_{a}[+1]&a>b.
\end{cases}\]
\end{theorem}

\begin{proof}
Suppose $a<b$. Then by Theorem \ref{thm:CC*directsum} and \ref{thm:C*Cdirectprod}:
\begin{align*} 
\mathcal{A}^y&= \mathsf{1}^y_{[0,-b-1]}\P^{(y+a+1)}\P^{(-b-y-1)^t}[b+y+1] \xrightarrow{\mathsf{D_y}} \left\lbrace\bigoplus_{R\geq(-b,0)} \P^{(y+a+1)}\P^{(R)^t}\Q^{(R+b+1,1^y)}[y-R],\mathsf{d}_y
\right\rbrace \\
\mathcal{A}^v&=\left\lbrace  \bigoplus_{R\geq(0,a+1)}\P^{(v-b)^t}\P^{(R)}\Q^{(v+1,1^{R-a-1})}[R-a+b-v],
\mathsf{d}_v
\right\rbrace \xrightarrow{\mathsf{D}_v} \mathsf{1}_{[0,a]}^{v} \P^{(v-b)^t} \P^{(a-v)}[b-v]
\end{align*}

However, $\P^{(y+a+1)}\P^{(-b-y-1)}[b+y+1]$ and $\P^{(v-b)^t}\P^{(a-v)}[b-v]$ are nonzero if and only if $-b-1\geq y \geq -a-1$ and $a\geq v \geq b$, which implies $a\geq b$, but $b>a$ so neither of these terms can exist. Thus, 
\begin{align*}
&\B_{a+1} \otimes \B_{b+1}^* \simeq \text{Tot}^{\oplus} \left\lbrace\bigoplus_{R\geq(-b,0)} \P^{(y+a+1)}\P^{(R)^t}\Q^{(R+b+1,1^y)}[y-R],\mathsf{d}_y, \mathsf{d}^y
\right\rbrace_{y\geq (-a-1,0)} \\
&\B_{b}^* \tilde{\otimes} \B_{a} \simeq \text{Tot}^{\Pi}\left\lbrace  \bigoplus_{R'\geq(0,a+1)}\P^{(v-b)^t}\P^{(R')}\Q^{(v+1,1^{R'-a-1})}[R'-a+b-v],
\mathsf{d}_v, \mathsf{d}^v
\right\rbrace_{v\geq (b,0)}
\end{align*}
where the differentials are given by
$ \mathsf{d}_y=\hackcenter{\begin{tikzpicture} [yscale=-1.2,scale=.25]
\draw[line width=0.7mm,decoration={markings,mark=at position 1 with {\arrow[line width=.3mm]{>}}},
    postaction={decorate},shorten >=0.6pt] (4.5,2) to (4.5,8.5);
\draw[line width=0.7mm,decoration={markings,mark=at position .03 with {\arrow[line width=.3mm]{<}}},
    postaction={decorate},shorten >=0.6pt] (-1.5,2) to (-1.5, 8.5);
\draw[line width=0.7mm,decoration={markings,mark=at position .03 with {\arrow[line width=.3mm]{<}}},
    postaction={decorate},shorten >=0.6pt] (-4.5,2) to (-4.5, 8.5);
\draw[line width=0.7mm,decoration={markings,mark=at position 1 with {\arrow[line width=.3mm]{>}}},
    postaction={decorate},shorten >=0.6pt] (1.5,2) to (1.5, 8.5);
\draw[directed=.8] (4.2,5.5)..controls ++(0,1) and ++(0,1)..(-.8, 5.5);
\draw[fill =gray!60] (-2.7,4.5) rectangle (-0.4,5.5);
\draw[fill =white] (-5.7,4.5) rectangle (-3.4,5.5);
\draw[fill =gray!20] (0.4,6.7) rectangle (5.6,7.7);
\draw[fill =white] (5.6,4.5) rectangle (3.3,5.5);
\draw[fill =gray!20] (0.4,4) rectangle (5.6,3);
\end{tikzpicture}} \;$, $\; \mathsf{d}^y = (-1)^R \hackcenter{\begin{tikzpicture} [yscale=1.2,xscale=-1,scale=.25]
\draw[line width=0.7mm,decoration={markings,mark=at position 1 with {\arrow[line width=.3mm]{>}}},
    postaction={decorate},shorten >=0.6pt] (4.5,2) to (4.5,8.5);
\draw[line width=0.7mm,decoration={markings,mark=at position .03 with {\arrow[line width=.3mm]{<}}},
    postaction={decorate},shorten >=0.6pt] (-1.5,2) to (-1.5, 8.5);
\draw[line width=0.7mm,decoration={markings,mark=at position .03 with {\arrow[line width=.3mm]{<}}},
    postaction={decorate},shorten >=0.6pt] (-4.5,2) to (-4.5, 8.5);
\draw[line width=0.7mm,decoration={markings,mark=at position 1 with {\arrow[line width=.3mm]{>}}},
    postaction={decorate},shorten >=0.6pt] (1.5,2) to (1.5, 8.5);
\draw[directed=.8] (4.2,5.5)..controls ++(0,1) and ++(0,1)..(-.8, 5.5);
\draw[fill =gray!60] (-2.7,4.5) rectangle (-0.4,5.5);
\draw[fill =gray!60] (2.7,4.5) rectangle (0.4,5.5);
\draw[fill =gray!20] (-0.4,6.7) rectangle (-5.6,7.7);
\draw[fill =white] (5.6,4.5) rectangle (3.3,5.5);
\draw[fill =gray!20] (-0.4,4) rectangle (-5.6,3);
\end{tikzpicture}}$,

$\mathsf{d}_v=(-1)^{b-v}\hackcenter{\begin{tikzpicture} [yscale=1.2,xscale=1,scale=.25]
\draw[line width=0.7mm,decoration={markings,mark=at position .03 with {\arrow[line width=.3mm]{<}}},
    postaction={decorate},shorten >=0.6pt] (4.5,2) to (4.5,8.5);
\draw[line width=0.7mm,decoration={markings,mark=at position 1 with {\arrow[line width=.3mm]{>}}},
    postaction={decorate},shorten >=0.6pt] (-1.5,2) to (-1.5, 8.5);
\draw[line width=0.7mm,decoration={markings,mark=at position 1 with {\arrow[line width=.3mm]{>}}},
    postaction={decorate},shorten >=0.6pt] (-4.5,2) to (-4.5, 8.5);
\draw[line width=0.7mm,decoration={markings,mark=at position .03 with {\arrow[line width=.3mm]{<}}},
    postaction={decorate},shorten >=0.6pt] (1.5,2) to (1.5, 8.5);
\draw[directed=.8] (-.8,5.5)..controls ++(0,1) and ++(0,1)..(4.2, 5.5);
\draw[fill =white] (-2.7,4.5) rectangle (-0.4,5.5);
\draw[fill =gray!60] (-5.7,4.5) rectangle (-3.4,5.5);
\draw[fill =gray!20] (0.4,6.7) rectangle (5.6,7.7);
\draw[fill =gray!60] (5.6,4.5) rectangle (3.3,5.5);
\draw[fill =gray!20] (0.4,4) rectangle (5.6,3);
\end{tikzpicture}}\;$,  and  $\;\mathsf{d}^v= \hackcenter{\begin{tikzpicture} [yscale=-1.2,xscale=-1,scale=.25]
\draw[line width=0.7mm,decoration={markings,mark=at position .03 with {\arrow[line width=.3mm]{<}}},
    postaction={decorate},shorten >=0.6pt] (4.5,2) to (4.5,8.5);
\draw[line width=0.7mm,decoration={markings,mark=at position 1 with {\arrow[line width=.3mm]{>}}},
    postaction={decorate},shorten >=0.6pt] (-1.5,2) to (-1.5, 8.5);
\draw[line width=0.7mm,decoration={markings,mark=at position 1 with {\arrow[line width=.3mm]{>}}},
    postaction={decorate},shorten >=0.6pt] (-4.5,2) to (-4.5, 8.5);
\draw[line width=0.7mm,decoration={markings,mark=at position .03 with {\arrow[line width=.3mm]{<}}},
    postaction={decorate},shorten >=0.6pt] (1.5,2) to (1.5, 8.5);
\draw[directed=.8] (-.8,5.5)..controls ++(0,1) and ++(0,1)..(4.2, 5.5);
\draw[fill =white] (-2.7,4.5) rectangle (-0.4,5.5);
\draw[fill =white] (2.7,4.5) rectangle (0.4,5.5);
\draw[fill =gray!20] (-0.4,6.7) rectangle (-5.6,7.7);
\draw[fill =gray!60] (5.6,4.5) rectangle (3.3,5.5);
\draw[fill =gray!20] (-0.4,4) rectangle (-5.6,3);
\end{tikzpicture}}
$.

By Remark \ref{rem:totalcomsym} we can rewrite $\B_{b}^* \tilde{\otimes}\B_{a}$ in term of its rows instead of its columns, so that:
\[\B_{b}^* \tilde{\otimes} \B_{a}\simeq \text{Tot}^{\Pi}\left\lbrace  \bigoplus_{v\geq (b,0)}\P^{(v-b)^t}\P^{(R')}\Q^{(v+1,1^{R'-a-1})}[R'-a+b-v],
\mathsf{d}_{R'}, \mathsf{d}^{R'}
\right\rbrace_{R'\geq (0,a+1)}\]
Where $\mathsf{d}_{R'} = \mathsf{d}^{v}$ and $\mathsf{d}^{R'} = \mathsf{d}_v$. So then, if we send $R'\to y+a+1$ and $v-b \to R$ in $\B_{b}^*\tilde{\otimes}\B_{a}$ and recall that $\P^t\P \simeq \P\P^t$,  we have:
\[\B_{b}^* \tilde{\otimes} \B_{a}\simeq \text{Tot}^{\Pi}\left\lbrace  \bigoplus_{R\geq (0,-b)}\P^{(y+a+1)}\P^{(R)^t}\Q^{(R+b+1,1^{y})}[y+1-R],
\mathsf{d}_{y}, \mathsf{d}^{y}
\right\rbrace_{y\geq (0,-a-1)}
\]
where the differentials are exactly those obtained for $\B_{a+1}\otimes \B_{b+1}^*$. 

In particular, $\B_{a+1}\otimes \B_{b+1}^*$ and $\B_b^* \tilde{\otimes}\B_a$ correspond to the total complex and the completed total complex of the same bi-complex over $\K(\H_n)$. Thus Proposition \ref{prop:locallyfinite} applies and the result follows. 

\medskip

If $a>b$ then through an entirely dual construction we derive the following result. 
\begin{align*}
&\B_{b}^* {\otimes} \B_{a}\simeq \text{Tot}^{\oplus}\left\lbrace \bigoplus_{R'\geq (0,b+1)} \P^{(x+1,1^{R'-b-1})}\Q^{(R')}\Q^{(x-a)^t}[b-R'+x-a], \mathsf{d}_x, \mathsf{d}^x
\right\rbrace_{x\geq(0,a)}\\
&\B_{a+1}\tilde{\otimes}\B_{b+1}^* \simeq \text{Tot}^{\Pi}\left\lbrace \bigoplus_{R\geq (-a,0)}\P^{(R+a+1,1^w)}\Q^{(R)^t}\Q^{(w+b+1)}[R-w]
, \mathsf{d}_w,\mathsf{d}^w
 \right\rbrace_{w\geq(0,-b-1)}
 \end{align*}
Once again, these bi-complexes correspond to the usual and the completed total complex of the same bi-complex. Since we are working over $\K^b(\H)$ then $\B_{a+1}\tilde{\otimes}\B_{b+1}^* = \B_{a+1}{\otimes}\B_{b+1}^*$. Thus, we obtain $\B_{a+1}{\otimes}\B_{b+1}^* \simeq \B_{b}^* {\otimes} \B_{a}[+1]$.
\end{proof}

\begin{theorem}\label{thm:CC*distingishedTriangle}
 For any $n \in \N$, then over $\K(\H_n)$ we have:
\begin{equation}
\B_a^* \otimes \B_a \simeq \emph{Cone}(\B_{a+1}\otimes \B_{a+1}^* \to \1) \qquad \text{and} \qquad \B_{a+1} \otimes \B_{a+1}^*[1] \simeq \emph{Cone}(\1 \to \B^*_{a}\otimes \B_{a})
\end{equation}
\begin{equation}
\B_{a+1}\otimes \B_{a+1}^* \simeq \emph{Cone}(\B_{a}^*\otimes \B_{a} \to \1) \qquad \text{and} \qquad \B_{a}^* \otimes \B_{a}[1] \simeq \emph{Cone}(\1 \to \B_{a+1}\otimes \B_{a+1}^*).
\end{equation}
Thus, $\B_{a+1}\otimes\B_{a+1}^* \rightarrow \1 \rightarrow \B_{a}^*\otimes\B_{a}$ and $\B_{a}^*\otimes\B_{a} \rightarrow \1 \rightarrow \B_{a+1}\otimes\B_{a+1}^*$ are distinguished triangles in $\K(\H_n)$.
\end{theorem}

\begin{proof}
Since we are working over $\K(\H_n)$ it follows $\B_a$ and $\B_b^*$ are finite and thus all tensor products and completed tensor products agree. 

Suppose $a\geq 0$. Then by Theorems  \ref{thm:CC*directsum} and \ref{thm:C*Cdirectprod} we have $
\B_{a+1}\otimes \B_{a+1}^*\simeq \text{Tot}^{\oplus} \left\lbrace \text{Cone}(\mathsf{D}_y), \begin{pmatrix}
\d_y&0\\\mathsf{D_y}&\d_y
\end{pmatrix},\d^y\right\rbrace_{y\geq 0}$ and $
\B_{a}^*\otimes \B_{a}\simeq \text{Tot}^{\oplus} \left\lbrace \text{Cone}(\mathsf{D}_v), \begin{pmatrix}
\d_v&0\\\mathsf{D_v}&\d_v
\end{pmatrix},\d^v\right\rbrace_{v\geq a}
.$ In particular, 
\begin{align*}
\text{Cone}(\mathsf{D}_y) &= \left\lbrace\bigoplus_{R\geq0} \P^{(y+a+1)}\P^{(R')^t}\Q^{(R'+a+1,1^y)}[y-R'],\mathsf{d}_y
\right\rbrace \\
\text{Cone}(\mathsf{D}_v)&=
\left\lbrace  \bigoplus_{R\geq a+1}\P^{(v-a)^t}\P^{(R)}\Q^{(v+1,1^{R-a-1})}[R-v],
\mathsf{d}_v
\right\rbrace \xrightarrow{\mathsf{D}_v} \delta_{v,a}\1[0]
\end{align*}
Moreover, $\mathsf{D}_v$ is zero unless $v=a$ and $R=a+1$. Thus the bi-complex for $\B_a^*\otimes\B_a$ has only one arrow going into $\1[0]$ and no arrows coming out of it. Consequently, we can write:
\[\B_a^* \otimes \B_a \simeq \text{Tot}^{\oplus} \left\lbrace \bigoplus_{R\geq a+1} \P^{(v-a)^t}\P^{(R)}\Q^{(v+1,1^{R-a-1})}[R-v],\d_v,\d^v\right\rbrace_{v\geq a} \xrightarrow{\mathsf{D}_v} \;\; \delta_{v,a} \1[0]\]
But then, re-indexing so that $R'\mapsto v-a$ and $y+a+1\mapsto R$ and comparing the differentials we immediately obtain that:
\[
\B_a^*\otimes\B_a \simeq \B_{a+1}\otimes \B_{a+1}^*[1] \xrightarrow{\mathsf{D}_v} \delta_{v,a} \1[0]\]
Likewise, if $a<0$ we find that Cone$(\mathsf{D}_v)$ has no identity terms whereas Cone$(\mathsf{D}_y)$ has a term equal to $\1$ that is the source of one arrow and the target of none. By an identical argument we find that $\B_{a+1}^*\otimes\B_{a+1}[1] \simeq \1 \xrightarrow{\mathsf{D}_y} \B_{a}\otimes\B_a^*\simeq$ Cone$(\mathsf{D}_y)$. Thus, there exists a distinguished triangle
\[\cdots \to \B_{a+1}\otimes\B_{a+1}^* \to \1 \to \B_a^*\otimes \B_a \to \B_{a+1}\otimes \B_{a+1}^* \to \cdots\]
The remaining statements all follow by dual arguments. 
\end{proof}


\section{Fock Space Idempotents}\label{sec:SigmaComplexes}

Cautis, Licata, and Sussan studied certain complexes $\Sigma_i$ in the Heisenberg category from \cite{CL-Heis} and showed they induced categorical braid group actions \cite{CLS-BraidAction}. Motivated by this, Cautis and Sussan introduced analogous chain complexes  $\Sigma^-, \Sigma^+ \in \K(\H)$  and made several conjectures about their properties \cite{CS-BosonFermion}. In this section we prove these conjectures. 

Define the following biadjoint chain complexes in $\K(\H)$:
\begin{equation}\label{eq:sigma-minus}
\Sigma^-:= \left\lbrace \bigoplus_{\substack{\lambda \vdash n}} \P^{\lambda}\Q^{\lambda^t}[n], \mathsf{d}^-\right\rbrace \in \K^-(\H) \; \text{ with }\; \mathsf{d}^-:\P^{\mu(s^+)} \Q^{\mu^t(s^+)} \xrightarrow{\iota_s\otimes \iota_s} \P^{\mu}\P\Q\Q^{\mu^t} \xrightarrow{\1\otimes adj \otimes \1} \P^{\mu}\Q^{\mu^t}
\end{equation} 
\begin{equation} \label{eq:sigma-plus}
\Sigma^+:=\left\lbrace \bigoplus_{\substack{\lambda \vdash n}} \P^{\lambda}\Q^{\lambda^t}[-n], \mathsf{d}^+\right\rbrace \in \K^+(\H)\; \text{ with }\;\mathsf{d}^+:\P^{\mu} \Q^{\mu^t} \xrightarrow{\1\otimes adj \otimes \1} \P^{\mu}\P\Q\Q^{\mu^t} \xrightarrow{\rho_s\otimes \rho_s} \P^{\mu(s^+)}\Q^{\mu^t(s^+)}
\end{equation}
where $\mu$ is any partition of $n-1$, the maps $\iota_s$ and $\rho_s$ defined as in (\ref{eq:PlambdaPmerge}) and $adj$ equal to the adjunction maps given by the cap and cup morphisms. In particular, we have that $\mathsf{d}^-_n = \sum_{\mu \vdash n-1} \sum_{\mu+\square} \1_{\P^{(\mu+\square)}} \otimes \iota_\mu^{(\mu+\square)\otimes 1}$, where $\sum_{\mu+\square}$ is summing all possible $\mu(s+)$ and $\iota_\mu^{(\mu+\square)\otimes 1}$ is given by the morphisms in Proposition \ref{prop:QlambdaPswap}. 

\begin{lemma}\label{lem:sigma-null}
The chain complexes $\Sigma^-\P$ and $\Q \Sigma^-$ are contractible.  
\end{lemma}

\begin{proof}
Tensoring $\Sigma^-$ with $\P$, by Propositions \ref{prop:QlambdaPswap} and \ref{prop:PlambdaPmerge} we have the chain group isomorphisms
\begin{align*}
\bigoplus_{\substack{\lambda \vdash n}}\P^{\lambda}\Q^{\lambda^t}\P
\cong \bigoplus_{\substack{\lambda \vdash n}}\P^{\lambda}\P\Q^{\lambda^t} \oplus \bigoplus_{\substack{\lambda \vdash n \\ \lambda = \mu+\square}} \P^{\lambda}\Q^{\mu^t}
\cong \bigoplus_{\substack{\lambda \vdash n}}\P^{\lambda}\P\Q^{\lambda^t} \oplus \bigoplus_{\mu\vdash n-1}\bigoplus_{\mu(s^+)} \P^{\mu(s^+)}\Q^{\mu^t}.
\end{align*}
Let $\mathcal{D}_n:=  \bigoplus_{\lambda\vdash n} \P^{\lambda}\P\Q^{\lambda^t}\cong\bigoplus_{\lambda\vdash n} \bigoplus_{\lambda(s^+)} \P^{\lambda(s^+)}\Q^{\lambda^t}$. Then, $\Sigma^-  \cong \left\lbrace \bigoplus_n \mathcal{D}_n\oplus \mathcal{D}_{n-1}, \mathsf{D}_n\right\rbrace$ for some differential $\mathsf{D}_n$. In particular, we will show that the differential restricted as follows $\mathsf{D}_{n}:\mathcal{D}_{n-1} \to \mathcal{D}_{n-1}$ is an isomorphism for all $n> 0$. 

First since the differential $\mathsf{d}^-\otimes \1:\P^{\lambda}\Q^{\lambda^t}\P\to \P^{\mu}\Q^{\mu^t}\P$ is zero whenever $\lambda \neq \mu+ \square$, then for any $\mu \vdash n-1$, $\mathsf{D}_{n}: \P^{\mu(s^+)}\Q^{\mu^t} \to \P^{\mu}\P\Q^{\mu^t}$   is given by a nonzero multiple of the map below: 
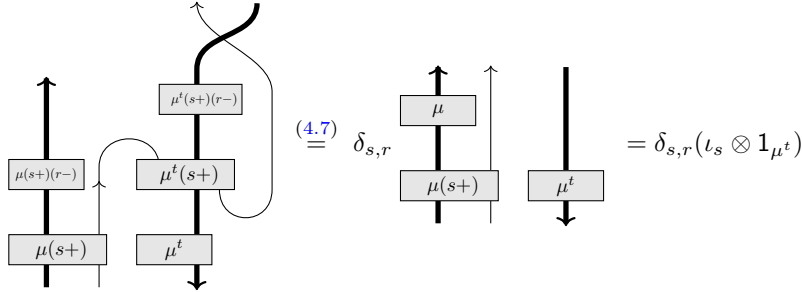
\[\hackcenter{\begin{tikzpicture}
\draw [line width=0.7mm,decoration={markings,mark=at position 1 with {\arrow[line width=.3mm]{>}}},
    postaction={decorate},shorten >=0.6pt] (-1,-.3) to (-1,2.5);
\draw [line width=0.7mm,decoration={markings,mark=at position 0.02 with {\arrow[line width=.3mm]{<}}},
    postaction={decorate},shorten >=0.6pt] (1,-.3) to (1,2.6)..controls ++(0,.5) and ++(0,-.5) ..(1.8,3.5);
\draw [directed =.5](-.3,-.3) to (-.3,1.3)..controls ++(0,.5) and ++(0,.5) .. (.5,1.3);
\draw [->](1.3,1)..controls ++(0,-.5) and ++(0,-.5) .. (2,1) to (2, 2.3)..controls ++(0,.5) and ++(0,-.5) ..(1,3.5);
\draw [fill=gray!20] (-1.5,0) rectangle (-.2,.4);
\draw [fill=gray!20] (.2,0) rectangle (1.2,.4);
\node at (-.8,.2)[scale=.75] {$\mu(s+)$};
\node at (.7,.2)[scale=.75]{$\mu^t$};
\begin{scope}[shift={(0,1)}]
\draw [fill=gray!20] (-1.5,0) rectangle (-.5,.4);
\draw [fill=gray!20] (.2,0) rectangle (1.5,.4);
\node at (-1,.2)[scale=.5] {$\mu(s+)(r-)$};
\node at (.9,.2)[scale=.75]{$\mu^t(s+)$};
\end{scope}
\begin{scope}[shift={(0,2)}]
\draw [fill=gray!20] (.5,0) rectangle (1.6,.4);
\node at (1.1,.2)[scale=.5]{$\mu^t(s+)(r-)$};
\end{scope}
\end{tikzpicture}}
\;\;\overset{\eqref{heis: right twist curl}}{=}\;\delta_{s,r}\;
\hackcenter{\begin{tikzpicture}
\draw [line width=0.7mm,decoration={markings,mark=at position 1 with {\arrow[line width=.3mm]{>}}},
    postaction={decorate},shorten >=0.6pt] (-1,-.3) to (-1,1.8);
\draw [line width=0.7mm,decoration={markings,mark=at position .04 with {\arrow[line width=.3mm]{<}}},
    postaction={decorate},shorten >=0.6pt] (.7,-.3) to (.7,1.8);
\draw [->](-.3,-.3) to (-.3,1.8);
\draw [fill=gray!20] (-1.5,0) rectangle (-.2,.4);
\draw [fill=gray!20] (.2,0) rectangle (1.2,.4);
\node at (-.8,.2)[scale=.75] {$\mu(s+)$};
\node at (.7,.2)[scale=.75]{$\mu^t$};
\begin{scope}[shift={(0,1)}]
\draw [fill=gray!20] (-1.5,0) rectangle (-.5,.4);
\node at (-1,.2)[scale=.75] {$\mu$};
\end{scope}
\end{tikzpicture}}\;\; = \delta_{s,r}(\iota_{s} \otimes \1_{\mu^t})\]

\noindent This is because if $s \neq r$, the diagram on the left has a right twist curl which by  \eqref{heis: right twist curl} is zero. Consequently, if we denote $\iota_{s^+}$ by $\iota_{\mu+\square}^{\mu \otimes 1}$ then the restriction $\mathsf{D}_n:\mathcal{D}_{n-1} \to \mathcal{D}_{n-1}$ is given by $\mathsf{D}_n = \sum_{\mu \vdash n-1} \iota_{\mu+\square}^{\mu \otimes 1} \otimes \1_{\mu^t}$. Now consider the map
\[ \tilde{\mathsf{D}_n}:=
 \sum_{\mu \vdash n-1} \sum_{\mu+\square} \rho^{\mu+\square}_{\mu \otimes 1} \otimes \1_{\mu^t}:\bigoplus_{\substack{\mu \vdash n-1}}\P^{\mu}\P\Q^{\mu^t} \rightarrow \bigoplus_{\substack{\mu \vdash n-1 \\\mu +\square}} \P^{\mu +\square}\Q^{\mu^t}.\]
Computing its composition with $\mathsf{D}_n$ we find that
\begin{align*}
\mathsf{D}_n\tilde{\mathsf{D}_n} &= \left(\sum_{\mu \vdash n-1} \iota_{\mu+\square}^{\mu \otimes 1} \otimes \1_{\mu^t}\right)\left(\sum_{\mu' \vdash n-1} \sum_{\mu'+\square} \rho^{\mu'+\square}_{\mu' \otimes 1} \otimes \1_{{\mu'}^t}\right)\\
&=\sum_{\substack{\mu \vdash n-1 \\ \mu' \vdash n-1}}\left(\sum_{\mu'+\square}  \iota_{\mu+\square}^{\mu \otimes 1} \circ  \rho^{\mu'+\square}_{\mu' \otimes 1}\right) \otimes \left( \1_{\mu^t} \circ \1_{{\mu'}^t}\right) \\
&=\sum_{\substack{\mu \vdash n-1 \\ \mu' \vdash n-1}}\left( \sum_{\mu'+\square}  \iota_{\mu+\square}^{\mu \otimes 1} \circ  \rho^{\mu'+\square}_{\mu' \otimes 1}\right) \otimes \left( \delta_{\mu,\mu'}\1_{{\mu}^t}\right) \\
&=\sum_{\mu \vdash n-1} \left( \sum_{\mu+\square}  \iota_{\mu+\square}^{\mu \otimes 1} \circ  \rho^{\mu+\square}_{\mu \otimes 1}\right) \otimes \1_{{\mu}^t}\\
&=\sum_{\mu \vdash n-1} \1_{\mu \otimes 1} \otimes \1_{{\mu}^t}\\
&= \1_{\bigoplus_{\substack{\mu \vdash n-1}}\P^{\mu}\P\Q^{\mu^t}}.
\end{align*}
Similarly, we can compute that $\tilde{\mathsf{D}_n}\mathsf{D}_n= \1_{\bigoplus_{\substack{\mu \vdash n-1 \\\mu +\square}} \P^{(\mu + \square)} \Q^{\mu^t}}$.
Thus, $\tilde{\mathsf{D}_n}$ and $\mathsf{D}_n$ are mutual inverses. Defining the nullhomotopy to be $\mathsf{H}_n=\begin{pmatrix}
0 & \tilde{\mathsf{D}_n}\\0&0
\end{pmatrix}$ yields the desired result. 

The proof for $\Q\Sigma^- \cong 0$ follows identically by interchanging the terms in the tensor product and noting that $(\mu +\square)^t = \mu^t + \square$.
\end{proof}

\begin{lemma}\label{cor:sigma-plus-null}
$\Sigma^+\P$ and $\Q \Sigma^+$ are contractible. 
\end{lemma}

\begin{proof}
The result follows from an identical argument to Lemma \ref{lem:sigma-null}. 
\end{proof}

\begin{theorem}[Conjecture 4.1 in \cite{CS-BosonFermion}] \label{thm:SigmaPlambdaNull}Given any nontrivial partition $\lambda$, the nullhomotopies $\Sigma^-\P^{\lambda} \cong 0 \cong \Sigma^+\P^{\lambda}$ and $\Q^{\lambda}\Sigma^- \cong 0 \cong \Q^{\lambda}\Sigma^+$ hold in $\K(\H)$. Thus, $\Sigma^-$ and $\Sigma^-$ are categorical projectors for $\mathsf{V}_{Fock}$.
\end{theorem}

\begin{proof}
Since for any pair of nontrivial partitions $\mu$ and $\lambda = \mu+\square$ there are unique maps $\P^{\lambda} \hookrightarrow \P\P^{\mu}$ and $\Q^{\lambda} \hookrightarrow \Q^{\mu}\Q$ given by $\rho$ in Proposition \ref{prop:PlambdaPmerge}. Thus by Lemma \ref{lem:sigma-null}, it follows that $\Sigma^- \P^{\lambda} \hookrightarrow\Sigma^-\P \P^{\mu} \cong 0$. The remaining homotopies follow identically from Lemma \ref{cor:sigma-plus-null}. 
\end{proof}

\begin{corollary}\label{cor:SigmaCreduction}
Given complexes:
\[
\C^-\simeq \dots \to \C_k\to \C_{k-1} \to \dots \to \C_1\to\1 \qquad \in \K^-(\H) \]
\[
\C^+ \simeq \1 \to \C_{-1} \to \dots \to \C_{k-1} \to \C_{k} \to \dots \qquad \in \K^+(\H)
\]
Such that for all $k>0$ we have $\C_r = \bigoplus_{i\in I} \P^{\lambda^i} \mathsf{X} \Q^{\mu^i}$ where $I \subset \N$ is some finite indexing set, $\lambda^i$ and $\mu^i$ are nontrivial partitions, and $\mathsf{X} \in \H$. Then the following homotopy equivalences hold in $\K(\H)$:
\begin{align*}
&\bullet\Sigma^-\otimes \C^- \cong \Sigma^- \cong \C^-\otimes \Sigma^-  && \bullet\Sigma^+\tilde{\otimes} \C^+ \cong \Sigma^+ \cong \C^+\tilde{\otimes} \Sigma^+\\
&\bullet\Sigma^- \tilde{\otimes} \C^+ \cong \Sigma^- \cong \C^+ \tilde{\otimes} \Sigma^- &&\bullet\Sigma^+ \otimes \C^- \cong \Sigma^+ \cong \C^- \otimes \Sigma^+.
\end{align*}
In particular, $\Sigma^- \otimes \Sigma^- \cong \Sigma^-$ and $\Sigma^+ \tilde{\otimes} \Sigma^+ \cong \Sigma^+$, and thus $\Sigma^\pm$ are idempotents. 
\end{corollary}

\begin{proof}
The homotopies follow directly from Theorem \ref{thm:SigmaPlambdaNull} and Proposition \ref{prop:SimultSimp}. 
\end{proof}

\begin{theorem}[Conjecture 4.2 in \cite{CS-BosonFermion}] \label{thm:SigmaISO} Given any $n \in \N$, the chain complexes $\Sigma^+$ and $\Sigma^-$ are homotopy equivalent in $\K(\H_n)$. 
\end{theorem}

\begin{proof}
Since $\Sigma^\pm$ become finite on $\K(\H_n)$ they are homologically locally finite and thus  by Proposition \ref{prop:locallyfinite} and Corollary \ref{cor:SigmaCreduction} it immediately follows that $\Sigma^+ \cong \Sigma^- \otimes \Sigma^+ \cong \Sigma^- \tilde{\otimes} \Sigma^+ \cong \Sigma^-$. 
\end{proof}

Since the action of $\Sigma^+$ and $\Sigma^-$ is integrable on categorical Fock space, then  Theorems \ref{thm:SigmaPlambdaNull} and \ref{thm:SigmaISO} imply that both $\Sigma^+$ and $\Sigma^-$ can be used to project onto $\H_0 = \bigoplus \k[S_n]-$mod. This is in line with the decategorified picture where any $\mathfrak{h}$ module generated by highest weight vectors decomposes into a direct sum of the trivial module, i.e. Fock Space.



\nocite{*}

\begin{thebibliography}{DJKM82}

\bibitem[CL11]{CL-Vertex}
Sabin Cautis and Anthony Licata.
\newblock {Vertex operators and 2-representations of quantum affine algebras}.
\newblock {\em https://arxiv.org/abs/1112.6189}, 2011.

\bibitem[CL12a]{CL-Heis}
Sabin Cautis and Anthony Licata.
\newblock Heisenberg categorification and {H}ilbert schemes.
\newblock {\em Duke Math. J.}, 161(13):2469--2547, 2012.

\bibitem[CL12b]{CL-Loop}
Sabin Cautis and Anthony Licata.
\newblock Loop realizations of quantum affine algebras.
\newblock {\em J. Math. Phys.}, 53(12):123505, 18, 2012.

\bibitem[CLS14]{CLS-BraidAction}
Sabin Cautis, Anthony Licata, and Joshua Sussan.
\newblock Braid group actions via categorified {H}eisenberg complexes.
\newblock {\em Compos. Math.}, 150(1):105--142, 2014.

\bibitem[CS15]{CS-BosonFermion}
Sabin Cautis and Joshua Sussan.
\newblock On a categorical {B}oson-{F}ermion correspondence.
\newblock {\em Comm. Math. Phys.}, 336(2):649--669, 2015.

\bibitem[Cvi08]{Birdtracks}
Predrag Cvitanovic.
\newblock {\em Group theory: Birdtracks, Lie's and exceptional groups}.
\newblock Princeton University Press, 2008.

\bibitem[DJKM82]{Solitions}
Etsuro Date, Michio Jimbo, Masaki Kashiwara, and Tetsuji Miwa.
\newblock Transformation groups for soliton equations: Iv. a new hierarchy of
  soliton equations of kp-type.
\newblock {\em Physica D: Nonlinear Phenomena}, 4(3):343 -- 365, 1982.

\bibitem[EH17]{EH-CatDiag}
Ben Elias and Matt Hogancamp.
\newblock Categorical diagonalization.
\newblock {\em https://arxiv.org/abs/1707.04349}, 07 2017.

\bibitem[FJW00]{Frenkel2000}
Igor~B. Frenkel, Naihuan Jing, and Weiqiang Wang.
\newblock Vertex representations via finite groups and the McKay
  correspondence.
\newblock {\em International Mathematics Research Notices}, 2000:195--222, Dec
  2000.

\bibitem[FK80]{Frenkel-Kac}
I.~B. Frenkel and V.~G. Kac.
\newblock {Basic Representations of Affine Lie Algebras and Dual Resonance
  Models}.
\newblock {\em Invent. Math.}, 62:23--66, 1980.

\bibitem[FPS16]{FPS}
Igor Frenkel, Ivan Penkov, and Vera Serganova.
\newblock A categorification of the boson-fermion correspondence via
  representation theory of {$sl(\infty)$}.
\newblock {\em Comm. Math. Phys.}, 341(3):911--931, 2016.

\bibitem[Fre81]{Frenkel-BF}
I.B Frenkel.
\newblock Two constructions of affine lie algebra representations and
  boson-fermion correspondence in quantum field theory.
\newblock {\em Journal of Functional Analysis}, 44(3):259 -- 327, 1981.

\bibitem[Ful96]{Fulton-YoungTableux}
William Fulton.
\newblock {\em Young Tableaux: With Applications to Representation Theory and
  Geometry}.
\newblock London Mathematical Society Student Texts. Cambridge University
  Press, 1996.

\bibitem[Gab00]{Gab-Intro}
Matthias~R. Gaberdiel.
\newblock An introduction to conformal field theory.
\newblock {\em Reports on Progress in Physics}, 63, 2000.

\bibitem[Gab05]{Gab-VOA}
Matthias~R. Gaberdiel.
\newblock {2D conformal field theory and vertex operator algebras}.
\newblock {\em https://arxiv.org/abs/hep-th/9910156}, 2005.

\bibitem[Gar03]{Garcia}
Adriano Garcia.
\newblock Young seminormal representation, Murphy elements, and content
  evaluations.
\newblock {\em
  http://www.math.ucsd.edu/~garsia/somepapers/Youngseminormal.pdf}, 2003.

\bibitem[Gei76]{Geissinger}
Ladnor Geissinger.
\newblock Hopf algebras of symmetric functions and class functions.
\newblock {\em Combinatoire at repr\'{e}sentation du groupe sym\'{e}trique},
  579:168--181, 04 1976.

\bibitem[Hog17]{Ho-Idemp}
Matt Hogancamp.
\newblock Idempotents in triangulated monoidal categories.
\newblock {\em https://arxiv.org/abs/1703.01001}, 03 2017.

\bibitem[HSS16]{Sussan-Sav-Heisenberg}
David Hill and Joshua Sussan.
\newblock A categorification of twisted {H}eisenberg algebras.
\newblock {\em Adv. Math.}, 295:368--420, 2016.

\bibitem[Jin91a]{Jing-VOA-HL}
Naihuan Jing.
\newblock Vertex operators and Hall-Littlewood symmetric functions.
\newblock {\em Advances in Mathematics}, 87(2):226 -- 248, 1991.

\bibitem[Jin91b]{Jing-VOA-Spin}
Naihuan Jing.
\newblock Vertex operators, symmetric functions, and the spin group $\gamma_n$.
\newblock {\em Journal of Algebra}, 138(2):340 -- 398, 1991.

\bibitem[Kac90]{Kac1}
V.G. Kac.
\newblock {\em Infinite-dimensional Lie Algebras}.
\newblock Cambridge University Press, third edition, 1990.

\bibitem[Kac98]{VOAintro}
Victor Kac.
\newblock {\em Vertex algebras for beginners}, volume~10 of {\em University
  Lecture series}.
\newblock Amer. Math. Society, 2nd edition, 1998.

\bibitem[Kho10]{Kh-H}
Mikhail Khovanov.
\newblock Heisenberg algebra and a graphical calculus.
\newblock {\em Fundamenta Mathematicae}, 225, 09 2010.

\bibitem[Lau09]{Lauda}
Aaron Lauda.
\newblock unpublished notes.
\newblock 07 2009.

\bibitem[LS12]{LicSav-Survey}
Anthony Licata and Alistair Savage.
\newblock A survey of {H}eisenberg categorification via graphical calculus.
\newblock {\em Bull. Inst. Math. Acad. Sin. (N.S.)}, 7(2):291--321, 2012.

\bibitem[Mac95]{MacDonald}
I.G. Macdonald.
\newblock {\em Symmetric Functions and Hall Polynomials}.
\newblock Oxford Mathematical Monographs. The Clarendon Press Oxford University
  Press, second edition, 01 1995.

\bibitem[MJD00]{Miwa-Jimbo-Date}
T.~Miwa, M.~Jimbo, and E.~Date.
\newblock {\em Solitions}, volume 135 of {\em Cambridge Tracts in Mathematics}.
\newblock Cambridge University Press, 2000.

\bibitem[Rai14]{Raicu2014}
Claudiu Raicu.
\newblock Products of Young symmetrizers and ideals in the generic tensor
  algebra.
\newblock {\em Journal of Algebraic Combinatorics}, 39(2):247--270, Mar 2014.

\bibitem[RS17]{Savage-WreathProd}
Daniele Rosso and Alistair Savage.
\newblock A general approach to {H}eisenberg categorification via wreath
  product algebras.
\newblock {\em Math. Z.}, 286(1-2):603--655, 2017.

\bibitem[RZ16]{Rios-Zert}
Rodolfo R\'{i}os~Zertuche.
\newblock An introduction to the half-infinite wedge.
\newblock {\em Contemporary Mathematics}, 657:198--237, 2016.

\bibitem[Sav07]{Sav-BF}
Alistair Savage.
\newblock A geometric boson-fermion correspondence.
\newblock {\em https://arxiv.org/abs/math/0508438}, 07 2007.

\bibitem[Seg81]{Segal}
G~Segal.
\newblock Unitary representations of some infinite dimensional groups.
\newblock {\em Communications in Mathematical Physics}, 80:301--342, 1981.

\bibitem[Ste95]{Stern}
Eugene Stern.
\newblock Semi-infinite wedges and vertex operators.
\newblock {\em International Mathematics Research Notices}, pages 201--220, 05
  1995.

\bibitem[Tia17]{Tian-BF}
Yin Tian.
\newblock Towards a categorical boson-fermion correspondence.
\newblock {\em https://arxiv.org/abs/1710.11579}, 10 2017.

\bibitem[Tin11]{Tingley}
Peter Tingley.
\newblock Notes on Fock space.
\newblock {\em http://webpages.math.luc.edu/~ptingley/}, 2011.

\bibitem[tKL91]{deLeur-BosFerm}
Fons ten Kroode and J~Leur.
\newblock Bosonic and fermionic realizations of the affine algebra
  $\widehat{\rm gl}_n$.
\newblock {\em Communications in Mathematical Physics}, 137:67--107, 03 1991.

\bibitem[Yan18]{YANAGIDA201855}
Shintarou Yanagida.
\newblock Boson-fermion correspondence from factorization spaces.
\newblock {\em Journal of Geometry and Physics}, 124:55 -- 63, 2018.

\bibitem[Zel81]{Zelevinsky}
Andrey~V. Zelevinsky.
\newblock {\em Representations of Finite Classical Groups}, volume 869 of {\em
  Lecture Notes in Mathematics}.
\newblock Springer-Verlag, 1981.
\end{thebibliography}
\end{document}